\documentclass[11pt,a4paper,openright]{amsbook}
\usepackage[T1]{fontenc}
\usepackage{amssymb}
\usepackage{amsmath}
\usepackage{amsthm}
\usepackage{color}
\usepackage{graphicx}
\usepackage[usenames,dvipsnames]{xcolor}
\usepackage{shadethm}
\usepackage[margin=1.0in]{geometry}
\usepackage{enumerate}
\usepackage{dsfont}
\usepackage{hyperref}
\usepackage[titletoc]{appendix}
\usepackage{mathrsfs}

\usepackage{hyperref}
\hypersetup{colorlinks=true, citecolor=purple, linktocpage, linkcolor=blue, urlcolor=cyan}

\newtheorem{thm}{Theorem}[section]
\newtheorem{lem}[thm]{Lemma}
\newtheorem{prop}[thm]{Proposition}
\newtheorem{cor}[thm]{Corollary}
\newtheorem{exer}[thm]{Exercise}
\newtheorem{defn}[thm]{Definition}

\newtheorem{ex}[thm]{Example}
\newtheorem{rem}[thm]{Remark}

\numberwithin{equation}{section}

\newcommand{\R}{\mathbb{R}}
\newcommand{\A}{\mathcal{A}}
\newcommand{\B}{\mathcal{B}}
\newcommand{\N}{\mathbb{N}}
\newcommand{\C}{\mathbb{C}}
\newcommand{\Rbar}{\overline{\mathbb{R}}}

\newcommand{\Q}{\mathbb{Q}}
\newcommand{\E}{\mathbb{E}}
\newcommand{\p}{\mathbb{P}}
\newcommand{\one}{\mathds{1}}
\newcommand{\0}{\mathcal{O}}

\newcommand{\F}{\mathcal{F}}

\title[Lectures on Probability Theory]{Lectures on Probability Theory}
\author[N. Moshayedi]{Nima Moshayedi\vspace{20cm}}
\address{Institut f\"ur Mathematik\\ Universit\"at Z\"urich\\ 
Winterthurerstrasse 190
CH-8057 Z\"urich}
\email[N.~Moshayedi]{nima.moshayedi@math.uzh.ch}

\begin{document}

\maketitle

\newpage

\chapter*{Abstract}
These notes have been written for the courses \emph{Stochastics I} and \emph{Stochastics II} given in the spring semester 2014 and spring semester 2015 at the University of Zurich. 
The first part of the notes gives an introduction to probability theory. It explains the notion of random events and random variables, probability measures, expectation, distributions, characteristic function, independence of random variables, types of convergence and limit theorems. The first part is separated into two different chapters. The first chapter is about combinatorial aspects of probability theory and the second chapter is the actual introduction to probability theory, which contains the modern probability language. The second part covers conditional expectations, martingales and Markov chains, which are easily accessible after reading the first part. The chapters are exactly covered in this order and go into some more details of the respective topic. 

\vspace{0.5cm}

Nima Moshayedi, \today

\newpage
\tableofcontents
\newpage

\part{The Modern Probability Language}

\chapter*{Introduction}

We think of probability theory as a mathematical module for random events. Therefore probability theory is widely used in many different topics, for example:

\begin{itemize}

\item{Biology}

\item{Economics / Insurance / Stochastic finance theory}

\item{Physics (for example statistical mechanics / Quantum mechanics)}

\item{Mathematics (Random matrix theory / Number theory / Group theory / etc.)}

\end{itemize}

{\bf Random Experiments:} Random experiments are experiments whose output cannot be surely predicted in advance. The theory of probability aims towards a mathematical theory which describes such phenomena. There are four important objects for probability theory:

\vspace{0.5cm}

\begin{itemize}

\item{{\bf State space $\Omega$:} The state space represents all possible outcomes of the experiment.

\vspace{0.5cm}

\begin{ex} There are many different state spaces:
 
\begin{enumerate}[$(i)$]

\item{A toss of a coin: $\Omega=\{H,T\}$ (head or tail)}

\item{Two successive tosses of a coin: $\Omega=\{HH,HT,TT,TH\}$}

\item{A toss of two dice: $\Omega=\{(i,j)\mid1\leq i\leq 6,1\leq j\leq 6\}$}

\item{The lifetime of a bulb: $\Omega=\R_+$}
\end{enumerate}
\end{ex}

}

\vspace{0.2cm}
\item{{\bf Events:} An event is a property which can observed either to hold or not to hold after the experiment is done. In mathematical terms an event is a subset of $\Omega$. We shall denote by $\A$ the family of all events. If $A$ and $B$ are two events, then

\begin{enumerate}[$(i)$]

\item{The contrary event is interpreted as $A^C$.
}

\item{The event $A$ or $B$ is interpreted as $A\cup B$.
}
\item{The event $A$ and $B$ is interpreted as $A\cap B$.
}

\item{The sure event is $\Omega$
}

\item{The impossible event is $\emptyset$
}
\end{enumerate}

}

\vspace{0.3cm}

\item{{\bf Probability Measure:} With each event $A$, one associates a number by $\p[A]$ and call it the probability of $A$. The number measures the likelihood of the event $A$ to be realized a priori before performing the experiment and $\p[A]\in[0,1]$. Imagine that these numbers are frequencies. Let us repeat the same experiment $n$ times. Denote by $f_n(A)$ the frequency with which $A$ is realized (i.e. number of times the event $A$ occurs divided by $n$) and we write $\p[A]=\lim_{n\to\infty}f_n(A)$. With this interpretation, we have to expect that 

\begin{enumerate}[$(i)$]

\item{$\p[\Omega]=1$}
\item{$\p[A\cup B]=\p[A]+\p[B]$ if $A\cap B=\emptyset$}

\end{enumerate}

In terms of mathematical modeling we have 

$$\p:\A\to[0,1],\hspace{0.5cm}\A\hspace{0.2cm}\sigma\text{-Algebra},$$

which is a probability measure.

In this notes, a mathematical model for our random experiment is a triple

$$(\Omega,\A,\p)$$

with $\Omega,\A,\p$ as we have seen.

}

\item{{\bf Random Variable:} A  random variable (r.v.) is a quantity which depends on the outcome of the experiment. In mathematical terms, this is a map from the space $\Omega$ into some space $E$ (and $E$ is very often $\R$, $\R^d$,$\N$,$\mathbb{Z}$) and this map $X:\Omega\to E$ should be measurable. We always think of $E$ as a measure space that is $E$ is endowed with a $\sigma$-Algebra $\mathcal{E}$, and from measure theory we know that one can transport $\p$ on $E$: for all $B\in\mathcal{E},\hspace{0.3cm}\p^X$ (or $\p_X$) is a probability measure on $\mathcal{E}$ such that for all $B\in\mathcal{E}$, 

$$\p^X[B]=\p_X[B]=\p[X^{-1}(B)]=\p\left[\underbrace{\left\{\omega\in\Omega\mid X(\omega)\in B\right\}}_{\in \A}\right]$$

and $\p_X$ is called the law or the distribution of $X$.

\vspace{0.2cm}

\begin{ex} We got the following examples:

\vspace{0.3cm}
\begin{enumerate}[$(i)$]

\item{Toss of two dice:

$$\Omega=\{(i,j)\mid 1\leq i\leq 6,1\leq j\leq 6\}$$
$$\A=\mathcal{P}(\Omega)(= 2^{\Omega})$$
$$\forall \omega\in\Omega,\hspace{0.3cm}\p[\{\omega\}]=\frac{1}{36}$$
$$\forall A\subset\Omega,\hspace{0.3cm}\p[A]=\frac{\vert A\vert}{36}$$

The map $\Omega\to(\N,\mathcal{P}(\N))$, $(i,j)\mapsto i+j$ is a r.v. for all $B\subset \N$, 

$$\p_X[B]=\frac{\vert\{(i,j)\mid i+j\in B\}\vert}{36}.$$

For instance if $B=\{2\}$ then $\p_X[\{2\}]=\p[\{1,1\}]=\frac{1}{36}.$

To conclude the introduction, let us emphasize the fact that $\Omega$ can be more involved and the construction of $\p$ "completed".

}

\item{We throw a die until we get a 6. Here the choice of $\Omega$ is less obvious. The number of time, we have to throw the die, is not bounded. A natural choice for $\Omega$ would be 

$$\Omega=\{1,2,3,4,5,6\}^{\N^{\times}},$$

$\omega\in\Omega$ that is $\omega=(\omega_1,...,\omega_n,...)$ with $\omega_j=\{1,2,3,4,5,6\}$ for all $j\geq 1$. $\A$ will be the product $\sigma$-Algebra i.e. the smallest $\sigma$-Algebra containing all sets of the form $\{\omega\mid\omega_1=i_1,...,\omega_n=i_n\}$, $n\geq 1$, $i_1,...,i_n\in\{1,2,3,4,5,6\}$ and $\p$ is the unique probability measure on $\Omega$.

\item{We are interested in the motion of a particle in the space, and this particle is subject to some random perturbations. If the time interval is $[0,1]$ a natural space of outcomes can be the space of continuous functions from $[0,1]$ into $\R^d$ and we set 

$$\Omega=\mathcal{C}([0,1],\R^3),$$ 

$\omega\in\Omega$, is a possible trajectory of the particle: $\omega:[0,1]\to\R^3$. Take $\A$ to be the Borel $\sigma$-Algebra when we consider the supremumnorm on $\Omega$ $\left(\|\omega\|_\infty=\sup_{t\in[0,1]}\vert\omega(t)\vert \right)$. A famous example for $\p$ is the so called \emph{Wiener} measure (denoted by $\mathbb{W}$). Under this measure, a typical trajectory is a \emph{Brownian motion}.
}
}
\end{enumerate}
\end{ex}
}
\end{itemize}

\chapter{Elements of Combinatorial Analysis and Simple Random Walks}
In this chapter we are going to give an introduction to combinatorial aspects of probability theory and describe several simple methods for the computation of different problems, where the usage of these is an important aspect. We will also already give some of the most basic examples of discrete distributions and consider combinatorial aspects of random events such as a first look at a simple random walk. We will look at the reflection principle and basic terminology of random walks as a combinatorial theory. Finally, we will consider probabilities of events for special random walks.

\section{Summability conditions}
Let $I$ be an index set, and  $(a_i)_{i\in I}\subset \Rbar_+^I$ (i.e. $a_i\in[0,\infty]$, for all $ i\in I$). We will write
$$\sum_{i\in I}a_i:=\sup_{J\subset I\atop J\hspace{0.1cm}\text{finite}}\sum_{j\in J}a_j.$$
The family $(a_i)_{i\in I}$ is said to be summable if $\sum_{i\in I}a_i$ is finite. The following elementary properties are important for such a family $(a_i)_{i\in I}$.
\begin{enumerate}[$(i)$]
\item{For $H\subseteq I$ we have 
$$\sum_{h\in H}a_h\leq \sum_{i\in I}a_i$$
and $(a_h)_{h\in H}$ is summable if $(a_i)_{i\in I}$ is summable.
}
\item{For a family $(a_i)_{i\in I}$ to be summable, it is necessary and sufficient that the Cauchy criterion is satisfied, i.e. for all $\epsilon>0$, there is a finite set $J\subset I$, with all finite subsets $K\subset I$ satisfying $K\cap J=\varnothing$, we get 
$$\sum_{i\in K}a_i<\epsilon.$$
}
\item{If the family $(a_i)_{i\in I}$ is summable, the set $\{i\in I\mid a_i\not=0\}$ is at most countable, as the union of the sets $I_n=\left\{i\in I\mid a_i\geq \frac{1}{n}\right\}$. Moreover, $I_n$ has at most $\left[n\left(\sum_{i\in I}a_i\right)\right]$ elements, where $[x]$ denotes the integer part of $x$.
}
\end{enumerate}
To avoid confusion we want to emphasize that If $A_1,...,A_n$ are disjoint sets, we will sometimes write $\sum_{n\geq 1}A_n$ instead of $\bigsqcup_{n\geq 1}A_n$.
\begin{lem}
Let $(a_i)_{i\in I}\subset\bar\R_+^I$ be a family of numbers in $\bar\R_+$. If $I=\sum_{\lambda\in L}H_{\lambda}$ and $q_\lambda=\sum_{i\in H_{\lambda}}a_i$, then 
$$\sum_{i\in I}a_i=\sum_{\lambda\in L}q_\lambda,$$
where $L$ is an Index set and $(H_\lambda)_{\lambda\in L}$ is a partition of $I$.
\end{lem}
\begin{proof}
Exercise.
\end{proof}
In the case $I=\N$, we say that $\sum_{n\in\N}u_n$ converges absolutely if $\sum_{n\in\N}\vert u_n\vert$ converges. In this case we can modify the order in which the terms are taken without changing the convergence nor the sum of the series.
\begin{thm}[Stirling]
For all $n\in\N\setminus\{0\}$ we have 
$$n!=\kappa n^{n+\frac{1}{2}}\exp\left(-n+\frac{\theta(n)}{12n}\right),$$
where $1-\frac{1}{12n+1}\leq \theta(n)\leq 1$ and $\kappa=\int_{-\infty}^\infty e^{-\frac{1}{2}w^2}dw=\sqrt{2\pi}$.
\end{thm}
\section{Finite Probability Spaces}
\subsection{General Probability Spaces}
\begin{defn}[Probability measure]
Let $(\Omega,\A)$ be a measurable space. A probability measure defined on $\A$ is an application 
$$\p:\A\to[0,1],$$ 
such that 
\begin{enumerate}[$(i)$]
\item{$\p[\Omega]=1$.}
\item{for all $(A_n)_{n\in \N}$, with $A_n\cap A_m=\varnothing$ for $n\not=m$, we have 
$$\p\left[\bigcup_{n\in\N}A_n\right]=\sum_{n\in \N}\p[A_n].$$
}
\end{enumerate}
\end{defn}
\begin{rem}
We want to recall the following properties without a proof.
\begin{itemize}
\item{For $A,B\in\A$ with $A\subset B$ we get $\p[A]\leq \p[b]$.
}
\item{For $A\in\A$ we get $\p[A^C]=1-\p[A]$.
}
\item{If $(A_n)_{n\in\N}\subset\A$ and $A_n\uparrow A$ as $n\to\infty$, then $\p[A_n]\uparrow \p[A]$ as $n\to\infty$.
}
\item{If $(A_n)_{n\in\N}\subset\A$ and $A_n\downarrow A$ as $n\to\infty$, then $\p[A_n]\downarrow \p[A]$ as $n\to\infty$.
}
\end{itemize}
\end{rem}
\begin{thm}
Let $\mu$ be a measure on a measureable space $(\Omega,\A)$. 
\begin{enumerate}[$(i)$]
\item{If $(A_n)_{n\in\N}\subset\A$ is an increasing family of measurable sets, then
$$\mu\left( \bigcup_{n\in\N}A_n \right)=\lim_{n\to\infty}\mu(A_n).$$
Conversely, if an additive function satisfies the property above, then it is a measure.
}
\item{Let $\mu$ be an additive, positive and bounded function on $\A$. Then $\sigma$-additivity is equivalent to say that for all decreasing sequences $(B_n)_{n\in\N}\subset\A$, we get the implication
$$\bigcap_{n\in\N}B_n=\varnothing\Longrightarrow \lim_{n\to\infty}\mu(B_n)=0.$$
}
\end{enumerate}
\end{thm}
\subsection{Probability measures on finite spaces}
There are two cases where $(\Omega,\A)$ is a measurable space, either $\Omega$ is finite or $\A$ is finite. An important question is how to define probability measures on finite or countable spaces. If $I$ is at most countable, then having a probability measure on $\mathcal{P}(I)$ is equivalent to having a family $(a_i)_{i\in I}$ of positive real numbers such that $\sum_{i\in I}a_i=1$. Moreover, for $A\subset I$, we can define a measure $\mu$ by 
$$\mu(A)=\sum_{i\in A}a_i,$$
where $\mu(\{i\})=a_i$ for all $i\in I$.
\begin{ex}
We got the following examples:
\subsubsection{The Poisson distribution} The Poisson distribution with parameter $\lambda>0$, is a probability measure $\Pi_\lambda$ on $(\N,\mathcal{P}(\N))$, defined for all $n\geq0$ by 
$$\Pi_{\lambda}(\{ n\})=e^{-\lambda}\frac{\lambda^n}{n!}.$$
For $A\subset \N$ we thus get 
$$\Pi_\lambda(A)=\sum_{n\in A}e^{-\lambda}\frac{\lambda^n}{n!}.$$
In order to check that it is indeed a probability distribution, we need to check that $\sum_{i\geq 0}\Pi_\lambda(\{i\})=1$. indeed, we have 
$$\sum_{i\geq 0}e^{-\lambda}\frac{\lambda^i}{i!}=e^{-\lambda}\sum_{i\geq 0}\frac{\lambda^i}{i!}=1.$$
\subsubsection{The Riemann $\zeta$-function} Another example can be obtained by considering the function $\zeta(s)=\sum_{n\geq 1}n^{-s}$ for $s>1$, which we can deform to a probability measure. Indeed, if we set $a_n=\frac{1}{\zeta(s)}\cdot \frac{1}{n^s}$ for all $n\geq 1$, we get 
$$\sum_{n\geq 1}a_n=\frac{1}{\zeta(s)}\sum_{n\geq 1}n^{-s}=\frac{1}{\zeta(s)}\cdot \zeta(s)=1.$$
\subsubsection{The Geometric distribution} Similarly one can define the geometric distribution with parameter $r\in(0,1)$ on $(\N\setminus\{0\},\mathcal{P}(\N\setminus\{0\}))$ by 
$$\gamma_r(\{n\})=(1-r)r^{n-1}.$$
\subsubsection{The Uniform distribution} For a finite state space $\Omega$, we can define the uniform distribution for any measurable set $A\in\A$ by 
$$\p[A]=\frac{\vert A\vert}{\vert\Omega\vert}.$$
For $\omega\in\Omega$ we get $\p[\{\omega\}]=\frac{1}{\vert\Omega\vert}.$ Take for example $\Omega=\{1,2,...,N\}$. Then $\p[\{k\}]=\frac{1}{N}$. 
\begin{rem}
There is no uniform measure on $\N$, since $\N$ is not finite. Thus from now on we shall only consider the case of a finite state space $\Omega$ or a finite $\sigma$-Algebra $\A$.
\end{rem}
\end{ex}
\begin{thm}
Let $\A$ be a finite $\sigma$-Algebra on a state space $\Omega$. Then there exists a partition $A_1,...,A_n$ of $\Omega$, such that $A_k\in\A$ for all $k\in\{1,...,n\}$ and with the property that any $A\in \A$ can be written as a union of the events $A_1,...,A_n$. The sets are called the atoms of $\A$.
\end{thm}
\begin{proof}
Fix $\omega\in\Omega$. Define the set $C(\omega):=\bigcap_{B\in\A,\atop\omega\in B}B$. Then by finite intersection it follows that $C(\omega)\in\A$. In fact $C(\omega)$ is the smallest element of $\A$, which contains $\omega$. Moreover, it is easy to see that $\omega\sim\omega'$ if and only if $C(\omega)=C(\omega')$ defines an equivalence relation on $\Omega$ and that $\omega'\in C(\omega)$ if and only if $C(\omega)=C(\omega')$. Hence if $A$ is an equivalence class and $\omega\in A$, then $A=C(\omega)$. Therefore, one can take the events $A_1,...,A_n$ as the equivalence classes.
\end{proof}
\section{Basics of combinatorial Analysis}
\subsection{Sampling}
Let us consider a population of size $n\in\N$, i.e. a set $S=\{S_1,...,S_n\}$ with $n$ elements. We call any ordered sequence $(S_{i_1},...,S_{i_r})$ of $r$ elements of $S$ a sample of size $r$, drawn from this population. Then there are two possible procedures.
\begin{enumerate}[$(i)$]
\item{\emph{Sampling with replacement}, i.e. the same element can be drawn more than once.
}
\item{\emph{Sampling without replacement}. Here an element one choses is removed from the population. In this case $r\leq n$.
}
\end{enumerate}
In either case, our experiment is described by a sample space $\Omega$ in which each individual points represents a sample of size $r$. In the first case, we get $\vert\Omega\vert=n^r$, where an element $\omega\in\Omega$ is of the form $\omega=(S_{i_1},...,S_{i_r})$. In the second case we get 
$$\vert\Omega\vert=n(n-1)\dotsm(n-r+1)=\frac{n!}{(n-r)!}.$$
\begin{ex}
We got the following examples:
\begin{enumerate}[$(i)$]
\item{If by a "ten letter word" is meant a (possibly meaningless) sequence of ten letters, then such a word represents a sample from the population of 26 letters. There are $26^{10}$ such words.
}
\item{Tossing a coin $r$ times is one way of obtaining a sample of size $r$ drawn from the population of the letters $H$ and $T$. Usually we assign equal probabilities to all samples, namely $n^{-r}$ is sampling with replacement and $\frac{1}{n(n-1)...(n-r+1)}$ is a sampling without replacement.
}
\item{A random sample of size $r$ with replacement is taken from a population of size $n$, and we assume that $r\leq n$. The probability $p$ that in the sample no element appears twice is
$$p=\frac{n(n-1)\dotsm (n-r+1)}{n^r}\hspace{0.2cm}\left(\p[A]=\frac{\vert A\vert}{\vert\Omega\vert}\right).$$
Consequently, the probability of having four different numbers chosen from $\{0,1,2,3,...,9\}$ is (here $n=10$, $r=4$)
$$\frac{10\times 9\times 8\times 7}{10^4}\approx 0.5.$$
In sampling without replacement the probability for any fixed element of the population to be included in a random sample of size $r$ is 
$$1-\frac{(n-1)(n-2)\dotsm (n-r)}{\underbrace{n(n-1)\dotsm (n-r+1)}_{\text{probability of the complementary event}}}=1-\frac{n-r}{n}=\frac{r}{n}.$$ 
We have used that for $A\in\A$, $\p[A]=1-\p[A^C]$. Similarly, in sampling with replacement the probability that a fixed element of the population to be included in a random sample of size $r$ is 
$$1-\left(\frac{n-1}{n}\right)^r=1-\left(1-\frac{1}{n}\right)^2$$
}
\end{enumerate}
\end{ex}
\subsection{Subpopulations}
Let $S=\{S_1,...,S_n\}$ be a population of size $n$. We call subpopulations of size $r$ any set of $r$ elements, distinct or not, chosen from $S$. Elements of a subpopulation are not ordered, i.e. two samples of size $r$ from $S$ correspond to the same subpopulation if they only differ by the order of the elements. Here we have two cases as well. The case $\underline{without}$ replacement: We must have $r\leq n$. There are $\binom{n}{r}=\frac{n!}{r!(n-r)!}$. We have that $\binom{n}{r}=\binom{n}{n-r}$. We always use the convention $0!=1$.
\begin{ex} We got the following examples:
\subsubsection{Poker} There are $\binom{52}{5}=2'598'960$ hands at poker. Let us calculate the probability $p$ that a hand at poker contains 5 different face values. The face values can be chosen in $\binom{13}{5}$ different ways, and corresponding to each case we can choose one of the four suits. It follows that 
$$p=\frac{4^5\cdot\binom{13}{5}}{\binom{52}{5}}\approx 0.5.$$
\subsubsection{Senator problem} Each of the 50 states has two senators. We consider the event that in a committee of 50 senators chose at random
\begin{itemize}
\item{a given state is represented.}
\item{all states are represented.}
\end{itemize}
In the first case, it is better to calculate the probability $q$ of the complementary event that the given state is not represented
$$q=\frac{\binom{98}{50}}{\binom{100}{50}}=\frac{50\times 49}{100\times 99}.$$
For the second case we note that a comittee including a senator of all states can be chosen in $2^{50}$ different ways. The probability that all states are included is 
$$\frac{2^{50}}{\binom{100}{50}}\approx 4.126\cdot 10^{-14}.$$
\subsubsection{An occupancy problem} Consider a random distribution of $r$ balls in $n$ cells. To find the probability $p_k$ that a specified cell contains exactly $k$ balls $(k=0,1,...,r)$. We note that $k$ balls can be chosen in $\binom{r}{k}$ different ways and the remaining $(r-k)$ balls can be placed into the remaining $(n-1)$ cells in $(n-1)^{r-k}$ so we get 
$$p_k=\binom{r}{k}(n-1)^{r-k}\frac{1}{n^r}\hspace{0.2cm}\left(\p[A]=\frac{\vert A\vert}{\vert \Omega\vert}\right)$$
We call $\tilde p=\frac{1}{n}$, then 
$$p_k=\binom{r}{k}\tilde p^k(1-\tilde p)^{r-k}.$$
This is called the binomial distribution of parameters $p$ and $r$. It is a probability distribution on $\{0,1,...,r\}$ and 
$$\sum_{k=0}^rp_k=\sum_{k=0}^r\binom{r}{k}\tilde p^k(1-\tilde p)^{r-k}=(\tilde p+1-\tilde p)^r=1.$$
For example in a coin tossing game we have $r$ tosses and we set $H\leftrightarrow 1$ and $T\leftrightarrow 0$. We have the state space 
$$\Omega=\{(\omega_j)_{1\leq j\leq r}\},\hspace{0.2cm}\omega_j\in\{0,1\}.$$
Moreover, we have that 
$$S_r=\sum_{j=1}^r\omega_j,$$
where $S_r$ is the number of heads. Assume that $\p(H)=p\in[0,1]$, then 
$$\p[S_r=k]=\binom{r}{k}p^k(1-p)^{r-k}.$$
\begin{thm}
Let $r_1,...,r_k$ be integers such that $r_1+...+r_k=n$ $(r_i\geq 0,r_i\in\N)$. The number of ways in which a population of $n$ elements can be partitioned into $k$ subpopulations of which the first contains $r_1$ elements, the second contains $r_2$ elements etc., is given by 
$$\frac{n!}{r_1!r_2!\dotsm r_k!}$$
\end{thm}
\begin{proof}
For the first part we note that 
$$\frac{n!}{r_1!\dotsm r_k!}=\binom{n}{r_1}\binom{n-r_1}{r_2}\binom{n-r_1-r_2}{r_3}\dotsm\binom{n-r_1-r_2-...-r_{k-2}}{r_k-1}.$$
For the second part we see that an induction on $n$ shows that 
\begin{equation}
(U_1+...+U_k)^n=\sum_{\substack{r_1\geq 0,...,r_k\geq 0\\ r_1+...+r_k=n}}\frac{n!}{r_1!\dotsm r_k!}U_1^{r_1}\dotsm U_k^{r_k}.
\end{equation}
But the left hand side is also 
$$\sum_{1\leq k_1,...,k_n\leq k}U_{k_1}\dotsm U_{k_n}.$$
Now ordering the terms and compering with $(1)$ yields the result.
\end{proof}
\end{ex}
\begin{ex}
A throw of twelve dice can result in $6^{12}$ different outcomes. The event that each face can appear twice is 
$$\frac{12!}{2^6\cdot 6^{12}}.$$
\end{ex}
The case $\underline{with}$ replacement: We take the point of view of occupancy problems. Our model is that of placing randomly $r$ balls into $n$ cells. Such an event is completely described by its occupancy numbers $r_1,...,r_n$ where $r_k$ stands for the number of balls in the $k$-th cell. Every $n$-tupel of integers satisfying $r_1+...+r_k=r$ describes a possible configuration. Two distributions are distinguishable if the occupancy numbers are different. We denote by $A_{r,n}$ the number of distinguishable distributions. It is also the number of subpopulations of size $r$ with replacement from a population $S=\{S_1,...,S_n\}$ of size $n$ and it is characterized by the number $r_k$ of appearances of the individual $S_k$ with the restriction $r_1+...+r_n=r$.
\begin{thm}
$$A_{r,n}=\binom{n+r-1}{r}=\binom{n+r-1}{n-r}.$$
\end{thm}
\begin{proof} There are two different proofs.
\begin{enumerate}[$(i)$]
\item{We represent the balls by  $\bigotimes$ and the cells by the $n$ spaces between $n+1$ bars. Then 
$$\mid \bigotimes \bigotimes \bigotimes \mid\bigotimes\mid\hspace{0.3cm}\mid\hspace{0.3cm}\mid\hspace{0.3cm}\mid \bigotimes \bigotimes \bigotimes \bigotimes\mid$$
is used as a symbol for a distribution of $r=8$ balls in $n=6$ cells. The occupancy numbers are 
$$3,1,0,0,0,4\Longrightarrow 3+1+0+0+0+4=8.$$
Such a symbol necessarily starts and ends with a bar, but the remaining $(n-1)$ bars and $r$ balls can appear in an arbitrary order. We have that 
$$\binom{n+r-1}{r}$$
is the number of ways of selecting $r$ places (for the balls) out of $n+r-1$.
}
\item{For $|t_i|<1$ with $1\leq i\leq n$
\begin{align*}
\prod_{1\leq i\leq n}\frac{1}{1-t_i}=\prod_{1\leq i\leq n}\left(\sum_{m_i\geq 0}t_i^{m_i}\right)&=\sum_{0\leq m_1,...,m_n}t_1^{m_1}\cdot t_2^{m_2}\dotsm t_n^{m_n}\\
&=\sum_{r\geq 0}\sum_{m_1+...+m_n=r\atop m_1\geq 0,...,m_n\geq 0}t_1^{m_1}\cdot t_2^{m_2}\dotsm t_n^{m_n}
\end{align*}
For $t_1=t_2=\dotsm=t_n=t$ we have 
$$\frac{1}{(1-t)^n}=\sum_{r\geq 0}\sum_{m_1+...+m_n=r\atop m_1\geq 0,...,m_n\geq 0}t^r=\sum_{r\geq 0}A_{r,n}\cdot t^r.$$
Taylor's formula gives us that 
$$\frac{1}{(1-r)^n}=\sum_{r\geq 0}\binom{n+r-1}{r}t^n\Longrightarrow A_{r,n}=\binom{n+r-1}{r}.$$
}
\end{enumerate}
\end{proof}
\begin{ex}The partial derivatives of order $r$ of a $C^{\infty}$ function $f(x_1,...,x_n)$ of $n$ variables do not depend on the order of differentiation but only on the number of times that each variable appears. Hence there exists $\binom{n+r-1}{r}$ different partial derivatives of order $r$.
\end{ex}
\begin{rem}
$A_{r,n}$ is also the number of different integer solutions to the equation $r_1+...+r_n=r$.
\end{rem}
\subsection{Combination of events}
We write a permutation $\sigma\in\Sigma_N$ in the following way
$$\sigma=\begin{pmatrix}1&2&\dotsm &N\\\sigma(1)&\sigma(2)&\dotsm&\sigma(N)\end{pmatrix}.$$
We know that the permutation group has cardinality $\vert\Sigma_N\vert=N!$. The probability of a permutation is given by 
$$\p[\sigma]=\frac{1}{N!}.$$
We go back to a probability space $(\Omega,\A,\p)$. If $A_1,A_2\in\A$ with $A=A_1\cup A_2$, we get
$$\p[A_1\cup A_2]=\p[A_1]+\p[A_2]-\p[A_1\cap A_2].$$
What about $\p[A]$ when $A=\bigcup_{i=1}^NA_i$, where $A_1,...,A_N\subset \A$? We shall note $\p_i=\p[A_i]$, $\forall 1\leq i\leq N$ and $\p_{ij}=\p[A_i\cap A_j]$ for $i<j$. For $i_1<...<i_k$ we have that 
$$\p_{i_1,...,i_k}=\p\left[\bigcap_{j=1}^kA_{i_j}\right].$$
We note that
\begin{align*}
S_1&=\sum_{1\leq i\leq N}\p_i\\
S_2&=\sum_{1\leq i<j\leq N}\p_{ij}\\
&\vdots\\
S_r&=\sum_{1\leq i_1<...<i_r\leq N}\p_{i_1,...,i_r}
\end{align*}
Moreover, $S_r$ has $\binom{N}{r}$ terms. For $N=2$ we get 
$$\p[A]=S_1-S_2.$$
\begin{thm}[Inclusion-Exclusion]
$$\p\left[\bigcup_{i=1}^nA_i\right]=S_1-S_2+S_3-S_4+...\pm S_N\Longrightarrow \p\left[\bigcup_{i=1}^NA_i\right]=\sum_{r=1}^N(-1)^{r-1}S_r$$
\end{thm}
\begin{ex} Two equivalent decks of $N$ cards. Each are put into random order and matched against each other. If a card occupies the same place in both decks we speack of a match. We want to compute the probability of having at least one match. Let us number the cards $1,...,N$ with 
$$\sigma=\begin{pmatrix}1&2&\dotsm &N\\\sigma(1)&\sigma(2)&\dotsm&\sigma(N)\end{pmatrix},$$
where the first line denotes the first deck and the second line denotes the second deck. A match for $k$ corresponds to $k=\sigma(k)$. In that case we call $k$ a fix point of the permutation $\sigma$. We look for the number of permutations of $\{1,...,N\}$ (out of the $N$) which have at least 1 fixed point. Let $A_k$ be the event $k=\sigma(k)$. Clearly 
$$\p[A_k]=\p_k=\frac{(N-1)!}{N!}=\frac{1}{N}.$$
similarly if $i<j$, 
$$\p_{ij}=\p[A_i\cap A_j]=\frac{(N-2)!}{N!}=\frac{1}{N(N-1)}.$$
More generally 
$$\p_{i_1,...,i_r}=\frac{(N-r)!}{N!}$$
and hence 
$$S_r=\sum_{1\leq i_1<...<i_r\leq N}\p_{i_1,...,i_r}=\sum_{1\leq i_1<...<i_r\leq N}\frac{(N-r)!}{N!}=\frac{(N-r)!}{N!}\sum_{1\leq i_1<...<i_r\leq N}1=\frac{(N-r)!}{N!}\binom{N}{r}=\frac{1}{r!}.$$
If $\p_1$ is the probability of the least fixed point, then 
$$\p_1=\p\left[\bigcup_{i=1}^NA_i\right]=1-\frac{1}{2!}+\frac{1}{3!}-...\pm\frac{1}{N!}.$$
\end{ex}
\section{Random Walks}
\subsection{The Reflection Principle}
From a formal point of view, we shall be concerned with arrangements of finitely many $+1$ and $-1$. Consider $n=p+q$ symbols $\epsilon_1,...,\epsilon_n$, where $\epsilon_j\in\{-1,+1\}$ for all $1\leq  j\leq  n$. Suppose that there are $p$ $\{+1\}$'s and $q$ $\{-1\}$'s. Then $S_k=\epsilon_1+...+\epsilon_k$ represents the difference between the number of $\{+1\}$'s and $\{-1\}$'s at the first $k$ places. 
\begin{equation}
S_k-S_{k-1}=\epsilon_k=\pm1\hspace{0.5cm}S_0=0,\hspace{0.1cm}S_n=p-q
\end{equation}
\begin{figure}
\begin{center}
\includegraphics[height=6cm, width=10cm]{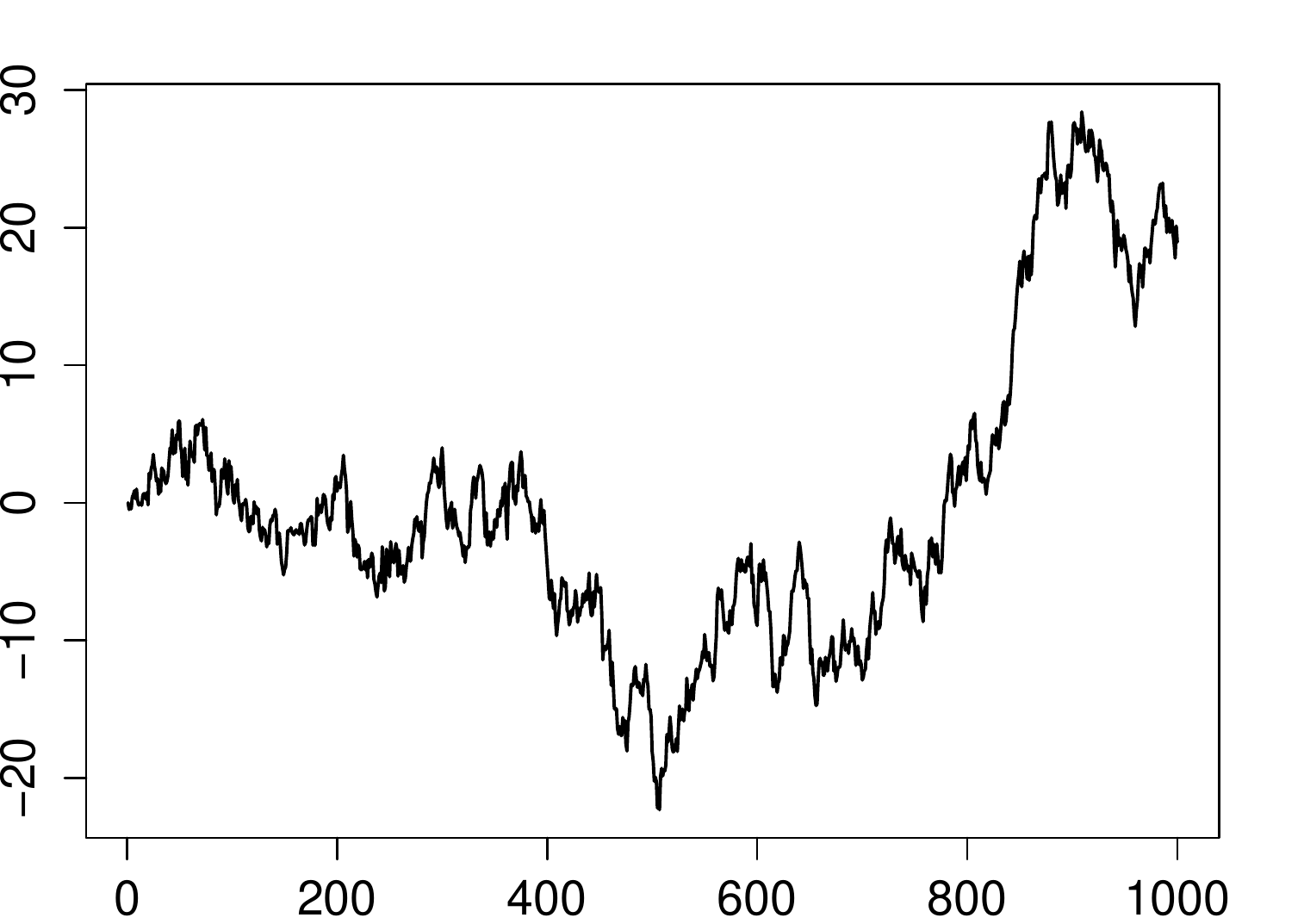}
\end{center}
\caption{Example for a random walk}
\end{figure}
The arrangement $(\epsilon_1,...,\epsilon_n)$ will be represented by a polygonal line whose $k$'th side has slope $\epsilon_k$, and whose $k$'th vertex has ordinate $S_k$. Such lines will be called a path. We shall use $(t,x)$ for coordinates.
\begin{defn}
Let $n>0$ and $x$ be integers. A path $(S_0,S_1,...,S_n)$ from the origin to the point $(n,x)$ is a polygonal line whose vertices have abscissas $0,1,2,...,n$ and ordinates $S_0,S_1,...,S_n$ satisfying (10) with $S_n=x$. We shall refer to $n$ as the length of the path. There are $2^n$ paths of length $n$. Moreover, we have 
\begin{align*}
n&=p+q\\
x&=p-q
\end{align*}
\end{defn}
A path from the origin to an arbitrary point $(n,x)$ exists only if $n$ and $x$ are of the form as in the definition. In this case the $p$ places for the positive $\epsilon_k$'s can be chosen from the $n=p+q$ available places in 
$$N_{n,x}=\binom{p+q}{p}=\binom{p+q}{q}$$
different ways.
\vspace{0.5cm}
{$Convention:$} $N_{n,x}=0$ whenever $n$ and $x$ are not of the form as in the definition. this implies that, $N_{n,x}$ represents the number of different paths from the origin to an arbitrary point $(n,x)$.
\vspace{0.5cm}
\begin{ex}[Ballot theorem]
Suppose that in a ballot candidate $P$ scores $p$ votes and candidate $Q$ scores $q$ votes, where $p>q$. the probability that throughout the counting there are always more votes for $P$ than for $Q$ equals $\frac{p-q}{p+q}$. The whole voting record may be represented by a path of length $p+q$ in which $\epsilon_k=+1$, if the $k$'th vote is for $P$ and $\epsilon_k=-1$ otherwise. Conversely every path from the origin to the point $(p+q,p-q)$ can be interpreted as a voting with the given totals $p$ and $q$. $S_k$ is the number of votes by which $P$ leads just after the $k$'th vote. The candidate $P$ leads throughout the voting if $S_1>0,S_2>0,...,S_n>0$ (in the ballot theorem, it is implicitly assumed that all paths are equally probable).
\end{ex}
\begin{figure}
\begin{center}
\includegraphics[height=8cm, width=12cm]{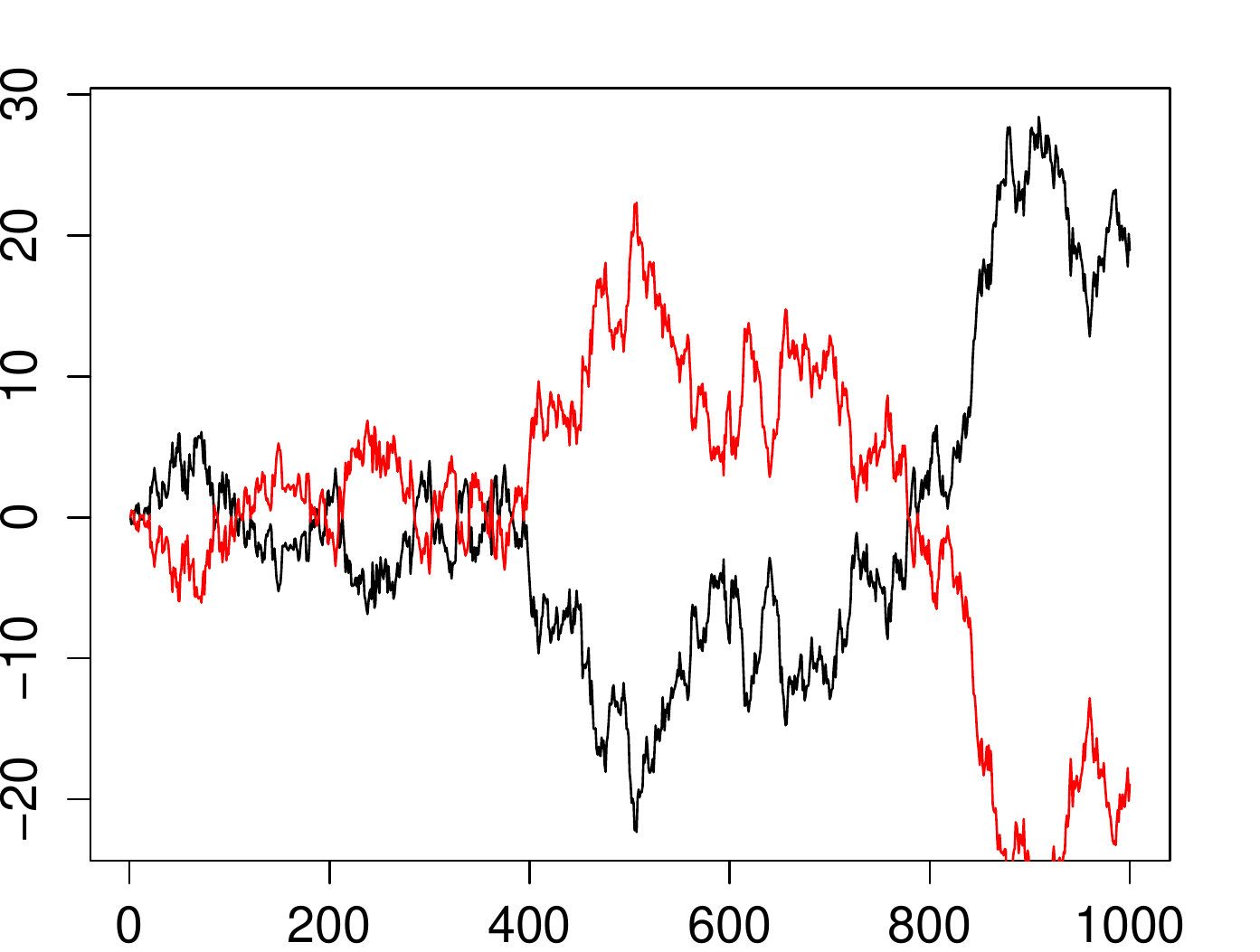}
\end{center}
\caption{Example for the reflection principle}
\end{figure}
Let $A=(a,\alpha)$ and $B=(b,\beta)$ be two independent points with $b>a\geq 0$, $a,b\in\N$, $\alpha>0$, $\beta>0$, $\alpha,\beta\in\N$. By reflection of $A$ on the $t$-axis we mean the point $A'=(a,-\alpha)$.
\begin{lem}[The reflection principle]
The number of paths from $A$ to $B$ which touch or cross the $t$-axis equals the number of paths from $A'$ to $B$.
\end{lem}
\begin{proof}
Consider a path $(S_a=\alpha,S_{a+1},...,S_b=B)$ from $A$ to $B$ having one or more vertices on that axis. Let $t$ be the abscissa of the first such vertex that is $S_a>0,S_{a+1}>0,...,S_{k-1}>0,S_k>0$. Then $(-S_a,-S_{a+1},...,-S_{k-1},S_k,S_{k+1},...,S_b)$ is a path leading from $A'$ to $B$ and having $T=(t,0)$ as its vertex on the $t$-axis. This gives a one to one correspondence between all paths from $A'$ to $B$ and paths from $A$ to $B$ that have a vertex on the $t$-axis.
\end{proof}
Let us now prove the ballot theorem. Let $n$ and $x$ be positive integers. There are exactly $\frac{x}{n}N_{n,x}$ paths $(S_1,...,S_n=x)$ such that 
$$S_1>0,S_2>0,...,S_n>0.$$
Indeed clearly there exists exactly as many admissible paths as there are paths from $(1,1)$ to $(n,x)$, which neither or cross the $t$-axis. From the previous lemma the number of such paths equals 
$$N_{n-1,x-1}-N_{n-1,x+1}=\binom{p+q-1}{p-1}-\binom{p+q-1}{p},$$
where we used that $N_{n,x}=\binom{p+q}{p}$. So we get
\begin{align*}
\binom{p+q-1}{p-1}-\binom{p+q-1}{p}&=\frac{(p+q-1)!}{(p-1)!q!}-\frac{(p+q-1)!}{p!(q-1)!}\\
&=\frac{p}{p+q}\cdot\frac{(p+q)!}{p!q!}-\frac{q}{p+q}\cdot \frac{(p+q)!}{p!q!}\\
&=\frac{p-q}{p+q}N_{n,x}=\frac{x}{n}N_{n,x}.
\end{align*}
\subsection{Random Walk terminology}
We set $S_0=0$, $S_n=X_1+...+X_n$, $X_j\in\{-1,+1\}$. Our state space would be $\Omega_n$, containing all possible paths. Therefore we get that $\vert\Omega\vert=2^n$. We set our $\sigma$-Algebra $\A=\mathcal{P}(\Omega_n)$. We set $\p$ to be our uniform probability measure. Consider the event $\{S_n=r\}$ (at time $n$ the particle is at the point $r$). We shall also speak about a visit to $r$ at time $n$. the number $N_{n,x}$ of paths from the origin to $N_{n,r}$ is given by 
$$\binom{n}{\frac{n+r}{2}},$$ 
with $n=p+q$, $r=p-q$ and hence $p=\frac{n+r}{2}$. Here we interpret $\binom{n}{\frac{n+r}{2}}$ as 0 if $n+r$ is not an even integer between 0 and $n$. Hence we get that
$$\p_{n,r}=\p[S_n=r]=\binom{n}{\frac{n+r}{2}}\cdot 2^n.$$
A return to the origin occurs at time $k$ if $S_k=0$. Here $k$ is necessarily an even integer, that we note $k=2\nu$ with $\nu$ an integer. The probability of return to the origin is 
$$\p_{2\nu,0}.$$
We shall denote it by $U_{2\nu}$, which is now given by
$$\boxed{U_{2\nu}=\p[S_{2\nu}=0]=\binom{2\nu}{\nu}\cdot 2^{-2\nu}}$$
Stirling's formula implies that $U_{2\nu}\sim\frac{1}{\sqrt{\pi r}}$. Among the returns to the origin, the first return receives special attention. A first return occurs at $2\nu$ if 
$$S_1\not=0,S_2\not=0,...,S_{2\nu-1}\not=0,S_{2\nu}=0.$$
We denote the probability of this event by $f_{2\nu}$.
\begin{lem}
\label{lem3}
For all $n\geq 0$ we have
$$\p[S_1\not=0,S_2\not=0,...,S_{2n-1}\not=0,S_{2n}\not=0]=\p[S_{2n}=0]=U_{2n}.$$
\end{lem}
\begin{rem}
When the event on the left hand side occurs, either all $S_j>0$ or all $S_j<0$. Since these events are equally probable it follows that 
$$\p[S_1>0,...,S_{2n}>0]=\frac{1}{2}U_{2n}.$$
\end{rem}
\begin{proof}[Proof of Lemma \ref{lem3}] We have 
$$\p[S_1>0,...,S_{2n}>0]=\sum_{r=1}^\infty\p[S_1>0,...,S_{2n}=2r],$$
where all the terms with $r>n$ are zero. By the ballot theorem, the number of paths $[S_1>0,...,S_{2n-1}>0,S_{2n}=2r]$ is equal to $N_{2n-1,2r+1}$ and thus
\begin{align*}
\p[S_1>0,...,S_{2n-1}>0,S_{2n}=2r]&=\frac{1}{2}\left(\p_{2n-1,2r-1}-\p_{2n-1,2r+1}\right)\\
\p[S_1>0,S_2>0,...,S_{2n}>0]&=\frac{1}{2}\sum_{r=1}^n\left(\p_{2n-1,2r-1}-\p_{2n-1,2r+1}\right)
\end{align*}
$$\frac{1}{2}\left(\p_{2n-1,1}-\p_{2n-1,3}+\p_{2n-1,3}-\p_{2n-1,5}\pm...+\p_{2n-1,2n-1}-\underbrace{\p_{2n-1,2n+1}}_{=0}\right)$$
$$=\frac{1}{2}\p_{2n-1,1}=\frac{1}{2}U_{2n}.$$
Therefore we get 
$$\p_{2n-1,1}=\p[S_{2n-1}=1]=\binom{2n-1}{\frac{2n-1+1}{2}}2^{-(2n-1)}=2^{-2n}\cdot 2\cdot \frac{(2n-1)!}{n!(n-1)!}=2^{2n}\binom{2n}{2}=U_{2n}$$
\end{proof}
Saying that the first return to the origin occurs at $2^n$ amounts $S_1\not=0,S_2\not=0,...,S_{2n-1}\not=0,S_{2n}=0$, we have
$$\p[S_1\not=0,S_2\not=0,...,S_{2n-1}\not=0,S_{2n}=0]=f_{2n},$$
and
$$\{S_1\not=0,S_2\not=0,...,S_{2n-1}\not=0\}=\{S_1\not=0,...,S_{2n-1}\not=0,S_{2n}\not=0\}\cup\{S_1\not=0,...,S_{2n-1}\not=0,S_{2n}=0\},$$
which implies that
$$U_{2n-2}=U_{2n}+f_{2n},$$
and hence for all $n\geq 1$ we have
$$ f_{2n}=U_{2n-2}-U_{2n}.$$
Therefore we get the relation
$$\boxed{f_{2n}=\frac{1}{2n-1}U_{2n}}$$
\begin{thm}
The probability that up to time $2n$, the last visit to the origin occurs at $2k$ is $$\alpha_{2k,2n}=U_{2k}N_{2n-2k},\hspace{0.2cm}k=0,1,...,n$$
\end{thm}
\begin{proof}
We are concerned with paths satisfying $S_{2k}=0,S_{2k+1}\not=0,...,S_{2n}\not=0$. The first $2k$ vertices of such paths can be chosen in $2^{2k}N_{2k}$ different ways. Taking the point $(2k,0)$ as new origin and using the last lemma, we see that the next $(2n-k)$ vertices can be chosen in $2^{2n-2k}N_{2n-2k}$ different ways. Therefore we get $$\alpha_{2k,2n}=\frac{1}{2^{2n}}\left(2^{2k}N_{2k}2^{2n-2k}N_{2n-2k}\right)=U_{2k}N_{2n-2k}.$$
\end{proof}

\chapter{The Modern Probability Language}
This is the main chapter, covering a basic introduction of the modern probability theory. We will discuss the concept of distributions and the notion of expectation at the beginning. Afterwards, the concept of moments, variance and covariance and several properties of those will be covered. We will continue with the concept of the characteristic function and independence, both for $\sigma$-Algebras and for random variables. Futhermore, we look at the Borel-Cantelli lemma and move on to the weak and strong law of large numbers. The concept of different convergences will follow and finally we are going to spend time on the central limit theorem. After one has read this chapter, the more advanced structures and notions of probability theory will be accessible and lead to a fundamental understanding of modern probability theory.

\section{General Definitions}
\subsection{Law of a Random Variable}
\begin{defn}[Random Variable]
Let $(\Omega,\A,\p)$ be a probability space. Let $(E,\mathcal{E})$ be a measurable space. A measurable map $X:(\Omega,\A,\p)\to (E,\mathcal{E})$ is called a random variable (and is noted r.v.) with values in $E$. 
\end{defn}
\begin{defn}[Law/Distribution]
The law or distribution of a random variable is the image measure of $\p$ by $X$, and is usually noted $\p_X$. It is hence a probability measure on $(E,\mathcal{E})$.
$$\p_X[B]=\p[X^{-1}(B)]=\p[X\in B]=\p\left[\{\omega\in\Omega\mid X(\omega)\in B\}\right]$$
\end{defn}
If $\mu$ is a probability measure on $(\R^d,\B(\R^d))$, (or even on a more general space $(E,\mathcal{E})$), there is a canonical way of constructing a r.v. $X$ such that $\p_X=\mu$ as a map
$$X:(\R^d,\B(\R^d),\mu)\to\R^d.$$
There are two special cases. 
\begin{enumerate}[$(i)$]
\item{\emph{Discrete r.v.:} Let $E$ be a countable space and $\mathcal{E}=\mathcal{P}(E)$. The law of $X$ is given by 
$$\p_X:=\sum_{x\in E}P(x)\delta_x,$$ 
where $P(x)=\p[X=x]$ and $\delta_x$ is the Dirac measure of $x$, meaning that for all $A\subset E$,
$$\delta_x(A)=\begin{cases}1&\text{if $x\in A$}\\ 0&\text{if $x\not\in A$}\end{cases}$$
We note that if $\p_X[E]=1$, then 
$$\sum_{x\in E}P(x)\delta_x(E)=\sum_{x\in E}P(x)=1.$$
Indeed, for all $B\in E$ we have that
$$\p_X[B]=\p[X\in B]=\p\left[\bigcup_{x\in B}\{X=x\}\right]=\sum_{x\in B}\p[X=x]=\sum_{x\in E}P(x)\delta_x(B).$$
}
\vspace{0.5cm}
\item{\emph{Continuous r.v.:} A random variable $X$ with values in $(\R^d,\B(\R^d))$ is said to have a density if $\p_X\ll \lambda$, where $\lambda$ is the lebesgue measure on $\R^d$. The Radon-Nikodym theorem says there exists $P:\R^d\to\R$, measurable such that for al $B\in \B(\R^d)$
$$\p_X[B]=\int_BP(x)dx.$$
In particular, $\int_{\R^d}P(x)dx=\p_X(\R^d)=1$. Moreover the map $P$ is unique up to sets of lebesgue measure 0. $P$ is called the density of $X$. If $d=1$, then 
$$\p[\alpha\leq X\leq \beta]=\p_X[[\alpha,\beta]]=\int_\alpha^\beta P(x)dx.$$
}
\end{enumerate}
\begin{defn}[Expected Value/Expectation]
Let $(\Omega,\A,\p)$ be a probability space. Let $X$ be a real valued r.v. (i.e. with values in $\R$). The expectation of such a r.v. is defined as 
$$\E[X]=\int_{\Omega}X(\omega)d\p(\omega)=\int_\R xd\p_X(x),$$
which is well defined in the following two cases.
\begin{itemize}
\item{If $x\geq 0$, and then $\E[X]\in[0,\infty]$.
}
\item{If $\E[X]=\int_{\Omega}\vert X(\omega)\vert d\p(\omega)<\infty$.
}
\end{itemize}
\end{defn}
We extend this definition to the case of a r.v. $X=(X_1,...,X_d)$ taking values in $\R^d$ by defining 
$$\E[X]=(\E[X_1],...,\E[X_d])$$
provided each $\E[X_i]$ is well defined.
\begin{rem}
If $B\in\A$ and $X=\one_{B}$, then 
$$0\leq \E[X]=\E[\one_B]=\p[B]\leq 1.$$
In general, $\E[X]$ is interpreted as the average or the mean of the r.v. $X$. If $X$ takes values in $\{x_1,...,x_n,...\}$ then 
$$\E[X]=\sum_{n=1}^\infty x_n\p[X=x],$$
whenever it is well defined.
\end{rem}
The expectation is a special case of an integral with respect to a positive measure. In particular,  
\begin{itemize}
\item{For all $X,Y$ integrable and $a,b\in\R$ we have
$$\E[aX+bY]=a\E[X]+b\E[Y].$$
}
\item{If $C$ is a constant and $\E[X]=C$, then 
$$\int_\Omega Cd\p(\omega)=C\p[\Omega]=C.$$
}
\item{If $X\geq 0$ and $\E[X]\geq 0$ and if $X\leq Y$ both integrable then 
$$\E[X]\leq \E[Y].$$
}
\item{(\emph{Monotne convergence}) If $(X_n)_{n\geq 1}$ is a sequence of real valued r.v.'s, and if $X_n\geq 0$ for all $n\geq 1$ and $X_n\uparrow X$ as $n\to\infty$, then 
$$\E[X_n]\uparrow\E[X]\hspace{0.3cm}\text{as $n\to\infty$}.$$
}
\vspace{0.5cm}
\item{(\emph{Fatou}) If $(X_n)_{n\geq 1}$ is a sequence of real valued r.v.'s with $X_n\geq 0$ for all $n\geq 1$, then 
$$\E\left[\liminf_{n\to\infty}X_n\right]\leq \liminf_{n\to\infty}\E[X_n].$$
}
\vspace{0.5cm}
\item{(\emph{Dominated convergence}) If $(X_n)_{n\geq 1}$ is a sequence of real valued r.v.'s with $\vert X_n\vert \leq Z$ for all $n\geq 1$, such that $\E[Z]<\infty$, for another real valued r.v. $Z$, and $X_n\xrightarrow{n\to\infty}X$ a.e., then 
$$\E[X_n]\xrightarrow{n\to\infty}\E[X].$$ 
}
\end{itemize}
\begin{rem}
In probability theory we say almost sure convergence and write a.s., rather than almost everywhere. If $X_n\xrightarrow{n\to\infty}X$ a.s., then we mean 
$$\p\left[\{\omega\in\Omega\mid X_n(\omega)\xrightarrow{n\to\infty}X(\omega)\}\right]=1.$$
\end{rem}
\vspace{0.5cm}
\begin{prop}
\label{random}
Let $X$ be a r.v. with values in $(E,\mathcal{E})$. If $f:E\to [0,\infty]$ is measurable, then 
$$\E[f(X)]=\int_E f(x)d\p_X(x).$$
Similarly, if $f:E\to \R$ is such that $\E[f(X)]<\infty$, then 
$$\E[f(X)]=\int_E f(x)d\p_X(x).$$
\end{prop}
\begin{rem}
$f(X)$ is also a r.v.
\end{rem}
\begin{proof}[Proof of Proposition \ref{random}]
In the case $f=\one_B$ with $B\in\mathcal{E}$ we get that 
$$\E[f(X)]=\p[X\in B]=\p_X[B]$$
from the definition of the distribution of a r.v. Then by linearity, the result is true for positive simple functions. And then we use the fact that for $f\geq 0$ measurable, $\exists (f_n)_{n\in\N}$, where the $f_n$'s are simple and positive such that $f_n\uparrow f$ as $n\to\infty$ and we apply the monotone convergence theorem.
\end{proof}
\begin{rem}
One often uses the proposition to compute the law of a r.v. $X$. If one is able to write $\E[X]=\int f d\nu$ for a sufficiently large class of functions $f$, then one can deduce that $\p_X=\nu$. The idea is to be able to take $f=\one_B$, for then $\E[f(X)]=\p_X[B]=\nu(B)$.
\end{rem}
\begin{ex} Assume that $\p_X$ is absolutely continuous with density $h(x)=\frac{1}{\sqrt{2\pi}}e^{-\frac{x^2}{2}}$ for $x\in\R$ and $Y=X^2$. Then one can ask about the distribution of $Y$. Let $f:\R\to[0,\infty]$ be measurable. Then 
$$\E[f(Y)]=\E[f(X^2)]=\int_{-\infty}^\infty f(x^2)\frac{1}{\sqrt{2\pi}}e^{-\frac{x^2}{2}}dx.$$
We can write 
$$\int_{-\infty}^\infty f(x^2)\frac{1}{\sqrt{2\pi}}e^{-\frac{x^2}{2}}dx=2\int_{0}^\infty f(x^2)\frac{1}{\sqrt{2\pi}}e^{-\frac{x^2}{2}}dx.$$
Now we can set $y=x^2$. Then $dy=2xdx$ and hence $dx=\frac{dy}{2\sqrt{y}}$. Now we can write
$$2\int_0^\infty f(y)\frac{e^{-\frac{y}{2}}}{2\sqrt{2\pi y}}dy=\int_0^\infty f(y)\frac{e^{-\frac{y}{2}}}{\sqrt{2\pi y}}dy,$$
which implies that
$$d\nu(y)=\frac{e^{-\frac{y}{2}}}{\sqrt{2\pi y}}\one_{\{y>0\}}dy.$$
So we see that the distribution of $Y$ is given by $\frac{e^{-\frac{y}{2}}}{\sqrt{2\pi y}}\one_{\{y>0\}}$.
\end{ex}
\begin{prop}
\label{random2}
Let $X=(X_1,...,X_d)\in\R^d$ be a r.v. Assume that $X$ has density $P(x_1,...,x_d).$ Then $\forall j\in\{1,...,n\}$, $X_j$ has density
$$P_j(x)=\int_{\R^{d-1}}P(x_1,...,x_{j-1},x_{j},x_{j+1},...,x_d)dx^1\dotsm dx^{j-1}dx^{j+1}\dotsm dx^d$$
\end{prop}
\begin{rem}
Let $d=2$ and $X=(X_1,X_2)$. Then $P_1(x)=\int_\R P(x,y)dy$ and $P_2(x)=\int_\R P(x,y)dx$.
\end{rem}
\begin{proof}[Proof of Proposition \ref{random2}]
Let $\pi_j:(x_1,...,x_d)\mapsto x_j$. From Fubini's theorem we get that $\forall f:\R\to \R^+$, Borel measurable
$$\E[f(X_j)]=\E[f(\pi_j(X))]=\int_{\R^{d}}f(x_j)P(x_1,...,x_d)dx^1\dotsm dx^d$$
$$=\int_\R f(x_j)\underbrace{\left(\int_{\R^{d-1}}P(x_1,...,x_{j-1},x_j,x_{j+1},...,x_d)dx^1\dotsm dx^{j-1}dx^{j+1}\dotsm dx^d\right)dx^j}_{d\nu(x_j)=P(x_j)dx^j}$$
By renaming $x_j=y$, we get 
$$\E[f(X_j)]=\int_\R f(y)P_j(y)dy.$$
Hence the distribution of $X_j$ has density $P_j(y)$ on $\R$.
\end{proof}
\begin{rem}
If $X=(X_1,...,X_d)\in\R^d$ is a r.v., then the distribution $\p_{X_j}$ are called the margins of $X$. The last proposition shows us that the margins are determined by 
$$\p_{X=(X_1,...,X_d)},$$
but the converse is wrong. For example take $Q$ to be a density on $\R$ and observe that $P(x_1,x_2)=Q(x_1)Q(x_2)$ is also a density on $\R^2$. We have already seen that we can construct (in a canonical way) a r.v. $X=(X_1,X_2)\in\R^2$ such that $\p_X$ has $P(x_1,x_2)$ as density. Now the margins of $X$, namely $\p_{X_1}$ and $\p_{X_2}$, have density $q(x)$. We now observe that the $r.v.$'s $X=(X_1,X_2)$ and $X'=(X_1,X_1)$ have the same margin but they are different. $\p_X$ has support in $\R^2$, while $\p_{X'}$ has support in the diagonal of $\R^2$, which is of Lebesgue measure 0 in $\R^2$. In general we have $\p_X\not=\p_{X'}$.
\end{rem}
\section{Classical Probability distributions}
Let $(\Omega,\A,\p)$ denote a probability space and let $X:(\Omega,\A,\p)\to(E,\mathcal{E})$ be a r.v. taking values in some measureable space $(E,\mathcal{E})$.
\subsection{Discrete distributions}
\subsubsection{The uniform distribution} Let $\vert E\vert<\infty$. A r.v. $X$ with values in $E$ is said to be uniform on $E$ if $\forall x\in E$ 
$$\p[X=x]=\frac{1}{\vert E\vert}.$$
\subsubsection{The Bernoulli distribution with parameter $p\in[0,1]$} This is a r.v. $X$ with values in $\{0,1\}$ such that 
$$\p[X=1]=p,\hspace{0.2cm}\p[X=0]=1-p.$$  
The r.v. $X$ can be interpreted as the outcome of a coin toss. The expectation of $X$ is then given by $$\E[X]=0\cdot\p[X=0]+1\cdot\p[X=1].$$
\subsubsection{The Binomial distribution $\B(n,p)$, $n\in \N$, $n\geq 1$, $p\in[0,1]$} This is the distribution of a r.v. $X$ taking its values in $\{0,1,...,n\}$ such that 
$$\p[X=k]=\binom{n}{k}p^k(1-p)^{n-k}.$$
\begin{figure}[h!]
\begin{center}
\includegraphics[height=8cm, width=10cm]{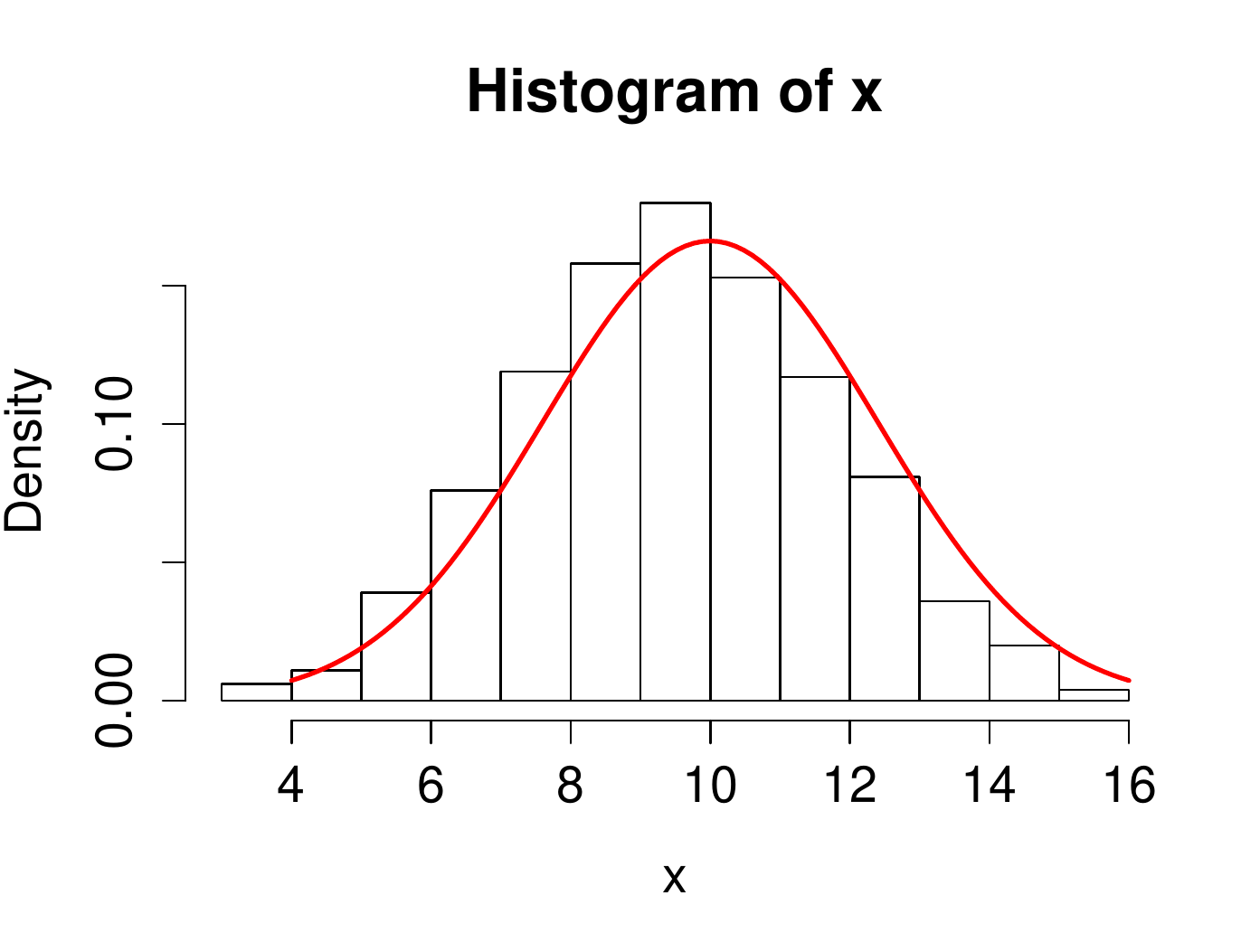}
\end{center}
\caption{Histogram of a binomial distributed r.v.}
\end{figure}
The r.v. $X$ is interpreted as the number of heads of the $n$ tosses of the previous case. One has to check that its a probability distribution:
$$\sum_{k=0}^n\p[X=k]=\sum_{k=0}^n\binom{n}{k}p^k(1-p)^{n-k}=(p+(1-p))^n=1.$$
The expected value for the binomial distribution is given by 
\begin{align*}
\E[X]=\sum_{k=0}^nk\p[X=k]&=\sum_{k=0}^nk\binom{n}{k}p^k(1-p)^{n-k}=np\sum_{k=0}^nk\frac{(n-1)!}{(n-k)!k!}p^{k-1}(1-p)^{(n-1)-(k-1)}\\
&=np\sum_{k=1}^n\frac{(n-1)!}{(n-k)!(k-1)!}p^{k-1}(1-p)^{(n-1)-(k-1)}\\
&=np\sum_{k=1}^n\binom{n-1}{k-1}p^{k-1}(1-p)^{(n-k)-(k-1)}\\
&=np\sum_{l=0}^{n-1}\binom{n-1}{l}p^l(1-p)^{(n-1)-l},\hspace{0.1cm}(l:=k-1)\\
&=np\sum_{l=0}^m\binom{m}{l}p^l(1-p)^{m-l},\hspace{0.2cm}(m:=n-1)=np(p+(1-p))^m=np
\end{align*}
\subsubsection{The Geometric distribution with parameter $p\in[0,1]$} This is a r.v. $X$ with values in $\N$ such that 
$$\p[X=k]=(1-p)p^k.$$
The r.v. $X$ can be interpreted as the number of heads obtained before tail shows for the first time. It is also a probability distribution, since
$$\sum_{k=0}^\infty\p[X=k]=\sum_{k=0}^\infty(1-p)p^k=(1-p)\sum_{k=0}^\infty p^k=\frac{1-p}{1-p}=1.$$
\subsubsection{The Poisson distribution with parameter $\lambda>0$} This is a r.v. $X$ with values in $\N$ such that 
$$\p[X=k]=e^{-\lambda}\frac{\lambda^k}{k!},\hspace{0.2cm}k\in\N.$$
\begin{figure}[h!]
\begin{center}
\includegraphics[height=8cm, width=10cm]{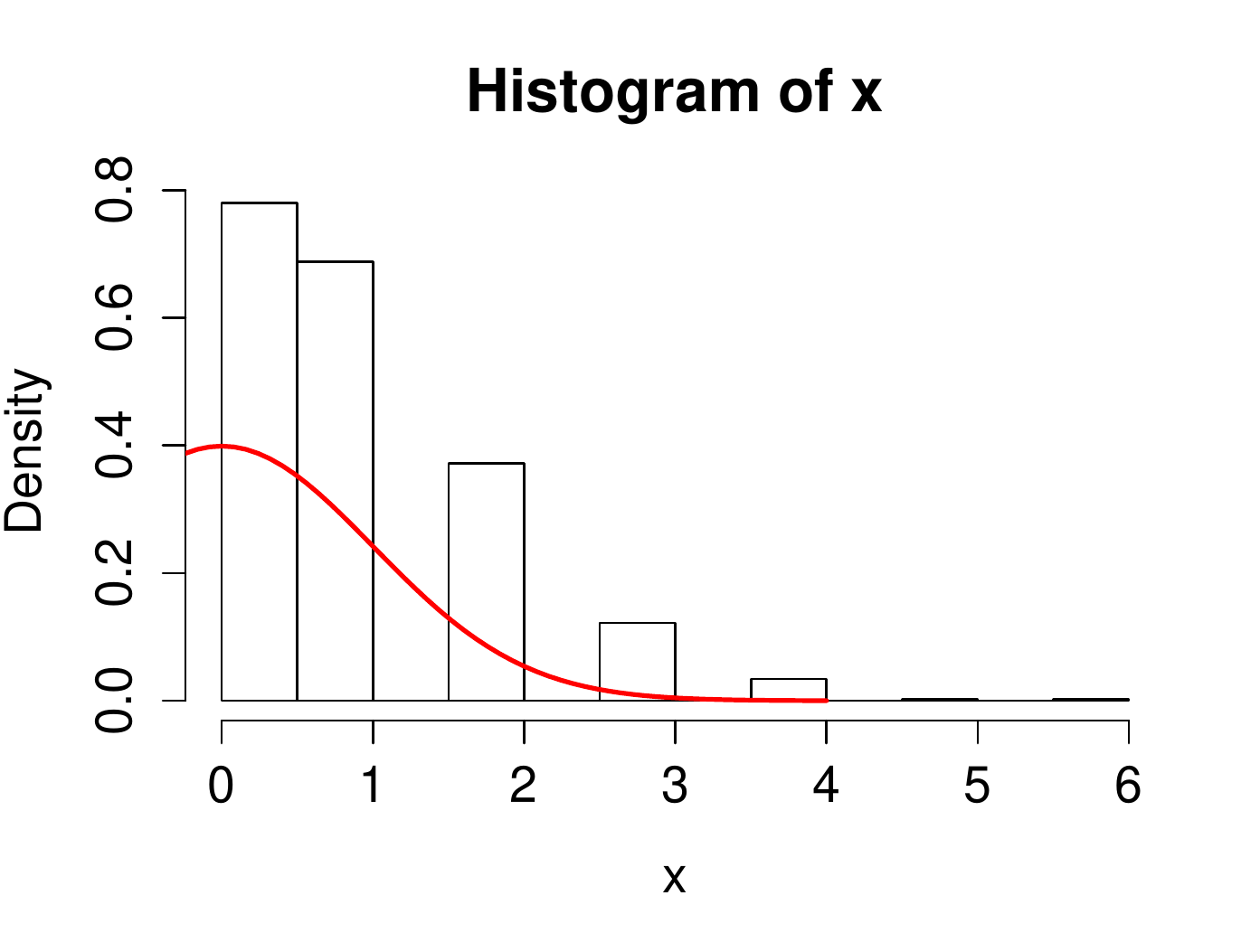}
\end{center}
\caption{Histogram of a poisson distributed r.v.}
\end{figure}
The Poisson distribution is very important, both from the point of view of applications and from the theoretical point of view. Intuitively it describes the number of rare events that have occurred during a long period. If $X_n\sim \B(n,p_n)$ and if $np_n\xrightarrow{n\to\infty}\lambda>0$, i.e. $p_n\sim \frac{\lambda}{n}$ for $n\geq 1$, then for every $k\in\N$
$$\p[X_n=k]\xrightarrow{n\to\infty} e^{-\lambda}\frac{\lambda^k}{k!}.$$
The expected value is then given by 
$$\E[X]=\sum_{k=0}^\infty k\frac{\lambda^k}{k!}e^{-\lambda}=\lambda e^{-\lambda}\sum_{k=1}^\infty\frac{\lambda^{k-1}}{(k-1)!}=\lambda e^{-\lambda}\sum_{j=0}\frac{\lambda^j}{j!}=\lambda.$$
\subsection{Absolutely continuous distributions}
Let now $E\subset\R$. The question here is about the densities $P(x)$ of a certain distributed r.v. in the continuous case.
\subsubsection{The uniform distribution on $[a,b]$} The density of a continuous, uniformly distributed r.v. $X$ is given by 
$$P(x)=\frac{1}{b-a}\one_{[a,b]}(x).$$
\begin{figure}[h!]
\begin{center}
\includegraphics[height=8cm, width=10cm]{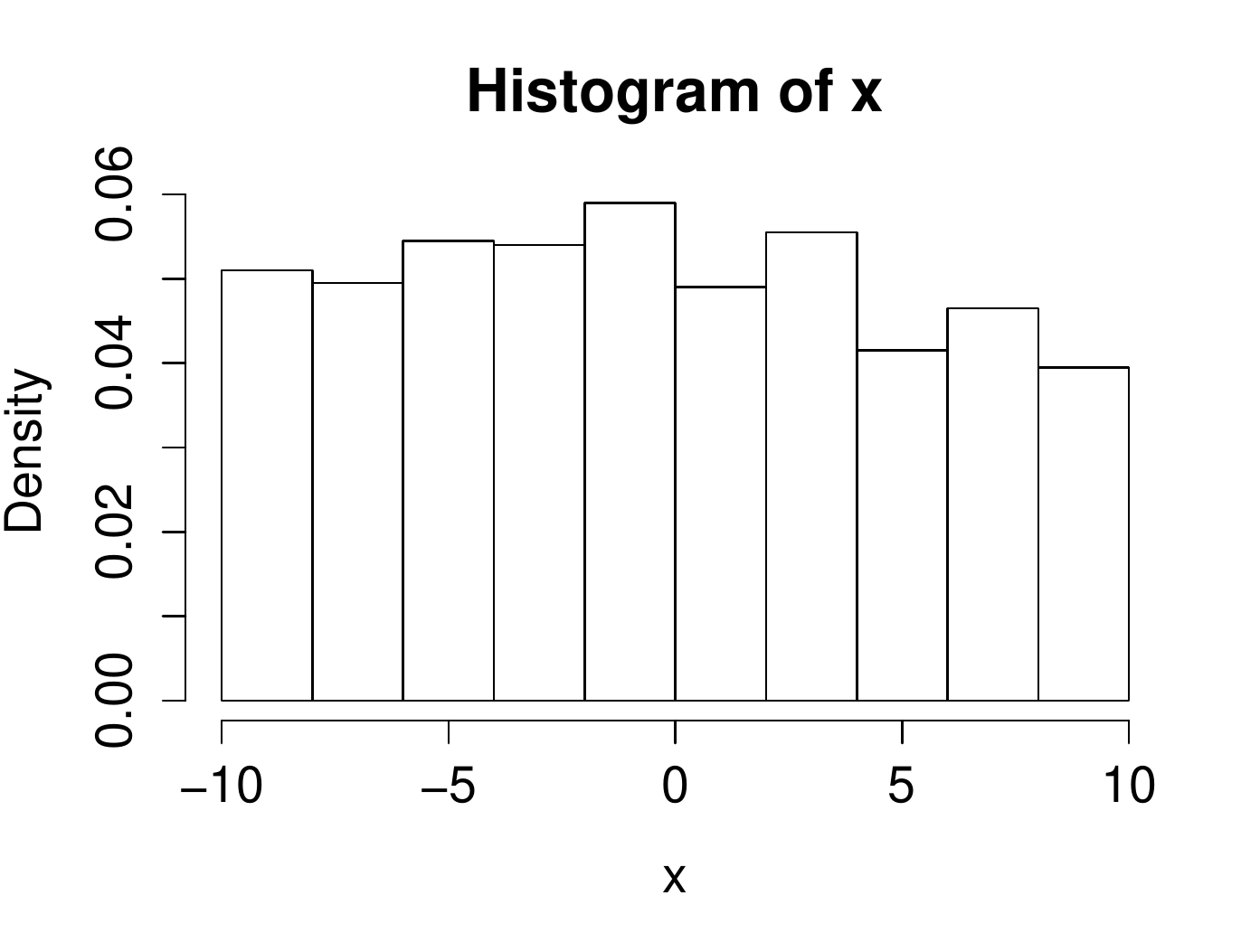}
\end{center}
\caption{Histogram of a uniformly distributed r.v.}
\end{figure}
We want to check that it is a probability density. We have to check that $\int_\R P(x)dx=1$, so we have
$$\int_{-\infty}^\infty P(x)dx=\int_{-\infty}^\infty \frac{1}{b-a}\one_{[a,b]}(x)dx=\frac{1}{b-a}\int_{-\infty}^\infty\one_{[a,b]}(x)dx=\frac{1}{b-a}(b-a)=1.$$
Hence it's a probability density. If $X$ is uniform on $[a,b]$, then $\vert X\vert\leq \vert a\vert +\vert b\vert<\infty$ a.s. and $\E[\vert a\vert +\vert b\vert ]=\vert a\vert +\vert b\vert<\infty\Longrightarrow \E[X]<\infty$. The expectation is given by
$$\E[X]=\int_{-\infty}^{\infty}xP(x)dx=\int_{-\infty}^\infty\frac{1}{b-a}\one_{[a,b]}(x)dx=\frac{1}{b-a}\int_a^bxdx=\frac{1}{b-a}\frac{1}{2}(b^2-a^2)=\frac{a+b}{2}.$$
\subsubsection{The Exponential distribution with parameter $\lambda>0$} The density is given by 
$$P(x)=\lambda e^{-\lambda x}\one_{\R^+}(x),$$
with $X\geq 0$ a.s. The expectation is given by
$$\E[X]=\int_{-\infty}^\infty xP(x)dx=\int_{0}^\infty x\lambda e^{-\lambda x}dx=\lambda\int_0^\infty xe^{-\lambda x}dx.$$
With $u=\lambda x$ we get $dx=\frac{du}{\lambda}$ and hence 
$$\lambda\int_0^\infty \frac{u}{\lambda}e^{-u}\frac{du}{\lambda}=\frac{1}{\lambda}\int_0^\infty ue^{-u}du=\frac{1}{\lambda}.$$
If $a,b>0$, then 
$$\p[X>a+b]=\int_{a+b}^\infty \lambda e^{-\lambda x}dx=\lambda\left[-\frac{1}{\lambda}e^{-\lambda x}\right]_{a+b}^\infty=e^{-\lambda(a+b)}=e^{-\lambda a}e^{-\lambda b}=\p[X>a]\p[X>b].$$
\begin{figure}[h!]
\begin{center}
\includegraphics[height=8cm, width=10cm]{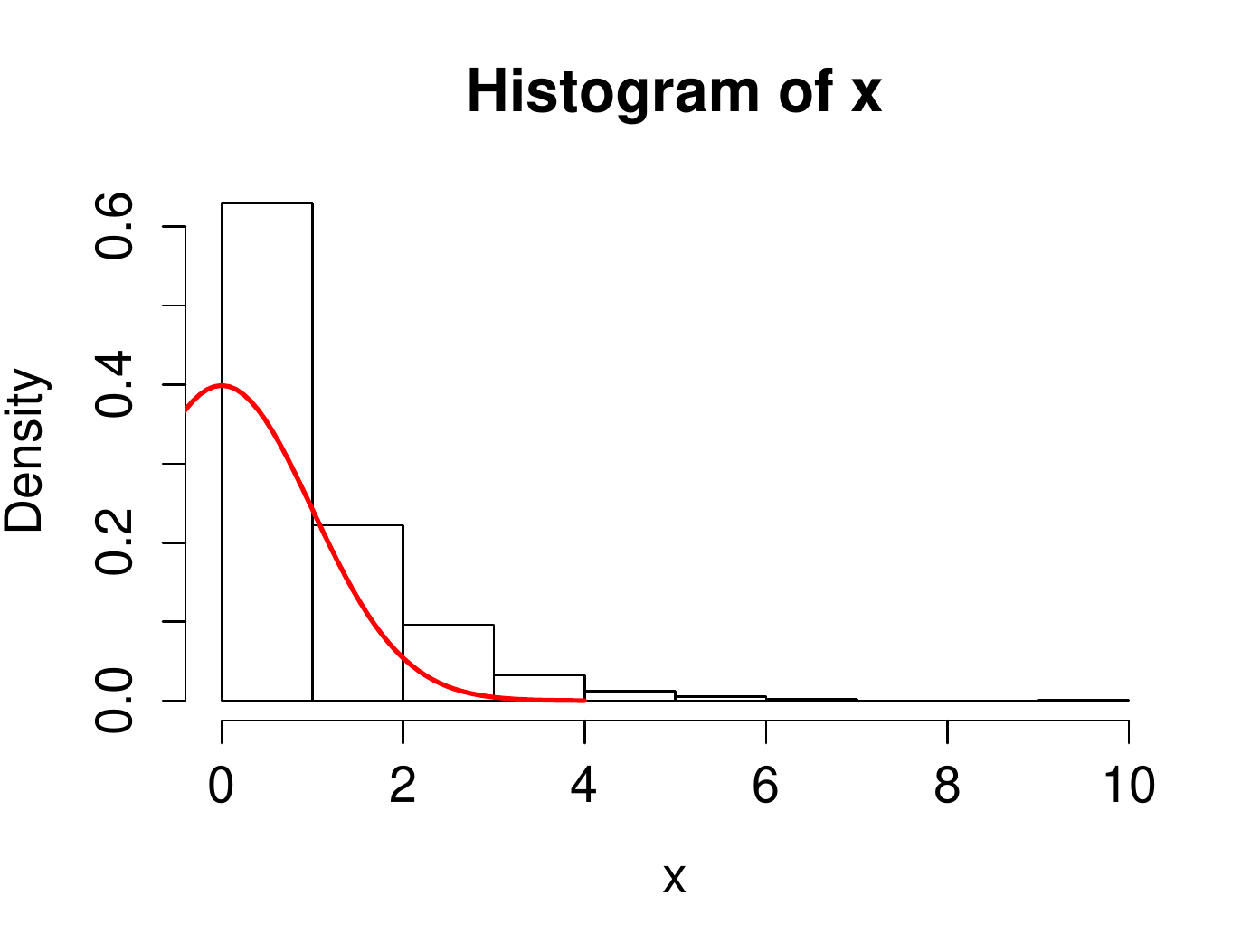}
\end{center}
\caption{Histogram of an exponentially distributed r.v.}
\end{figure}
Note that 
$$\p[X<0]=\E[\one_{\{X<0\}}]=\int_{-\infty}^\infty\one_{\{x<0\}}P(x)dx=\int_{-\infty}^\infty\one_{\{x<0\}}\lambda e^{-\lambda x}\one_{\{x\geq 0\}}dx=0$$
and also that 
$$\p[X=x]=\int_{-\infty}^{\infty} \one_{\{y=x\}}P(y)dy=0.$$
\subsubsection{The Gaussian distribution $\mathcal{N}(m,\sigma^2)$, $m\in\R$, $\sigma>0$} The density is given by 
$$P(x)=\frac{1}{\sigma\sqrt{2\pi}}\exp\left(-\frac{(x-m)^2}{2\sigma^2}\right).$$
This is the most important distribution in probability theory. We have to check that $P(x)$ is a probability density, i.e.
$$\int_{-\infty}^\infty\frac{1}{\sigma\sqrt{2\pi}}\exp\left(-\frac{(x-m)^2}{2\sigma^2}\right) dx=1.$$
\begin{figure}[h!]
\begin{center}
\includegraphics[height=8cm, width=10cm]{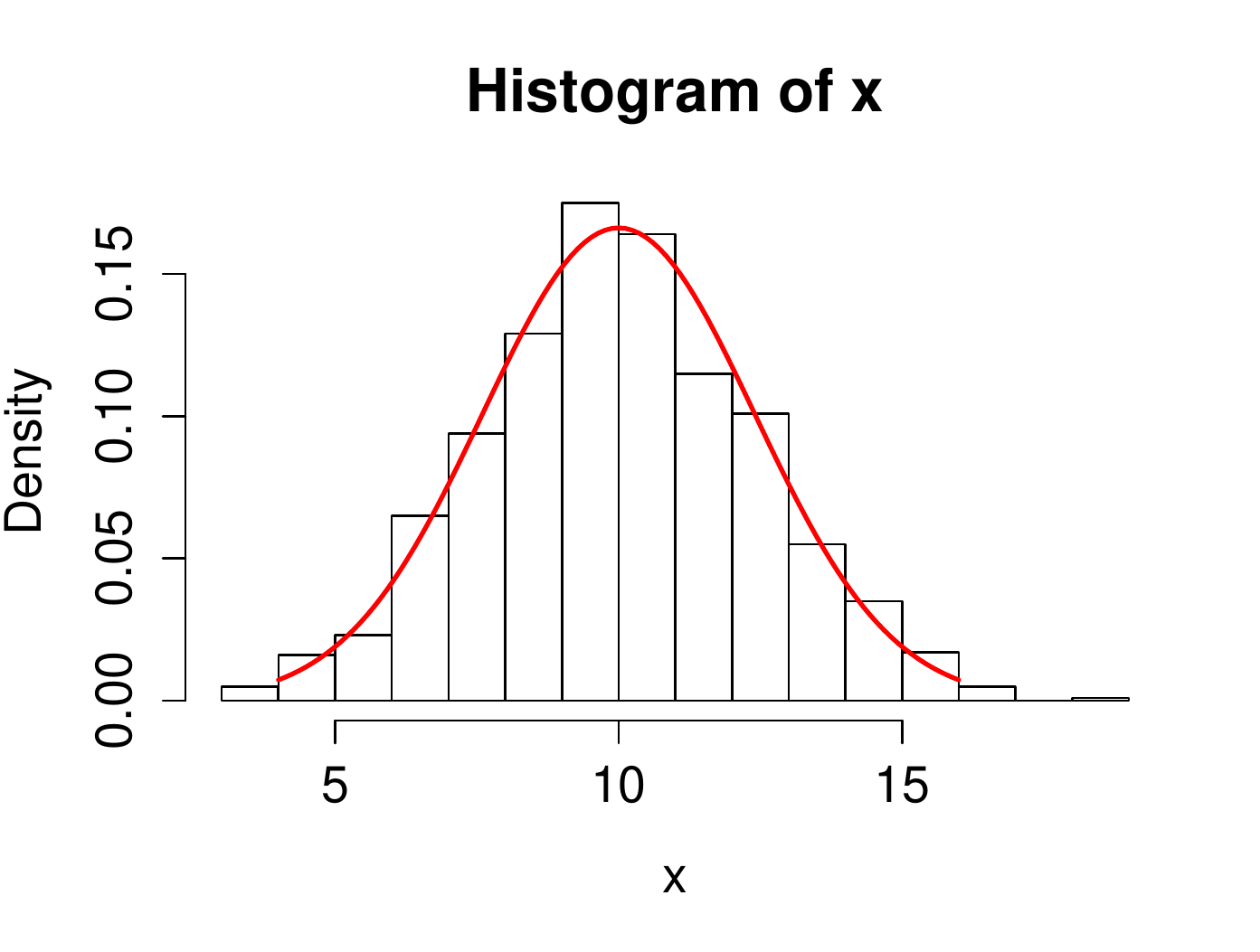}
\end{center}
\caption{Histogram of a Gaussian distributed r.v.}
\end{figure}
We set $u=x-m$ and hence $du=dx$. So we get
$$\int_{-\infty}^{\infty}\frac{1}{\sigma\sqrt{2\pi}}e^{-\frac{u^2}{2\sigma^2}}du.$$
Now we set $t=\frac{u}{\sigma}$ and hence $du=\sigma dt$. So now we get 
$$\int_{-\infty}^\infty \frac{1}{\sigma\sqrt{2\pi}}e^{-\frac{t^2}{2}}\sigma dt=\frac{1}{\sqrt{2\pi}}\int_{-\infty}^\infty e^{-\frac{t^2}{2}}dt=\frac{\sqrt{2\pi}}{\sqrt{2\pi}}=1.$$
We have used the fact that $\int_{-\infty}^{\infty}e^{-\frac{x^2}{2}}dx=\sqrt{2\pi}$, by change of coordinates from cartesian coordinates to polar coordinates. Consider $\mathcal{N}(0,1)$ with density $P(x)=\frac{1}{\sqrt{2\pi}}e^{-\frac{x^2}{2}}$. It is called the standard Gaussian distribution ($m=0$, $\sigma=1$). We note that if $X$ is distributed according to $\mathcal{N}(m,\sigma^2)$, then 
$$\E[X]=m,\hspace{0.5cm}\text{and}\hspace{0.5cm}\E[(X-m)^2]=\sigma^2.$$
Indeed we have
$$\E[\vert X\vert]=\int_{-\infty}^{\infty}\vert x\vert\frac{1}{\sigma\sqrt{2\pi}}\exp\left(-\frac{(x-m)^2}{2\sigma^2}\right)dx<\infty$$
and therefore
$$\E[X]=\int_{-\infty}^\infty x\frac{1}{\sigma\sqrt{2\pi}}\exp\left(-\frac{(x-m)^2}{2\sigma^2}\right)dx.$$
We set $u=x-m$ and hence $du=dx$. So we get
$$\int_{-\infty}^{\infty}\frac{(u+m)}{\sigma\sqrt{2\pi}}e^{-\frac{u^2}{2\sigma^2}}du=\underbrace{\int_{-\infty}^\infty\frac{u}{\sigma\sqrt{2\pi}}e^{-\frac{u^2}{2\sigma^2}}du}_{=0}+\underbrace{m\int_{-\infty}^\infty \frac{1}{\sigma\sqrt{2\pi}}e^{-\frac{u^2}{2\sigma^2}}du}_{=m}.$$
Therefore we get $\E[X]=m$. One can show similarly that $\E[(X-m)^2]=\sigma^2$.
\subsection{The distribution function}
Let $X:\Omega\to\R$ be a real valued r.v. The distribution function of $X$ is the function
$$F_X:\R\to[0,1],\hspace{0.5cm}t\mapsto F_X(t):=\p[X\leq t]=\p_X[(-\infty,t)].$$
We claim that $F_X$ is increasing and right continuous. Meaning that 
$$\lim_{t\to-\infty}F_X(t)=0,\hspace{1cm}\lim_{t\to \infty}F_X(t)=1$$
We can thus write 
$$\p[a\leq X\leq b]=F_X(b)-\underbrace{F_X(a^-)}_{\lim_{t\to a\atop t<a}F_X(t)}\hspace{0.5cm}\text{and thus}\hspace{0.5cm}\p[a<X<b]=F(b^-)-F_X(a),$$
Moreover, for a single value we get 
$$\p[X=a]=F_X(a)-F_X(a^-),$$
which is called the jump of the function $F_X$.
If $X$ and $Y$ are two r.v.'s, such that $F_X(t)=F_Y(t)$, then $\p_X=\p_Y$ (this is a consequence of the monotone class theorem). If $F$ is an increasing and right continuous function, then the set 
$$A=\{a\in\R\mid F(a)\not=F(a^-)\}$$
is at most countable. If $\p_X$ is absolutely continuous, then 
$$\p_X[\{ a\}]=\p[X=a]=0,$$
which implies that for all $a\in\R$ we have $F_X(a)=F_X(a^-)$ and hence $F_X$ is continuous. An alternative point of view is to say that, if $P(x)$ is the density of of $\p_X$, then 
$$F_X(x)=\p_X[(-\infty,t]]=\int_\R\one_{(-\infty,t]}(x)P(x)dx=\int_{-\infty}^tP(x)dx,$$
is a continuous function of $t$.
\subsection{$\sigma$-Algebras generated by a Random Variable} Let $(\Omega,\A,\p)$ be a probability space. Let $X$ be a r.v. taking values in $(E,\mathcal{E})$, i.e. $X:(\Omega,\A,\p)\to(E,\mathcal{E})$. The $\sigma$-Algebra generated by $X$, denoted by $\sigma(X)$, is by definition the smallest $\sigma$-Algebra, which makes $X$ measurable. So we have
$$\sigma(X)=\{A=X^{-1}(B)\mid B\in\mathcal{E}\}.$$
\begin{rem}
One can of course extend this definition to the case of a family of r.v.'s $X_i$ for $i\in I$, taking values in $(E_i,\mathcal{E}_i)$. In this case we have 
$$\sigma((X_i)_{i\in I})=\sigma(\{X_i^{-1}(B)\mid B_i\in\mathcal{E}_i,i\in I\}).$$
\end{rem}
\begin{prop}
Let $(\Omega,\A,\p)$ be a probability space. Let $X$ be a r.v. with values in a measure space $(E,\mathcal{E})$ and let $Y$ be a real valued r.v. Then the following are equivalent.
\begin{enumerate}[$(i)$]
\item{$Y$ is $\sigma(X)$-measurable.
}
\item{There exists a measurable map $f:(E,\mathcal{E})\to(\R,\B(\R))$, such that
$$Y=f(X).$$
}
\end{enumerate}
\end{prop}
\begin{proof}
So we have the following cases 
$$X:(\Omega,\A,\p)\longrightarrow(E,\mathcal{E}),\hspace{0.5cm}X:(\Omega,\sigma(X),\p)\longrightarrow (E,\mathcal{E}),\hspace{0.5cm}Y:(\Omega,\A,\p)\longrightarrow (\R,\B(\R))$$
\begin{itemize}
\item{$(2)\Longrightarrow (1)$: This follows from the fact that the composition of two measurable maps is measurable.
}
\item{$(1)\Longrightarrow(2)$: Assume that $Y$ is $\sigma(X)$-measurable. Assume first $Y$ is simple, i.e. 
$$Y=\sum_{i=1}^n\lambda_i\one_{A_i}(x),\hspace{0.2cm}\forall i\{1,...,n\},\lambda_i\in\R,A_i\in\sigma(X).$$
Now by definition of $\sigma(X)$, there is a $\B_i\in\mathcal{E}$, such that $A_i\in X_i^{-1}(B_i)$, $\forall i\in\{1,...,n\}$. So it follows that 
$$Y=\sum_{i=1}^n\lambda_i\one_{A_i}=\sum_{i=1}^n\lambda_i\one_{B_i}\circ X=f\circ X,$$
where $f=\sum_{i=1}^n\lambda_i\one_{B_i}$ is $\mathcal{E}$-measurable. More generally, if $Y$ is $\mathcal{E}$-measurable, there exists a seqence $(Y_n)$ of simple functions such that $Y_n$ is $\sigma(X)$-measurable and $Y_n\xrightarrow{n\to\infty} Y$. The above implies $Y_n=f_n(X)$ when $f_n:E\to\R$ is a measurable map. For $x\in E$, set 
$$f(x)=\begin{cases}\lim_{n\to\infty}f_n(x),&\text{if the limit exists}\\ 0,&\text{otherwise}\end{cases}$$
Then $f$ is measurable. Moreover for all $\omega \in\Omega$ we get
$$X(\omega)\in\left\{x\mid \lim_{n\to\infty}f_n(x)\hspace{0.2cm}\text{exists}\right\},$$
since $\lim_{n\to\infty}f_n(X(\omega))=\lim_{n\to\infty}Y_n(\omega)=Y(\omega)$ and $f(X(\omega))=\lim_{n\to\infty}f_n(X(\omega))$. Hence $Y=f(X)$.
}
\end{itemize}
\end{proof}
\section{Moments of Random Variables}
\subsection{Moments and Variance} Let $(\Omega,\A,\p)$ be a probability space. Let $X$ be a r.v. and let $p\geq 1$ be an integer (or even a real number). The $p$-th moment of $X$ is by definition $\E[X^p]$, which is well defined when $X\geq 0$ or $\E[\vert X\vert^p]<\infty$, which is by definition
$$\E[\vert X\vert^p]=\int_\Omega\vert X(\omega)\vert^pd\p(\omega)<\infty.$$
When $p=1$, we get the expected value. We say that $X$ is $centered$ if $\E[X]=0$. The spaces $L^p(\Omega,\A,\p)$ for $p\in[1,\infty)$ are defined as we have seen in the course $Measure$ $and$ $Integral$. From Hölder's inequality we can observe that 
$$\E[\vert XY\vert]\leq \E[\vert X\vert^p]^{\frac{1}{p}}\E[\vert Y\vert^q]^{\frac{1}{q}},$$
whenever $p,q\geq 1$ and $\frac{1}{p}+\frac{1}{q}=1$. If we take $Y=1$ above, we obtain 
$$\E[X]\leq \E[\vert X\vert^p]^{\frac{1}{p}},$$
which means $\|X\|_1\leq \|X\|_p$. This can be extended as $\|X\|_r\leq \|X\|_p$ if $r\leq p$. So it follows that $L^p(\Omega,\A,\p)\subset L^r(\Omega,\A,\p)$. For $p=q=2$ we get the Cauchy-Schwarz inequality as follows
$$\E[\vert XY\vert]\leq \E[\vert X\vert^2]^{\frac{1}{2}}\E[\vert Y\vert^2]^{\frac{1}{2}}.$$
With $Y=1$ we have $\E[\vert X\vert]^2=\E[X^2]$.
\begin{defn}[Variance]
Let $(\Omega,\A,\p)$ be a probability space. Consider a r.v. $X\in L^2(\Omega,\A,\p)$. The variance of $X$ is defined as 
$$Var(X)=\E[(X-\E[X])^2]$$
and the standard deviation of $X$ is given by 
$$\sigma_X=\sqrt{Var(X)}.$$
\end{defn}
\begin{rem}
Informally, the variance represents the deviation of $X$ around its mean $\E[X]$. Note that $Var(X)=0$ if and only if $X$ is a.s. constant.
\end{rem}
\begin{prop} 
$$Var(X)=\E[X^2]-\E[X]^2\hspace{0.2cm}\text{and for all $a\in\R$ we get}\hspace{0.2cm}\E[(X-a)^2]=Var(X)+(\E[X]-a)^2.$$ 
Consequently, we get
$$Var(X)=\inf_{a\in R}\E[(X-a)^2].$$ 
\end{prop}
\begin{proof}
$$Var(X)=\E[(X-\E[X])^2]=\E[X^2-\{\E[X]X+\E[X]^2\}]=\E[X^2]-2\E[X]\E[X]+\E[X]^2=\E[X^2]-\E[X]^2$$
Moreover, we have
$$\E[(X-a)^2]=\E[(X-\E[X])+(\E[X]-a)^2]=Var(X)+(\E[X]-a)^2,$$
which implies that for all $a\in\R$
$$\E[(X-a)^2]\geq Var(X)$$
and there is equality when $a=\E[X]$.
\end{proof}
\begin{rem}
\label{random3}
It follows that if $X$ is centered (i.e. $\E[X]=0$), we get $Var(X)=\E[X^2]$. Moreover, the following two simple inequalities are very often used.
\begin{enumerate}[$(i)$]
\item{(\emph{Markov inequality}) If $X\geq 0$ and $a\geq 0$ then 
$$\boxed{\p[X>a]\leq \frac{1}{a}\E[X]}.$$
}
\item{(\emph{Tchebisheff inequality})
$$\boxed{\p[\vert X-\E[X]\vert>a]\leq \frac{1}{a^2}Var(X)}.$$
}
\end{enumerate}
\end{rem}
\begin{proof}[Proof of Remark \ref{random3}] We want to show both inequalities.
\begin{enumerate}[$(i)$]
\item{Note that 
$$\p[X\geq a]=\E[\one_{\{X\geq a\}}]\leq \E\left[\frac{X}{a}\underbrace{\one_{\{X\geq a\}}}_{\leq 1}\right]\leq \E\left[\frac{X}{a}\right].$$
}
\item{This follows from (1) because $\vert X-\E[X]\vert$ is a positive r.v. and hence 
$$\p[\vert X-\E[X]\vert\geq a]=\p[\vert X-\E[X]\vert^2\geq a^2]\leq \frac{1}{a^2}\underbrace{\E[\vert X-\E[X]\vert^2]}_{Var(X)}.$$
}
\end{enumerate}
\end{proof}
\begin{defn}[Covariance]
Let $(\Omega,\A,\p)$ be a probability space. Consider two r.v.'s $X,Y\in L^2(\Omega,\A,\p)$. The covariance of $X$ and $Y$ is defined as 
$$Cov(X,Y)=\E[(X-\E[X])(Y-\E[Y])]=\E[XY]-\E[X]\E[Y].$$
\end{defn}
If $X=(X_1,...,X_d)\in\R^d$ is a r.v. such that $\forall i\in\{1,...,d\}$, $X_i\in L^2(\Omega,\A,\p)$, then the covariance matrix of $X$ is defined as 
$$K_X=\left( Cov(X_i,X_j)\right)_{1\leq i,j\leq d}.$$
Informally speaking, the covariance between $X$ and $Y$ measures the correlation between $X$ and $Y$. Note that $Cov(X,X)=Var(X)$ and from Cauchy-Schwarz we get 
$$\left\vert Cov(X,Y)\right\vert\leq \sqrt{Var(X)}\cdot\sqrt{Var(Y)}.$$
The application $(X,Y)\mapsto Cov(X,Y)$ is a bilinear form on $L^2(\Omega,\A,\p)$. We also note that $K_X$ is symmetric and positive, i.e. if $\lambda_1,...,\lambda_d\in\R$, $\lambda=(\lambda_1,...,\lambda_d)^T$, then 
$$\left\langle K_X\lambda,\lambda\right\rangle=\sum_{i,j=1}^d\lambda_i\lambda_jCov(X_i,X_j)\geq 0.$$
So we get
\begin{align*}
\sum_{i,j=1}^d\lambda_i\lambda_jCov(X_i,X_j)&=Var\left(\sum_{j=1}^d\lambda_jX_j\right)=\E\left[\left(\sum_{j=1}^d\lambda_jX_j-\E\left[\sum_{j=1}^d\lambda_j X_j\right]\right)^2\right]\\
&=\E\left[\left(\sum_{j=1}^d\lambda_jX_j-\sum_{j=1}^d\lambda_j\E[X_j]\right)^2\right]=\E\left[\left(\sum_{j=1}^d (X_j-\E[X_j])\right)^2\right]\\
&=\E\left[\sum_{j=1}\lambda_j(X_j-\E[X_j])\sum_{i=1}^d\lambda_i(X_i-\E[X_i])\right]\\
&=\sum_{i,j=1}^d\lambda_i\lambda_j\E[(X_i-\E[X_i])(X_j-\E[X_j])]\geq 0
\end{align*}
\begin{exer} If $A$ is a matrix of size $n\times d$, and $Y=AX$, then prove that 
$$K_Y=AK_XA^T.$$
\end{exer}
\begin{rem}
Set $X=(X_1,...,X_d)^T$ and $XX^T=(X_iX_j)_{1\leq i,j\leq n}$, then informally 
$$K_X=\E[XX^T]=(\E[X_iX_j])_{1\leq i,j\leq n},$$
and for $Y=AX$ we get 
$$K_Y=\E[AXX^TA^T]=A\E[XX^T]A^T=AK_XA^T.$$
\end{rem}
\subsection{Linear Regression} Let $(\Omega,\A,\p)$ be a probability space. Let $X,Y_1,...,Y_n$ be r.v.'s in $L^2(\Omega,\A,\p)$. We want the best approximation of $X$ as an affine function of $Y_1,...,Y_n$. More precisely we want to minimize
$$\E[(X-\E[\beta_0+\beta_1Y_1+...+\beta_nY_n])^2]$$
over all possible choices of $(\beta_0,\beta_1,...,\beta_n)$.
\begin{prop}
Let $(\Omega,\A,\p)$ be a probability space. Let $X,Y\in L^1(\Omega,\A,\p)$ be two r.v.'s. Then
$$\inf_{(\beta_0,\beta_1,...,\beta_n)\in\R^{n+1}}\E[(X-(\beta_0+\beta_1Y_1+...+\beta_nY_n))^2]=\E[(X-Z)^2],$$
where $Z=\E[X]+\sum_{j=1}^n\alpha_j(Y_j-\E[Y_j])$ and the $\alpha_j$'s are solutions to the system
$$\sum_{j=1}^n\alpha_jCov(Y_j,Y_k)=Cov(X,Y_k)_{1\leq k\leq n}.$$
In particular if $K_Y$ is invertible, we have $\alpha=Cov(X,Y)K_Y^{-1}$, where 
$$Cov(X,Y)=\begin{pmatrix}Cov(X,Y_1)\\ \vdots \\ Cov(X,Y_n)\end{pmatrix}$$.
\end{prop}
\begin{proof}
Let $H$ be the linear subspace of $L^2(\Omega,\A,\p)$ spanned by $\{1,Y_1,...,Y_n\}$. Then we know that the r.v. $Z$, which minimizes
$$\|X-U\|_2=\E[(X-U)^2]$$
for $U\in H$, is the orthogonal projection of $X$ on $H$. We can thus write
$$Z=\alpha_0+\sum_{j=1}^n\alpha_j(Y_j-\E[Y_j]).$$
The orthogonality of $X-Z$ to $H$ can be written as $\E[(X-Z)\cdot 1]=0$. Therefore $\E[X]=\E[Z]$ and thus $\alpha_0=\E[X]$. Moreover, we get $\E[(X-Z)(Y_k-\E[Y_k])]=0$, which implies that for all $k\in\{1,...,n\}$ we get $\E[(X-\E[X]+\E[Z]-Z)]\cdot(Y_k-\E[Y_k])=0$.
\end{proof}
\begin{rem}
When $n=1$, we have 
$$Z=\E[X]+\frac{Cov(X,Y)}{Var(Y)}(Y-\E[Y]).$$
\end{rem}
\section{The Characteristic function}
\begin{defn}[Characteristic function]
Let $(\Omega,\A,\p)$ be a probability space. Let $X$ be a r.v. with values in $\R^d$, i.e. $X:(\Omega,\A,\p)\to\R^d$. Then we can look at the characteristic function of $X$, which is given by the Fourier transform
$$\Phi_X:\R^d\to\C,\hspace{0.5cm}\xi\mapsto \Phi_X(\xi)=\E\left[e^{i\langle \xi,X\rangle}\right]=\int_{\R^d}e^{i\langle \xi, x\rangle}d\p_X(x)$$
where $\langle\xi,X\rangle=\sum_{k=1}^d\xi_kX_k$.
\end{defn}
\begin{rem}
For $d=1$ and $\xi\in\R$, we get 
$$\Phi_X(\xi)=\E\left[e^{i\xi X}\right]=\int_\R e^{i\xi x}d\p_X(x).$$
\end{rem}
\begin{rem}
$\Phi_X(\xi)$ is continuous on $\R^d$ and bounded. For boundedness note 
$$\vert\Phi_X(\xi)\vert=\left\vert\E\left[e^{i\langle\xi,X\rangle}\right]\right\vert\leq \E\left[\underbrace{\left\vert e^{i\langle\xi,X\rangle}\right\vert}_{=1}\right]\leq 1.$$
Moreover, we know that $e^{i\langle\xi,x\rangle}$ is a continuous function of $\xi\in\R^d$ for every $x\in\R^d$.
$$\left\vert e^{i\langle\xi,X\rangle}\right\vert\leq 1,\hspace{0.2cm}\E[1]=1<\infty.$$
So it follows that $\Phi_X(\xi)$ is a continous function of $\xi$.
\end{rem}
\begin{thm}
The characteristic function uniquely characterizes probability distributions, meaning that for two r.v.'s $X$ and $Y$ satisfying
$$\Phi_X(\xi)=\Phi_Y(\xi),$$
for all $\xi\in\R^d$, we get that  
$$\p_X=\p_Y.$$
\end{thm}
\begin{proof}
No proof here.
\end{proof}
\begin{lem}
Let $X$ be a r.v. which is $\mathcal{N}(0,\sigma^2)$ distributed. Then 
$$\Phi_X(\xi)=\exp\left(-\frac{\sigma^2\xi^2}{2}\right),\hspace{0.2cm}\xi\in\R.$$ 
\end{lem}
\begin{proof}
According to the formula we get
$$\Phi_X(\xi)=\int_{-\infty}^{\infty}e^{i\xi x}e^{-\frac{x^2}{2\sigma^2}}\frac{dx}{\sigma\sqrt{2\pi}}.$$
Assume for simplicity $\sigma=1$ (change of variables: $\xi=\frac{x}{\sigma}$). Therefore we have
$$\Phi_X(\xi)=\int_{-\infty}^\infty\frac{1}{\sqrt{2\pi}} e^{-\frac{x^2}{2}}\cos(\xi x)dx+\underbrace{i\int_{-\infty}^\infty\frac{1}{\sqrt{2\pi}}e^{-\frac{x^2}{2}}\sin(\xi x)dx}_{0,\hspace{0.1cm}\text{by parity}}=\int_{-\infty}^\infty \frac{1}{\sqrt{2\pi}}e^{-\frac{x^2}{2}}\cos(\xi x)dx.$$
Hence we have
$$\frac{d\Phi_X}{d\xi}(\xi)=-\int_{-\infty}^{\infty}\frac{1}{\sqrt{2\pi}}xe^{-\frac{x^2}{2}}\sin(\xi x)dx.$$
We have used the fact that $e^{-\frac{x^2}{2}}\cos(\xi x)$ is $C^\infty$ in both variables and that $\left\vert x\sin(\xi x)e^{-\frac{x^2}{2}}\right\vert\leq \vert x\vert e^{-\frac{x^2}{2}}$, which is integrable on $\R$. Using integration by parts we get 
$$\frac{d\Phi_X}{d\xi}(\xi)=\underbrace{\left[\frac{1}{\sqrt{2\pi}}e^{-\frac{x^2}{2}}\sin(\xi x)\right]_{-\infty}^\infty}_{=0}-\xi\int_{-\infty}^\infty \frac{1}{\sqrt{2\pi}} e^{-\frac{x^2}{2}}\cos(\xi x)dx.$$
So we have the following Cauchy-problem
$$\begin{cases}\frac{d\Phi_X}{d\xi}(\xi)=-\xi \Phi_X(\xi)\\ \Phi_X(0)=1\end{cases}$$
Solving the differential equation, we get
$$\Phi_X(\xi)=e^{-\frac{-\xi^2}{2}}.$$
\end{proof}
\begin{prop}
Let $X=(X_1,...,X_d)\in\R^d$ such that $\E[\vert X\vert^2]<\infty$, where $\vert\cdot\vert$ denotes the euclidean norm. Then 
$$\lim_{\vert\xi\vert\to0}\Phi_X(\xi)=1+i\sum_{j=1}^d\E[X_j]\xi_j-\frac{1}{2}\sum_{j=1}^d\sum_{k=1}^d\xi_j\xi_k\E[X_jX_k]+o(\vert\xi\vert^2),$$
\end{prop}
\begin{proof}
Note that we can write
$$\frac{\partial\Phi_X(\xi)}{\partial\xi_j}=i\E\left[X_je^{i\langle\xi,X\rangle}\right].$$
This follows from the differentiation under the integral sign with $\left\vert X_je^{i\langle \xi,X\rangle}\right\vert\leq \vert X_j\vert$, which is integrable. Since 
$$\E[\vert X_iX_k\vert]\leq \E[\vert X_j\vert^2]^\frac{1}{2}\E[\vert X_k\vert^2]^{\frac{1}{2}}<\infty,$$
we have 
$$\frac{\partial^2\Phi_X(\xi)}{\partial\xi_j\partial \xi_k}=-\E\left[\underbrace{X_jX_ke^{i\langle\xi,X\rangle}}_{\leq \vert X_jX_k\vert\in L^1}\right].$$
Taking $\xi=0$ we get that $\frac{\partial\Phi_X}{\partial\xi_j}(0)=i\E[X_j]$ and $\frac{\partial^2\Phi_X}{\partial\xi_j\partial\xi_k}(0)=-\E[X_jX_k]$. we see that the equation in the proposition is the taylor-expansion at order 2 near 0 of the $C^2$-function $\Phi_X(\xi)$. 
\end{proof}
\begin{rem}
From the proof we see that when $d=1$, we have 
$$\E[\vert X\vert]<\infty\Longrightarrow\E[X]=i\Phi_X'(0)\hspace{0.5cm}\text{and}\hspace{0.5cm}\E[X^2]<\infty\Longrightarrow\E[X^2]=-\Phi_X''(0).$$
\end{rem}

\section{Independence}
\subsection{Independent events} Let $(\Omega,\A,\p)$ be a probability space. If $A,B\in\A$, we say that $A$ and $B$ are independent if 
$$\p[A\cap B]=\p[A]\p[B].$$
\begin{ex}[Throw of a die] We have the state space $\Omega=\{1,2,3,4,5,6\}$, $\omega\in\Omega$. Hence we have $\p[\{\omega\}]=\frac{1}{6}$. Now let $A=\{1,2\}$ and $B=\{1,3,5\}$. Then 
$$\p[A\cap B]=\p[\{1\}]=\frac{1}{6}\hspace{0.5cm}\text{and}\hspace{0.5cm}\p[A]=\frac{1}{3},\hspace{0.2cm}\p[B]=\frac{1}{2}$$
Therefore we get
$$\p[A\cap B]=\p[A]\p[B].$$
Hence we get that $A$ and $B$ are independent.
\end{ex}
\begin{defn}[Independence of events]
We say that the $n$ events $A_1,...,A_n\in\A$ are independent if $\forall \{j_1,...,j_l\}\subset\{1,...,n\}$ we have
$$\p[A_{j_1}\cap A_{j_2}\cap\dotsm \cap A_{j_l}]=\p[A_{j_1}]\dotsm \p[A_{j_l}].$$
\end{defn}
\begin{rem}
It is not enough to have $\p[A_1\cap\dotsm\cap A_n]=\p[A_1]\dotsm\p[A_n]$. It is also not enough to check that $\forall \{i,j\}\subset\{1,...,n\}$, $\p[A_i\cap A_j]=\p[A_i]\p[A_j]$. For instance, let us consider two tosses of a coin and consider events $A,B$ and $C$ given by
$$A=\{\text{$H$ at the first throw}\},\hspace{0.3cm}B=\{\text{$T$ at the first throw}\},\hspace{0.3cm}C=\{\text{same outcome for both tosses}\}$$
The events $A,B$ and $C$ are two by two independent but $A,B$ and $C$ are not independent events.
\end{rem}
\begin{prop}
The $n$ events $A_1,...,A_n\in\A$ are independent if and only if 
$$(*)\hspace{0.3cm}\p[B_1\cap\dotsm \cap B_n]=\p[B_1]\dotsm\p[B_n]\hspace{0.2cm}$$
for all $B_i\in\sigma(A_i)=\{\emptyset,A_i,A_i^C,\Omega\}$, $\forall i\in\{1,...,n\}$.
\end{prop}

\begin{proof}
If the above is satisfied and if $\{j_1,...,j_l\}\subset\{1,...,n\}$, then for $i\in\{j_1,...,j_l\}$ take $B_i=A_i$ and for $i\not\in\{j_1,...,j_l\}$ take $B_i=\Omega$. So it follows that 
$$\p[A_{j_1}\cap\dotsm \cap A_{j_l}]=\p[A_{j_1}]\dotsm\p[A_{j_l}].$$
Conversely, assume that $A_1,...,A_n\in\A$ are independent and we want to deduce $(*)$. We can assume that $\forall i\in\{1,...,n\}$ we have $B_i\not=\emptyset$ (for otherwise the identity is trivially satisfied). If $\{j_1,...,j_l\}=\{i\mid B_i\not=\Omega\}$, we have to check that 
$$\p[B_{j_1}\cap\dotsm\cap B_{j_l}]=\p[B_{j_1}]\dotsm\p[B_{j_l}],$$
as soon as $B_{j_k}=A_{j_k}$ or $B_{j_k}=A_{j_k}^C$. Finally it's enough to show that if $C_1,...,C_p$ are independent events, then 
$$C_1^C,C_2,...,C_p$$
are also independent. But if $1\not\in\{i_1,...,i_q\}$, for all $\{i_1,...,i_q\}\subset\{1,...,p\}$, then from the definition of independence we have 
$$\p[C_{i_1}\cap\dotsm\cap C_{i_q}]=\p[C_{i_1}]\dotsm\p[C_{i_q}].$$
If $1\in\{i_1,...,i_q\}$, say $1=i_1$, then 
\begin{align*}
\p[C_{i_1}^C\cap C_{i_2}\cap\dotsm\cap C_{i_q}]&=\p[C_{i_1}\cap\dotsm\cap C_{i_q}]-\p[C_1\cap C_{i_2}\cap\dotsm\cap C_{i_q}]\\
&=\p[C_{i_2}]\dotsm\p[C_{i_q}]-\p[C_1]\p[C_{i_2}]\dotsm\p[C_{i_q}]\\
&=(1-\p[C_1])\p[C_{i_2}]\dotsm\p[C_{i_q}]=\p[C_1^C]\p[C_{i_2}]\dotsm\p[C_{i_q}]
\end{align*}
\end{proof}
\begin{defn}[Conditional probability]
Let $(\Omega,\A,\p)$ be a probability space. Let $A,B\in\A$ such that $\p[B]>0$. The conditional probability of $A$ given $B$ is then defined as
$$\p[A\mid B]=\frac{\p[A\cap B]}{\p[B]}.$$
\end{defn}
\begin{thm}
Let $(\Omega,\A,\p)$ be a probability space. Let $A,B\in\A$ and suppose that $\p[B]>0$.
\begin{enumerate}[$(i)$]
\item{$A$ and $B$ are independent if and only if 
$$\p[A\mid B]=\p[A].$$
}

\item{The map 
$$\A\to [0,1],\hspace{0.5cm}A\mapsto \p[A\mid B]$$
defines a new probability measure on $\A$ called the conditional probability given $B$.
}
\end{enumerate}
\end{thm}
\begin{proof} We need to show both points.
\begin{enumerate}[$(i)$]
\item{If $A$ and $B$ are independent, then 
$$\p[A\mid B]=\frac{\p[A\cap B]}{\p[B]}=\frac{\p[A]\p[B]}{\p[B]}=\p[A]$$
and conversely if $\p[A\mid B]=\p[A]$, we get that $$\p[A\cap B]=\p[A]\p[B],$$
and hence $A$ and $B$ are independent.
}
\item{Let $\Q[A]=\p[A\mid B]$. We have 
$$\Q[\Omega]=\p[\omega\mid B]=\frac{\p[\omega\cap B]}{\p[B]}=\frac{\p[B]}{\p[B]}=1.$$
Take $(A_n)_{n\geq 1}\subset \A$ as a disjoint family of events. Then 
\begin{align*}
\Q\left[\bigcup_{n\geq 1}A_n\right]&=\p\left[\bigcup_{n\geq 1}A_n\mid B\right]=\frac{\p\left[\left(\bigcup_{n\geq 1}A_n\right)\cap B\right]}{\p[B]}=\p\left[\bigcup_{n\geq 1}(A_n\cap B)\right]\\
&=\sum_{n\geq 1}\frac{\p[A_n\cap B]}{\p[B]}=\sum_{n\geq 1}\Q[A_n].
\end{align*}
}
\end{enumerate}
\end{proof}
\begin{thm}
Let $(\Omega,\A,\p)$ be a probability space. Let $A_1,...,A_n\in\A$ with $\p[A_1\cap\dotsm\cap A_n]>0$. Then 
$$\p[A_1\cap\dotsm\cap A_n]=\p[A_1]\p[A_2\mid A_1]\p[A_3\mid A_1\cap A_2]\dotsm\p[A_n\mid A_1\cap\dotsm\cap A_{n-1}].$$
\end{thm}
\begin{proof}
We prove this by induction. For $n=2$ it's just the definition of the conditional probability. Now we want to go from $n-1$ to $n$. Therefore set $B=A_1\cap \dotsm\cap A_{n-1}$. Then 
$$\p[B\cap A_n]=\p[A_n\mid B]\p[B]=\p[A_n\mid B]\p[A_1]\p[A_\mid A_1]\dotsm\p[A_{n-1}\mid A_1\cap\dotsm\cap A_{n-2}].$$
\end{proof}
\begin{thm}
Let $(\Omega,\A,\p)$ be a probability space. Let $\left(E_{n}\right)_{n\geq 1}$ be a finite or countable measurable partition of $\Omega$, such that $\p[E_n]>0$ for all $n$. If $A\in\A$, then 
$$\p[A]=\sum_{n\geq 1}\p[A\mid E_n]\p[E_n].$$
\end{thm}
\begin{proof} Note that
$$A=A\cap\Omega=A\cap\left(\bigcup_{n\geq 1}E_n\right)=\bigcup_{n\geq 1}(A_n\cap E_n).$$
Now since the $(A\cap E_n)_{n\geq 1}$ are disjoint, we can write
$$\p[A]=\sum_{n\geq 1}\p[A\cap E_n]=\sum_{n\geq 1}\p[A\mid E_n]\p[E_n].$$
\end{proof}
\begin{thm}[Baye]
Let $(\Omega,\A,\p)$ be a probability space. Let $(E_n)_{n\geq 1}$ be a finite or countable partition of $\Omega$ and assume that $\p[A]>0.$ Then 
$$\p[E_n\mid A]=\frac{\p[A\mid E_n]\p[E_n]}{\sum_{n\geq 1}\p[A\mid E_n]\p[E_n]}.$$
\end{thm}
\begin{proof}
By the previous theorem we know that 
$$\p[A]=\sum_{n\geq 1}\p[A\mid E_n]\p[E_n],\hspace{0.3cm}\p[E_n\mid A]=\frac{\p[E_n\cap A]}{\p[A]},\hspace{0.3cm}\p[A\mid E_n]=\frac{\p[A\cap E_n]}{\p[E_n]}$$
Therefore, combining things, we get 
$$\p[E_n\mid A]=\frac{\p[E_n\cap A]}{\p[A]}=\frac{\p[A\mid E_n]\p[E_n]}{\sum_{n\geq 1}\p[A\mid E_n]\p[E_n]}.$$
\end{proof}
\subsection{Independent Random Variables and independent $\sigma$-Algebras}
\begin{defn}[Independence of $\sigma$-Algebras]
Let $(\Omega,\A,\p)$ be a probability space. We say that the sub $\sigma$-Algebras $\B_1,...,\B_n$ of $\A$ are independent if for all $ A_1\in\B_1,..., A_n\in\B_n$ we get 
$$\p[A_1\cap\dotsm \cap A_n]=\p[A_1]\dotsm\p[A_n].$$
Let now $X_1,...,X_n$ be $n$ r.v.'s with values in measureable spaces $(E_1,\mathcal{E}_1),...,(E_n,\mathcal{E}_n)$ respectively. We say that the r.v.'s $X_1,...,X_n$ are independent if the $\sigma$-Algebras $\sigma(X_1),...,\sigma(X_n)$ are independent. This is equivalent to the fact that for all $F_1\in\mathcal{E}_1,...,F_n\in\mathcal{E}_n$ we have
$$\p[\{X_1\in F_1\}\cap\dotsm\cap\{X_n\in F_n\}]=\p[X_1\in F_1]\dotsm \p[X_n\in F_n].$$
(This comes from the fact that for all $i\in\{1,...,n\}$ we have that $\sigma(X_i)=\{X_i^{-1}(F)\mid F\in\mathcal{E}_i\}$)
\end{defn}
\begin{rem}
If $\B_1,...,\B_n$ are $n$ independent sub $\sigma$-Algebras and if $X_1,...,X_n$ are independent r.v.'s such that $X_i$ is $\B_i$ measurable for all $i\in\{1,...,n\}$, then $X_1,...,X_n$ are independent r.v.'s (This comes from the fact that for all $i\in\{1,...,n\}$ we have that $\sigma(X_i)\subset \B_i$).
\end{rem}
\begin{rem}
The $n$ events $A_1,...,A_n\in\A$ are independent if and only if $\sigma(A_1),...,\sigma(A_n)$ are independent.
\end{rem}
\begin{thm}[Independence of Random Variables]
\label{thm7}
Let $(\Omega,\A,\p)$ be a probability space. Let $X_1,...,X_n$ be $n$ r.v.'s. Then $X_1,...,X_n$ are independent if and only if the law of the vector $(X_1,...,X_n)$ is the product of the laws of $X_1,...,X_n$, i.e.
$$\p_{(X_1,...,X_n)}=\p_{X_1}\otimes\dotsm\otimes \p_{X_n}.$$
Moreover, for every measurable map $f_i:(E_i,\mathcal{E}_i)\to\R_+$ defined on a measurable space $(E_i,\mathcal{E}_i)$ for all $i\in\{1,...,n\}$, we have 
$$\E\left[\prod_{i=1}^nf_i(X_i)\right]=\prod_{i=1}^n\E[f_i(X_i)].$$
\end{thm}
\begin{proof}
Let $F_i\in\mathcal{E}_i$ for all $i\in\{1,...,n\}$. Thus we have
$$\p_{(X_1,...,X_n)}(F_1\times\dotsm \times F_n)=\p[\{X_1\in F_1\}\cap\dotsm\cap\{X_n\in F_n\}]$$
and on the other hand 
$$\left(\p_{X_1}\otimes\dotsm\otimes\p_{X_n}\right)(F_1\times\dotsm\times F_n)=\p_{X_1}[F_1]\dotsm\p_{X_n}[F_n]=\prod_{i=1}^n\p_{X_i}[F_i]=\prod_{i=1}^n\p[X_i\in F_i].$$
If $X_1,...,X_n$ are independent, then $$\p_{(X_1,...,X_n)}(F_1\times\dotsm \times F_n)=\prod_{i=1}^n\p[X_i\in F_i]=\left(\p_{X_1}\otimes\dotsm\otimes \p_{X_n}\right)(F_1\times\dotsm\times F_n),$$ which implies that $ \p_{(X_1,...,X_n)}$ and $\p_{X_1}\otimes\dotsm\otimes \p_{X_n}$ are equal on rectangles. Hence the monotone class theorem implies that 
$$\p_{(X_1,...,X_n)}=\p_{X_1}\otimes\dotsm\otimes\p_{X_n}.$$
Conversely, if $\p_{(X_1,...,X_n)}=\p_{X_1}\otimes\dotsm\otimes\p_{X_n}$, then for all $F_i\in\mathcal{E}_i$, with $i\in\{1,...,n\}$, we get that
$$\p_{(X_1,...,X_n)}(F_1\times\dotsm\times F_n)=\left(\p_{X_1}\otimes\dotsm\otimes\p_{X_n}\right)(F_1\times\dotsm\times F_n)$$
and therefore
$$\p[\{X_1\in F_1\}\cap\dotsm\cap\{X_n\in F_n\}]=\p[X_1\in F_1]\dotsm\p[X_n\in F_n].$$
This implies that $X_1,...,X_n$ are independent. For the second assumption we get  
$$\E\left[\prod_{i=1}^nf_i(X_i)\right]=\int_{E_1\times\dotsm\times E_n}\prod_{i=1}^nf_i(X_i)\underbrace{P_{X_1}dx_1\dotsm P_{X_n}dx_n}_{\p_{X_1,...,X_n}(dx_1\dotsm dx_n)}=\prod_{i=1}^n\int_{E_i}f_i(x_i)P_{X_i}dx_i=\prod_{i=1}^n\E[f_i(X_i)],$$
where we have used the first part and Fubini's theorem.
\end{proof}
\begin{rem}
We see from the proof above that as soon as for all $i\in\{1,...,n\}$ we have $\E[\vert f_i(X_i)\vert]<\infty$, it follows that 
$$\E\left [\prod_{i=1}^n f_i(X_i)\right]=\prod_{i=1}^n\E[ f_i(X_i) ].$$
Indeed, the previous result shows that 
$$\E\left[\prod_{i=1}^n\vert f_i(X_i)\vert\right]=\prod_{i=1}^n\E[\vert f_i(X_i)\vert]<\infty$$
and thus we can apply Fubini's theorem. In particular if $X_1,...,X_n\in L^1(\Omega,\A,\p)$ and independent, we get that
$$\E\left[\prod_{i=1}^nX_i\right]=\prod_{i=1}^n\E[X_i].$$
\end{rem}
\begin{cor}
Let $(\Omega,\A,\p)$ be a probability space. Let $X_1$ and $X_2$ be two independent r.v.'s in $L^2(\Omega,\A,\p)$. Then we get 
$$Cov(X_1,X_2)=0.$$
\end{cor}
\begin{proof}
Recall that if $X\in L^2(\Omega,\A,\p)$, we also have that $X\in L^1(\Omega,\A,\p)$. Thus
$$Cov(X_1,X_2)=\E[X_1X_2]-\E[X_1]\E[X_2]=\E[X_1]\E[X_2]-\E[X_1]\E[X_2]=0.$$
\end{proof}
\begin{rem}
Note that the converse is not true! Let $X_1\sim\mathcal{N}(0,1)$. We can also take for $X_1$ any symmetric r.v. in $L^2(\Omega,\A,\p)$ with density $P(x)$, such that $P(-x)=P(x)$. Recall that being in $L^2(\Omega,\A,\p)$ simply means 
$$\E[X^2]=\int_\R x^2 P(x)dx<\infty,$$
which implies that $P(x)=P(-x)$ and thus $\E[X^2]=\int_\R x^2P(x)dx=0.$
Now consider a r.v. $Y$ with values in $\{-1,+1\}$. Then we get $\p[Y=1]=\p[Y=-1]=\frac{1}{2}$ and thus $Y$ is independent of $X_1$. Define $X_2:=YX_1$ and observe then 
$$Cov(X_1,X_2)=\E[X_1X_2]-\E[X_1]\E[X_2]=\E[YX_1^2]-\E[YX_1]\E[X_1]$$
and hence
$$\E[Y]\E[X_1^2]-\E[Y]\E^2[X_1]=0-0=0.$$
If $X_1$ and $X_2$ are independent, we note that $\vert X_1\vert$ and $\vert X_2\vert$ would also be independent. But $\vert X_2\vert =\vert Y\vert \vert X_1\vert=\vert X_1\vert$. This would mean that $\vert X_1\vert$ is independent of itself. So it follows that $\vert X_1\vert$ is equal to a constant a.s. If $c=\E[\vert X_1\vert]$, and we want to look at $\E[(\vert X_1\vert-c)^2]$, we now know that $\vert X_1\vert -c$ is independent of itself. Therefore we get
$$\E[(\vert X_1\vert-c)^2]=\E[\vert X_1\vert-c]\E[\vert X_1\vert-c]=0\Longrightarrow \vert X_1\vert=c\hspace{0.3cm}\text{a.s.}$$
This cannot happen since $\vert X_1\vert$ is the absolute value of a standard Gaussian distribution, which has a density given by 
$$P(x)=\frac{1}{\sqrt{2\pi}}e^{-\frac{x^2}{2}}.$$
\end{rem}
\begin{cor}
Let $(\Omega,\A,\p)$ be a probability space. Let $X_1,...,X_n$ be $n$ r.v.'s with values in $\R$.
\begin{enumerate}[$(i)$]
\item{Assume that for $i\in \{1,...,n\}$, $\p_{X_i}$ has density $P_i$ and that the r.v.'s $X_1,...,X_n$ are independent. Then the law of $(X_1,...,X_n)$ also has density given by $P(x_1,...,x_n)=\prod_{i=1}^nP_i(x_i)$.
}
\item{Conversely assume that the law of $(X_1,...,X_n)$ has density $P(x_1,...,x_n)=\prod_{i=1}^nq_i(x_i)$, where $q_i$ is Borel measurable and positive. Then the r.v.'s $X_1,...,X_n$ are independent and the law of $X_i$ has density $P_i=c_iq_i$, with $c_i>0$ for $i\in\{1,...,n\}$.
}
\end{enumerate}
\end{cor}
\begin{proof}
We only need to show $(ii)$. From Fubini we get 
$$\int_\R\prod_{i=1}^nq_i(x_i)dx_i=\prod_{i=1}^n\int_\R q_i(x_i)dx_i=\int_{\R^{n}}P(x_1,...,x_n)dx_1\dotsm dx_n=1.$$
which implies that $K_i:=\int_\R q_i(x_i)dx_i\in(0,\infty)$, for all $i\in\{1,...,n\}$. Now we know that the law of $X_i$ has density $P_i$ given by 
$$P_i(x_i)=\int_{\R^{n-1}}P(x_1,...,x_{i-1},x_i,x_{i+1},...,x_n)dx_1\dotsm dx_{i-1}dx_{i+1}\dotsm dx_n=\left(\prod_{j\not=i}K_j\right)q_i(x_i)=\frac{1}{K_i}q_i(x_i).$$
We can rewrite 
$$P(x_1,...,x_n)=\prod_{i=1}^nq_i(x_i)=\prod_{i=1}^nP_i(x_i).$$
Hence we get $P(x_1,...,x_n)=\p_{X_1}\otimes\dotsm \otimes \p_{X_n}$ and therefore $X_1,...,X_n$ are independent.
\end{proof}
\begin{ex} Let $U$ be a r.v. with exponential distribution. Let $V$ be a uniform r.v. on $[0,1]$. We assume that $U$ and $V$ are independent. Define the r.v.'s $X=\sqrt{U}\cos(2\pi V)$ and $Y=\sqrt{U}\sin(2\pi V)$. Then $X$ and $Y$ are independent. Indeed, for a measurable function $\varphi:\R^2\to \R_+$ we get
$$\E[\varphi(X,Y)]=\int_0^\infty\int_{0}^1\varphi(\sqrt{u}\cos(2\pi v),\sqrt{u}\sin(2\pi v))e^{u}dudv$$
$$=\frac{1}{\sqrt{\pi}}\int_{0}^\infty\int_0^{2\pi}\varphi(r\cos(\theta),r\sin(\theta))re^{-r^2}drd\theta,$$
which implies that $(X,Y)$ has density $\frac{e^{-x^2}e^{-y^2}}{\pi}$ on $\R\times\R$. With the previous corollary we get that $X$ and $Y$ are independent and $X$ and $Y$ have the same density $P(x)=\frac{1}{\sqrt{\pi}}e^{-x^2}$. This means that $X$ and $Y$ are independent. 
\end{ex}
\begin{rem}
We write $X\stackrel{law}{=}Y$ to say that $\p_X=\p_Y$. Thus in the example above we would have $$X\stackrel{law}{=}Y\sim\mathcal{N}(0,\frac{1}{2}).$$
\end{rem}
\subsubsection{Important facts}
Let $X_1,...,X_n$ be $n$ real valued r.v.'s. Then the following are equivalent
\begin{enumerate}[$(i)$]
\item{$X_1,...,X_n$ are independent.
}
\item{For $X=(X_1,...,X_n)\in\R^n$ we have
$$\Phi_X(\xi_1,...,\xi_n)=\prod_{i=1}^n\Phi_{X_i}(\xi_i).$$
}
\item{For all $a_1,..,a_n\in\R$, we have
$$\p[X_1\leq a_1,..,X_n\leq a_n]=\prod_{i=1}^n\p[X_i\leq a_i]$$ 

}
\item{If $f_1,...,f_n:\R\to\R_+$ are continuous, measurable maps with compact support, then
$$\E\left[\prod_{i=1}^nf_i(X_i)\right]=\prod_{i=1}^n\E[f_i(X_i)].$$
}
\end{enumerate}
\begin{proof}
First we show $(i)\Longrightarrow (ii)$. By definition and the iid property, we get 
$$\Phi_X(\xi_1,..,\xi_n)=\E\left[e^{i(\xi_1X_1+...+\xi_nX_n)}\right]=\E\left[e^{i\xi_1X_1}\dotsm e^{i\xi_nX_n}\right]=\prod_{i=1}^n\E[e^{i\xi X_1}]=\prod_{i=1}^n\Phi_{X_i}(\xi_i),$$ 
where the map $t\mapsto e^{it}$ is measurable and bounded.
Next we show $(ii)\Longrightarrow (i)$. Note that by theorem \ref{thm7} we have $\p_X=\p_Y$ if 
$$\Phi_X(\xi_1,...,\xi_n)=\Phi_Y(\xi_1,...,\xi_n).$$
Now if $\Phi_X(\xi_1,...,\xi_n)=\prod_{i=1}^n\Phi_{X_i}(\xi_i)$, we note that $\prod_{i=1}^n\Phi_{X_i}(\xi_i)$
is the characteristic function of the probability distribution if the probability distribution is $\p_{X_1}\otimes\dotsm \otimes\p_{X_n}$. Now from injectivity it follows that $\p_{(X_1,...,X_n)}=\p_{X_1}\otimes\dotsm\otimes\p_{X_n}$, which implies that $X_1,...,X_n$ are independent.
\end{proof}
\begin{prop}
Let $(\Omega,\A,\p)$ be a probability space. Let $\B_1,...,\B_n\subset\A$ be sub $\sigma$-Algebras of $\A$. For every $i\in\{1,...,n\}$, let $\mathcal{C}_i\subset\B_i$ be a family of subsets of $\Omega$ such that  $\mathcal{C}_i$ is stable under finite intersection and $\sigma(\mathcal{C}_i)=\B_i$. Assume that for all $C_i\in\mathcal{C}_i$ with $i\in\{1,...,n\}$ we have
$$\p\left[\prod_{i=1}^nC_i\right]=\prod_{i=1}^n\p[C_i].$$
Then $\B_1,...,\B_n$ are independent $\sigma$-Algebras.
\end{prop}
\begin{proof}
Let us fix $C_2\in \mathcal{C}_2,...,C_n\in\mathcal{C}_n$ and define 
$$M_1:=\left\{B_1\in\B_1\mid \p[B_1\cap C_2\cap\dotsm\cap C_2]=\p[B_1]\p[C_2]\dotsm \p[C_n]\right\}.$$
Now since $\mathcal{C}_1\subset M_1$ and $M_1$ is a monotone class, we get $\sigma(\mathcal{C}_1)=\B_1\subset M_1$ and thus $\B_1=M_1$. Let now $B_1\in\B_1,$ $C_3\in\mathcal{C}_3,...,C_n\in\mathcal{C}_n$ and define
$$M_2:=\{B_2\in\B_2\mid \p[B_2\cap B_1\cap C_3\cap\dotsm\cap C_n]=\p[B_2]\p[B_1]\p[C_3]\dotsm\p[C_n]\}.$$
Again, since $\mathcal{C}_2\subset M_2$, we get $\sigma(\mathcal{C}_2)=\B_2\subset M_2$ and thus $B_2=M_2$. By induction we complete the proof.
\end{proof}
$Consequence:$ Let $\B_1,...,\B_n$ be $n$ independent $\sigma$-Algebras and let $m_0=0<m_1<...<m_p=n$. Then the $\sigma$-Algebras
\begin{align*}
\mathcal{D}_1&=\B_1\lor\dotsm\lor\B_n=\sigma(\B_1,...,\B_n)=\sigma\left(\bigcup_{k=1}^n\B_k\right)\\
\mathcal{D}_2&=\B_{m_i+1}\lor\dotsm\lor\B_{n_2}\\
\vdots\\
\mathcal{D}_p&=\B_{n_{p-1}+1}\lor\dotsm\lor\B_{n_p}
\end{align*}
are also independent. Indeed, we can apply The previous proposition to the class of sets 
$$C_j=\{B_{n_{j-1}+1}\cap\dotsm\cap B_{n_j}\mid B_i\in\mathcal{C}_i, i\in\{n_{j-1}+1,...,n_j\}\}.$$
In particular if $X_1,...,X_n$ are independent r.v.'s, then 
\begin{align*}
Y_1&=(X_1,...,X_n)\\
\vdots\\
Y_p&=(X_{n_{p_1}},...,X_{n_p})
\end{align*}
are also independent.
\begin{ex} Let $X_1,...,X_4$ be real valued independent r.v.'s. Then $Z_1=X_1X_3$ and $Z_2=X_2^3+X_4$ are independent and $Z_3=\sigma(X_1,X_3)$ and $Z_4=\sigma(X_2,X_4)$ are measurable.
From above $\sigma(X_1,X_3)$ and $\sigma(X_2,X_4)$ are independent if for $X:\Omega\to\R$ we have that $Y$ is $\sigma(X)$ measurable if and only if $Y=f(X)$ with $f$ being a measurable map, i.e. if $Y$ is $\sigma(X_1,...,X_n)$ measurable, then $Y=f(X_1,....,X_n)$.
\end{ex}
\begin{prop}[Independence for an infinite family]
Let $(\Omega,\A,\p)$ be a probability space. Let $(\B_i)_{i\in I}$ be an infinite family of sub $\sigma$-Algebras of $A$. We say that the family $(\B_i)_{i\in I}$ is independent if for all $\{i_1,..,i_p\}\in I$, $\B_{i_1},...,\B_{i_p}$ are independent. If $(X_i)_{i\in I}$ is a family of r.v.'s we say that they are independent if $(\sigma(X_i))_{i\in I}$ is independent.
\end{prop}
\begin{prop}
Let $(\Omega,\A,\p)$ be a probability space. Let $(X_n)_{n\geq 1}$ be a sequence of independent r.v.'s. Then for all $p\in\N$ we get that $p_1=\sigma(X_1,...,X_p)$ and $p_2=\sigma(X_{p+1},...,X_n)$ are independent.
\end{prop}
\begin{proof}
Apply Proposition 5.9. to $\mathcal{C}_1=\sigma(X_1,...,X_p)$ and 
$\mathcal{C}_2=\bigcup_{k=p+1}^\infty\sigma(X_{p+1},...,X_n)\in\B_2$.
\end{proof}
\subsection{The Borel-Cantelli Lemma}
Let $(\Omega,\A,\p)$ be a probability space. Let $(A_n)_{n\in\N}$ be a sequence of events in $\A$. Recall that we can write
\[
\limsup_{n\to \infty} A_n=\bigcap_{n=0}^\infty\left(\bigcup_{k=n}^\infty A_k\right)\hspace{0.5cm}\text{and}\hspace{0.5cm}\liminf_{n\to \infty} A_n=\bigcup_{n=0}^\infty\left(\bigcap_{k=n}^\infty A_k\right).
\]
Moreover, both are again measurable sets. For $\omega\in\limsup_n A_n$ we get that $\omega\in\bigcup_{k=n}^\infty A_k$, for all $n\geq 0$. Moreover, for all $n\geq 0$, there exists a $k\geq n$ such that, $\omega\in A_n$ and $\omega$ is in infinitely many $A_k$'s. For $\omega\in\liminf_n A_n$, we get that for all $n\geq 0$ such that $\omega\in\bigcap_{k=n}^\infty A_k$, there exists $n\geq 0$, such that for all $k\geq n$ we have $\omega\in A_k$, which shows that $\liminf_nA_n\subset \limsup_nA_n$.
\begin{lem}[Borel-Cantelli]
Let $(\Omega,\A,\p)$ be a probability space. Let $(A_n)_{n\in\N}\in\A$ be a family of measurable sets.
\begin{enumerate}[$(i)$]
\item{If $\sum_{n\geq 1}\p[A_n]<\infty$, then 
$$\p\left[\limsup_{n\to\infty} A_n\right]=0,$$
which means that the set $\{n\in\N\mid \omega\in A_n\}$ is a.s. finite.
}
\item{If $\sum_{n\geq 1}\p[A_n]=\infty$, and if the events $(A_n)_{n\in\N}$ are independent, then 
$$\p\left[\limsup_{n\to\infty} A_n\right]=1,$$
which means that the set $\{n\in\N\mid \omega\in A_n\}$ is a.s. finite.
}
\end{enumerate}
\end{lem}
\begin{proof} We need to show both points.
\begin{enumerate}[$(i)$]
\item{If $\sum_{n\geq 1}\p[A_n]<\infty,$ then, by Fubini, we get
$$\E\left[\sum_{n\geq 1}\one_{A_n}\right]=\sum_{n\geq 1}\p[A_n],$$
which implies that $\sum_{n\geq 1}\one_{A_n}<\infty$ and $\one_{A_n}\not=0$ a.s. for finite numbers of $n$.
}
\item{Fix $n_0\in\N$ and note that for all $n\geq n_0$ we have
$$\p\left[\bigcap_{k=n_0}^nA_k^C\right]=\prod_{k=n_0}^n\p[A_k^C]=\prod_{k=n_0}^n\p[1-A_n].$$
Now we see that 
$$\sum_{n\geq 1}\p[A_n]=\infty$$
and thus
$$\p\left[\bigcap_{k=n_0}^nA_k^C\right]=0.$$
Since this is true for every $n_0$ we have that 
$$\p\left[\bigcup_{n=0}^\infty\bigcap_{k=n_0}^\infty A_k^C\right]\leq \sum_{n\geq 1}\p[A_k^C]=0.$$
Hence we get
$$\p\left[\bigcup_{n=0}^\infty\bigcap_{k=n_0}^\infty A_k^C\right]=\p\left[\bigcap_{n=0}^\infty\bigcup_{k=n}^\infty A_k\right]=\p\left[\limsup_{n\to\infty} A_n\right]=1.$$
}
\end{enumerate}
\end{proof}
\subsubsection{Application 1} Let $(\Omega,\A,\p)$ be a probability space. There does not exist a probability measure on $\N$ such that the probability of the set of multiples of an integer $n$ is $\frac{1}{n}$ for $n\geq 1$. Let us assume that such a probability measure exists. Let $\tilde{p}$ denote the set of prime numbers. For $p\in\tilde{p}$ we note that $A_p=p\N$, i.e. the set of all multiples of $p$. We first show that the sets $(A_p)_{p\in\tilde{p}}$ are independent. Indeed let $p_1,...,p_n\in\tilde{p}$ be distinct. Then we have
$$\p[p_1\N\cap\dotsm\cap p_n\N]=\p[p_1,...,p_n\N]=\frac{1}{p_1\dotsm p_n}=\p[p_1\N]\dotsm\p[p_n\N].$$
Moreover it is known that 
$$\sum_{p\in\tilde{p}}\p[p\N]=\sum_{p\in\tilde{p}}\frac{1}{p}=\infty.$$
The second part of the Borel-Cantelli lemma implies that all integers $n$ belong to infinitely many $A_p$'s. So it follows that $n$ is divisible by infinitely many distinct prime numbers.
\subsubsection{Application 2} Let $(\Omega,\A,\p)$ be a probability space. Let $X$ be an exponential r.v. with parameter $\lambda=1$. Thus we know that $X$ has density $e^{-x}\one_{\R_+}(x)$. Now consider a sequence $(X_n)_{n\geq 1}$ of independent r.v.'s with the same distribution as $X$, i.e. for all $n\geq 1$,we have $X_n\sim X$. Then $\limsup_n \frac{X_n}{\log(n)}=1$ a.s., i.e. there exists an $N\in\A$ such that $\p[N]=0$ and  for $\omega\not\in N$ we get 
$$\limsup_{n\to\infty} \frac{X_n(\omega)}{\log(n)}=1.$$
Therefore we can compute the probability
$$\p[X>t]=\int_t^\infty e^{-x}dx=e^{-t}.$$
Now let $\epsilon>0$ and consider the sets $A_n=\{X_n>(1+\epsilon)\log(n)\}$ and $B_n=\{X_n>\log(n)\}$. Then
$$\p[A_n]=\p[X_n>(1+\epsilon)\log(n)]=\p[X>(1+\epsilon)\log(n)]=e^{-(1+\epsilon)\log(n)}=\frac{1}{n^{1+\epsilon}}.$$
This implies that 
$$\sum_{n\geq 1}\p[A_n]<\infty.$$
With the Borel-Cantelli lemma we get that $\p\left[\limsup_{n\to\infty} A_n\right]=0$. Let us define 
$$N_{\epsilon}=\limsup_{n\to\infty} A_n.$$ Then we have $\p[N_\epsilon]=0$ for $\omega\not\in N_{\epsilon}$, which implies that there exists an $n_0(\omega)$ such that for all $n\geq n_0$ we have 
$$X_n(\omega)\leq (1+\epsilon)\log(n)$$ 
and thus for $\omega\not\in N_{\epsilon}$, we get $\limsup_{n\to\infty}\frac{X_n(\omega)}{\log(n)}\leq 1+\epsilon$. Moreover, let 
$$N'=\bigcup_{\epsilon\in \Q_+}N_{\epsilon}.$$ 
Therefore we get $\p[N']\leq \sum_{\epsilon\in\Q_+}\p[N_{\epsilon}]=0$ for $\omega\not\in N'$. Hence we get 
$$\limsup_{n\to\infty}\frac{X_n(\omega)}{\log(n)}\leq 1.$$
Now we note that the $B_n$'s are independent, since $B_n\in\sigma(X_n)$ and the fact that the $X_n$'s are independent. Moreover, 
$$\p[B_n]=\p[X_n>\log(n)]=\p[X>\log(n)]=\frac{1}{n},$$ 
which gives that
$$\sum_{n\geq 1}\p[B_n]=\infty.$$
Now we can use Borel-Cantelli to get 
$$\p\left[\limsup_{n\to\infty} B_n\right]=1.$$ 
If we denote $N''=\left(\limsup_{n\to\infty} B_n\right)^C$, then for $\omega\not\in N''$ we get that $X_n(\omega)>\log(n)$ for infinitely many $n$. So it follows that for $\omega\not\in N''$ we have 
$$\limsup_{n\to\infty}\frac{X_n(\omega)}{\log(n)}\geq 1.$$ 
Finally, take $N=N'\cup N''$ to obtain $\p[N]=0$. Thus for $\omega\not\in N$ we get
$$\limsup_{n\to\infty} \frac{X_n(\omega)}{\log(n)}=1.$$
\subsection{Sums of independent Random Variables}
Let us first define the convolution of two probability measures. If $\mu$ and $\nu$ are two probability measures on $\R^d$, we denote by $\mu*\nu$ the image of the measure $\mu\otimes\nu$ by the application 
$$\R^d\times\R^d\to\R^d,\hspace{1cm}(x,y)\mapsto x+y.$$
Moreover, for all measurable maps $\varphi:\R^d\to \R_+$, we have
$$\int_{\R^d}\varphi(z)(\mu*\nu)(dz)=\iint_{\R^d}\varphi(x+y)\mu(dx)\nu(dy).$$
\begin{prop}
Let $(\Omega,\A,\p)$ be a probability space. Let $X$ and $Y$ be two independent r.v.'s with values in $\R^d$. Then the following hold.
\begin{enumerate}[$(i)$]
\item{The law of $X+Y$ is given by $\p_X*\p_Y$. In particular if $X$ has density $f$ and $Y$ has density $g$, then $X+Y$ has density $f*g$, where $*$ denotes the convolution product, which is given by 
\[
f*g(\xi)=\int_{\R^d} f(x)g(\xi-x)dx.
\]
}
\item{
$\Phi_{X+Y}(\xi)=\Phi_X(\xi)\Phi_Y(\xi).$
}
\item{If $X$ and $Y$ are in $L^2(\Omega,\A,\p)$, we get
$$K_{X+Y}=K_X+K_Y.$$
In particular when $d=1$, we obtain 
$$Var(X+Y)=Var(X)+Var(Y).$$
}
\end{enumerate}
\end{prop}
\begin{proof} We need to show all three points.
\begin{enumerate}[$(i)$] 
\item{If $X$ and $Y$ are independent r.v.'s, then $\p_{(X,Y)}=\p_X\otimes\p_Y$. Consequently, for all measurable maps $\varphi:\R^d\to\R_+$, we have 
\begin{multline*}
\E[\varphi(X+Y)]=\iint_{\R^d}\varphi(X+Y)\p_{(X,Y)}(dxdy)=\iint_{\R^d}\varphi(X+Y)\p_X(dx)\p_{Y}(dy)\\=\int_{\R^d}\varphi(\xi)(\p_X*\p_Y)(d\xi).
\end{multline*}
Now since $X$ and $Y$ have densities $f$ and $g$ respectively, we get
$$\E[\varphi(Z=X+Y)]=\iint_{\R^d}\varphi(X+Y)f(x)*g(y)dxdy=\iint_{\R^d}\varphi(\xi)\left(\int_{\R^d} f(x)g(\xi-x)dx\right)d\xi.$$
Since this identity here is true for all measurable maps $\varphi:\R^d\to\R_+$, the r.v. $Z:=X+Y$ has density 
$$h(\xi)=(f*g)(\xi)=\int_{\R^d}f(x)g(\xi-x)dx.$$
}
\item{By definition of the characteristic function and the independence property, we get
\[
\Phi_{X+Y}(\xi)=\E\left[e^{i\xi(X+Y)}\right]=\E\left[e^{i\xi X}e^{i\xi Y}\right]=\E\left[e^{i\xi X}\right]\E\left[e^{i\xi Y}\right]=\Phi_X(\xi)\Phi_Y(\xi).
\]
}
\item{If $X=(X_1,...,X_d)$ and $Y=(Y_1,...,Y_d)$ are independent r.v.'s on $\R^d$, we get that $Cov(X_i,Y_j)=0$, for all $0\leq  i,j\leq  d$. By using the multi linearity of the covariance we get that 
$$Cov(X_i+Y_i,X_j+Y_j)=Cov(X_i,X_j)+Cov(Y_j+Y_j),$$
and hence $K_{X+Y}=K_X+K_Y$. For $d=1$ we get 
\begin{align*}
Var(X+Y)&=\E[((X+Y)-\E[X+Y])^2]=\E[((X-\E[X])+(Y-\E[Y]))^2]\\
&=\underbrace{\E[(X-\E[X])^2]}_{Var(X)}+\underbrace{\E[(Y-\E[Y])^2]}_{Var(Y)}+\underbrace{2\E[(X-\E[X])(Y-\E[Y])]}_{2Cov(X,Y)}
\end{align*}
Now since $Cov(X,Y)=0$, we get the result.
}
\end{enumerate}
\end{proof}
\begin{thm}[Weak law of large numbers]
Let $(\Omega,\A,\p)$ be a probability space. Let $(X_n)_{n\geq 1}$ be a sequence of independent r.v.'s. Moreover, write $\mu=\E[X_n]$ for all $n\geq1$ and assume $\E[(X_n-\mu)^2]\leq  C$ for all $n\geq1$ and for some constant $C<\infty$. We also write $S_n=\sum_{j=1}^nX_j$ and $\tilde X_n=\frac{S_n}{n}$ for all $n\geq 1$. Then for all $\epsilon>0$
\[
\p[\vert \tilde X_n-\mu\vert >\epsilon]\xrightarrow{n\to\infty}0.
\]
Thus, we also have
\[
\E[S_n]=\frac{1}{n}\E\left[\sum_{j=1}^nX_j\right]=\frac{1}{n}n\E[X_j]=\E[X_j].
\]
\end{thm}
\begin{proof}
We note that 
\[
\E[(S_n-n\mu)^2]=\sum_{j=1}^n\E[(X_j-\mu)^2]\leq  nC.
\]
Hence for $\epsilon>0$ we get by Markov's inequality
\[
\p[\vert \tilde X-\mu>\epsilon]=\p[(S_n-n\mu)^2>(n\epsilon)^2]\leq  \frac{\E[(S_n-n\mu)^2]}{n^2\epsilon^2}\leq  \frac{C}{n\epsilon^2}\xrightarrow{n\to\infty}0
\]
\end{proof}
%
%
%
%
%
%
%
%
%
%
%
%
%
%
\begin{cor}
Let $(\Omega,\A,\p)$ be a probability space. Let $(A_n)_{n\geq 1}\in \A$ be a sequence of independent events with the same probabilities, i.e. $\p[A_n]=\p[A_m]$, for all $n,m\geq 1$. Then 
$$\lim_{n\to\infty}\frac{1}{n}\sum_{i=1}^\infty \one_{A_i}=\p[A_1]\hspace{0.5cm}a.s.$$
\end{cor}
\begin{proof}
Note that by the weak law of large numbers, we get for a sequence of independent r.v.'s $(X_n)_{n\geq 1}$ with the same expectation for all $n\geq 1$
\[
\lim_{n\to\infty}\E\left[\frac{1}{n}\sum_{j=1}^nX_j\right]=\E[X_1]
\]
and thus we can take $X_j=\one_{A_j}$, since we know that $\E[\one_A]=\p[A]$.
\end{proof}
\section{Finding the distribution of some Random Variables}
\subsection{The case of Sums of independent Random Variables}
Let $X$ be a Poisson distributed r.v. with parameter $\lambda>0$. We already know that for all $\xi\in \R$, we then have $\E\left[e^{i\xi X}\right]=\exp(-\lambda(1-e^{i\xi}))$. Let $X_1,...,X_n$ be $n$ independent r.v.'s for $n\in\N$ with $X_j$ being a Poisson distributed r.v. with parameter $\lambda_j>0$ for all $1\leq  j\leq  n$. Now let $S_n=\sum_{j=1}^nX_j$. We want to figure out what the law of $S_n$ is. We have 
\begin{align*}
\E\left[e^{i\xi S_n}\right]&=\E\left[ e^{i\xi(X_1+...+X_n)}\right]=\E\left[e^{i\xi X_1}\dotsm e^{i\xi X_n}\right]=\E\left[e^{i\xi X_1}\right]\dotsm\E\left[e^{i\xi X_n}\right]\\
&=\exp(-\lambda_1(1-e^{i\xi}))\dotsm \exp(-\lambda_n(1-e^{i\xi}))\\
&=\exp(-(\lambda_1+...+\lambda_n)(1-e^{i\xi}))\\
&=\exp(-Y(1-e^{i\xi})),
\end{align*}
with $Y=\lambda_1+...+\lambda_n$. Since the characteristic function uniquely characterizes the probability distributions, we can conclude that 
$$S_n\sim \Pi(Y=\lambda_1+...+\lambda_n).$$
Let now $X$ be a r.v. with $X\sim\mathcal{N}(m,\sigma^2)$. Then we know
$$\E\left[e^{i\xi X}\right]=e^{i\xi m}e^{-i\sigma^2\frac{\xi^2}{2}}.$$
Now let $X_1,...,X_n$ be $n$ independent r.v.'s for $n\in \N$, such that $X_j\sim\mathcal{N}(m_j,\sigma_j^2)$ for all $0\leq  j\leq  n$. Set again $S_n=X_1+...+X_n$. Therefore 
$$\E\left[e^{i\xi S_n}\right]=e^{im_1\xi}e^{-\sigma^2_1\frac{\xi^2}{2}}\dotsm e^{im_n\xi}e^{-\sigma_n^2\frac{\xi^2}{2}}=e^{i(m_1+...+m_n)\xi}e^{-(\sigma_1^2+....+\sigma_n^2)\frac{\xi^2}{2}},$$
which implies, because of the same argument as above, that
$$S_n\sim \mathcal{N}(m_1+...+m_n,\sigma_1^2+....+\sigma_n^2).$$
\subsection{Using change of variables}
Let $g:\R^n\to\R^n$ be a measurable function given as $g(x)=(g_1(x),...,g_n(x))$ for $x\in\R^n$. Then the jacobian of $g$ is given by 
$$J_g(x)=\left(\frac{\partial g_i(x)}{\partial x_j}\right)_{1\leq i,j\leq n}.$$
Recall that for $g:G\subset \R^n\to\R^n$, where $G$ is a open subset of $\R^n$, with $J_g$ injective such that $\det(J_g(x))\not=0$ for all $x\in G$, we have for every measurable and positive map $f:\R^n\to\R_+$, or for every integrable $f=\one_{G}$, that
$$\int_{g(G)}f(y)dy=\int_Gf(g(x))\vert \det(J_g(x))\vert dx,$$
where $g(G)=\{y\in\R^n\mid \exists x\in G; g(x)=y\}$. 
\begin{thm}
Let $(\Omega,\A,\p)$ be a probability space. Let $X=(X_1,...,X_n)$ be a r.v. on $\R^n$ for $n\in\N$, having a joint density $f$. Let $g:\R^n\to\R^n\in C^1(\R^n)$ be an injective measurable map, such that $\det(J_g(x))\not=0$ for all $x\in\R^n$. Then $Y=g(X)$ has the density 
$$f_Y(y)=\begin{cases}f_X(g^{-1}(y))\vert\det{(J_{g^{-1}}(y))\vert},& y\in g(x)\\ 0,&\text{otherwise}\end{cases}$$
\end{thm}
\begin{proof}
Let $B\in\B(\R^n)$ be a Borel set and $A=g^{-1}(B)$. Then we have 
$$\p[X\in A]=\int_Af_X(x)dx=\int_{g^{-1}(B)}f_X(x)dx=\int_Bf_X(g^{-1}(x))\vert \det(J_{g^{-1}}(x))\vert dx.$$
But we know $\p[Y\in B]=\p[X\in A]$, for all $B\in\B(\R^n)$ with  
$$\p[Y\in B]=\int_Bf_X(g^{-1}(x))\vert\det(J_{g^{-1}}(x))\vert dx.$$ 
It follows that $Y$ has density given by 
$$f_Y(x)=f_X(g^{-1}(x))\vert \det(J_{g^{-1}}(x))\vert.$$
\end{proof}
\begin{ex}We got the following examples:
\begin{enumerate}[$(i)$]
\item{Let $X$ and $Y$ be two independent $\mathcal{N}(0,1)$ distributed r.v.'s. We want to know what is the joint union distribution of $(U,V)=(X+Y,X-Y)$. Therefore, let $g:\R^2\to\R^2$ be given by $(x,y)\mapsto (x+y,x-y)$. The inverse $g^{-1}:\R^2\to\R^2$ is then given by $(u,v)\mapsto \left(\frac{u+v}{2},\frac{u-v}{2}\right)$. We have the following jacobian
$$J_{g^{-1}}=\begin{pmatrix}\frac{1}{2}&\frac{1}{2}\\ \frac{1}{2}&-\frac{1}{2}\end{pmatrix},\hspace{0.3cm}\det(J_{g^{-1}})=-\frac{1}{2}.$$
Moreover we get 
\begin{align*}
f_{(U,V)}(u,v)&=f_{(X,Y)}\left(\frac{u+v}{2},\frac{u-v}{2}\right)\det(J_{g^{-1}})=\left(\frac{1}{\sqrt{2\pi}}e^{-\frac{1}{2}\left(\frac{u+v}{2}\right)^2}\frac{1}{\sqrt{2\pi}}e^{-\frac{1}{2}\left(\frac{u-v}{2}\right)^2}\right)\left(\frac{1}{2}\right)\\
&=\frac{1}{\sqrt{4\pi}}e^{-\frac{u^2}{4}}\frac{1}{\sqrt{4\pi}}e^{-\frac{v^2}{4}}.
\end{align*}
Thus $U$ and $V$ are independent and $U\stackrel{law}{=}V\sim\mathcal{N}(0,2)$.
}
\item{Let $(X,Y)$ be a r.v. on $\R^2$ with joint density $f$. We want to find the density of $Z=XY$. In this case, consider $h:\R^2\to\R^2$, given by $(x,y)\mapsto xy$. We then define $g:\R^2\to\R^2$, given by $(x,y)\mapsto (xy,x)$. We write $S_0=\{(x,y)\mid x=0,y\in\R\}$ and $S_1=\R^2\setminus S_0$. Now $g$ is injective from $S_1$ to $\R\setminus\{0\}$ and $g^{-1}(u,v)=\left(v,\frac{u}{v}\right)$. The jacobian is thus given by 
$$J_{g^{-1}}(u,v)=\begin{pmatrix}0&\frac{1}{v}\\ 1&-\frac{u}{v^2}\end{pmatrix},\hspace{0.3cm}\det(J_{g^{-1}}(u,v))=-\frac{1}{v}.$$
Moreover we have
$$f_{(U,V)}(u,v)=f_{(X,Y)}\left(u,\frac{u}{v}\right)\frac{1}{\vert v\vert}\one_{V\not=0}.$$
Therefore, we get
$$f_U(u)=\int_{\R}f_{(X,Y)}(u,v)\frac{1}{\vert v\vert}dv.$$
}
\end{enumerate}
\end{ex}
\section{Convergence of Random Variables}
\subsection{Types of Convergences}
We have already seen the notion of a.s. convergence. There are different types of convergences for r.v.'s in probability theory. Let us recall the notion of a.s. convergence. 
\begin{defn}[Almost sure convergence]
Let $(\Omega,\A,\p)$ be a probability space. Let $(X_n)_{n\geq 1}$ be a sequence of r.v.'s and let $X$ be a r.v. with values in $\R^d$. Then
$$\lim_{n\to\infty\atop a.s.}X_n=X\Longleftrightarrow\p\left[\left\{\omega\in\Omega\mid \lim_{n\to\infty}X_n(\omega)=X(\omega)\right\}\right]=1.$$
\end{defn}
\begin{rem}
Another very important convergence type is the $L^p$-convergence as it is described in measure theory. Recall that convergence in $L^p$ for $p\in[1,\infty)$ in the probability language means 
$$\lim_{n\to\infty\atop L^p}X_n=X\Longleftrightarrow \lim_{n\to\infty}\E\left[\vert X_n-X\vert^p\right]=0.$$
\end{rem}
\begin{defn}[Convergence in probability]
Let $(\Omega,\A,\p)$ be a probability space. We say that the sequence $(X_n)_{n\geq 1}$ converges in probability to $X$ if for all $\epsilon>0$ 
$$\lim_{n\to\infty\atop \p}X_n=X\Longleftrightarrow\lim_{n\to\infty}\p[\vert X_n-X\vert>\epsilon]=0.$$
\end{defn}
\begin{prop}
Let $(\Omega,\A,\p)$ be a probability space. Let $\mathcal{L}^0_{\R^d}(\Omega,\A,\p)$ be the space of r.v.'s with values in $\R^d$ and let $L^0_{\R^d}(\Omega,\A,\p)$ be the quotient of $\mathcal{L}^0_{\R^d}(\Omega,\A,\p)$ by the equivalence relation $X\sim Y:\Longleftrightarrow X=Y$ a.s. Then the map
\begin{align*}
d:L^0_{\R^d}(\Omega,\A,\p)\times L^0_{\R^d}(\Omega,\A,\p)&\longrightarrow \R_+\\ 
(X,Y)&\longmapsto d(X,Y)=\E[\vert X-Y\vert\land 1]
\end{align*}
defines a distance (metric) in $L^0_{\R^d}(\Omega,\A,\p)$, which is compatible with convergence in probability, i.e.
$$\lim_{n\to\infty\atop \p}X_n=X\Longleftrightarrow \lim_{n\to\infty}d(X_n,X)=0.$$
Moreover $L^0_{\R^d}(\Omega,\A,\p)$ is complete for the metric $d$.
\end{prop}
\begin{proof}
It's easy to see that $d$ defines a distance. If $\lim_{n\to\infty\atop \p}X_n=X$, then for all $ \epsilon>0$ we get
$$\lim_{n\to\infty}\p[\vert X_n-X\vert >\epsilon]=0.$$
Fix $\epsilon >0$. Then 
$$\E[\vert X_n-X\vert \land 1]=\E[(\vert X_n-X\vert\land 1)\one_{\vert X_n-X\vert \leq \epsilon}]+\E[(\vert X_n-X\vert\land 1)\one_{\vert X_n-X\vert >\epsilon}]\leq \epsilon+\E[\one_{\vert X_n-X\vert>\epsilon}]\xrightarrow{n\to\infty} 0,$$
for $\epsilon$ arbitrary small. Conversely assume that $\lim_{n\to\infty}d(X_n,X)=0$. Then for all $\epsilon\in (0,1)$ we have
$$\p[\vert X_n-X\vert>\epsilon]\leq \frac{1}{\epsilon}\E[\vert X_n-X\vert\land 1]=\frac{1}{\epsilon}d(X_n,X).$$
Now we show completeness. Let $(X_k)_{k\geq 0}$ be a Cauchy sequence for $d$. Then there exists a subsequence $Y_k=X_{n_k}$, such that $d(Y_k,Y_{k+1})\leq \frac{1}{2^k}$. It follows that 
\[
\E\left[\sum_{k=1}^\infty(\vert Y_k-Y_{k+1}\vert\land1)\right]=\sum_{k=1}^\infty d(Y_k,Y_{k+1})<\infty,
\]
which implies that $\sum_{k=1}^\infty (\vert Y_k-Y_{k+1}\vert\land 1)<\infty$ and hence $\sum_{k=1}^\infty\vert Y_k-Y_{k+1}\vert<\infty$.
The r.v. $X=Y_1+\sum_{k=1}^\infty Y_{k+1}-Y_k=X_{n_1}+\sum_{k=1}^\infty X_{n_k+1}-X_{n_k}$ is well defined and $(X_n)_{n\geq 1}$ converges to $X$.
\end{proof}
\begin{prop}
Let $(\Omega,\A,\p)$ be a probability space. If $(X_n)_{n\geq 1}$ converges a.s. or in $L^p$ to $X$, it also converges in probability to $X$. Conversely, if $(X_n)_{n\geq 1}$ converges to $X$ in probability, then there exists a subsequence $(X_{n_k})_{k\geq 1}$ of $(X_n)_{n\geq 1}$ such that 
$$\lim_{n\to\infty\atop a.s.}X_n=X.$$
\end{prop}
\begin{proof}
Consider $d(X_n,X)$. We need to prove that $\lim_{n\to\infty\atop a.s.}X_n=X$ or $\lim_{n\to\infty\atop L^p}X_n=X$, which implies that $\lim_{n\to\infty}d(X_n,X)=0$. If $\lim_{n\to\infty\atop a.s.}X_n=X$, then we apply Lebesgue's dominated convergence theorem (we can do this, because $\vert X_n-X\vert\land 1\leq 1$ and $\E[1]<\infty$) to obtain that $\lim_{n\to\infty}\E[\vert X_n-X\vert\land 1]=\E[\lim_{n\to\infty}(\vert X_n-X\vert\land 1)]=0$. If $\lim_{n\to\infty\atop L^p}X_n=X$, we can use the fact that for all $p\geq 1$,
$$\E[\vert X_n-X\vert\land 1]\leq \underbrace{\E[\vert X_n-X\vert]}_{\| X_n-X\|_1}\leq \|X_n-X\|_p\xrightarrow{n\to\infty}0.$$
\end{proof}
\begin{prop}
Let $(\Omega,\A,\p)$ be a probability space. Let $(X_n)_{n\geq 1}$ be a sequence of r.v.'s and let $\lim_{n\to\infty\atop\p}X_n=X$. Assume there is some $r<1$, such that $(X_n)_{n\geq 1}$ is bounded in $L^r$, i.e.
$$\sup_{n}\E[\vert X_n\vert^r]<\infty.$$
Then for every $p\in[1,\infty)$, we get that  $\lim_{n\to\infty\atop L^p}X_n=X$.
\end{prop}
\begin{proof}
The fact that $(X_n)_{n\geq 1}$ is bounded in $L^r$ implies that there is some $C>0$, such that for all $n\geq 1$
$$\E[\vert X_n\vert^r]\leq C.$$
With Fatou's lemma we get 
$$\E[\vert X\vert]\leq C.$$
So it follows that $X\in L^r$. Now we apply Hölder's inequality to obtain for $p\in[1,r),$
\begin{align*}
\E[\vert X_n-X\vert^p]&=\E[\vert X_n-X\vert^p\one_{\{\vert X_n-X\vert \leq \epsilon\}}]+\E[\vert X_n-X\vert^p\one_{\{\vert X_n-X\vert>\epsilon\}}]\\
&\leq \epsilon^p+\E[\vert X_n-X\vert^r]^{\frac{p}{r}}\p[\vert X_n-X\vert>\epsilon]^{1-\frac{p}{r}}\\
&\leq \epsilon^p+2^pC^{\frac{p}{r}}\p[\vert X_n-X\vert>\epsilon]\xrightarrow{n\to\infty}0.
\end{align*}
\end{proof}
\subsection{The strong law of large numbers}
\begin{thm}[Kolmogorov's 0-1 law]
Let $(\Omega,\A,\p)$ be a probability space. Let $(X_n)_{n\geq 1}$ be a sequence of independent r.v.'s with values in arbitrary measure spaces. For $n\geq 1$, define the $\sigma$-Algebra
$$\B_n:=\sigma(X_k\mid k\geq n).$$
The tail $\sigma$-Algebra $\B_\infty$ is defined as 
$$\B_\infty:=\bigcap_{n=1}^\infty\B_n.$$
Then $\B_\infty$ is trivial in the sense that for all $B\in \B_\infty$ we get that $\p[B]\in\{0,1\}$.
\end{thm}
\begin{rem}
We can easily see that a r.v. which is $\B_\infty$-measurable is constant a.s. indeed its distribution function can only take the values 0 and 1.
\end{rem}
\begin{proof}
Define $\mathcal{D}_n:=\sigma(X_k\mid k\leq n)$. We have already observed that $\mathcal{D}_n$ and $\B_{n+1}$ are independent and hence since $\B_\infty\subset \B_{n+1}$, we get that for all $n\geq 1$, $\mathcal{D}_{n}$ and $\B_{\infty}$ are also independent. This implies that for all $A\in\bigcup_{n=1}^\infty\mathcal{D}_n$ and for all $B\in\B_\infty$ we get
$$\p[A\cap B]=\p[A]\p[B].$$
Since $\bigcup_{n=1}^\infty \mathcal{D}_n$ is stable under finite intersection, we obtain that $\sigma\left(\bigcup_{n=1}^\infty\mathcal{D}_n\right)$ is independent of $\B_\infty$ and 
$$\sigma\left(\bigcup_{n=1}^\infty\mathcal{D}_n\right)=\sigma\left(X_1,X_2,...\right).$$
We also note that the fact that $\B_\infty\subset \sigma(X_1,X_2,...)$ implies that $\B_\infty$ is independent of itself. Thus it follows that for all $B\in\B_{\infty}$, we get $\p[B]=\p[B\cap B]=\p[B]\p[B]=\p[B]^2$. Hence $\p[B]=\p[B]^2$ and therefore $\p[B]\in\{0,1\}$.
\end{proof}
\begin{rem}
If $(X_n)_{n\geq 1}$ is a sequence of independent r.v.'s, then $\limsup_{n}\frac{X_1+...+X_n}{n}$ $(\in[-\infty,\infty])$ is $\B_\infty$ measurable. It follows that $\frac{1}{n}(X_1+...+X_n)$ converges a.s. Moreover, its limit is a.s. constant.
\end{rem}
\begin{prop}
Let $(\Omega,\A,\p)$ be a probability space. Let $(X_n)_{n\geq 1}$ be a sequence of independent r.v.'s with the same distribution 
$$\p[X_n=1]=\p[X_n=-1]=\frac{1}{2}$$
for all $n\geq 1$and set $S_n=\sum_{j=1}^nX_j$. Then
$$\begin{cases}\sup_{n\geq 1}S_n=\infty&\text{a.s.}\\ \inf_{n\geq 1}S_n=-\infty&\text{a.s.}\end{cases}$$
\end{prop}
\begin{proof}
We first need to show that for $p\geq 1$ we get $\p[-p\leq \inf_n S_n\leq \sup_n S_n\leq p]=0$. This is a good exercise \footnote{Hint: Borel-Cantelli}. Now take $p\to\infty$ to obtain 
$$\p[\{\inf_n S_n>-\infty\}\cap\{\sup_n S_n<\infty\}]=0,$$
and therefore $\p[\{\inf_n S_n=-\infty\}\cup\{\sup_n S_n=\infty\}]=1$. So it follows 
$$1\leq \p[\inf_n S_n=-\infty]+\p[\sup_n S_n=\infty].$$
By symmetry, we get $\p[\inf_n S_=-\infty]=\p[\sup_n S_n=\infty]$ and hence $\p[\sup_n S_n=\infty]>0$. Now note that $\{\sup_n S_n=\infty\}\in\B_\infty$. Indeed, for all $k\geq 1$ we get $\{\sup_n S_n=\infty\}=\{\sup_{n\geq k}(X_k,X_{k+1}+...+X_n)=\infty\}\in\B_k$. Since $\{\sup_n S_n=\infty\}\in\B_\infty$ it follows that $\p[\sup_n S_n =\infty]\in\{0,1\}$, but we have just seen that $\p[\sup_n S_n=\infty]\geq \frac{1}{2}>0$, which implies then $\p[\sup_n S_n=\infty]=1$.
\end{proof}
\begin{thm}[Strong law of large numbers]
Let $(\Omega,\A,\p)$ be a probability space. Let $(X_n)_{n\geq 1}$ be a sequence of iid r.v.'s, such that $X_i\in L^1(\Omega,\A,\p)$ for all $i\in\{1,...,n\}$. Then 
$$\lim_{n\to\infty\atop a.s.}\frac{1}{n}(X_1+...+X_n)=\E[X_1].$$
Moreover, for $\bar X_n:=\frac{1}{n}\sum_{j=1}^n(X_j-\E[X_j])$ we have
$$\p\left[\limsup_{n\to\infty}\vert\bar X_n\vert=0\right]=1.$$
\end{thm}
\begin{rem}
The assumption $\E[X_1]<\infty$ is important, but if $X_1\geq 0$ and $\E[X_1]=\infty$ we can apply the theorem to $X_1\land k$ for $k>0$, and obtain that the theorem also holds with $\E[X_1]=\infty$.
\end{rem}
\begin{proof}
Let $S_n=X_1+...+X_n$ with $S_0=0$ and take $a>\E[X_1]$. Define $M=\sup_{n>0}(S_n-na)$. We shall show that $M<\infty$ a.s. Since we obviously have $S_n\leq na+M$, it follows immediately that $\frac{S_n}{n}\leq a$ a.s. Choosing $a\searrow \E[X_1]$ we obtain that $\limsup_n\frac{S_n}{n}\leq \E[X_1]$. Replacing $(X_n)_{n\geq 1}$ with $(-X_n)_{n\geq 1}$, we also get $\liminf_n\frac{S_n}{n}\geq \E[X_1]$ a.s. So it follows that 
$$\liminf_n\frac{S_n}{n}=\limsup_n\frac{S_n}{n}=\E[X_1]\hspace{0.2cm}a.s.$$ 
Hence we only need to show that $M<\infty$ a.s. We first note that $\{M<\infty\}\in\B_\infty$. Indeed, for all $k\geq 0$ we get that $\{M<\infty\}=\{\sup_{n\in\N}(S_n-an)<\infty\}=\{\sup_{n\geq k}(S_n-S_k-(n-k)a)<\infty\}$. So it follows that $\p[M<\infty]\in\{0,1\}$. Now we need to show that $\p[M<\infty]=1$ or equivalently $\p[M=\infty]<1$. We do it by contradiction. For $k\in\N$, set $M_k=\sup_{0\leq n\leq k}(S_n-na)$ and $M'_k=\sup_{0\leq n\leq k}(S_{n+1}-S_n-na)$. Then $M_k$ and $M'_k$ have the same distribution. Indeed, $(X_1,...,X_k)$ and $(X_2,...,X_{k+1})$ have the same distribution and $M_k=F_k(X_1,...,X_k)$ and $M'_k=F_k(X_2,...,X_{k+1})$ with some map $F_k:\R^k\to\R$. Moreover, $M=\lim_{k\to\infty}\uparrow M_k$ and therefore $M'=\lim_{k\to\infty}M_k$. Since $M_k$ and $M_k'$ have the same distribution, $M$ and $M'$ also have the same distribution. Indeed $\p[M'\leq X]=\lim_{k\to\infty}\downarrow \p[M_k'\leq X]=\lim_{k\to\infty}\downarrow \p[M_k\leq X]=\p[M\leq X]$. So $M$ and $M'$ have the same distribution function. Moreover, $M_{k+1}=\sup\{0,\sup_{1\leq n\leq k+1}(S_n-na)\}=\sup\{0,M_k'+X_1-a\}$, which implies that $M_{k+1}=M_k'-\inf\{a-X_1,M_k'\}$. Now we can use the fact that $M_k'$ and $M_k$ are bounded to obtain
\begin{align*}
\E[\inf\{a-X_1,M_k'\}]&=\E[M_k']-\E[M_{k+1}]\\
&\leq \vert a-X_1\vert\\
&\leq \vert a\vert +\vert X_1\vert\in \mathcal{L}^1(\Omega,\A,\p)
\end{align*}
and apply the dominated convergence theorem to obtain 
$$\E[\inf\{a-X_1,M'\}]\leq 0.$$
If we had $\p[M\leq \infty]=1$, then since $M'$ and $M$ have the same distribution we would also have $\p[M'=\infty]=1$, in which case $\inf\{a-X_1,M'\}<a-X_1$ and $\E[a-X_1]>0$ and this contradicts 
$$\E[\inf\{a-X_1,M'\}]\leq 0.$$
\end{proof}
\section{More convergence in probability, $L^p$ and almost surely}
\begin{prop}
Let $(\Omega,\A,\p)$ be a probability space. Let $(X_n)_{n\geq 1}$ be a sequence of r.v.'s and assume that for all $\epsilon>0$ we have 
$$\sum_{n\geq 1}\p[\vert X_n-X\vert >\epsilon]<\infty.$$
Then 
$$\lim_{n\to\infty\atop a.s.}X_n=X.$$
\end{prop}
\begin{proof}
Take $\epsilon_k=\frac{1}{k}$ for $k\in\N$ with $k\geq 1$. Now with the Borel-Cantelli lemma we get 
\[
\p\left[\limsup_n\left\{\vert X_n-X\vert>\frac{1}{k}\right\}\right]=0,
\]
which implies that $\p\left[\bigcup_{k\geq 1}\limsup_n\left\{\vert X_n-X\vert>\frac{1}{k}\right\}\right]=0$ and hence 
\[
\p\left[\underbrace{\bigcap_{k\geq 1}\liminf_n\left\{\vert X_n-X\vert\leq \frac{1}{k}\right\}}_{\Omega'}\right]=1.
\]
Moreover, we have that $\p[\Omega']=1$ and for $\omega\in\Omega'$ we get that for all $k\geq 1$ there is  $n_0(\omega)\in\N\setminus\{0\}$ such that for $n\geq n_0(\omega)$ we get that $\vert X_n(\omega)-X(\omega)\vert\leq \frac{1}{k}$, i.e. $\lim_{n\to\infty}X_n(\omega)=X(\omega)$ for $\omega\in\Omega'$.
\end{proof}
\begin{ex} Let $(\Omega,\A,\p)$ be a probability space. Let $(X_n)_{n\geq 1}$ be a sequence of r.v.'s such that $\p[X_n=0]=1-\frac{1}{1+n^2}$ and $\p[X_n=1]=\frac{1}{1+n^2}$. Then for all $\epsilon>0$ we get $\p[\vert X_n\vert>\epsilon]=\p[X_n>\epsilon]=\frac{1}{1+n^2}$, so it follows 
$$\sum_{n\geq 1}\p[\vert X_n\vert>\epsilon]<\infty,$$ 
which implies that $\lim_{n\to\infty\atop a.s.}X_n=0.$
\end{ex}
\begin{prop}
Let $(\Omega,\A,\p)$ be a probability space. Let $(X_n)_{n\geq 1}$ be a sequence of r.v.'s. Then 
$$\lim_{n\to\infty\atop a.s.}X_n=X\Longleftrightarrow \lim_{n\to\infty\atop \p}\sup_{m>n}\vert X_m-X\vert=0.$$
\end{prop}
\begin{proof}
Exercise.
\end{proof}
\begin{ex} Let $(Y_n)_{n\geq 1}$ be iid r.v.'s such that $\p[Y_n\leq X]=1-\frac{1}{1+X}$ for $X\geq 0$ and $n\geq 1$. Take $X_n=\frac{Y_n}{n}$ and let $\epsilon>0$. Then 
$$\p[\vert X_n\vert>\epsilon]=\p[\vert Y_n\vert>n\epsilon]=\frac{1}{1+n\epsilon}\xrightarrow{n\to\infty}0,$$
and thus $\lim_{n\to\infty\atop \p}X_n=0$. Moreover, we have
$$\p\left[\sup_{m\geq n}\vert X_m\vert>\epsilon\right]=1-\p\left[\sup_{m\geq n}\vert X_n\vert\leq \epsilon\right]=1-\prod_{m\geq n}^\infty\left(1-\frac{1}{1+m\epsilon}\right),$$
but $\prod_{m\geq n}^\infty\left(1-\frac{1}{1+m\epsilon}\right)=0$. Hence $\p[\sup_{m\geq n}\vert X_n\vert>\epsilon]\not\rightarrow 0$ as $n\to\infty$ and therefore $(X_n)_{n\geq 1}$ doesn't converge a.s. to $X$. 
\end{ex}
\begin{lem}
Let $(\Omega,\A,\p)$ be a probability space. Let $(X_n)_{n\geq 1}$ be a sequence of r.v.'s. Then $\lim_{n\to\infty\atop \p}X_n=X$ if and only if for very subsequence of $(X_n)_{n\geq 1}$, there exists a further subsequence which converges a.s.
\end{lem}
\begin{proof}
If $\lim_{n\to\infty\atop\p}X_n=X$, then any of its subsequences also converge in probability. We already know that there exists a subsequence which converges a.s. Conversely, if $\lim_{n\to\infty\atop\p}X_n=X$, then there is  an $\epsilon>0$, some $n_k\in\N$ and a $\nu>0$ such that for all $k\geq 1$ we get 
$$\p[\vert X_{n_k}-X\vert>\epsilon]>\nu$$ 
and therefore we cannot extract a subsequence from $(X_{n_k})_{k\geq 1}$ which would converge a.s.
\end{proof}
\begin{prop}
Let $(\Omega,\A,\p)$ be a probability space. Let $(X_n)_{n\geq 1}$ be a sequence of r.v.'s and $g:\R\to\R$ a continuous map. Moreover, assume that $\lim_{n\to\infty\atop\p}X_n=X$. Then 
$$\lim_{n\to\infty\atop \p}g(X_n)=g(X).$$
\end{prop}
\begin{proof}
Any subsequence $g((X_{n_k})_{k\geq 1})$ and $(X_{n_k})_{k\geq 1}$ converges in probability. So it follows that there exists a subsequence $(X_{m_k})_{k\geq 1}$ of $(X_{n_k})_{k\geq 1}$ such that 
$$\lim_{n\to\infty\atop a.s.} X_n=X\hspace{0.3cm}\text{and}\hspace{0.3cm}\lim_{k\to\infty\atop a.s.}g(X_{m_k})=g(X)$$ 
because $g$ is continuous. Now with the previous lemma we get that
$$\lim_{n\to\infty\atop \p}g(X_n)=g(X).$$
\end{proof}
\begin{prop}
Let $(\Omega,\A,\p)$ be a probability space. Let $(X_n)_{n\geq 1}$ and $(Y_n)_{n\geq 1}$ be sequences of r.v.'s such that $\lim_{n\to\infty\atop \p}X_n=X$ and $\lim_{n\to\infty\atop\p} Y_n=Y$. Then 
\begin{enumerate}[$(i)$]
\item{$\lim_{n\to\infty\atop\p} X_n+Y_n=X+Y$}
\item{$\lim_{n\to\infty\atop \p}X_n\cdot Y_n=X\cdot Y$}
\end{enumerate}
\end{prop}
\begin{proof}
We need to show both points.
\begin{enumerate}[$(i)$]
\item{Let $\epsilon>0$. Then $\vert X_n-X\vert\leq \frac{\epsilon}{2}$ and $\vert Y_n-Y\vert\leq \frac{\epsilon}{2}$ implies that $\vert (X_n+Y_n)-(X+Y)\vert\leq \epsilon$, and thus we get 
$$ \p[\vert X_n+Y_n-(X+Y)\vert>\epsilon]\leq \p\left[\vert X_n-X\vert>\frac{\epsilon}{2}\right]+\p\left[\vert Y_n-Y\vert>\frac{\epsilon}{2}\right].$$
}
\item{We apply proposition 8.4 to the continuous map $g(X)=X^2$. Hence we get
$$2X_nY_n=(X_n+Y_n)^2-X_n^2-Y_n^2.$$
}
\end{enumerate}
\end{proof}
\section{Convergence in Law}
We denote by $C_b(\R^d)$ the space of bounded and continuous functions $\varphi:\R^d\to\R$. Moreover, we endow $C_b(\R^d)$ with the supremums norm $\|\varphi\|_\infty=\sup_{x\in\R^d}\vert\varphi(x)\vert$. The space $(C_b(\R^d),\|\cdot\|_\infty)$ forms a Banach space, i.e. it is a complete normed vector space. Next, we want to introduce the notion of law convergence in terms of probability measures.
\begin{defn}[Weak and Law convergence]
\begin{enumerate}[$(i)$]
The following hold.
\item{Let $(\mu_n)_{n\geq 1}$ be a sequence of probability measures on $\R^d$. We say that $(\mu_n)_{n\geq 1}$ is converging weakly to a probability measure $\mu$ on $\R^d$, and we write
$$\lim_{n\to\infty\atop w}\mu_n=\mu,$$
if for all $\varphi\in C_b(\R^d)$ we have
$$\lim_{n\to\infty}\int_{\R^d} \varphi d\mu_n=\int_{\R^d} \varphi d\mu.$$
}
\item{Let $(\Omega,\A,\p)$ be a probability space. A sequence of r.v.'s $(X_n)_{n\geq 1}$, taking values in $\R^d$, is said to converge in law to a r.v. $X$ with values in $\R^d$ and we write 
$$\lim_{n\to\infty\atop law}X_n=X,$$
if $\lim_{n\to\infty\atop w}\p_{X_n}=\p_X$, or equivalently if for all $\varphi\in C_b(\R^d)$ we have 
$$\lim_{n\to\infty}\E[\varphi(X_n)]=\E[\varphi(X)]\Longleftrightarrow \lim_{n\to\infty}\int_{\R^d}\varphi(x)d\p_{X_n}(x)=\int_{\R^d}\varphi(x)d\p_{X}(x).$$
}
\end{enumerate}
\end{defn}
\begin{rem} One has to consider the following: 
\begin{enumerate}[$(i)$]
\item{There is an abuse of language when we say that $\lim_{n\to\infty\atop law}X_n=X$ because the r.v. $X$ is not determined in a unique way, only $\p_X$ is unique.
}
\item{Note also that the r.v.'s $X_n$ and $X$ need not be defined on the same probability space $(\Omega,\A,\p)$.
}
\item{The space of probability measures on $\R^d$ can be viewed as a subspace of $C_b(\R^d)^*$ (the dual space of $C_b(\R^d)$). The weak convergence then corresponds to convergence for the \emph{weak*}-topology. 
}
\item{It is enough to show that $\lim_{n\to\infty}\E[\varphi(X_n)]=\E[\varphi(X)]$ or $\lim_{n\to\infty}\int_{\R^d}\varphi(x)d\p_{X_n}(x)=\int_{\R^d} \varphi(x)d\p_X(x)$ is satisfied for all $\varphi\in C_c(\R^d)$, where $C_c(\R^d)$ is the space of continuous functions with compact support. That is, $\varphi\in C_c(\R^d)$ if $supp(\varphi):=\overline{\{x\in\R^d\mid \varphi(x)\not=0\}}$ is compact.
}
\end{enumerate}
\end{rem}
\begin{ex} We got the following examples:
\begin{enumerate}[$(i)$]
\item{If $X_n$ and $X\in\mathbb{Z}^d$ for all $n\geq 1$, then $\lim_{n\to\infty\atop law}X_n=X$ if and only if for all $X\in\mathbb{Z}^d$ we have
$$\lim_{n\to\infty}\p[X_n=x]=\p[X=x].$$
To see this, we use point (4) of the remark above. Let therefore $\varphi\in C_c(\R^d)$. Then 
$$\E[\varphi(X_n)]=\sum_{k\in\mathbb{Z}^d}\varphi(k)\p[X_n=k].$$
Since $\varphi$ has compact support, i.e. $\varphi(x)=0$ for $\vert x\vert \leq C$ for some $C\geq 0$, we get
$$\E[\varphi(X_n)]=\sum_{k\in\mathbb{Z}^d\atop \vert k\vert\leq C}\varphi(k)\p[X_n=k].$$
Hence we have
$$\lim_{n\to\infty}\E[\varphi(X_n)]=\lim_{n\to\infty}\sum_{k\in\mathbb{Z}^d\atop \vert k\vert\leq C}\varphi(k)\p[X_n=k]=\sum_{k\in\mathbb{Z}^d\atop \vert k\vert \leq C}\varphi(k)\p[X=k]=\E[\varphi(X)].$$
}
\item{If $X_n$ has density $\p_{X_n}(dx)=P_n(x)dx$ for all $n\geq1$ and if we assume that 
$$\lim_{n\to\infty\atop a.e.}P_n(x)=P(x),$$
then there is a $q\geq 0$ such that $\int_{\R^d}q(x)dx<\infty$ and $P_n(x)\leq q(x)$ a.e. Then an application of the dominated convergence theorem shows that 
$$\int_{\R^d} P(x)dx=1,$$
and thus there exists a r.v. $X$ with density $P$ such that $\lim_{n\to\infty\atop law}X_n=X$ and for $\varphi\in C_b(\R^d)$ we get
$$\E[\varphi(X_n)]=\int_{\R^d}\varphi(x)P_n(x)dx,$$
and $\vert\varphi(x)P_n(x)\vert\leq \underbrace{\|\varphi\|_\infty q(x)}_{\in\mathcal{L}^1(\R^d)}$. So with the dominated convergence theorem we get 
$$\lim_{n\to\infty}\int_{\R^d}\varphi(x)P_n(x)dx=\int_{\R^d}\varphi(x)P(x)dx=\E[\varphi(X)].$$
}
\item{Let $X_n\sim \mathcal{N}(0,\sigma_n^2)$ such that $\lim_{n\to\infty}\sigma_n=0$. Then $\lim_{n\to\infty\atop law}X_n=0$ and 
$$\E[\varphi(X_n)]=\int_\R \varphi(x)e^{-\frac{x^2}{2\sigma_n^2}}\frac{1}{\sigma_n\sqrt{2\pi}}dx.$$
Now using that $u=\frac{x}{\sigma_n}$, we get $dx=\sigma_n du$ and hence we have 
$$\E[\varphi(X_n)]=\int_\R \varphi(x)e^{-\frac{x^2}{2\sigma_n^2}}\frac{1}{\sigma_n\sqrt{2\pi}}dx=\int_\R \varphi(\sigma_n u)e^{-\frac{u^2}{2}}\frac{1}{\sqrt{2\pi}}du.$$
Moreover, we have $\vert \varphi(\sigma_n u)e^{-\frac{u^2}{2}}\vert\leq \underbrace{\|\varphi\|_\infty e^{-\frac{u^2}{2}}}_{\in\mathcal{L}^1(\R)}$. Hence we get
$$\lim_{n\to\infty}\E[\varphi(X_n)]=\int_{\R}\varphi(0)e^{-\frac{u^2}{2}}\frac{1}{\sqrt{2\pi}}du.$$
}
\end{enumerate}
\end{ex}
\begin{prop}
Let $(\Omega,\A,\p)$ be a probability space. Let $(X_n)_{n\geq 1}$ be a sequence of r.v.'s and assume that $\lim_{n\to\infty\atop \p}X_n=X$. Then $\lim_{n\to\infty\atop law}X_n=X$.
\end{prop}
\begin{proof}
We first note that if $\lim_{n\to\infty\atop a.s.}X_n=X$ then $\lim_{n\to\infty}\E[\varphi(X_n)]=\E[\varphi(X)]$ for every $\varphi\in C_b(\R^d)$. Let us now assume that $(X_n)_{n\geq 1}$ does not converge in law to $X$. Then there is a $\varphi\in C_b(\R^d)$ such that $\E[\varphi(X_n)]$ does not converge to $\E[\varphi(X)]$. We can hence extract a subsequence $(X_{n_k})_{k\geq 1}$ from $(X_n)_{n\geq 1}$ and find an $\epsilon>0$ such that 
$$\vert\E[\varphi(X_{n_k})]-\E[\varphi(X)]\vert>\epsilon.$$
But this contradicts the fact that we can extract a further subsequence $(X_{n_{k_l}})_{l\geq 1}$ from $(X_{n_k})_{k\geq 1}$ such that 
$$\lim_{l\to\infty\atop a.s.} X_{n_{k_{l}}}=X.$$
\end{proof}
Let $(\Omega,\A,\p)$ be a probability space. Let $(X_n)_{n\geq 1}$ be a sequence of r.v.'s, A natural question would be to ask whether, under these condition, we have a $B\in\B(\R)$ such that $\lim_{n\to\infty}\p[X_n\in B]=\p[X\in B]$. If we take $B=\{0\}$ and use the previous example, we would get 
$$\lim_{n\to\infty}\underbrace{\p[X_n=0]}_{=0}\not=\underbrace{\p[X=0]}_{=1},$$
which shows that the answer to the question is negative.
\begin{prop}
Let $(\mu_n)_{n\geq 1}$ be a sequence of probability measures on $\R^d$ and $\mu$ be a probability measure on $\R^d$. Then the following are equivalent.
\begin{enumerate}[$(i)$]
\item{$\lim_{n\to\infty\atop w}\mu_n=\mu$.
}
\item{For all open subsets $G\subset \R^d$ we have 
$$\limsup_n\mu_n(G)\geq \mu(G).$$
}
\item{For all closed subsets $F\subset\R^d$ we have 
$$\limsup_n\mu_n(F)\leq \mu(F).$$
}
\item{For all Borel measurable sets $B\in\B(\R^d)$ with $\mu(\partial B)=0$ we have
$$\lim_{n\to\infty}\mu_n(B)=\mu(B).$$
}
\end{enumerate}
\end{prop}
\begin{proof}
We immediately note that $(ii)\Longleftrightarrow (iii)$ by taking complements.
First we show $(i)\Longrightarrow (ii):$ Let $G$ be an open subset of $\R^d$. Define $\varphi_p(x):=p(d(x,G^C)\land 1)$. Then $\varphi_p$ is continuous, bounded, $0\leq \varphi_p(x)\leq \one_{G}(x)$ for all $x\in\R^d$ and $\varphi_p\uparrow \one_{G}$ (note that $d(x,F)=\inf_{y\in F}d(x,y)$) as $p\to\infty$. Moreover, $F$ is closed if and only if $d(x,F)=0$. We also get that $\varphi_p(x)=0$ on $G^C$ and $0\leq \varphi_p(x)\leq 1\leq \one_{G}(x)$ for all $x\in\R^d$. Therefore we get
\[
\liminf_n\mu_n(G)\geq \sup_{p}\left(\liminf_nF\int_{\R^d}\varphi_pd\mu_n\right)=\sup_p\int_{\R^d} \varphi_pd\mu=\int \one_G d\mu=\mu(G).
\]
Now we show that $(ii)$ and $(iii)\Longrightarrow (iv):$ For Borel measurable set $B\in\B(\R^d)$ with $\mathring{B}\subset B\subset \bar B$ we get 
$$\limsup_n\mu_n(B)\leq \limsup_n\mu_n(\bar B)\leq \mu(\bar B),$$
$$\liminf_n\mu_n(B)\geq \liminf_n\mu_n(\mathring{B})\geq \mu(\mathring{B}).$$
Therefore it follows that 
$$\mu(\mathring{B})\leq \liminf_n \mu_n(B)\leq \limsup_n\mu_n(B)\leq \mu(B).$$
Moreover, if $\mu(\partial B)=0$, we get that $\mu(\bar B)=\mu(\mathring{B})=\mu(B)$ and thus $\lim_{n\to\infty}\mu_n(B)=\mu(B)$.
Now we show $(iv)\Longrightarrow (i):$ Let therefore $\varphi\in C_b(\R^d)$. We can always use that $\varphi=\varphi^+-\varphi^-$ and so, without loss of generality, we may assume that $\varphi\geq 0$. Let $\varphi\geq 0$ and $K\geq 0$ be such that $0\leq \varphi\leq K$. Then
$$\int_{\R^d}\varphi(x)d\mu(x)=\int_{\R^d}\underbrace{\left(\int_0^K\one_{\{t\leq \varphi(x)\}}dt\right)}_{K\land \varphi(x)=\varphi(x)}d\mu(x)=\int_0^K\mu(E_t^\varphi)dt,$$
where $E_t^\varphi:=\{x\in\R^d\mid \varphi(x)\geq t\}$. Similarly, we have
$$\int_{\R^d}\varphi(x)d\mu_n(x)=\int_0^K\mu_n(E_t^\varphi)dt.$$
Now we can note that $\partial E_t^\varphi\subset\{x\in\R^d\mid \varphi(x)=t\}$. Moreover, there are at most countably many values of $t$ for which $\mu\left(\{x\in\R^d\mid \varphi(x)=t\}\right)>0$. Indeed, for an integer$K\geq 1$ we get that $\mu\left(\{x\in\R^d\mid \varphi(x)=t\}\right)\geq \frac{1}{K}$. This can happen for at most $K$ distinct values of $t$. Thence we have
$$\lim_{n\to\infty}\mu_n(E_t^\varphi)=\mu(E_t^\varphi)dt\hspace{0.2cm}\text{a.e.},$$
which implies that 
$$\lim_{n\to\infty}\int_{\R^d} \varphi(x)d\mu_n(x)=\int_0^K\mu_n(E_t^\varphi)dt\xrightarrow{n\to\infty} \int_0^K\mu(E_t^\varphi)dt=\int_{\R^d}\varphi(x)d\mu(x).$$
\end{proof}
\emph{Consequences:} We look at the case of $d=1$. Let $(X_n)_{n\geq 1}$ be a sequence of r.v.'s with values in $\R$ and let $X$ be a r.v. with values in $\R$. One can show that
$$\lim_{n\to\infty\atop law}X_n=X\Longleftrightarrow \lim_{n\to\infty}F_{X_n}(t)=F_X(t).$$
\begin{prop}
Let $(\mu_n)_{n\geq 1}$ and $\mu$ be probability measures on $\R^d$. Let $H\subset C_b(\R^d)$ such that $\bar H\supset C_c(\R^d)$. Then the following are equivalent.
\begin{enumerate}[$(i)$]
\item{$\lim_{n\to\infty\atop w}\mu_n=\mu.$
}
\item{For all $\varphi\in C_c(\R^d)$ we have
$$\lim_{n\to\infty}\int_{\R^d} \varphi d\mu_n=\int_{\R^d} \varphi d\mu.$$
}
\item{For all $\varphi\in H$ we have 
$$\lim_{n\to\infty}\int_{\R^d}\varphi d\mu_n=\int_{\R^d} \varphi d\mu.$$
}
\end{enumerate}
\end{prop}
\begin{proof}
It is obvious that $(i)\Longrightarrow (ii)$ and $(i)\Longrightarrow (iii)$. 
Therefore we first show $(ii)\Longrightarrow (i):$ Let therefore $\varphi\in C_b(\R^d)$ and let $(f_k)_{k\geq 1}\in C_c(\R^d)$ with $0\leq f_k\leq 1$ and $f_k\uparrow 1$ as $k\to\infty$. Then for all $k\geq 1$ we get that $\varphi f_k\in C_c(\R^d)$ and hence 
$$\lim_{n\to\infty}\int_{\R^d}\varphi f_kd\mu_n=\int_{\R^d} \varphi f_k d\mu.$$
Moreover, we have 
$$\left\vert \int_{\R^d} \varphi d\mu-\int_{\R^d}\varphi f_k d\mu\right\vert\leq \sup_{x\in\R^d}\vert\varphi(x)\vert\left(1-\int_{\R^d} f_kd\mu\right)$$
and also 
$$\left\vert \int_{\R^d} \varphi d\mu_n-\int_{\R^d} \varphi f_k d\mu_n\right\vert\leq \sup_{x\in\R^d}\vert \varphi(x)\vert \left(1-\int_{\R^d} f_k d\mu_n\right).$$ 
Hence, for all $k\geq 1$ we get
\begin{align*}
\limsup_n \left\vert\int_{\R^d}\varphi d\mu-\int_{\R^d} \varphi  d\mu_n\right\vert&\leq \sup_{x\in \R^d}\vert \varphi(x)\vert \limsup_n\left[\left(1-\int_{\R^d} f_k d\mu_n\right)+\left(1-\int_{\R^d} f_k d\mu\right)\right]\\
&=2\sup_{x\in\R^d}\vert\varphi(x)\vert\left(1-\int f_kd\mu\right)\xrightarrow{k\to\infty}0.
\end{align*}
Now we show $(iii)\Longrightarrow (ii):$ Let therefore $\varphi\in C_c(\R^d)$. Then there is a sequence $(\varphi_k)_{k\geq 1}\subset H$ such that $\|\varphi-\varphi_k\|_\infty\leq \frac{1}{k}$ for all $k\geq 1$. This implies that
\begin{multline*}
\limsup_n\left\vert\int \varphi d\mu_n-\int \varphi d\mu\right\vert\\\leq \limsup_n\left(\left\vert\int \varphi d\mu_n-\int\varphi_k d\mu_n\right\vert+\underbrace{\left\vert\int \varphi_k d\mu_n-\int\varphi_k d\mu\right\vert}_{\xrightarrow{n\to\infty}0} +\left\vert \int \varphi_k d\mu-\int\varphi d\mu\right\vert\right)\leq \frac{2}{k}.
\end{multline*}
The claim follows now for $k\to\infty$.
\end{proof}
\begin{thm}[Lèvy]
Let $(\Omega,\A,\p)$ be a probability space. Let $(\mu_n)_{n\geq 1}$ be a sequence of probability measures on $\R^d$ associated to a sequence of real r.v.'s $(X_n)_{n\geq 1}$. Moreover, let $\hat{\mu}_n(\xi)=\int_{\R^d}e^{i\xi x}d\mu_n(x)$ and $\Phi_X(\xi)=\E[e^{i\xi x}]$. Then for all $\xi\in\R^d$ we get
$$\lim_{n\to\infty\atop w}\mu_n=\mu \Longleftrightarrow\lim_{n\to\infty}\hat\mu_n(\xi)=\hat\mu(\xi).$$ 
Equivalently, for all $\xi\in\R^d$ we get 
$$\lim_{n\to\infty\atop law}X_n=X\Longleftrightarrow\lim_{n\to\infty}\Phi_{X_n}(\xi)=\Phi_X(\xi).$$
\end{thm}
\begin{proof}
It is obvious that $\lim_{n\to\infty\atop w}\mu_n=\mu$ implies that $\lim_{n\to\infty}\hat{\mu}_n(\xi)=\hat{\mu}(\xi)$. Therefore $e^{i\xi X}$ is continuous and bounded. For notation conventions we deal with the case $d=1$. Let therefore $f\in C_c(\R^d)$. For $\sigma>0$ we also note $g_\sigma(x)=\frac{1}{\sigma\sqrt{2\pi}}e^{-\frac{x^2}{2\sigma^2}}$.
\begin{exer} 
Show that $g_\sigma *f\xrightarrow{\sigma\to 0}f$ uniformly on $\R$.
\end{exer}
\begin{exer}
Show that if $\nu$ is a probability measure, then
$$\int_\R g_\sigma * fd\nu=\int_\R f(x)(g_\sigma *\nu)(x)dx=\int_\R f(x)\frac{1}{\sigma\sqrt{2\pi}}\int_\R e^{i\xi x}g_{\frac{1}{\sigma}}(\xi)\hat{\nu}(\xi)d\xi dx.$$
\end{exer}
Since $\lim_{n\to\infty}\hat{\mu}_n(\xi)=\hat{\mu}(\xi)$, we get by the dominated convergence theorem that 
$$\int_\R e^{i\xi x}g_{\frac{1}{\sigma}}(\xi)\hat{\mu}_n(\xi)d\xi\xrightarrow{n\to\infty}\int_\R e^{i\xi x}g_{\frac{1}{\sigma}}(\xi)\hat{\mu}(\xi)d\xi.$$
These quantities are bounded by 1, and hence we can apply the dominated convergence theorem to obtain 
$$\int_\R g_\sigma * fd\mu_n\xrightarrow{n\to\infty} \int_\R g_\sigma * fd\mu.$$
Let now $H:=\{\varphi=g_\sigma *f\mid \sigma>0,f\in C_c(\R^d)\}\subset C_b(\R^d)$. Since $f\in C_c(\R^d)$ we get that $\|g_\sigma *f-f\|_{\infty}\xrightarrow{n\to\infty}0$ and thus $\bar H\supset C_c(\R^d)$. The result now follows from the previous proposition.
\end{proof}
\begin{thm}[Lévy]
Let $(\mu_n)_{n\geq 1}$ be a sequence of probability measures on $\R^d$ with characteristic functions $(\Phi_n)_{n\geq 1}$. If $\Phi_n$ converges pointwise to a function $\Phi$ which is continuous at 0, then 
$$\lim_{n\to\infty\atop w}\mu_n=\mu$$
for some probability measure $\mu$ on $\R^d$.
\end{thm}
\begin{proof}
No proof here.
\end{proof}
\begin{ex} Let $(X_n)_{n\geq 1}$ be a sequence of poisson r.v.'s with parameter $\lambda$. Moreover, consider the sequence $Z_n=\frac{X_n-1}{\sqrt{n}}$. Then we have
\begin{multline*}
E\left[e^{iu Z_n}\right]=\E\left[e^{iu\left(\frac{X_n-u}{\sqrt{n}}\right)}\right]=e^{-i\frac{u}{\sqrt{n}}}\E\left[e^{iu\frac{X_n}{\sqrt{n}}}\right]=e^{-i\frac{u}{\sqrt{n}}}e^{\lambda\left(e^{i \frac{u}{\sqrt{n}}}-1\right)}\\=e^{-i\frac{u}{\sqrt{n}}}e^{\lambda\left(i\frac{u}{\sqrt{n}}-\frac{u^2}{2n}+O\left(\frac{1}{n}\right)\right)}\xrightarrow{n\to\infty}e^{-\frac{u^2}{2}}.
\end{multline*}
Since $\E\left[e^{iu\mathcal{N}(0,1)} \right]=e^{-\frac{u^2}{2}}$, we deduce that $\lim_{n\to\infty\atop law}Z_n=\lim_{n\to\infty\atop law}\frac{X_n-u}{\sqrt{n}}=\mathcal{N}(0,1)$. Before stating and proving the central limit theorem, we give two extra results on convergence in law.
\end{ex}
\begin{thm}
Let $(\Omega,\A,\p)$ be a probability space. Let $(X_n)_{n\geq 1}$ be a sequence of r.v.'s and $X$ a r.v. and assume that $\lim_{n\to\infty\atop law}X_n=X$ and that $X$ is a.s. equal to a constant $a$. Then 
$$\lim_{n\to\infty\atop\p}X_n=X.$$
\end{thm}
\begin{proof}
Let $f(x):=\vert x-a\vert\land 1$. Then $f$ is a continuous and bounded map and therefore $\lim_{n\to\infty}\E[f(X_n)]=\E[f(X)]=0$, i.e. $\lim_{n\to\infty}\E[\vert X_n-a\vert \land 1]=0$ which implies that $\lim_{n\to\infty\atop \p}X_n=X$.
\end{proof}
\begin{thm}
Let $(\Omega,\A,\p)$ be a probability space. Let $(X_n)_{n\geq 1}$ be a sequence of r.v.'s and $X$ be r.v. in $\R^d$. Assume that $X_n$ has density $f_n$ for all $n\geq 1$ and $X$ has density $f$. Moreover, assume that $\lim_{n\to\infty}f_n(x)=f(x)$, a.e. Then 
$$\lim_{n\to\infty\atop law}X_n=X.$$
\end{thm}
\begin{proof}
We need to show that $\E[h(X_n)]\xrightarrow{n\to\infty}\E[h(X)]$, where $h:\R^d\to\R$ is a bounded and measurable map and
$$\E[h(X_n)]=\int_{\R^d}h(x)f_n(x)dx,$$
$$\E[h(X)]=\int_{\R^d}h(x)f(x)dx.$$
Let $h:\R^d\to\R$ be a bounded and measurable map. Moreover, let $\alpha=\sup_{x\in\R^d}\vert h(x)\vert.$ Set $h_1(x)=h(x)+\alpha\geq 0$ and $h_2(x)=\alpha-h(x)\geq 0$. So it follows that $h_1f_n\geq 0$ and $h_2f_n\geq 0$. Moreover, we also get that
$$h_1f_n\xrightarrow{n\to\infty}h_1f\hspace{0.2cm}\text{a.e.}$$
$$h_2f_n\xrightarrow{n\to\infty}h_2f\hspace{0.2cm}\text{a.e.}$$
With Fatou's lemma we get 
$$\E[h_1(X)]=\int_{\R^d}h_1(x)f(x)dx\leq \liminf_n\int_{\R^d}h_1(x)f_n(x)dx=\liminf_n\E[h_1(X_n)].$$
Similarly we get $\E[h_2(X)]\leq \liminf_n\E[h_2(X_n)]$. Now substitute $h_1(x)=h(x)+\alpha$ and $h_2(x)=\alpha-h(x)$, where we use $\liminf_n(-a_n)=\liminf_n(a_n)$ to obtain
$$\limsup_n\E[h(X_n)]\leq \E[h(X)]\leq \liminf_n\E[h(X_n)],$$
which implies that
$$\liminf_n\E[h(X_n)]=\limsup_n\E[h(X_n)]=\E[h(X)].$$
\end{proof}
\section{The Central limit theorem (real case)}
\begin{thm}[Central limit theorem (CLT)]
Let $(\Omega,\A,\p)$ be a probability space. Let $(X_n)_{n\geq 1}$ be a sequence of iid r.v.'s with values in $\R$. We assume that $\E[X_k^2]<\infty$ (i.e. $X_k\in L^2(\Omega,\A,\p)$) and let $\sigma^2=Var(X_k)$ for all $k\in\{1,...,n\}$. Then for all $k\in\{1,...,n\}$ we get
$$\sqrt{n}\left(\frac{1}{n}\sum_{i=0}^nX_i-\E[X_k]\right)\xrightarrow{n\to\infty\atop law}\mathcal{N}(0,\sigma^2).$$
Equivalently, for all $a,b\in\bar\R$ with $a<b$ and for all $k\in\{1,...,n\}$ we get
$$\lim_{n\to\infty}\p\left[\sum_{i=0}^nX_i\in \left[n\E[X_k]+a\sqrt{n},n\E[X_k]+b\sqrt{n}\right]\right]=\frac{1}{\sigma\sqrt{2\pi}}\int_a^be^{-\frac{x^2}{2\sigma^2}}dx.$$
\end{thm}
\begin{ex} If $\E[X_k]=0$ for all $k\in\{1,...,n\}$, then
$$\lim_{n\to\infty}\p\left[\frac{\sum_{i=1}^nX_i-n\E[X_k]}{\sqrt{n}}\in[a,b]\right]=\frac{1}{\sigma\sqrt{2\pi}}\int_a^be^{-\frac{x^2}{2\sigma^2}}dx.$$
\end{ex}
\begin{proof}
Without loss of generality we can assume that $\E[X_k]=0$, for all $k\in\{1,...,n\}$. Define now a sequence $Z_n=\frac{\sum_{i=1}^nX_i}{\sqrt{n}}$. Then we can obtain
$$\Phi_{Z_n}(\xi)=\E\left[e^{i\xi Z_n}\right]=\E\left[e^{i\xi \left(\frac{\sum_{i=1}^nX_i}{\sqrt{n}}\right)}\right]=\prod_{j=1}^n\E\left[e^{i\xi \frac{X_j}{\sqrt{n}}}\right]=\E\left[e^{i\xi\frac{X_k}{\sqrt{n}}}\right]^n=\Phi_{X_k}^n(\xi).$$
We have already seen that 
$$\Phi_{X_k}(\xi)=1+n\xi\E[X_k]-\frac{1}{2}\xi^2\E[X_k^2]+O(\xi^2)=1-\frac{\sigma^2}{2}\xi^2+O(\xi^2).$$
Finally, for fixed $\xi$, we have 
$$\Phi_{X_k}\left(\frac{\xi}{\sqrt{n}}\right)=1-\frac{\sigma^2\xi^2}{2n}+O\left(\frac{1}{n}\right)$$
$$\lim_{n\to\infty}\Phi_{X_k}^n(\xi)=\lim_{n\to\infty}\left(1-\frac{\sigma^2\xi^2}{2}+O\left(\frac{1}{n}\right)\right)^n=\E\left[e^{i\xi \mathcal{N}(0,\sigma^2)}\right].$$
\end{proof}
\begin{figure}
\begin{center}
\includegraphics[height=8cm, width=10cm]{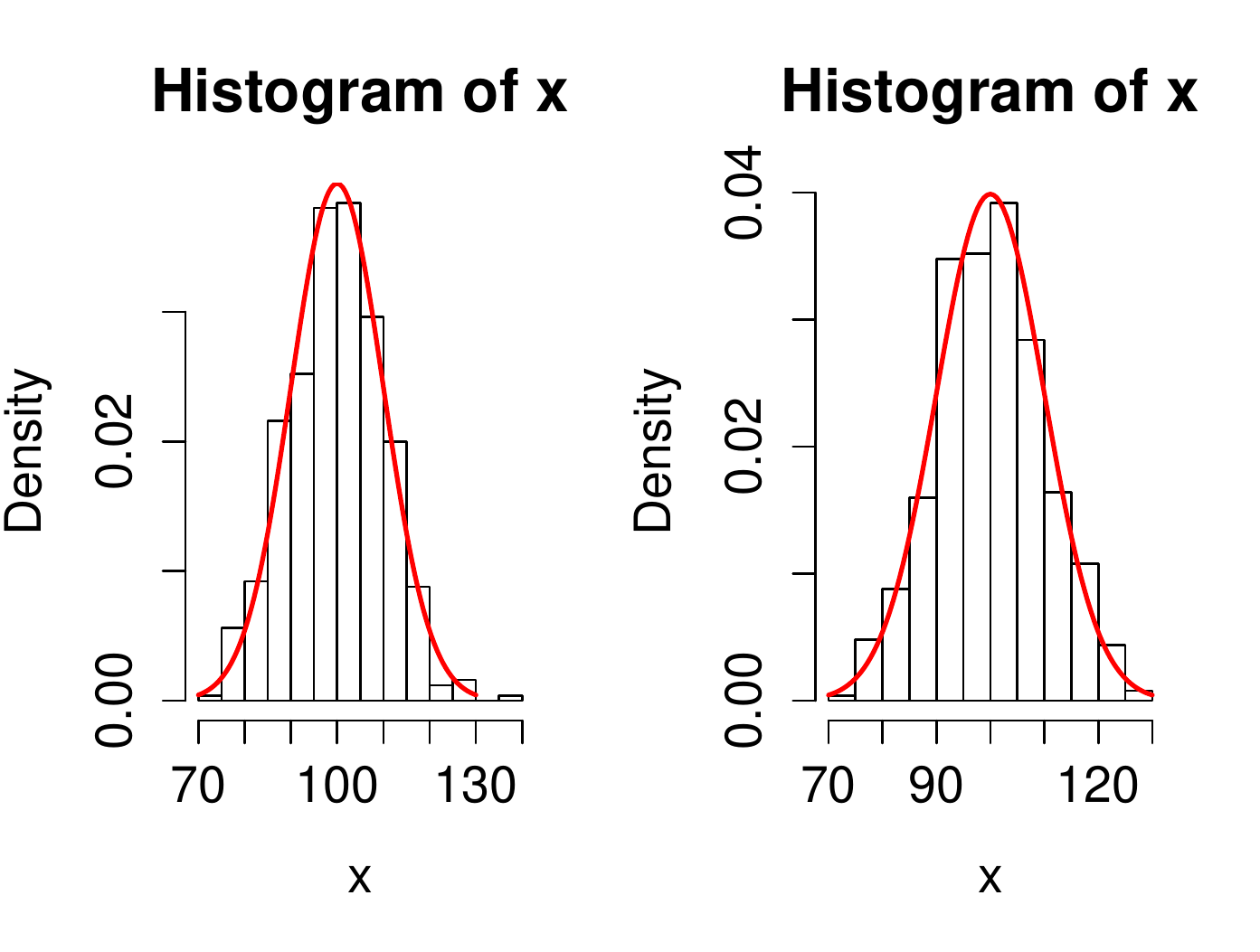}
\end{center}
\caption{For the illustration of the CLT, the distribution of a Gaussian distributed r.v. $X$ with $\mu=100$ and $\sigma=10$ ($X\sim \mathcal{N}(\mu,\sigma^2)$).}
\end{figure}
\begin{figure}
\begin{center}
\includegraphics[height=8cm, width=10cm]{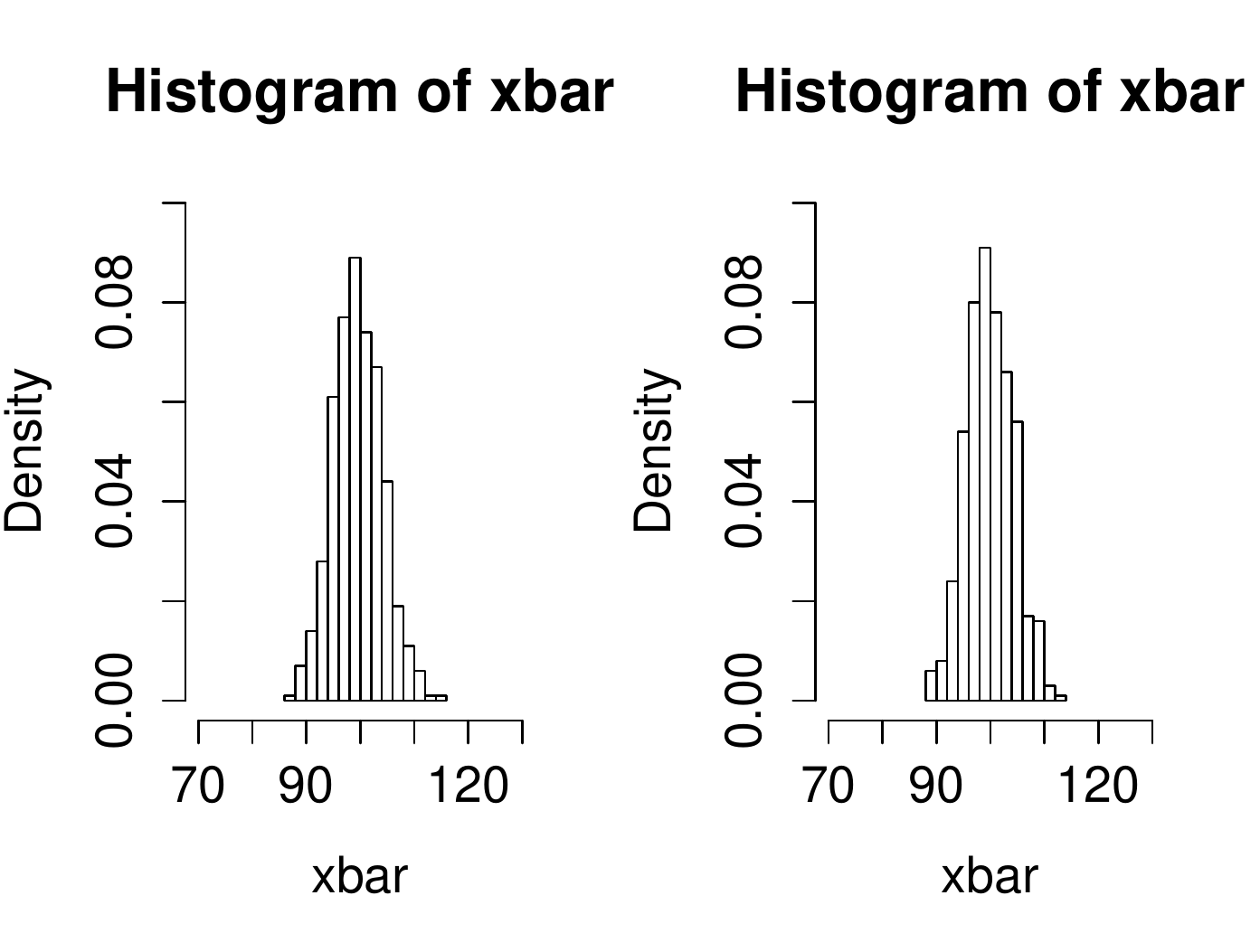}
\end{center}
\caption{The distribution of the mean of $n=5$ i.i.d. Gaussian r.v.'s $(X_n)$ with $\mu=100$ and $\sigma^2=10$ ($X_k\sim \mathcal{N}(\mu,\sigma^2)$). The figure illustrates the behavior of the distribution of the mean, which is given by $\bar X_n=\frac{1}{n}\sum_{k=1}^nX_k$.}
\end{figure}

\begin{figure}
\begin{center}
\includegraphics[height=8cm, width=10cm]{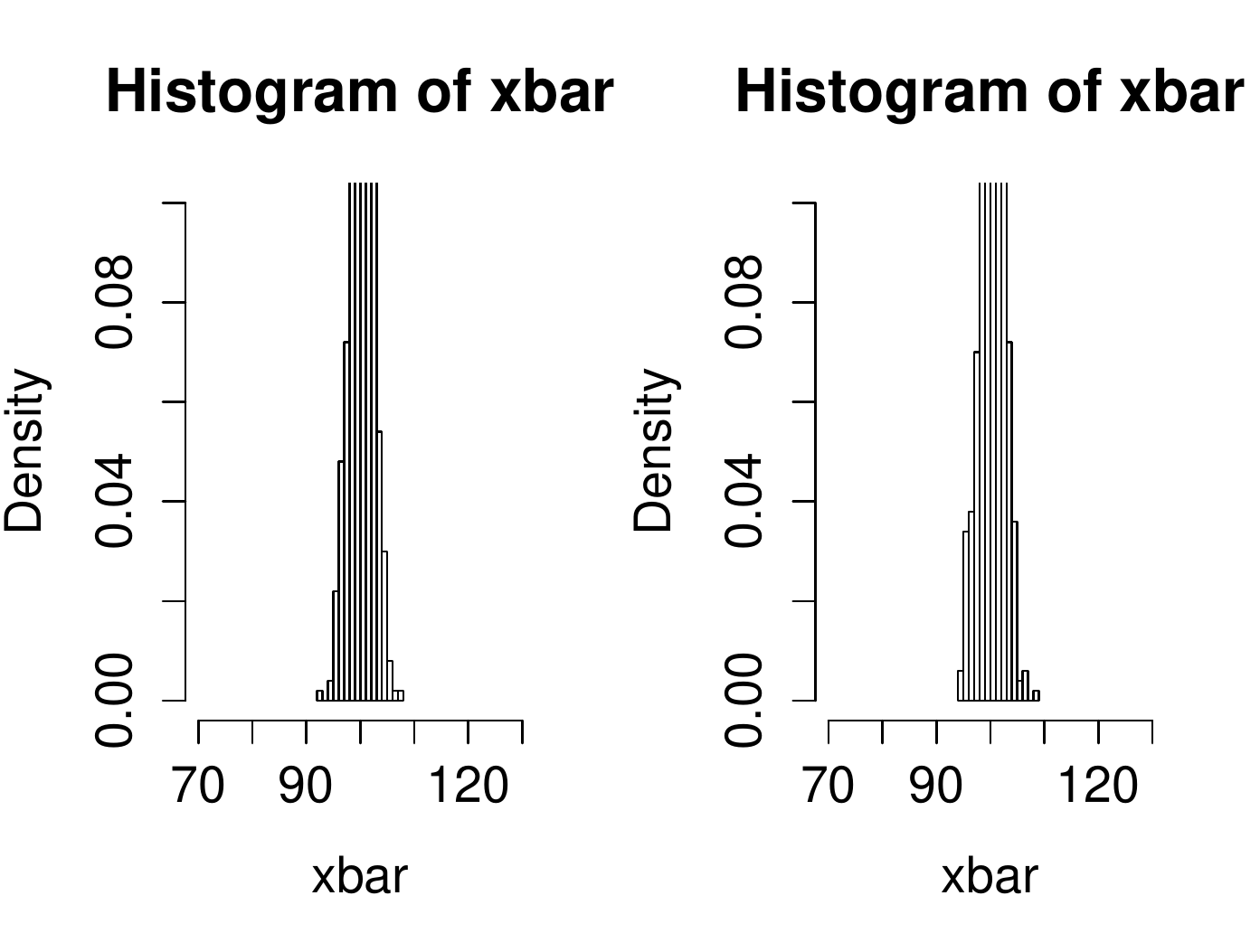}
\end{center}
\caption{The distribution of the mean of $n=20$ i.i.d. Gaussian r.v.'s $(X_n)$ with $\mu=100$ and $\sigma^2=10$ ($X_k\sim \mathcal{N}(\mu,\sigma^2)$). The figure illustrates the behavior of the distribution of the mean, which is given by $\bar X_n=\frac{1}{n}\sum_{k=1}^nX_k$.}
\end{figure}

\begin{figure}
\begin{center}
\includegraphics[height=8cm, width=10cm]{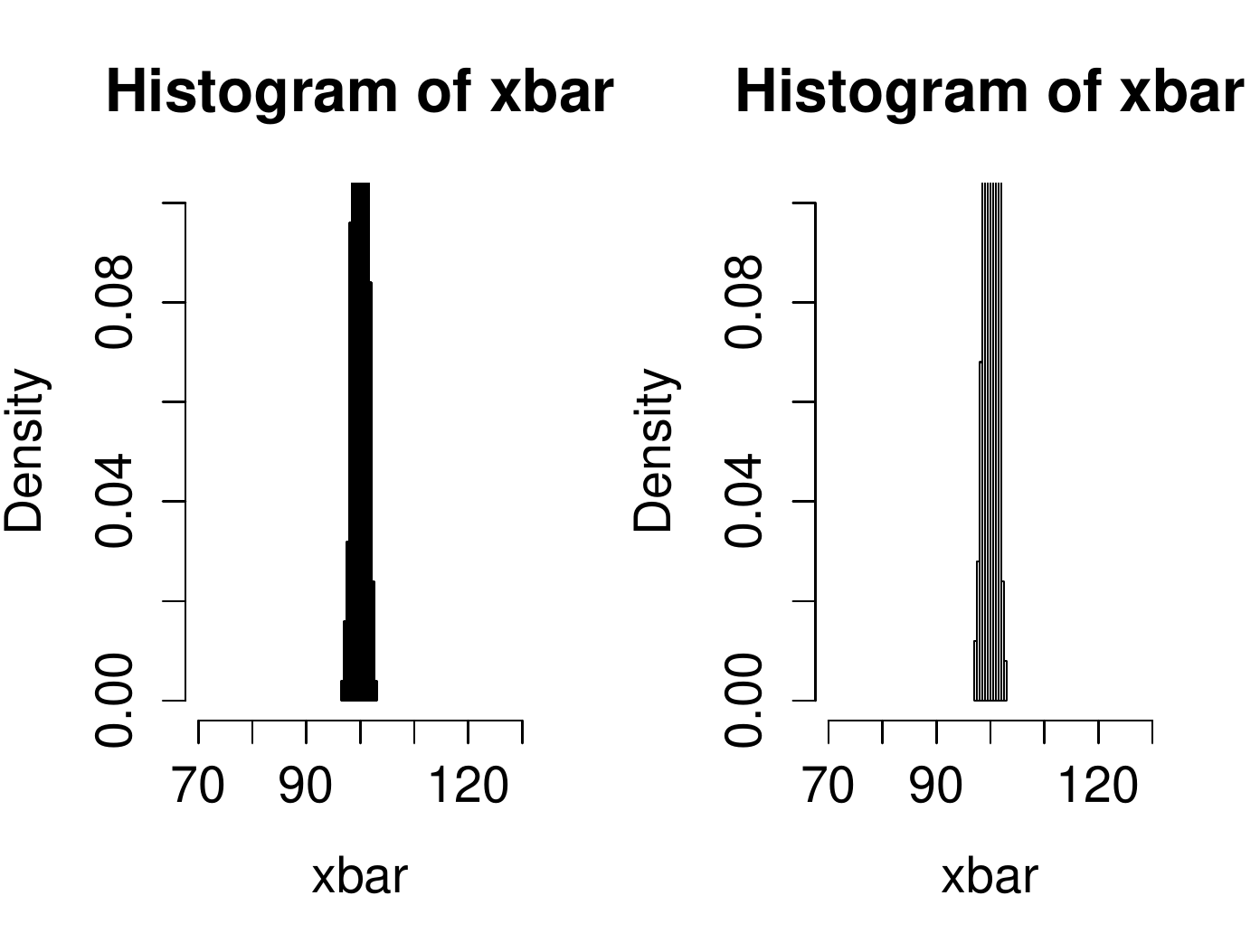}
\end{center}
\caption{The distribution of the mean of $n=100$ i.i.d. Gaussian r.v.'s $(X_n)$ with $\mu=100$ and $\sigma^2=10$ ($X_k\sim \mathcal{N}(\mu,\sigma^2)$). The figure illustrates the behavior of the distribution of the mean, which is given by $\bar X_n=\frac{1}{n}\sum_{k=1}^nX_k$.}
\end{figure}

\begin{figure}
\begin{center}
\includegraphics[height=8cm, width=10cm]{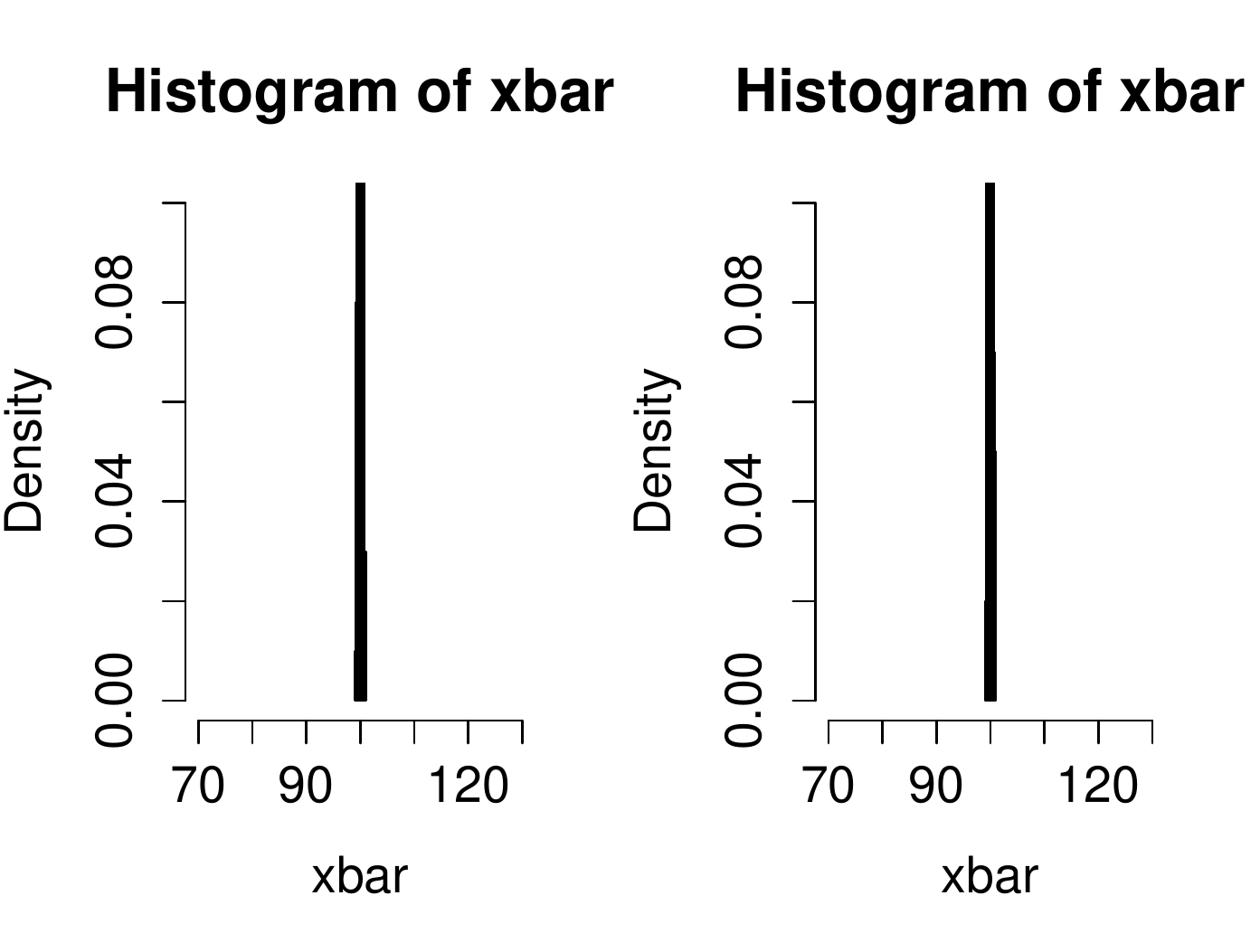}
\end{center}
\caption{The distribution of the mean of $n=1000$ i.i.d. Gaussian r.v.'s $(X_n)$ with $\mu=100$ and $\sigma^2=10$ ($X_k\sim \mathcal{N}(\mu,\sigma^2)$). The figure illustrates the behavior of the distribution of the mean, which is given by $\bar X_n=\frac{1}{n}\sum_{k=1}^nX_k$. Clearly we see that it's converging to $\mu$ for $n\to\infty$.}
\end{figure}

\begin{thm}
Let $(X_n)_{n\geq 1}$ be independent r.v.'s but not necessarily i.i.d. We assume that $\E[X_j]=0$ and that $\E[X_j^2]=\sigma_j^2<\infty$, for all $j\in\{1,...,n\}$. Assume further that $\sup_n\E[\vert X_n\vert^{2+\delta}]<\infty$ for some $\delta>0$, and that $\sum_{j=1}^\infty \sigma_j^2<\infty.$ Then 
$$\frac{\sum_{j=1}^nX_j}{\sqrt{\sum_{j=1}^n \sigma_j^2}}\xrightarrow{n\to\infty\atop law}\mathcal{N}(0,1).$$
\end{thm}
\begin{ex} We got the following examples:
\begin{enumerate}[$(i)$]
\item{Let $(X_n)_{n\geq 1}$ be i.i.d. r.v.'s with $\p[X_n=1]=p$ and $\p[X_n=0]=1-p$. Then $S_n=\sum_{i=1}^nX_i$ is a binomial r.v. $\B(p,n)$. We have $\E[S_n]=np$ and $Var(S_n)=np(1-p)$. Now with the strong law of large numbers we get $\frac{S_n}{n}\xrightarrow{n\to\infty\atop a.s.}p$ and with the central limit theorem we get 
$$\frac{S_n-np}{\sqrt{np(1-p)}}\xrightarrow{n\to\infty\atop law}\mathcal{N}(0,1).$$
}
\item{Let $\mathcal{P}$ be the set of prime numbers. For $p\in\mathcal{P}$, define $\B_p$ as $\p[\B_p=1]=\frac{1}{p}$ and $\p[\B_p=0]=1-\frac{1}{p}$. We take the $(\B_p)_{p\in\mathcal{P}}$ to be independent and 
$$W_n=\sum_{\mathcal{P}\subset\N\atop p\in\mathcal{P}}\B_p,$$
the probabilistic model for the total numbers of distinct prime divisors of $n:=W(0)$. It's a simple exercise to  check that $(W_n)_{n\geq 1}$ satisfies the assumption of theorem 13.2 and using the fact that $\sum_{p\leq n\atop p\in\mathcal{P}}\frac{1}{p}\sim \log\log n$ and we obtain 
$$\frac{W_n-\log\log n}{\sqrt{\log\log n}}\xrightarrow{n\to\infty\atop law}\mathcal{N}(0,1).$$
\begin{figure}
\begin{center}
\includegraphics[height=8cm, width=10cm]{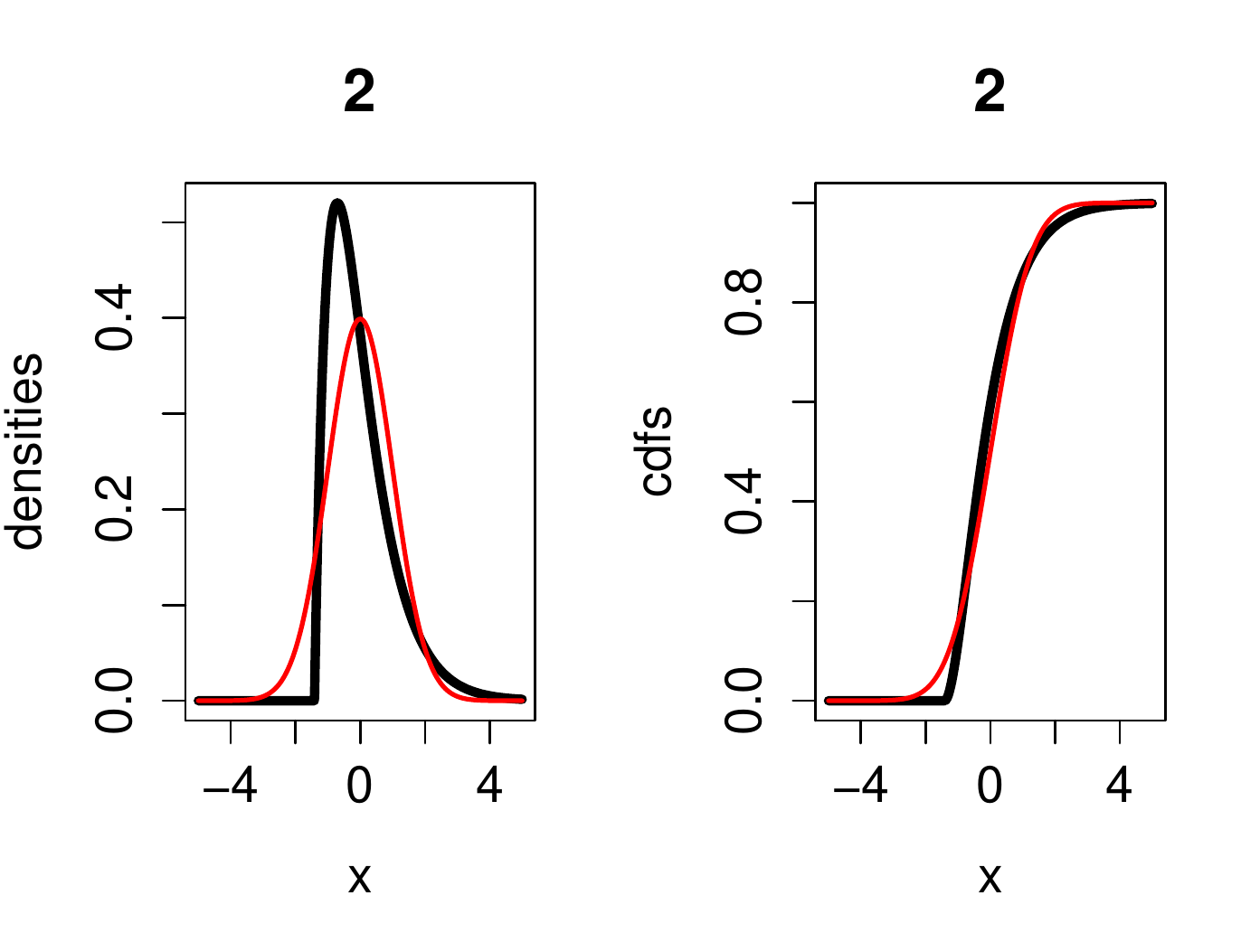}
\end{center}
\caption{An illustration of the CLT, where the r.v.'s $(X_n)$ are i.i.d. exponentially distributed. Also with the cumulative distribution functions. Here we have $n=2$ exponentially distributed r.v's $\left(X_k\sim \lambda e^{-\lambda}\right)$. On the left side, the black curve represents the density function of the r.v.'s and the red curve represents the density of a Gaussian r.v. $Y$ with $\mu=0$ ($Y\sim\mathcal{N}(0,\sigma^2)$). On the right side, the black curve represents the cumulative distribution function of the r.v.'s $X_k$ and the red curve represents the cumulative distribution function of a Gaussian r.v. $Y$.}
\end{figure}

\begin{figure}
\begin{center}
\includegraphics[height=8cm, width=10cm]{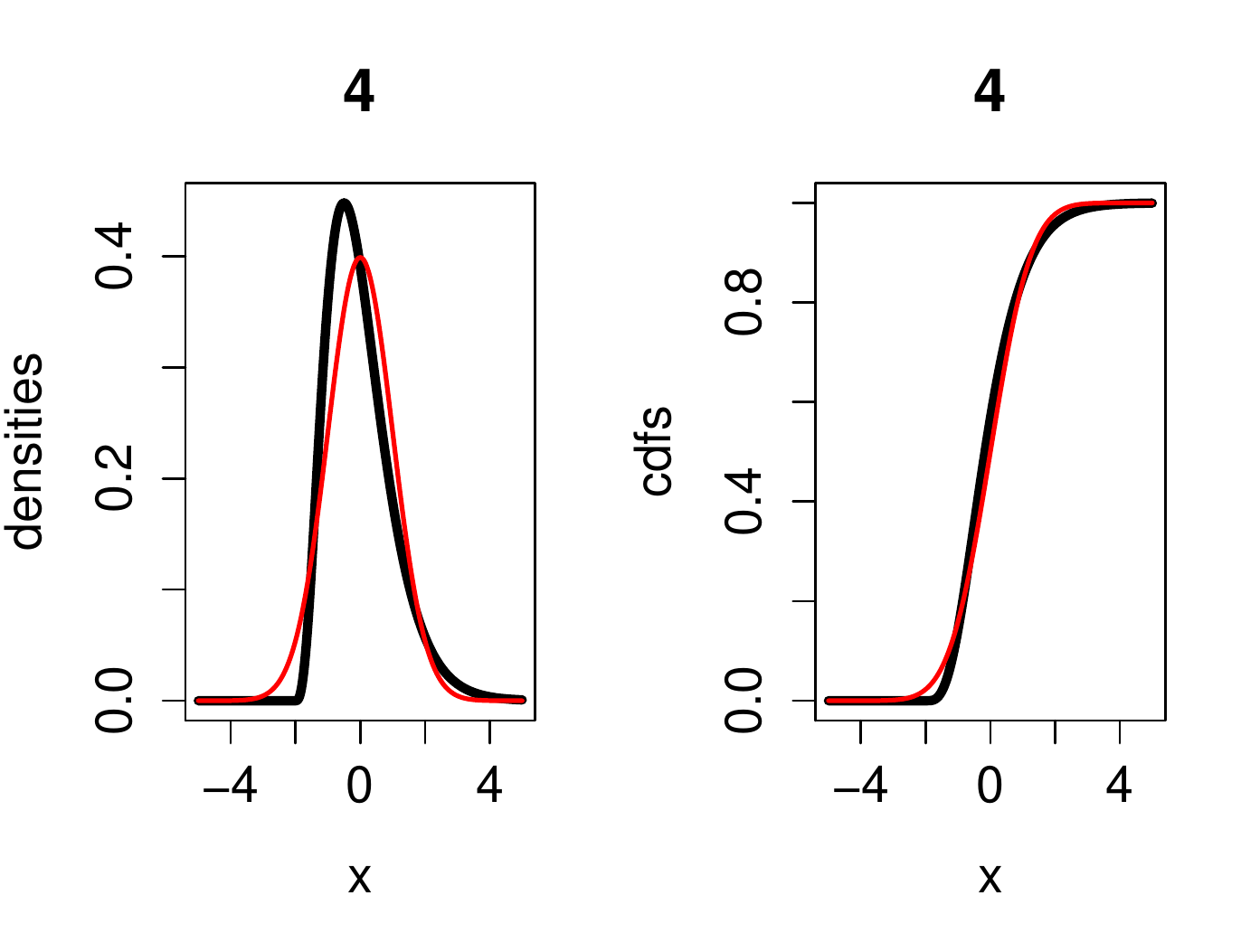}
\end{center}
\caption{An illustration of the CLT, where the r.v.'s $(X_n)$ are i.i.d. exponentially distributed. Also with the cumulative distribution functions. Here we have $n=4$ exponentially distributed r.v's $\left(X_k\sim \lambda e^{-\lambda}\right)$. On the left side, the black curve represents the density function of the r.v.'s and the red curve represents the density of a Gaussian r.v. $Y$ with $\mu=0$ ($Y\sim\mathcal{N}(0,\sigma^2)$). On the right side, the black curve represents the cumulative distribution function of the r.v.'s $X_k$ and the red curve represents the cumulative distribution function of a Gaussian r.v. $Y$.}
\end{figure}

\begin{figure}
\begin{center}
\includegraphics[height=8cm, width=10cm]{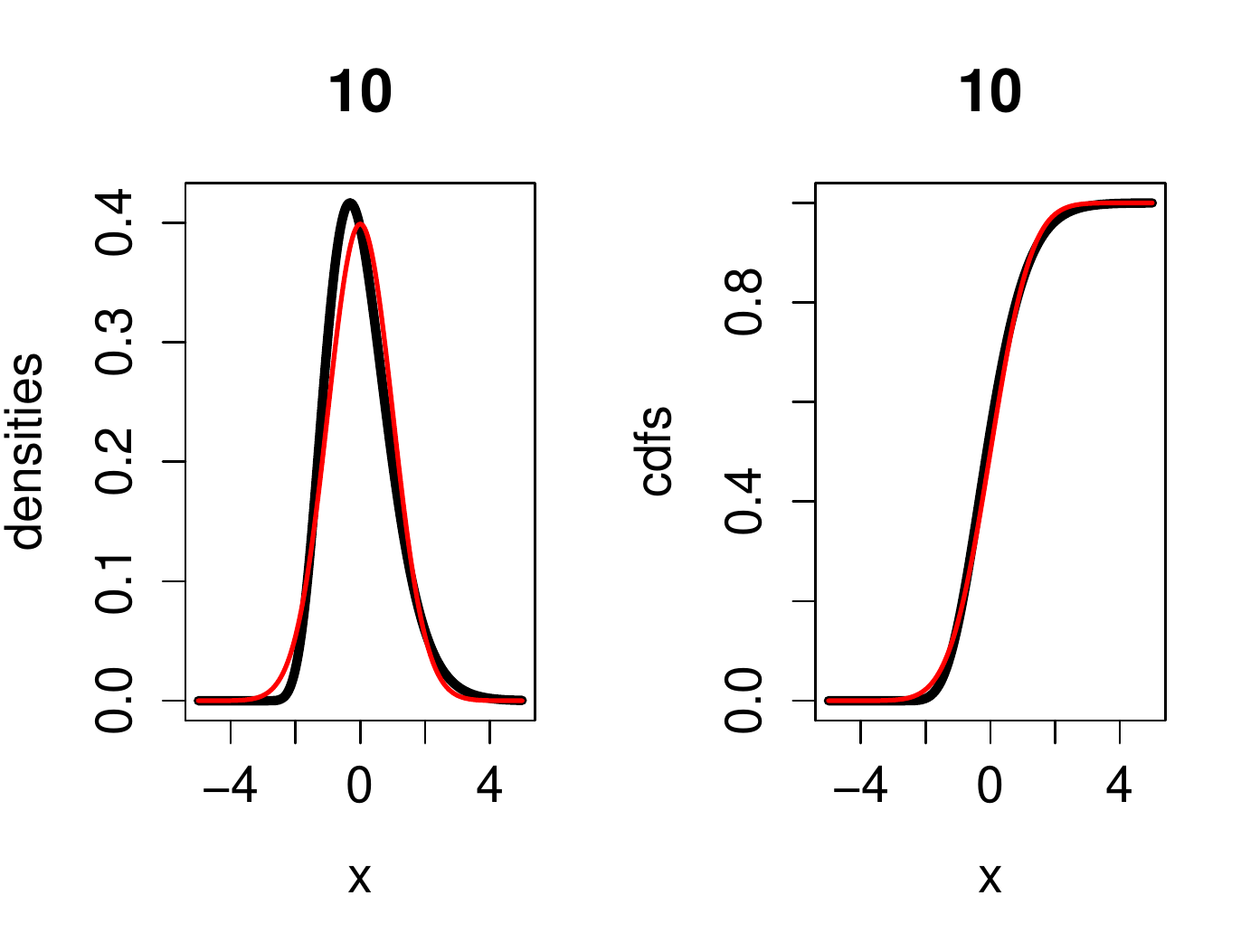}
\end{center}
\caption{An illustration of the CLT, where the r.v.'s $(X_n)$ are i.i.d. exponentially distributed. Also with the cumulative distribution functions. Here we have $n=10$ exponentially distributed r.v's $\left(X_k\sim \lambda e^{-\lambda}\right)$. On the left side, the black curve represents the density function of the r.v.'s and the red curve represents the density of a Gaussian r.v. $Y$ with $\mu=0$ ($Y\sim\mathcal{N}(0,\sigma^2)$). On the right side, the black curve represents the cumulative distribution function of the r.v.'s $X_k$ and the red curve represents the cumulative distribution function of a Gaussian r.v. $Y$.}
\end{figure}

\begin{figure}
\begin{center}
\includegraphics[height=8cm, width=10cm]{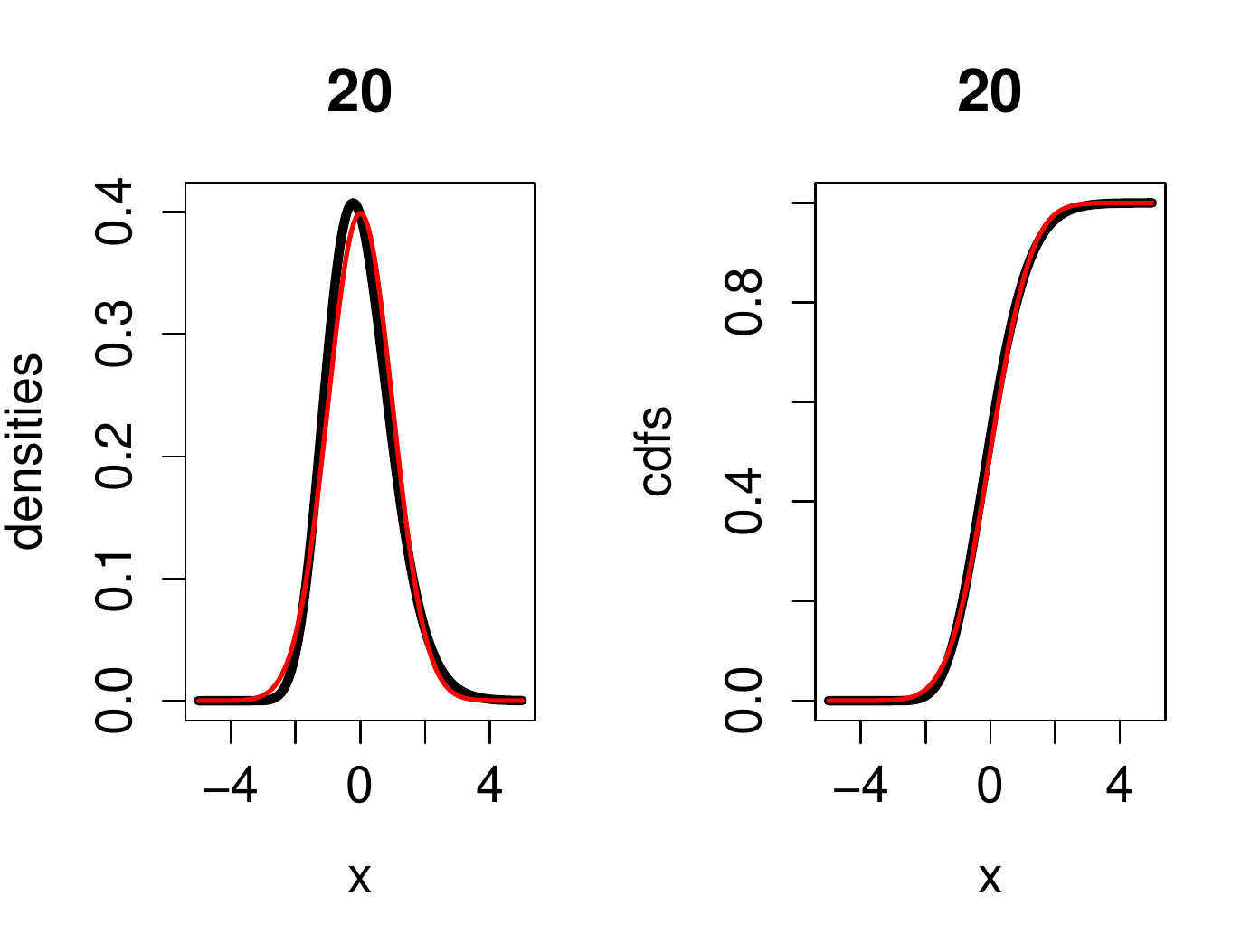}
\end{center}
\caption{An illustration of the CLT, where the r.v.'s $(X_n)$ are i.i.d. exponentially distributed. Also with the cumulative distribution functions. Here we have $n=20$ exponentially distributed r.v's $\left(X_k\sim \lambda e^{-\lambda}\right)$. On the left side, the black curve represents the density function of the r.v.'s and the red curve represents the density of a Gaussian r.v. $Y$ with $\mu=0$ ($Y\sim\mathcal{N}(0,\sigma^2)$). On the right side, the black curve represents the cumulative distribution function of the r.v.'s $X_k$ and the red curve represents the cumulative distribution function of a Gaussian r.v. $Y$. Now we can see that both, the density and the cumulative distribution function, are converging to a Gaussian density and a Gaussian cumulative distribution function for $n\to\infty$.}
\end{figure}

\begin{thm}[Erdös-Kac]
Let $N_n$ be a r.v. with uniformly distribution in $\{1,...,n\}$, then 
$$\frac{W(N_n)-\log\log n}{\sqrt{\log\log n}}\xrightarrow{n\to\infty\atop law}\mathcal{N}(0,1),$$
where $W(n)=\sum_{p\leq n\atop p\in\mathcal{P}}\one_{p|n}$.
\end{thm}
}
\item{Suppose that $(X_n)_{n\geq 1}$ are i.i.d. r.v.'s with distribution function $F(x)=\p[X_1\leq x]$. Let $Y_n(x)=\one_{X_n\leq x}$, where $(Y_n)_{n\geq 1}$ are i.i.d. Define $F_n(x)=\frac{1}{n}\sum_{k=1}^nY_k(x)=\frac{1}{n}\sum_{k=1}^n\one_{X_k\leq x}$. $F_n$ is called the empirical distribution function. With the strong law of large numbers we get $\lim_{n\to\infty\atop a.s.}F_n(x)=\E[Y_1(x)]$ and 
$$\E[Y_1(x)]=\E[\one_{X_1\leq x}]=\p[X_1\leq x]=F(x).$$
In fact, it is a theorem (Gliwenko-Cantelli) which says that 
$$\sup_{x\in\R}\vert F_n(x)-F(x)\vert\xrightarrow{n\to\infty\atop a.s.}0.$$
Next we note that 
$$\sqrt{n}(F_n(x)-F(x))=\sqrt{n}\left(\frac{1}{n}\sum_{k=1}^nY_k(x)-\E[Y_1(x)]\right)=\frac{1}{\sqrt{n}}\left(\sum_{k=1}^nY_k(x)-n\E[Y_1(x)]\right)$$
Now with the central limit theorem we get 
$$\sqrt{n}(F_n(x)-F(x))\xrightarrow{n\to\infty\atop law}\mathcal{N}(0,\sigma^2(x)),$$
where $\sigma^2(x)=Var(Y_1)=\E[Y_1^2(x)]=\E[Y_1(x)]=F(x)-F^2(x)=F(x)(1-F(x))$. Hence
$$\sqrt{n}(F_n(x)-F(x))\xrightarrow{n\to\infty\atop law}\mathcal{N}(0,F(x)(1-F(x))).$$
}
\end{enumerate}
\end{ex}
\begin{thm}[Berry-Esseen]
Let $(X_n)_{n\geq 1}$ be i.i.d. r.v's and suppose that $\E[\vert X_k\vert^3]<\infty$, $\forall k\in\{1,...,n\}$. Let 
$$G_n(x)=\p\left[\frac{\sum_{i=1}^nX_i-n\E[X_k]}{\sigma\sqrt{n}}\leq  x\right],\hspace{0.4cm}\forall k\in\{1,...,n\},$$ 
where $\sigma^2=\E[X_k^2]$ $\forall k\in\{1,...,n\}$ and $\Phi(x)=\p\left[\mathcal{N}(0,1)\leq x\right]=\int_{-\infty}^xe^{-\frac{u^2}{2}}\frac{1}{\sqrt{2\pi}}du$. Then
$$\sup_{x\in\R}\vert G_n(x)-\Phi(x)\vert\leq  C\frac{\E[\vert X_k\vert^3]}{\sigma^3\sqrt{n}},\hspace{0.4cm}\forall k\in\{1,...,n\}$$
where $C$ is a universal constant.
\end{thm}
\section{Multidimensional CLT}
\subsection{Gaussian Vectors}
\begin{defn}[Gaussian Random Vector]
An $\R^n$-valued r.v. $X=(X_1,...,X_n)$ is called a gaussian random vector if every linear combination $\sum_{j=1}^n\lambda_jX_j$, with $\lambda_j\in\R$, is a gaussian r.v. (Possibly degenerated $\mathcal{N}(\mu,0)=\mu$ a.s.).
\end{defn}
\begin{thm}
$X$ is an $\R^n$-valued gaussian r.v. if and only if its characteristic function has the form 
$$\varphi_X(u)=\exp\left(i\langle u,\mu\rangle-\frac{1}{2}\langle u,Qu\rangle\right),\hspace{2cm}(*)$$
where $\mu\in\R^n$ and $Q$ is a symmetric nonnegative semidefinit matrix of size $n\times n$. $Q$ is then the covariance matrix of $X$ and $\mu$ is the mean vector, i.e. $\E[X_j]=\mu_j$.
\end{thm}
\begin{proof}
Suppose that  $(*)$ holds. Let $Y=\sum_{j=1}^na_jX_j=\langle a,X\rangle$. For $v\in\R$, $\varphi_Y(v)=\varphi_X(va)=\exp\left(iv\langle a,\mu\rangle-\frac{v^2}{2}\langle a,Qa\rangle\right)\Longrightarrow Y\sim \mathcal{N}\left(\langle a,\mu\rangle,\langle a,Qa\rangle\right)\Longrightarrow X$ is a gaussian vector. Conversely assume that $X$ is a gaussian vector and let $Y=\sum_{j=1}^na_jX_j=\langle a,X\rangle$. Let $\omega=Cov(X)$ and note that $\E[Y]=\langle \sigma,\mu\rangle$ and $Var(Y)=\sigma^2(Y)=\langle a,Qa\rangle$. Since $Y$ is a gaussian r.v.
$$\varphi_Y(v)=\exp\left(iv\langle a,\mu\rangle-\frac{v^2}{2}\langle a,\omega a\rangle\right).$$
Now $\varphi_X(v)=\varphi_Y(1)=(*)$.
\end{proof}
\emph{Notation:} We write $X\sim \mathcal{N}(\mu,Q)$.
\begin{ex} Let $X_1,...,X_n$ be independent gaussian r.v.'s with $X_k\sim\mathcal{N}(\mu_k,\sigma_k^2)$. Then $X=(X_1,...,X_n)$ is a gaussian vector. Indeed, we have
\begin{align*}
\varphi_X(v_1,...,v_n)&=\E\left[e^{i(v_1X_1+...+v_nX_n)}\right]\\
&=\prod_{j=1}^n\E\left[e^{iv_jX_j}\right]\\
&=e^{i\langle v,\mu\rangle}-\frac{1}{2}\langle v,Qv\rangle,
\end{align*}
where $\mu=(\mu_1,...,\mu_n)$, $Q=\begin{pmatrix}\sigma_1^2&\dotsm &0\\ \vdots&\ddots&\vdots\\0&\dotsm&\sigma_n^2\end{pmatrix}$.
\end{ex}
\begin{cor}
Let $X$ be an $\R^n$-valued gaussian vector. The components $X_j$ of $X$ are independent if and only if $Q$ is a diagonal matrix. 
\end{cor}
\begin{proof}
Suppose $Q=\begin{pmatrix}\sigma_1^2&\dotsm &0\\ \vdots&\ddots&\vdots\\0&\dotsm&\sigma_n^2\end{pmatrix}$, then $(*)$ shows that 
$$\varphi_X(v_1,...,v_n)=\prod_{j=1}^n\varphi_{X_j}(v_j),$$
where $X_j\sim \mathcal{N}(\mu_j,\sigma^2_j)$. The result follows from the uniqueness of the r.v.
\end{proof}
\begin{thm}
Let $X$ be an $\R^n$-valued gaussian vector with mean $\mu$. Then there exists independent gaussian r.v.'s $Y_1,..,Y_n$ with 
$$Y_j=\mathcal{N}(0,\lambda_j),\hspace{0.2cm}\lambda_j\geq 0,\hspace{0.1cm}1\leq j\leq n$$
and an orthogonal matrix $A$ such that 
$$X=\mu+AY.$$
\end{thm}
\begin{rem}
It is possible that $\lambda_j=0$. In that case, it is also possible to get $Y_j=0$ a.s.
\end{rem}
\begin{proof}
There is an $A\in O(\R)$, such that $Q=A\Lambda A^*$, with $\Lambda=\begin{pmatrix}\lambda_1&\dotsm &0\\ \vdots&\ddots&\vdots\\0&\dotsm&\lambda_n\end{pmatrix}$, $\lambda_j\leq 0$. Set $Y=A^*(X-\mu)$. Then one can check that $Y$ is gaussian. So we get that $Cov(Y)=A^*QA=\Lambda$, which implies that $Y_1,...,Y_n$ are independent because $Cov(Y)$ is diagonal.
\end{proof}
\begin{cor}
An $\R^n$-valued gaussian vector $X$ has density on $\R^n$ if and only if $\det(Q)\not=0$.
\end{cor}
\begin{rem}
If $\det(Q)\not=0$, then $f_X(x)=\frac{1}{\sqrt{2\pi}\sqrt{\det(Q)}}e^{-\frac{1}{2}\langle x-\mu,Q^{-1}(\lambda-\mu)\rangle}$.
\end{rem}
\begin{thm}
Let $X$ be an $\R^n$-valued gaussian r.v. and let $Y$ be an $\R^m$-valued gaussian r.v. If $X$ and $Y$ are independent, then $Z=(X,Y)$ is an $\R^{n+m}$-valued gaussian vector.
\end{thm}
\begin{proof}
Let $u=(w,v)$, $w\in\R^n$ and $v\in\R^m$. Take $Q=\begin{pmatrix}Q^X&0\\ 0&Q^Y\end{pmatrix}$. Now we get
\begin{align*}
\varphi_Z(u)&=\varphi_X(w)\varphi_Y(v)\\
&=\exp\left(i\langle w,\mu X\rangle-\frac{1}{2}\langle w,Q^Xw\rangle\right)+\exp\left(i\langle v,\mu^Y\rangle-\frac{1}{2}\langle v,Q^Yv\rangle \right)\\
&=\exp\left(i\langle (w,v),(\mu^X,\mu^Y)\rangle-\frac{1}{2}\langle u,Qu\rangle\right),
\end{align*}
implying that $Z$ is a gaussian vector.
\end{proof}
\begin{thm}
Let $X$ be an $\R^n$-valued gaussian vector. Two components $X_j$ and $X_k$ of $X$ are independent if and only if $Cov(X_j,X_k)=0$.
\end{thm}
\begin{proof}
Consider $Y=(Y_1,Y_2)$, with $Y_1=X_j$ and $Y_2=X_k$. If $Y$ is a gaussian vector, then $Cov(Y_1,Y_2)=0$, which implies that $Y_1$ and $Y_2$ are independent.
\end{proof}
\emph{Warning!:}
Let $Y\sim\mathcal{N}(0,1)$ and for $a>0$ fix $Z=Y\one_{\vert Y\vert\leq a}-Y\one_{\vert Y\vert>a}$. Then $Z\sim\mathcal{N}(0,1)$. But $Y+Z=2Y\one_{\vert Y\vert \leq  a}$ is not gaussian because it is a bounded r.v. and it is not constant. Therefore $(Y,Z)$ is not a gaussian vector.

\part{Conditional Expectations, Martingales and Markov Chains}

\chapter*{Introduction}
If we consider a probability space $(\Omega,\F,\p)$ with a sequence of iid r.v.'s $(X_n)_{n\geq1}$, we can look at the expectation $\E[\vert X_1\vert]=\dotsm =\E[\vert X_n\vert]<\infty$. Now limit theorems play a central role, as we have seen in stochastics I. For example we have seen the strong law of large numbers, that is 
\[
\frac{X_1+\dotsm+X_n}{n}\xrightarrow{n\to \infty}\E[X_1] \hspace{0.3cm}\text{a.s.}
\]
Notice that $\E[X_1]=\dotsm =\E[X_n]$ since $(X_n)_{n\geq1}$ are iid. Another very important limit theorem is the central limit theorem (CLT), that is
\[
\sqrt{n}\left(\frac{X_1+\dotsm+X_n}{n}\right)\xrightarrow{n\to\infty\atop law}\mathcal{N}(0,1).
\]
That means that the distribution of the sum of the r.v.'s over $\sqrt{n}$ converges to a standard Gaussian distribution. Notice that it doesn't matter what the distribution of the $X_i$ is. The way we have proved this in stochastics I was the approach of the characteristic function. Take 
\[
\E\left[e^{it\left(\frac{X_1+\dotsm+X_n}{\sqrt{n}}\right)}\right]=\prod_{i=1}^n\E\left[e^{it\frac{X_1}{\sqrt{n}}}\right]=\left(\E\left[e^{it\frac{X_1}{\sqrt{n}}}\right]\right)^n\xrightarrow{n\to\infty}e^{-\frac{t^2}{2}}.
\]
Since we know that the characteristic function of a standard Gaussian is $e^{-\frac{t^2}{2}}$, we get the claim. 

\vspace{1cm}

Now the more interesting question is what kind of dependence structure can one put on a family of r.v.'s $(Z_n)_{n\geq1}$? 

\vspace{0.5cm}

This question will be discussed in detail in this notes and will lead to two very important notions in probability theory:

\begin{itemize}
\item{The notion of martingales}
\item{The notion of Markov-chains (ergodicity)}
\end{itemize}

The extra notion we get is that $n$ is a representative for the time, i.e. one can imagine stochastic processes changing with time (time evolution).

\vspace{0.5cm}

Therefore we consider tuples of the form $(Z_n,\F_n)_{n\geq1}$, where $\F_n$ is a $\sigma$-Algebra for all $n\geq 1$. It is often important to write the tuple down and to emphasize the $\F_n's$. Assume that the space of events in infinite time steps is known and denote it by $\F$. Then one will see that $\F_n\subset \F$ and moreover $\F_n\subset \F_{n+1}$. This is a very important fact and is known as filtration. 

\vspace{1cm}

Another very important thing is the notion of conditional expectation and conditional distribution, which we will cover as the first part of these notes. 

\chapter{Conditional expectations}
\section{$L^2(\Omega,\F,\p)$ as a Hilbert space and orthogonal projections}
In this section we will always work on a probability space $(\Omega,\F,\p)$. Consider the space $L^2(\Omega,\F,\p)$, which is given by
\[
L^2(\Omega,\F,\p):=\{X:\Omega\to \R \hspace{0.2cm}\text{a r.v.}\mid \E[X^2]<\infty\}.
\]
More precisely, consists only of equivalence classes, i.e. $X$ and $Y$ are identified if $X=Y$ a.s. We also know that we have a norm on this space given by 
\[
\|X\|_2=\E[X^2]^{1/2}.
\]
\begin{rem}
Recall that on $\mathcal{L}^2(\Omega,\F,\p)$ this would only be a semi-norm rather than a norm.
\end{rem}
We can also define an inner product on $L^2(\Omega,\F,\p)$ by 
\[
\langle X,Y\rangle=\E[XY].
\]
One can easily check that this satisfies the conditions of an inner product. It remains to show that $\vert\langle X,Y\rangle\vert<\infty$. By Cauchy-Schwarz we get 
\[
\vert\langle X,Y\rangle\vert\leq \|X\|_2\|Y\|_2=\E[X^2]^{1/2}\E[Y^2]^{1/2},
\]
and since we assume that the second moment of our r.v.'s exists, this is finite. Moreover one can see that 
\[
\sqrt{\langle X,X\rangle}=\E[X^2]^{1/2}=\|X\|_2.
\]
Now we know that $L^2(\p)$ is a Banach space and therefore a complete, normed vector space, i.e. every Cauchy sequence has a limit inside the space with respect to the norm. Since we have also an inner product on $L^2(\p)$, it is also a Hilbert space. We shall recall what a Hilbert space is.
\begin{defn}[Hilbert space]
An inner product space $(\mathcal{H},\langle\cdot,\cdot\rangle)$ is called a Hilbert space, if it is complete and the norm is derived from the inner product. If it is not complete it is called a pre-Hilbert space.
\end{defn}
\begin{rem}
It is in fact important that we have noted $L^2(\Omega,\F,\p)$, for instance if $\mathcal{G}$ is a $\sigma$-Algebra and $\mathcal{G}\subset\F$, then $L^2(\Omega,\mathcal{G},\p)\subset L^2(\Omega,\F,\p)$. When we have several $\sigma$-Algebras, we write explicitly the dependence of them, by noting for instance $L^2(\Omega,\mathcal{G},\p),L^2(\Omega,\F,\p)$, etc.
\end{rem}
\begin{ex}
Take $\mathcal{H}=\R^n$. This is indeed a Hilbert space with the euclidean inner product, i.e. if $X,Y\in\R^n$ then 
\[
\langle X,Y\rangle=\sum_{i=1}^nX_iY_i.
\]
Be aware that this is an example of a  finite dimensional Hilbert space, but $L^2(\p)$ is a infinite dimensional Hilbert space. Another example would be $\mathcal{H}=l^2(\N)$ the space of square summable sequences a subspace of $l^\infty(\N)$ which is the space of convergent sequences, i.e. 
\[
l^2(\N):=\left\{x=(x_n)_{n\geq1}\in l^\infty(\N)\mid \sum_{n=1}^\infty \vert x_n\vert^2<\infty\right\}.
\]
This is also an example of an infinite dimensional Hilbert space. It is clearly related to the $L^p(\mu)$ spaces, where we use the counting measure. The theory of Hilbert spaces is discussed in more detail in a course on functional analysis.

\vspace{1cm}

We want to give here more, maybe a bit harder, examples of Hilbert spaces used in functional analysis. 

\begin{enumerate}[$(i)$]
\item{{\bf Sobolev spaces:} Sobolev spaces, denoted by $H^s$ or $W^{s,2}$, are Hilbert spaces. These are a special kind of a function space in which differentiation may be performed, but that support the structure of an inner product. Because differentiation is permitted, Sobolev spaces are a convenient setting for the theory of partial differential equations. They also form the basis of the theory of direct methods in the calculus of variations. For $s$ a nonnegative integer and $\Omega\subset\R^n$, the Sobolev space $H^s(\Omega)$ contains $L^2$-functions whose weak derivatives of order up to $s$ are also in $L^2$. The inner product in $H^s(\Omega)$ is 
\[
\langle f,g\rangle_{H^s(\Omega)}=\int_\Omega f(x)\bar g(x)d\mu(x)+\int_\Omega Df(x)\cdot D\bar g(x)d\mu(x)+...+\int_\Omega D^sf(x)\cdot D^s\bar g(x)d\mu(x),
\]
where the dot indicates the dot product in the euclidean space of partial derivatives of each order. Slobber spaces can also be defined when $s$ is not an integer.
}
\item{{\bf Hardy spaces:} The Hardy spaces are function spaces, arising in complex analysis and harmonic analysis, whose elements are certain holomorphic functions in a complex domain. Let $U$ denote the unit disc in the complex plane. Then the Hardy space $\mathcal{H}^2(U)$ is defined as the space of holomorphic functions $f$ on $U$ such that the means 
\[
M_r(f)=\frac{1}{2\pi}\int_0^{2\pi}\vert f(re^{i\theta})\vert^2d\theta
\]
remains bounded for $r<1$. The norm on this Hardy space is defined by 
\[
\|f\|_{\mathcal{H}^2(U)}=\lim_{r\to 1}\sqrt{M_r(f)}.
\]
Hardy spaces in the disc are related to Fourier series. A function $f$ is in $\mathcal{H}^2(U)$ if and only if $f(z)=\sum_{n=0}^\infty a_nz^n$, where $\sum_{n=0}^\infty\vert a_n\vert^2<\infty$. Thus $\mathcal{H}^2(U)$ consists of those functions that are $L^2$ on the circle and whose negative frequency Fourier coefficients vanish.
}
\item{{\bf Bergman spaces:} The Bergman spaces are another family of Hilbert spaces of holomorphic functions. Let $D$ be a bounded open set in the complex plane (or a higher-dimensional complex space) and let $L^{2,h}(D)$ be the space of holomorphic functions $f$ in $D$ that are also in $L^2(D)$ in the sense that 
\[
\|f\|^2=\int_D\vert f(z)\vert^2 d\mu(z)<\infty,
\] 
where the integral is taken with respect to the Lebesgue measure in $D$. Clearly $L^{2,h}(D)$ is a subspace of $L^2(D)$; in fact, it is a closed subspace and so a Hilbert space in its own right. This is a consequence of the estimate, valid on compact subsets $K$ of $D$, that 
\[
\sup_{z\in K}\vert f(z)\vert\leq  C_K\|f\|_2,
\]
which in turn follows from Cauchy's integral formula. Thus convergence of a sequence of holomorphic functions in $L^2(D)$ implies also compact convergence and so the function is also holomorphic. Another consequence of this inequality is that the linear functional that evaluates a function $f$ at a point of $D$ is actually continuous on $L^{2,h}(D)$. The Riesz representation theorem (see notes on measure and integral) implies that the evaluation functional can be represented as an element of $L^{2,h}(D)$. Thus, for every $z\in D$, there is a function $\eta_z\in L^{2,h}(D)$ such that 
\[
f(z)=\int_Df(\zeta)\overline{\eta_z(\zeta)}d\mu(\zeta)
\]
for all $f\in L^{2,h}(D)$. The integrand $K(\zeta,z)=\overline{\eta_z(\zeta)}$ is known as the Bergman kernel of $D$. This integral kernel satisfies a reproducing property
\[
f(z)=\int_D f(\zeta) K(\zeta,z)d\mu(\zeta).
\]
A Bergman space is an example of a reproducing kernel Hilbert space, which is a Hilbert space of functions along with a kernel $K(\zeta,z)$ that verifies a reproducing property analogous to this one. The Hardy space $\mathcal{H}^2(D)$ also admits a reproducing kernel, known as the $Szeg\ddot{o}$ kernel. Reproducing kernels are common in other areas of mathematics as well.  For instance, in harmonic analysis the Poisson kernel is a reproducing kernel for the Hilbert space of square integrable harmonic functions in the unit ball. That the latter is a Hilbert space at all is a consequence of the mean value theorem for harmonic functions.
}
\end{enumerate}
\end{ex}
\begin{rem}
The notion of Hilbert spaces allow us to do basic geometry on them. Even if our space is an infinite dimensional vector space, we can still make sense of geometrical meanings, for example orthogonality, only by using the inner product on the space. 
\end{rem}
\begin{defn}[Orthogonal]
Two elements $X,Y$ in a Hilbert space $(\mathcal{H},\langle\cdot,\cdot\rangle)$ are said to be orthogonal if 
\[
\langle X,Y\rangle=0.
\]
\end{defn}
\begin{rem}
For a real valued Hilbert space $\mathcal{H}$ we get the following identity. For every $X,Y\in\mathcal{H}$ 
\[
\|X+Y\|^2=\langle X+Y,X+Y\rangle=\langle X,X\rangle+\langle X,Y\rangle+\langle Y,X\rangle+\langle Y,Y\rangle=\|X\|^2+\|Y\|^2+2\langle X,Y\rangle,
\]
and if $X$ and $Y$ are orthogonal, i.e. $\langle X,Y\rangle=0$, we get the usual pythagorean relation
\[
\|X+Y\|^2=\|X\|^2+\|Y\|^2.
\]
\end{rem}
\begin{thm}
Let $(X_n)_{n\geq1}$ and $(Y_n)_{n\geq 1}$ be two converging sequences in a Hilbert space $\mathcal{H}$ such that $X_n\xrightarrow{n\to\infty}X$ and $Y_n\xrightarrow{n\to\infty}Y$. Then 
\[
\langle X_n,Y_n\rangle\xrightarrow{n\to\infty}\langle X,Y\rangle.
\]
(In particular, $X_n=Y_n$ gives us that $\|X_n\|\xrightarrow{n\to\infty}\|X\|$)
\end{thm}
\begin{proof}
We can look at the difference, which is given by 
\begin{align*}
\vert \langle X,Y\rangle-\langle X_n,Y_n\rangle\vert&=\vert \langle X-X_n,Y\rangle+\langle X_n,Y\rangle-\langle X_n,Y_n\rangle\vert\\
&=\vert\langle X-X_n,Y\rangle+\langle X_n,Y-Y_n\rangle\vert\\
&\leq  \vert \langle X-X_n,Y\rangle\vert +\vert X_n,Y-Y_n\rangle\vert\\
&\leq  \|X-X_n\|\|Y\|+\|X_n\|\|Y-Y_n\|,
\end{align*}
where we have used the Cauchy-Schwarz inequality. Now since for $n$ large enough there is some $\varepsilon>0$ such that $\|X-X_n\|<\varepsilon$ and $\|Y-Y_n\|<\varepsilon$ by assumption, and the fact that $\|X_n\|$ is bounded independently of $n$, we get the claim.
\end{proof}
\begin{lem}[Parallelogram identity]
Let $\mathcal{H}$ be a Hilbert space. For all $X,Y\in\mathcal{H}$ we get 
\[
\|X+Y\|^2+\|X-Y\|^2=2(\|X\|^2+\|Y\|^2).
\]
Moreover if a norm satisfies the parallelogram identity, it can be derived from an inner product.
\end{lem}
\begin{proof}
Exercise\footnote{Use the fact $\|X+Y\|^2=\langle X+Y,X+Y\rangle$. This proof can be found in the notes of measure and integral}.
\end{proof}
\begin{defn}{\bf Closed linear subset}
Let $\mathcal{H}$ be a Hilbert space and let $\mathcal{L}\subset\mathcal{H}$ be a linear subset. $\mathcal{L}$ is called closed if for every sequence $(X_n)_{n\geq1}$ in $\mathcal{L}$ with $X_n\xrightarrow{n\to\infty}X$ we get that $X\in\mathcal{L}$.
\end{defn}
\begin{thm}
Let $\mathcal{H}$ be a Hilbert space and let $\Gamma\subset\mathcal{H}$ be a subset. Let $\Gamma^\perp$ denote the set of all elements of $\mathcal{H}$ which are orthogonal to $\Gamma$, i.e.
\[
\Gamma^\perp=\{X\in\mathcal{H}\mid \langle X,\gamma\rangle=0,\forall\gamma\in\Gamma\}.
\]
Then $\Gamma^\perp$ is a closed subspace of $\mathcal{H}$. We call $\Gamma^\perp$ the orthogonal complement of $\Gamma$.
\end{thm}
\begin{proof}
Let $\alpha,\beta\in \R$ and $X,X'\in\Gamma^\perp$. It is clear that for all $Y\in\Gamma$
\[
\langle\alpha X + \beta X',Y\rangle=0.
\]
Hence $\Gamma^\perp$ is a linear subspace of $\mathcal{H}$. Next we want to check whether it's closed. Take a sequence $(X_n)_{n\geq1}$ in $\Gamma^\perp$ such that $X_n\xrightarrow{n\to\infty}X$ with $X\in\mathcal{H}$. Now for all $Y\in\Gamma$ we get $\langle X_n,Y\rangle=0$ and $\langle X_n,Y\rangle\xrightarrow{n\to\infty}\langle X,Y\rangle$ because of the previous theorem. Hence $\langle X,Y\rangle=0$ and therefore $X\in \Gamma^\perp$ and the claim follows.
\end{proof}
\begin{defn}[Distance to a closed subspace]
Let $\mathcal{H}$ be a Hilbert space and let $X\in\mathcal{H}$. Moreover let $\mathcal{L}\subset\mathcal{H}$ be a closed subspace. The distance of $X$ to $\mathcal{L}$ is given by 
\[
d(X,\mathcal{L})=\inf_{Y\in\mathcal{L}}\|X-Y\|=\inf\{\|X-Y\|\mid Y\in\mathcal{L}\}.
\]
\end{defn}
\begin{rem}
Since $\mathcal{L}$ is closed, $X\in\mathcal{L}$ if and only if $d(X,\mathcal{L})=0$. 
\end{rem}
\subsection{Convex sets in uniformly Convex spaces} 
While the emphasis in this section is on Hilbert spaces, it is useful to isolate a more abstract property which is precisely what is needed for several proofs. 
\begin{defn}[Uniformly convex vector space]
A normed vector space $(V,\|\cdot\|)$ is called uniformly convex if for $X,Y\in V$
\[
\|X\|,\|Y\|\leq  1\Longrightarrow\left\|\frac{X+Y}{2}\right\|\leq  1-\psi(\|X-Y\|),
\]
where $\psi:[0,2]\to[0,1]$ is a monotonically increasing function with $\psi(r)>0$ for all $r>0$.
\end{defn}
\begin{lem}
A Hilbert space $(\mathcal{H},\langle\cdot,\cdot\rangle)$ is uniformly convex.
\end{lem}
\begin{proof}
For $X,Y\in \mathcal{H}$ with $\|X\|,\|Y\|\leq  1$ then by parallelogram identity we have
\[
\left\|\frac{X+Y}{2}\right\|=\sqrt{\frac|{1}{2}\|X\|^2+\frac{1}{2}\|Y\|^2-\frac{1}{2}\|X-Y\|^2}\leq \sqrt{1-\frac{1}{2}\|X-Y\|^2}=1-\psi(\|X-Y\|)
\]
as required, with $\psi(r)=1-\sqrt{1-\frac{1}{2}r^2}$.
\end{proof}
Heuristically, we can think of Definition 1.4 as having the following geometrical meaning. If vectors $X$ and $Y$ have norm (length) one, then their mid-point $\frac{X+Y}{2}$ has much smaller norm unless $X$ and $Y$ are very close together. This accords closely with the geometrical intuition from finite-dimensions spaces with euclidean distance. The following theorem, whose conclusion is illustrated in the figure, will have many important consequences for the study of Hilbert spaces.
\begin{thm}[Unique approximation of a closed convex set]
Let $(V,\|\cdot\|)$ be a Banach space with a uniformly convex norm, let $K\subset V$ be a closed convex subset and assume that $v_0\in V$. Then there exists a unique element $w\in K$ that is closest to $v_0$ in the sense that $w$ is the only element of $K$ with 
\[
\|w-v_0\|=d(v_0,K)=\inf_{k\in K}\|k-v_0\|.
\]
\end{thm}
\begin{proof}
By translating both the set $K$ and the point $v_0$ by $-v_0$ we may assume without loss of generality that $v_0=0$. We define 
\[
s=\inf_{k\in K}\|k-v_0\|=\inf_{k\in K}\|k\|.
\]
If $s=0$, then we must have $0\in K$ since $K$ is closed and the only choice is then $w=v_0=0$ (the uniqueness of $w$ is a consequence of the strict positivity of the norm). So assume that $s>0$. By multiplying by the scalar $\frac{1}{s}$ we have found a point $w\in K$ with norm 1, then its uniqueness is an immediate consequence of the uniform convexity: if $w_1,w_2\in K$ have $\|w_1\|=\|w_2\|=1$, then $\frac{w_1+w_2}{2}\in K$ because $K$ is convex. Also, $\left\|\frac{w_1+w_2}{2}\right\|=1$ by the triangle inequality and since $s=1$. By uniform convexity this implies that $w_1=w_2$. Turning to the existence, let us first sketch the argument. Choose a sequence $(k_n)$ in $K$ with $\| k_n\|\to 1$ as $n\to\infty$. Then the mid-points $\frac{k_n+k_m}{2}$ also lie in $K$, since $K$ is convex. However, this shows that the mid-point must have norm greater than or equal to 1, since $s=1$. Therefore $k_n$ and $k_m$ must be close together by uniform convexity. Making this precise, we will see that $(k_n)$ is a Cauchy sequence. Since $V$ is complete and $K$ is closed, this will give a point $w\in K$ with $\|w\|=1=s$ as required. To make this more precise, we apply uniform convexity to the normalized vectors
\[
 x_n=\frac{1}{s_n}k_n,
\]
where $s_n=\|k_n\|$. The mid-point of $x_n$ and $x_m$ can now be expressed as 
\[
\frac{x_m+x_n}{2}=\frac{1}{2s_m}k_m+\frac{1}{2s_n}k_n=\left(\frac{1}{2s_m}+\frac{1}{2s_n}\right (ak_m+bk_n)
\]
with 
\[
a=\frac{\frac{1}{2s_m}}{\frac{1}{2s_m}+\frac{1}{2s_n}}\geq0
\]
\[
b=\frac{\frac{1}{2s_n}}{\frac{1}{2s_m}+\frac{1}{2s_n}}\geq0
\]
and $a+b=1$. Therefore $ak_m+bk_n\in K$ by convexity and so 
\[
\left\|\frac{x_m+x_n}{2}\right\|=\left(\frac{1}{2s_m}+\frac{1}{2s_n}\right)\|ak_m+bk_n\|\geq \frac{1}{2s_m}+\frac{1}{2s_n}.
\]
Let $\psi$ be as in Definition 1.4 and fix $\varepsilon>0$. Choose $N=N(\varepsilon)$ large enough to ensure that $m\geq N$ implies that 
\[
\frac{1}{s_m}\geq 1-\psi(\varepsilon).
\]
Then $m,n\geq N$ implies that 
\[
\frac{1}{2s_m}+\frac{1}{s_n}\geq 1-\psi(\varepsilon),
\]
which together with the definition of uniform convexity gives
\[
1-\psi(\|x_m-x_n\|)\geq \left\|\frac{x_m+x_n}{2}\right\|\geq 1-\psi(\varepsilon).
\]
By monotonicity of the function $\psi$ this implies for all $m,n\geq N$ that $\|x_m-x_n\|\leq \varepsilon$, showing that $(x_n)$ is a Cauchy sequence. As $V$ is assumed to be complete, we deduce that $(x_n)$ converges to some $x\in V$. Since $s_n\to 1$ and $k_n\to s_nx_n$ as $n\to\infty$ it follows that $\lim_{n\to\infty}k_n=x$. As $K$ is closed the limit $x$ belongs to $K$ and by contradiction is an (and hence is the unique) element closest to $v_0=0$.
\begin{figure}
\begin{center}
\includegraphics[height=1.5cm, width=3cm]{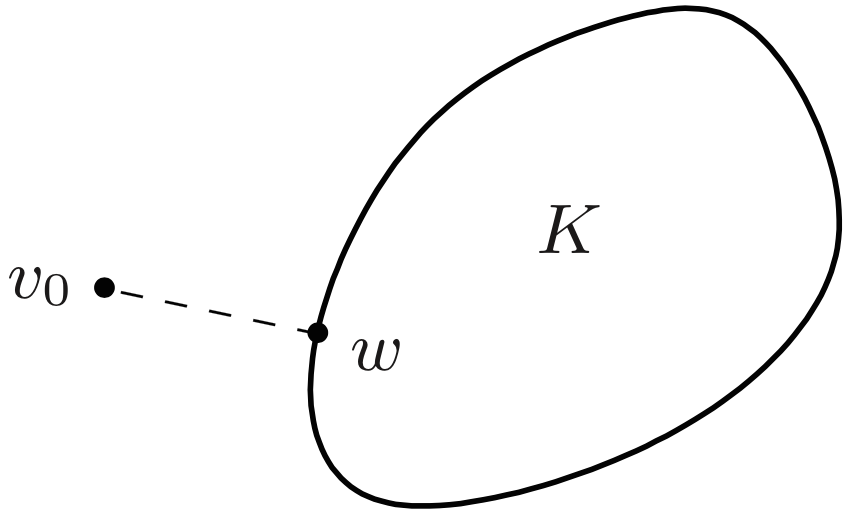}
\end{center}
\caption{The unique closest element of $K$ to $v_0$.}
\end{figure}
\end{proof}
\begin{rem}
This unique approximation is clearly true for Hilbert spaces, since they are uniformly convex spaces. 
\end{rem}
\begin{cor}[Orthogonal decomposition]
Let $\mathcal{H}$ be a Hilbert space and let $\mathcal{L}\subset\mathcal{H}$ be a closed subspace. Then $\mathcal{L}^\perp$ is a closed subspace with 
\[
\mathcal{H}=\mathcal{L}\oplus\mathcal{L}^\perp,
\]
meaning that every element $H\in \mathcal{H}$ can be written in the form 
\[
H=Y+Z
\]
with $Y\in \mathcal{L}$ and $Z\in\mathcal{L}^\perp$ and $Y$ and $Z$ are unique with these properties. Moreover, $Y=(Y^\perp)^\perp$ and 
\[
\|H\|^2=\|Y\|^2+\|Z\|^2
\]
if $H=Y+Z$ with $Y\in\mathcal{L}$ and $Z\in\mathcal{L}^\perp$.
\end{cor}
\begin{proof}
As $H\mapsto \langle H,Y\rangle$ is a (continuous linear) functional for each $Y\in\mathcal{L}$, the set $\mathcal{L}^\perp$ is an intersection of closed subspaces and hence is a closed subspace. Using positivity of the inner product, it is easy to see that $\mathcal{L}\cap\mathcal{L}^\perp=\{0\}$ and from this the uniqueness of the decomposition 
\[
H=Y+Z
\]
with $Y\in\mathcal{L}$ and $Z\in\mathcal{L}^\perp$ follows at once. So it remains to show the existence of this decomposition. Fix $H\in\mathcal{H}$ and apply the theorem of unique approximation with $K=\mathcal{L}$ to find a point $Y\in\mathcal{L}$ that is closest to $h$. Let $Z=H-Y$, so that for any $v\in \mathcal{L}$ and any scalar $t$ we have 
\[
\|Z\|^2\leq  \|H-\underbrace{(tv+Y)}_{\in\mathcal{L}}\|^2=\|Z-tv\|^2=\|Z\|^2-2t\langle v,Z\rangle +\vert t\vert^2\|Z\|^2.
\]
However, this shows that $t\langle v,Z\rangle=0$ for all scalars $t$ and $v\in\mathcal{L}$ and so $\langle v,Z\rangle=0$ for all $v\in\mathcal{L}$. Thus $Z\in\mathcal{L}^\perp$ and hence 
\[
\|H\|^2=\langle H,H\rangle=\langle Y+Z,Y+Z\rangle=\|Y\|^2+\|Z\|^2.
\]
It is clear from the definitions that $\mathcal{L}\subset (\mathcal{L}^\perp)^\perp$. If $v\in(\mathcal{L}^\perp)^\perp$ then 
\[
v=Y+Z
\]
for some $Y\in\mathcal{L}$ and $Z\in\mathcal{L}^\perp$ by the first part of the proof. However,
\[
0=\langle v,Z\rangle=\|Z\|^2
\]
implies that $v=Y$ and so $\mathcal{L}=(\mathcal{L}^\perp)^\perp$.
\end{proof}
\subsection{Orthogonal projection}
Let again $\mathcal{H}$ be a Hilbert space. The projection of an element $X\in\mathcal{H}$ onto a closed subspace $\mathcal{L}\subset\mathcal{H}$ is the unique point $Y\in\mathcal{L}$ such that 
\[
d(X,\mathcal{L})=\|X-Y\|.
\]
We denote this projection by
\[
\Pi:\mathcal{H}\to\mathcal{L},\hspace{0.3cm}X\mapsto \Pi X=Y
\]
\begin{thm}
Let $\mathcal{H}$ be a Hilbert space and let $\mathcal{L}\subset\mathcal{H}$ be a closed subspace. Then the projection operator $\Pi$ of $\mathcal{H}$ onto $\mathcal{L}$ satisfies
\begin{enumerate}[$(i)$]
\item{$\Pi^2=\Pi$}
\item{$\begin{cases}\Pi X=X,& X\in\mathcal{L}\\ \Pi X=0,& X\in\mathcal{L}^\perp\end{cases}$
}
\item{$(X-\Pi X)\perp \mathcal{L}$ for all $X\in\mathcal{H}$}
\end{enumerate}
\end{thm}
\begin{proof}
$(i)$ is clear. The first statement of $(ii)$ is clear from $(i)$. For the second statement of $(ii)$, if $X\in\mathcal{L}^\perp$, then for $Y\in\mathcal{L}$ we get 
\[
\|X-Y\|^2=\|X\|^2+\|Y\|^2.
\]
This is going to be minimized if $Y=0$. Hence then $\Pi X=0$. For $(iii)$, If $Y\in \mathcal{L}$ we get
\[
\|X-\underbrace{\Pi X}_{\in\mathcal{L}}\|^2\leq  \|X-\underbrace{\Pi X-Y}_{\in\mathcal{L}}\|^2=\|Y\|^2+\|X-\Pi X\|^2-2\langle X-\Pi X,Y\rangle.
\]
Therefore 
\begin{equation}
2\langle X-\Pi X,Y\rangle\leq  \|Y\|^2
\end{equation}
for all $Y\in \mathcal{L}$. Now since $\mathcal{L}$ is a linear space we get that for all $\alpha>0$
\[
\alpha Y\in\mathcal{L}.
\]
So in particular (1) is true when $Y$ is replaced by $\alpha Y$. Therefore we get 
\[
2\langle X-\Pi X,Y\rangle\leq  \alpha\|Y\|^2.
\]
Now let $\alpha\to 0$. Hence we obtain 
\[
\langle X-\Pi X,Y\rangle\leq  0
\]
for all $Y\in \mathcal{L}$. Since $-Y\in\mathcal{L}$ we get that 
\[
-\langle X-\Pi X,Y\rangle \leq  0.
\]
But this means $\langle X-\Pi X,Y\rangle=0$ and the claim follows.
\end{proof}
\begin{cor}
Let $\mathcal{H}$ be a Hilbert space and let $\mathcal{L}\subset\mathcal{H}$ be a closed subspace. Moreover let $\Pi$ be the projection operator of $\mathcal{H}$ onto $\mathcal{L}$. Then 
\[
X=(X-\Pi X)+\Pi X
\] 
is the unique representation of $X$ as the sum of an element of $\mathcal{L}$ and an element of $\mathcal{L}^\perp$. 
\end{cor}
\begin{proof}
This is just a consequence of Corollary 1.6.
\end{proof}
\begin{rem}
The uniqueness of the projection operator implies that, for $X_1\in\mathcal{L}$ and $X_2\in\mathcal{L}^\perp$
\[
\Pi(X_1+X_2)=X_1.
\]
\end{rem}
\begin{cor}
Let $\mathcal{H}$ be a Hilbert space and let $\mathcal{L}\subset\mathcal{H}$ be a closed subspace. Moreover let $\Pi$ be the projection operator of $\mathcal{H}$ onto $\mathcal{L}$. Then 
\begin{enumerate}[$(i)$]
\item{$\langle \Pi X,Y\rangle =\langle X,\Pi Y\rangle$ for all $X,Y\in\mathcal{H}$.}
\item{$\Pi(\alpha X+\beta Y)=\alpha\Pi X+\beta \Pi Y$ for all $\alpha,\beta\in\R$ and $X,Y\in\mathcal{H}$.}
\end{enumerate}
\end{cor}
\begin{proof}
For $(i)$, let $X,Y\in\mathcal{H}$, $X=X_1+X_2$ with $X_1\in\mathcal{L}$ and $X_2\in\mathcal{L}^\perp$ and $Y=Y_1+Y_2$ with $Y_1\in\mathcal{L}$ and $Y_2\in\mathcal{L}^\perp$. Then we get 
\[
\langle \Pi X,Y\rangle=\langle \Pi (X_1+X_2),Y_1+Y_2\rangle=\langle X_1,Y_1+Y_2\rangle=\langle X_1,Y_1\rangle
\]
\[
\langle X,\Pi Y\rangle=\langle X_1+X_2,\Pi (Y_1+Y_2)\rangle=\langle X_1+X_2,Y_1\rangle=\langle X_1,Y_1\rangle.
\] 
Therefore they are the same. For $(ii)$, take $\alpha,\beta\in\R$ and look at
\[
\alpha X+\beta Y=\underbrace{(\alpha X_1+\beta Y_1)}_{\in\mathcal{L}}+\underbrace{(\alpha X_2+\beta Y_2)}_{\in\mathcal{L}^\perp}.
\]
Hence we get 
\[
\Pi (\alpha X+\beta Y)=\alpha \Pi(X_1+X_2)+\beta\Pi (Y_1+Y_2)=\alpha \Pi X+\beta \Pi Y.
\]
\end{proof}
\section{The Conditional expectation}
\subsection{Conditional probability}
Let $(\Omega,\F,\p)$ be a probability space and let $A,B\in\F$ such that $\p[B]>0$. Then the conditional probability\footnote{One can look it up for more details in the stochastics I part.} of $A$ given $B$ is defined as 
\[
\p[A\mid B]=\frac{\p[A\cap B]}{\p[B]}.
\]
The important fact here is that the application $\F\to [0,1]$, $A\mapsto \p[A\mid B]$ defines a new probability measure on $\F$ called the conditional probability given $B$. There are several facts, which we need to recall:
\begin{enumerate}[$(i)$]
\item{If $A_1,...,A_n\in\F$ and if $\p\left[\bigcap_{k=1}^nA_k\right]>0$, then 
\[
\p\left[\bigcap_{k=1}^nA_k\right]=\prod_{j=1}^n\p\left[A_j\Big|\bigcap_{k=1}^{j-1}A_k\right].
\]
}
\item{Let $(E_n)_{n\geq 1}$ be a measurable partition of $\Omega$, i.e. for all $n\geq 1$ we have that $E_n\in\F$ and for $n\not=m$ we get $E_n\cap E_m=\varnothing$ and $\bigcup_{n\geq 1}E_n=\Omega$. Now for $A\in \F$ we get 
\[
\p[A]=\sum_{n\geq 1}\p[A\mid E_n]\p[E_n].
\]
}
\item{(Baye's formula)\footnote{Use the previous facts for the proof of Baye's formula. One can also look it up in the stochastics I part.} Let $(E_n)_{n\geq 1}$ be a measurable partition of $\Omega$ and $A\in\F$ with $\p[A]>0$. Then 
\[
\p[E_n\mid A]=\frac{\p[A\mid E_n]\p[E_n]}{\sum_{m\geq 1}\p[A\mid E_m]\p[E_m]}.
\]
}
\end{enumerate}
\begin{rem}
We can reformulate the definition of the conditional probability to obtain 
\begin{align*}
\p[A\mid B]\p[B]&=\p[A\cap B]\\
\p[B\mid A]\p[A]&=\p[A\cap B]
\end{align*}
Therefore one can prove the statements (1) to (3) by using these two equations\footnote{One also has to notice that if $A$ and $B$ are two independent events, then $\p[A\mid B]=\frac{\p[A\cap B]}{\p[B]}=\frac{\p[A]\p[B]}{\p[B]}=\p[A]$}.
\end{rem}
\subsection{Discrete construction of the conditional expectation}
Let $X$ and $Y$ be two r.v.'s on a probability space $(\Omega,\F,\p)$. Let $Y$ take values in $\R$ and $X$ take values in a countable discrete set $\{x_1,x_2,...,x_n,...\}$. The goal is to describe the expectation of the r.v. $Y$ by knowing the observed r.v. $X$. For instance, let $X=x_j\in\{x_1,x_2,...,x_n,...\}$. Therefore we look at a set $\{\omega\in\Omega\mid X(\omega)=x_j\}$ rather than looking at whole $\Omega$. For $\Lambda\in\F$, we thus define
\[
\Q[\Lambda]=\p[\Lambda\mid \{X=x_j\}],
\]
a new probability measure $\Q$, with $\p[X=x_j]>0$. Therefore it makes more sense to compute 
\[
\E_\Q[Y]=\int_\Omega Y(\omega)d\Q(\omega)=\int_{\{\omega\in\Omega\mid X(\omega)=x_j\}}Y(\omega)d\p(\omega)
\]
rather than 
\[
\E_\p[Y]=\int_\Omega Y(\omega)d\p(\omega)=\int_\R yd\p_Y(y).
\]
\begin{defn}[Conditional expectation ($X$ discrete, $Y$ real valued, single value case)]
Let $(\Omega,\F,\p)$ be a probability space. Let $X:\Omega\to\{x_1,x_2,...,x_n,...\}$ be a r.v. taking values in a discrete set and let $Y$ be a real valued r.v. on that space. If $\p[X=x_j]>0$, we can define the conditional expectation of $Y$ given $\{X=x_j\}$ to be 
\[
\E[Y\mid X=x_j]=\E_\Q[Y],
\]
where $\Q$ is the probability measure on $\F$ defined by 
\[
\Q[\Lambda]=\p[\Lambda\mid X=x_j],
\]
for $\Lambda\in\F$, provided that $\E_\Q[\vert Y\vert]<\infty$.
\end{defn}
\begin{thm}[Conditional expectation ($X$ discrete, $Y$ discrete, single value case)]
Let $(\Omega,\F,\p)$ be a probability space. Let $X$ be a r.v. on that space with values in $\{x_1,x_2,...,x_n,...\}$ and let $Y$ also be a r.v. with values in $\{y_1,y_2,...,y_n,...\}$. If $\p[X=x_j]>0$, we can write the conditional expectation of $Y$ given $\{X=x_j\}$ as
\[
\E[Y\mid X=x_j]=\sum_{k=1}^\infty y_k\p[Y=y_k\mid X=x_j].
\]
provided that the series is absolutely convergent.
\end{thm}
\begin{proof}
Apply the definitions above to obtain
\[
\E[Y\mid X=x_j]=\E_\Q[Y]=\sum_{k=1}^\infty y_k\Q[Y=y_k]=\sum_{k=1}^\infty y_k\p[Y=y_k\mid X=x_j]
\]
\end{proof}
Now let again $X$ be a r.v. with values in $\{x_1,x_2,...,x_n,...\}$ and $Y$ a real valued r.v. The next step is to define $\E[Y\mid X]$ as a function $f(X)$. Therefore we introduce the function 
\begin{equation}
f:\{x_1,x_2,...,x_n,...\}\to \R\hspace{1cm}f(x)=\begin{cases}\E[Y\mid X=x],&\p[X=x]>0\\ \text{any value in $\R$},&\p[X=x]=0\end{cases}
\end{equation}
\begin{rem}
It doesn't matter which value we assign to $f$ for $\p[X=x]=0$, since it doesn't affect the expectation because it's defined on a null set. For convention we want to assign to it the value 0.
\end{rem}
\begin{defn}[Conditional expectation ($X$ discrete, $Y$ real valued, complete case)]
Let $(\Omega,\F,\p)$ be a probability space. Let $X$ be a countably valued r.v. and let $Y$ be a real valued r.v. The conditional expectation of $Y$ given $X$ is defined by 
\[
\E[Y\mid X]=f(X), 
\]
with $f$ as in (2), provided that for all $j$: if $\Q_j[\Lambda]=\p[\Lambda\mid X=x_j]$, with $\p[X=x_j]>0$, we get $\E_{\Q_j}[\vert Y\vert]<\infty$.
\end{defn}
\begin{rem}
The above definition does not define $\E[Y\mid X]$ everywhere but rather almost everywhere, since on each set $\{X=x\}$, where $\p[X=x]=0$, its value is arbitrary.
\end{rem}
\begin{ex}
Let\footnote{Recall that this means that $X$ is Poisson distributed: $\p[X=k]=e^{-\lambda}\frac{\lambda^k}{k!}$ for $k\in\N$} $X\sim\Pi(\lambda)$. Let us consider a tossing game, where we say that when $X=n$, we do $n$ independent tossing of a coin where each time one obtains 1 with probability $p\in[0,1]$ and 0 with probability $1-p$. Define also $S$ to be the r.v. giving the total number of 1 obtained in the game. Therefore, if $X=n$ is given, we get that $S$ is binomial distributed with parameters $(p,n)$. We want to compute 
\begin{enumerate}[$(i)$] 
\item{$\E[S\mid X]$}
\item{$\E[X\mid S]$}
\end{enumerate}
\begin{rem}
It is more natural to ask for the expectation of the amount of 1 obtained for the whole game by knowing how many games were played. The reverse is a bit more difficult. Logically, we may also notice that it definitely doesn't make sense to say $S\geq X$, because we can not obtain more wins in a game than the amount of games that were played.
\end{rem}
\begin{enumerate}[$(i)$]
\item{First we compute $\E[S\mid X=n]$: If $X=n$, we know that $S$ is binomial distributed with parameters $(p,n)$ ($S\sim \B(p,n)$) and therefore we already know\footnote{If $X\sim \B(p,n)$ then $\E[X]=pn$. For further calculation, one can look it up in the stochastics I notes}
\[
\E[S\mid X=n]=pn.
\]
Now we need to identify the function $f$ defined as in (2) by 
\begin{align*}
f:\N&\longrightarrow\R\\
n&\longmapsto pn.
\end{align*}
Therefore we get by definition 
\[
\E[S\mid X]=pX.
\]
}
\item{
Next we want to compute $\E[X\mid S=k]$: For $n\geq k$ we have 
\[
\p[X=n\mid S=k]=\frac{\p[S=k\mid X=n]\p[X=n]}{\p[S=k]}=\frac{\begin{binom}nk\end{binom}p^k(1-p)^{n-k}e^{-\lambda}\frac{\lambda^n}{n!}}{\sum_{m=k}^\infty\begin{binom}mk\end{binom}p^k(1-p)^{m-k}e^{-\lambda}\frac{\lambda^m}{m!}},
\]
since $\{S=k\}=\bigsqcup_{m\geq k}\{S=k,X=m\}$. By some algebra we obtain that 
\[
\frac{\begin{binom}nk\end{binom}p^k(1-p)^{n-k}e^{-\lambda}\frac{\lambda^n}{n!}}{\sum_{m=k}^\infty\begin{binom}mk\end{binom}p^k(1-p)^{m-k}e^{-\lambda}\frac{\lambda^m}{m!}}=\frac{(\lambda(1-p))^{n-k}e^{-\lambda(1-p)}}{(n-k)!}
\]
Hence we get that 
\[
\E[X\mid S=k]=\sum_{n\geq k}n\p[X=n\mid S=k]=k+\lambda(1-p).
\]
Therefore $\E[X\mid S]=S+\lambda(1-p)$.
}
\end{enumerate}
\end{ex}
\subsection{Continuous construction of the conditional expectation}
Now we want to define $\E[Y\mid X]$, where $X$ is no longer assumed to be countably valued. Therefore we want to recall the following two facts:
\begin{defn}[$\sigma$-Algebra generated by a random variable]
Let $(\Omega,\F,\p)$ be a probability space. Let $X:(\Omega,\F,\p)\to (\R^n,\B(\R^n),\lambda)$ be a r.v. on that space. The $\sigma$-Algebra generated by $X$ is given by 
\[
\sigma(X)=X^{-1}(\B(\R^n))=\{A\in\Omega\mid A=X^{-1}(B),B\in\B(\R^n)\}.
\]
\end{defn}
\begin{thm}
Let $(\Omega,\F,\p)$ be a probability space. Let $X:(\Omega,\F,\p)\to(\R^n,\B(\R^n),\lambda)$ be a r.v. on that space and let $Y$ be a real valued r.v. on that space. $Y$ is measurable with respect to $\sigma(X)$ if and only if there exists a Borel measurable function $f:\R^n\to\R$ such that 
\[
Y=f(X). 
\]
\end{thm}
\begin{rem}
We want to make use of the fact that for the Hilbert space $L^2(\Omega,\F,\p)$ we get that $L^2(\Omega,\sigma(X),\p)\subset L^2(\Omega,\F,\p)$ is a complete subspace, since $\sigma(X)\subset\F$. This allows us to use the orthogonal projections and to interpret the conditional expectation as such a projection.
\end{rem}
\begin{defn}[Conditional expectation (as a projection onto a closed subspace)]
Let $(\Omega,\F,\p)$ be a probability space. Let $Y\in L^2(\Omega,\F,\p)$. Then the conditional expectation of $Y$ given $X$ is the unique element $\hat Y\in L^2(\Omega,\sigma(X),\p)$ such that for all $Z\in L^2(\Omega,\sigma(X),\p)$
\begin{equation}
\E[YZ]=\E[\hat Y Z].
\end{equation}
This result is due to the fact that if $Y-\hat Y\in L^2(\Omega,\F,\p)$ then for all $Z\in L^2(\Omega,\sigma(X),\p)$ we get $\langle Y-\hat Y,Z\rangle=0$. We write $\E[Y\mid X]$ for $\hat Y$.
\end{defn}
\begin{rem}
$\hat Y$ is the orthogonal projection of $Y$ onto $L^2(\Omega,\sigma(X),\p)$.
\end{rem}
\begin{rem}
Since $X$ takes values in $\R^n$, there exists a Borel measurable function $f:\R^n\to\R$ such that 
\[
\E[Y\mid X]=f(X)
\]
with $\E[f^2(X)]<\infty$. We can also rewrite (3) as: for all Borel measurable $g:\R^n\to\R$, such that $\E[g^2(X)]<\infty$, we get
\[
\E[Yg(X)]=\E[f(X)g(X)].
\]
\end{rem}
Now let $\mathcal{G}\subset\F$ be a sub $\sigma$-Algebra of $\F$ and consider the space $L^2(\Omega,\mathcal{G},\p)\subset L^2(\Omega,\F,\p)$. It is clear that $L^2(\Omega,\mathcal{G},\p)$ is a Hilbert space and thus we can project to it.
\begin{defn}[Conditional expectation (projection case)]
Let $(\Omega,\F,\p)$ be a probability space. Let $Y\in L^2(\Omega,\F,\p)$ and let $\mathcal{G}\subset\F$ be a sub $\sigma$-Algebra of $\F$. Then the conditional expectation of $Y$ given $\mathcal{G}$ is defined as the unique element $\E[Y\mid \mathcal{G}]\in L^2(\Omega,\mathcal{G},\p)$ such that for all $Z\in L^2(\Omega,\mathcal{G},\p)$
\begin{equation}
\label{4}
\E[YZ]=\E[\E[Y\mid \mathcal{G}]Z].
\end{equation}
\end{defn}
\begin{rem}
In (3) or (1), it is enough\footnote{Since we can always consider linear combinations of $\one_A$ and then apply density theorems to it} to restrict the test r.v. $Z$ to the class of r.v.'s of the form 
\[
Z=\one_A,\hspace{1cm}A\in\mathcal{G}. 
\]
\end{rem}
\begin{rem}
The conditional expectation is in $L^2$, so it's only defined a.s. and not everywhere in a unique way. So in particular, any statement like $\E[Y\mid\mathcal{G}]\geq0$ or $\E[Y\mid \mathcal{G}]=Z$ has to be understood with an implicit a.s.
\end{rem}
\begin{thm}
Let $(\Omega,\F,\p)$ be a probability space. Let $Y\in L^2(\Omega,\F,\p)$ and let $\mathcal{G}\subset \F$ be a sub $\sigma$-Algebra of $\F$.
\begin{enumerate}[$(i)$]
\item{If $Y\geq 0$, then $\E[Y\mid \mathcal{G}]\geq 0$}
\item{$\E[\E[Y\mid\mathcal{G}]]=\E[Y]$}
\item{The map $Y\mapsto\E[Y\mid\mathcal{G}]$ is linear.}
\end{enumerate}
\end{thm}
\begin{proof}
For $(i)$ take $Z=\one_{\{\E[Y\mid\mathcal{G}]<0\}}$ to obtain 
\[
\underbrace{\E[YZ]}_{\geq 0}=\underbrace{\E[\E[Y\mid \mathcal{G}]Z]}_{\leq  0}.
\]
This implies that $\p[\E[Y\mid \mathcal{G}]<0]=0$. For $(ii)$ take $Z=\one_{\Omega}$ and plug into (4). For $(iii)$ notice that linearity comes from the orthogonal projection operator. But we can also do it directly by taking $Y,Y'\in L^2(\Omega,\F,\p)$, $\alpha,\beta\in \R$ and $Z\in L^2(\Omega,\mathcal{G},\p)$ to obtain 
\[
\E[(\alpha Y+\beta Y')Z]=\E[YZ]+\beta\E[Y'Z]=\alpha\E[\E[Y\mid\mathcal{G}]Z]+\beta\E[\E[Y'\mid\mathcal{G}]Z]=\E[(\alpha\E[Y\mid\mathcal{G}]+\beta\E[Y'\mid\mathcal{G}])Z].
\]
Now we can conclude by using the uniqueness property that 
\[
\E[\alpha Y+\beta Y'\mid \mathcal{G}]=\alpha\E[Y\mid \mathcal{G}]+\beta\E[Y'\mid \mathcal{G}].
\]
\end{proof}
Now we want to extend the definition of the conditional expectation to r.v.'s in $L^1(\Omega,\F,\p)$ or to $L^+(\Omega,\F,\p)$, which is the space of non negative r.v.'s allowing the value $\infty$.
\begin{lem}
Let $(\Omega,\F,\p)$ be a probability space. Let $Y\in L^+(\Omega,\F,\p)$ and let $\mathcal{G}\subset \F$ be a sub $\sigma$-Algebra of $\F$. Then there exists a unique element $\E[Y\mid \mathcal{G}]\in L^+(\Omega,\mathcal{G},\p)$ such that for all $X\in L^+(\Omega,\mathcal{G},\p)$
\begin{equation}
\E[YX]=\E[\E[Y\mid \mathcal{G}]X]
\end{equation}
and this conditional expectation agrees with the previous definition when $Y\in L^2(\Omega,\F,\p)$. Moreover, if $0\leq  Y\leq  Y'$, then 
\[
\E[Y\mid \mathcal{G}]\leq \E[Y'\mid \mathcal{G}].
\]
\end{lem}
\begin{proof}
If $Y\leq  0$ and $Y\in L^2(\Omega,\F,\p)$, then we define $\E[Y\mid\mathcal{G}]$ as before. If $X\in L^+(\Omega,\mathcal{G},\p)$, we get that $X_n=X\land n$, is in $L^2(\Omega,\mathcal{G},\p)$ and is positive with $X_n\uparrow X$ for $n\to\infty$. Using the monotone convergence theorem we get 
\[
\E[YX]=\E[Y\lim_{n\to\infty}X_n]=\lim_{n\to\infty}\E[YX_n]=\lim_{n\to\infty}\E[\E[Y\mid\mathcal{G}]X_n]=\E[\E[Y\mid\mathcal{G}]\lim_{n\to\infty}X]=\E[\E[Y\mid\mathcal{G}]X].
\]
This shows that (5) is true whenever $Y\in L^2(\Omega,\F,\p)$ with $Y\geq 0$ and $X\in L^+(\Omega,\mathcal{G},\p)$. Now let $Y\in L^1(\Omega,\F,\p)$. Define $Y_m=Y\land m$. Hence we get $Y_m\in L^2(\Omega,\F,\p)$ and $Y_m\uparrow Y$ as $n\to\infty$. Each $\E[Y_m\mid\mathcal{G}]$ is well defined\footnote{because for $Y\in L^2$ and $U\in L^2$ we get $Y\geq U\Longrightarrow Y-U\geq 0\Longrightarrow \E[Y\mid\mathcal{G}]\geq \E[U\mid\mathcal{G}]$} and positive and increasing. We define 
\[
\E[Y\mid\mathcal{G}]=\lim_{n\to\infty}\E[Y_m\mid \mathcal{G}].
\]
Several applications of the monotone convergence theorem will give us for $X\in L^+(\Omega,\mathcal{G},\p)$
\[
\E[YX]=\lim_{m\to\infty}\E[Y_mX]=\lim_{m\to\infty}\E[\E[Y_m\mid\mathcal{G}]X]=\E[\E[Y\mid \mathcal{G}]X].
\]
Furthermore if $0\leq  Y\leq  Y'$, then $Y\land m\leq  Y'\land m$ and therefore 
\[
\E[Y\mid\mathcal{G}]\leq \E[Y'\mid\mathcal{G}].
\]
Now we need to show uniqueness\footnote{Note that for any $W\in L^+$, the set $E$ on which $W=\infty$ is a null set. For suppose not, then $\E[W]\geq \E[\infty \one_E]=\infty\p[E]$. But since $\p[E]>0$ this cannot happen}. Let $U$ and $V$ be two versions of $\E[Y\mid \mathcal{G}]$. Let 
\[
\Lambda_n=\{U<V\leq  n\}\in\mathcal{G}
\] 
and assume $\p[\Lambda_n]>0$. We then have
\[
\E[Y\one_{\Lambda_n}]=\underbrace{\E[U\one_{\Lambda_n}]=\E[V\one_{\Lambda_n}]}_{\E[(U-V)\one_{\Lambda_n}]=0}.
\]
This contradicts the fact that $\p[\Lambda_n]>0$. Moreover, $\{U<V\}=\bigcup_{n\geq 1}\Lambda_n$ and therefore 
\[
\p[U<V]=0
\]
and similarly $\p[V<U]=0$. This implies 
\[
\p[U=V]=1.
\]
\end{proof}
\begin{thm}
Let $(\Omega,\F,\p)$ be a probability space. Let $Y\in L^1(\Omega,\F,\p)$ and let $\mathcal{G}\subset\F$ be a sub $\sigma$-Algebra of $\F$. Then there exists a unique element $\E[Y\mid \mathcal{G}]\in L^1(\Omega,\mathcal{G},\p)$ such that for every $X$ bounded and $\mathcal{G}$-measurable
\begin{equation}
\E[YX]=\E[\E[Y\mid \mathcal{G}]X].
\end{equation}
This conditional expectation agrees with the definition for the $L^2$. Moreover it satisfies:
\begin{enumerate}[$(i)$]
\item{If $Y\geq 0$, then $\E[Y\mid\mathcal{G}]\geq 0$}
\item{The map $Y\mapsto \E[Y\mid\mathcal{G}]$ is linear.}
\end{enumerate}
\end{thm}
\begin{proof}
We will only prove the existence, since the rest is exactly the same as before. Write $Y=Y^+-Y^-$ with $Y^+,Y^-\in L^1(\Omega,\F,\p)$ and $Y^+,Y^-\geq 0$. So $\E[Y^+\mid\mathcal{G}]$ and $\E[Y^-\mid\mathcal{G}]$ are well defined. Now we set 
\[
\E[Y\mid\mathcal{G}]=\E[Y^+\mid \mathcal{G}]-\E[Y^-\mid\mathcal{G}].
\]
This is well defined because 
\[
\E[\E[Y^\pm\mid\mathcal{G}]]=\E[Y^\pm]<\infty
\]
if we let $X=\one_\Omega$ in the previous lemma and therefore $\E[Y^+\mid\mathcal{G}]$ and $\E[Y^-\mid \mathcal{G}]\in L^1(\Omega,\mathcal{G},\p)$. For all $X$ bounded and $\mathcal{G}$-measurable we can also write $X=X^+-X^-$ and it follows from the previous lemma that 
\[
\E[\E[Y^\pm\mid\mathcal{G}]X]=\E[Y^\pm X].
\]
This implies that $\E[Y\mid\mathcal{G}]$ satisfies (6).
\end{proof}
\begin{cor}
Let $(\Omega,\F,\p)$ be a probability space. Let $X\in L^1(\Omega,\F,\p)$ be a r.v. on that space. Then 
\[
\E[\E[X\mid\mathcal{G}]]=\E[X].
\]
\end{cor}
\begin{proof}
Take equation (4) and set $Z=\one_\Omega$.
\end{proof}
\begin{cor}
Let $(\Omega,\F,\p)$ be a probability space. Let $X\in L^1(\Omega,\F,\p)$ be a r.v. on that space. Then 
\[
\vert\E[X\mid\mathcal{G}]\vert\leq \E[\vert X\vert\mid\mathcal{G}].
\]
In particular 
\[
\E[\vert\E[X\mid\mathcal{G}]\vert]\leq \E[\vert X\vert].
\]
\end{cor}
\begin{proof}
We can always write $X=X^+-X^-$ and also $\vert X\vert=X^++X^-$. Therefore we get 
\[
\vert\E[X\mid\mathcal{G}]\vert=\vert\E[X^+\mid\mathcal{G}]-\E[X^-\mid\mathcal{G}]\vert\leq  \E[X^+\mid\mathcal{G}]+\E[X^-\mid\mathcal{G}]=\E[X^++X^-\mid\mathcal{G}]=\E[\vert X\vert\mid\mathcal{G}].
\]
\end{proof}
\begin{prop}
Let $(\Omega,\F,\p)$ be a probability space. Let $Y\in L^1(\Omega,\F,\p)$ be a r.v. on that space and assume that $Y$ is independent of the sub $\sigma$-Algebra $\mathcal{G}\subset\F$, i.e. $\sigma(Y)$ is independent of $\mathcal{G}$. Then 
\[
\E[Y\mid\mathcal{G}]=\E[Y].
\]
\end{prop}
\begin{proof}
Let $Z$ be a bounded and $\mathcal{G}$-measurable r.v. and therefore $Y$ and $Z$ are independent. Hence we get 
\[
\E[YZ]=\E[Y]\E[Z]=\E[\E[Y]Z].
\]
This implies that, since $\E[Y]$ is constant, that $\E[Y]\in L^1(\Omega,\mathcal{G},\p)$ and satisfies (4). Therefore by uniqueness we get that $\E[Y\mid\mathcal{G}]=\E[Y]$.

\end{proof}

\begin{thm}

Let $(\Omega,\F,\p)$ be a probability space. Let $X$ and $Y$ be two r.v.'s on that space and let $\mathcal{G}\subset\F$ be a sub $\sigma$-Algebra of $\F$. Assume further that at least one of these two holds:

\begin{enumerate}[$(i)$]
\item{$X,Y$ and $XY$ are in $L^1(\Omega,\F,\p)$ with $X$ being $\mathcal{G}$-measurable.}
\item{$X\geq 0$, $Y\geq 0$ with $X$ being $\mathcal{G}$-mearuable.}
\end{enumerate}

Then 

\[
\E[XY\mid\mathcal{G}]=\E[Y\mid\mathcal{G}]X.
\]

In particular, if $X$ is a positive r.v. or in $L^1(\Omega,\mathcal{G},\p)$ and $\mathcal{G}$-measurable, then 

\[
\E[X\mid\mathcal{G}]=X.
\]

\end{thm}

\begin{proof}

For $(ii)$ assume first that $X,Y\leq  0$. Let $Z$ be a positive and $\mathcal{G}$-measurable r.v. Then we can obtain 

\[
\E[(XY)Z]=\E[Y(XZ)]=\E[\E[Y\mid\mathcal{G}]XZ]=\E[(\E[Y\mid\mathcal{G}]X)Z].
\]

Note that $\E[(\E[Y\mid\mathcal{G}]X)Z]$ is a positive r.v. and $\mathcal{G}$-measurable. Hence $\E[XY\mid\mathcal{G}]=X\E[Y\mid\mathcal{G}]$. For $(i)$ we can write $X=X^++X^-$ and use $(ii)$. This is an easy exercise.

\end{proof}

\begin{rem}

Next we want to show that the classical limit theorems from measure theory also make sense in terms of the conditional expectation\footnote{Recall the classical limit theorems for integrals: $Monotone$ $convergence:$ Let $(f_n)_{n\geq 1}$ be an  increasing sequence of positive and measurable functions and let $f=\lim_{n\to\infty}\uparrow f_n$. Then $\int fd\mu=\lim_{n\to\infty}f_nd\mu$. $Fatou:$ Let $(f_n)_{n\geq 1}$ be a sequence of measurable and positive functions. Then $\int\liminf_n f_n d\mu\leq  \liminf_n \int f_nd\mu$. $Dominated$ $convergence:$ Let $(f_n)_{n\geq 1}$ be a sequence of integrable functions with $\vert f_n\vert\leq  g$ for all $n$ with $g$ integrable. Denote $f=\lim_{n\to\infty}f_n$. Then $\lim_{n\to\infty}\int f_nd\mu=\int fd\mu$}. 

\end{rem}

\begin{thm}[Limit theorems for the conditional expectation]

Let $(\Omega,\F,\p)$ be a probability space. Let $(Y_n)_{n\geq 1}$ be a sequence of r.v.'s on that space and let $\mathcal{G}\subset\F$ be a sub $\sigma$-Algebra of $\F$. Then we have:

\begin{enumerate}[$(i)$]
\item{(\emph{Monotone convergence}) Assume that $(Y_n)_{n\geq 1}$ is a sequence of positive r.v.'s for all $n$ such that $\lim_{n\to\infty}\uparrow Y_n=Y$ a.s. Then 

\[
\lim_{n\to\infty}\E[Y_n\mid\mathcal{G}]=\E[Y\mid \mathcal{G}].
\]

}
\item{(\emph{Fatou}) Assume that $(Y_n)_{n\geq 1}$ is a sequence of positive r.v.'s for all $n$. Then 

\[
\E[\liminf_n Y_n\mid\mathcal{G}]=\liminf_n\E[Y_n\mid\mathcal{G}].
\]

}

\item{(\emph{Dominated convergence}) Assume that $Y_n\xrightarrow{n\to\infty}Y$ a.s. and that there exists $Z\in L^1(\Omega,\F,\p)$ such that $\vert Y_n\vert\leq  Z$ for all $n$. Then 

\[
\lim_{n\to\infty}\E[Y_n\mid \mathcal{G}]=\E[Y\mid\mathcal{G}].
\]

}

\end{enumerate}
\end{thm}

\begin{proof}

We will only prove $(i)$, since $(ii)$ and $(iii)$ are proved in a similar way (it's a good exercise to do the proof). Since $(Y_n)_{n\geq 1}$ is an increasing sequence, it follows that 

\[
\E[Y_{n+1}\mid\mathcal{G}]\geq \E[Y_n\mid\mathcal{G}].
\]

Hence we can deduce that $\lim_{n\to\infty}\uparrow \E[Y_n\mid\mathcal{G}]$ exists and we denote it by $Y'$. Moreover, note that $Y'$ is $\mathcal{G}$-measurable, since it is a limit of $\mathcal{G}$-measurable r.v.'s. Let $X$ be a positive and $\mathcal{G}$-measurable r.v. and obtain then 

\[
\E[Y'X]=\E[\lim_{n\to\infty}\E[Y_n\mid\mathcal{G}]X]=\lim_{n\to\infty}\uparrow\E[\E[Y_n\mid\mathcal{G}]X]=\lim_{n\to\infty}\E[Y_n X]=\E[YX],
\]

where we have used monotone convergence twice and equation (4). Therefore we get 
\[
\lim_{n\to\infty}\E[Y_n\mid\mathcal{G}]=\E[Y\mid\mathcal{G}].
\]
\end{proof}

\begin{thm}[Jensen's inequality]
Let $(\Omega,\F,\p)$ be a probability space. Let $\varphi:\R\to\R$ be a real, convex function. Let $X\in\L^1(\Omega,\F,\p)$ such that $\varphi(X)\in L^1(\Omega,\F,\p)$. Then 
\[
\varphi(\E[X\mid\mathcal{G}])\leq  \E[\varphi(X)\mid\mathcal{G}]
\]
for all sub $\sigma$-Algebras $\mathcal{G}\subset\F$.
\end{thm}
\begin{proof}
Exercise.
\end{proof}

\begin{ex}

Let $(\Omega,\F,\p)$ be a probability space. Let $\varphi(X)=X^2$ and let $X\in L^2(\Omega,\F,\p)$. Then 

\[
(\E[X\mid \mathcal{G}])^2\leq  \E[X^2\mid\mathcal{G}]
\]

for all sub $\sigma$-Algebras $\mathcal{G}\subset \F$.

\end{ex}

\begin{thm}[Tower property]

Let $(\Omega,\F,\p)$ be a probability space. Let $X\in L^1(\Omega,\F,\p)$ be a positive r.v. on that space. Let $\mathcal{C}\subset\mathcal{G}\subset \F$ be a tower of sub $\sigma$-Algebras of $\F$. Then 

\[
\E[\E[X\mid\mathcal{G}]\mathcal{C}]=\E[X\mid\mathcal{C}]. 
\]

\end{thm}

\begin{proof}

Let $Z$ be a bounded and $\mathcal{C}$-measurable r.v. Then we obtain

\[
\E[XZ]=\E[\E[X\mid\mathcal{C}]Z].
\]

But $Z$ is also $\mathcal{G}$-measurable and hence we get

\[
\E[XZ]=\E[\E[X\mid\mathcal{G}]Z].
\]

Therefore, for all $Z$ bounded and $\mathcal{C}$-measurable r.v.'s, we get

\[
\E[\E[X\mid\mathcal{G}]Z]=\E[\E[X\mid\mathcal{C}]Z]
\]

and thus 

\[
\E[\E[X\mid\mathcal{G}]\mathcal{C}]=\E[X\mid\mathcal{C}].
\]

\end{proof}

\section{The Radon-Nikodym approach for the conditional expectation}

\begin{rem}

Before stating the Radon-Nikodym theorem, we recall some definitions from measure theory. Let $(\Omega,\B)$ be a measurable space. A measure $\nu$ is $absolutely$ $continuous$ with respect to another measure $\mu$, written $\nu\ll\mu$ if there exists some measurable $f\geq 0$ with $d\nu=fd\mu$, that is if there is a finite measurable $f\geq 0$ with 

\[
\nu(B)=\int_B fd\mu
\]

for all $B\in\B$. Two measures $\mu$ and $\nu$ are $singular$ with respect to each other if there exists disjoint measurable sets $A_1,A_2\subset \Omega$ with $\Omega=A_1\sqcup A_2$ and with $\nu(A_1)=0=\mu(A_2)$. Finally, recall that a measure $\mu$ is $\sigma$-finite if there is a decomposition of $\Omega$ into measurable sets,

\[
\Omega=\bigsqcup_{i=1}^\infty A_i
\]

with $\mu(A_i)<\infty$.

\end{rem}

\begin{thm}[Radon-Nikodym]
\label{RN}
Let $\mu$ and $\nu$ be two $\sigma$-finite measures on a measurable space $(\Omega,\B)$. Then $\nu$ can be decomposed as 

\[
\nu=\nu_{abs}+\nu_{sing}
\]

into the sum of two $\sigma$-finite measure with $\nu_{abs}\ll\mu$ being absolutely continuous with respect to $\mu$, and with $\nu_{sing}$ and $\mu$ being singular to each other (which will be written $\nu_{sing}\perp\mu$).

\end{thm}

\begin{rem}

The theorem implies that there exists another, more practical way of checking whether a given $\sigma$-finite measure $\nu$ is absolutely continuous with respect to another $\sigma$-finite measure $\mu$. If $\mu(N)=0$ implies that $\nu(N)=0$ for every measurable $N\subset \Omega$, then $\nu=\nu_{abs}$ is absolutely continuous. We also note that the density function $f$ with $fd\mu=d\nu$ is called the $Radon$-$Nikodym$ $derivative$ and is often written $f=\frac{d\nu}{d\mu}$.

\end{rem}

\begin{rem}

To prove this theorem, we need a theorem which gives us a nice relationship between a Hilbert space and its dual space. Actually we can identify a Hilbert space $\mathcal{H}$ with its dual space $\mathcal{H}^*$.

\end{rem}

\begin{lem}[Riesz-representation for Hilbert spaces]

For a Hilbert space $\mathcal{H}$, the map sending $h\in \mathcal{H}$ to $\phi(h)\in\mathcal{H}^*$ defined by 

\[
\phi(h)(x)=\langle x,h\rangle
\]

is a linear (resp. sesqui-linear in the complex case) isometric isomorphism between $\mathcal{H}$ and its dual space $\mathcal{H}^*$.

\end{lem}

\begin{proof}[Proof of Theorem \ref{RN}]

Suppose that $\mu$ and $\nu$ are both finite measures (the general case can be reduced to this case by using the assumption that $\mu$ and $\nu$ are both $\sigma$-finite). We define a new measure $m=\mu+\nu$ and will work with the real Hilbert space $\mathcal{H}=L^2(\Omega,m)$. On this Hilbert space we define a linear functional $\phi$ by 

\[
\phi(g)=\int gd\nu
\]

for $g\in\mathcal{H}$. For $g$ a simple function on $\Omega$, this is clearly well-defined and satisfies 

\[
\vert\phi(g)\vert=\left\vert\int gd\nu\right\vert\leq \int \vert g\vert d\nu\leq \int \vert g\vert dm\leq \| g\|_{\mathcal{H}}\|\one\|_{\mathcal{H}}
\]

where we have used the fact that $m=\mu+\nu$, that $\mu$ is a positive measure and the Cauchy-Schwartz inequality on $\mathcal{H}$. Since the simple functions are dense in $\mathcal{H}$, the functional extends to a functional on all of $\mathcal{H}$. By the Riesz-representation for Hilbert spaces there is some $k\in\mathcal{H}$ such that 

\begin{equation}
\int gd\nu=\phi(g)=\int gkdm.
\end{equation}

We claim that $k$ takes values in $[0,1]$ almost surely with respect to $m$. Indeed, for any $B\in\B$ we have 

\[
0\leq  \nu(B)\leq  m(B),
\]

so (using $g=\one_B$),

\[
0\leq  \int_B kdm\leq  m(B).
\]

Using the choices 

\[
B=\{\omega\in\Omega\mid k(\omega)<0\}
\]

and 

\[
B=\{\omega\in\Omega\mid k(\omega)>1\}
\]

implies the claim that $k$ takes $m$-almost surely values in $[0,1]$. Since $m=\mu+\nu$, we can reformulate (7) as 

\begin{equation}
\int g(1-k)d\nu=\int gkd\mu.
\end{equation}

This holds by construction for all simple functions $g$, and hence for all nonnegative measurable functions by monotone convergence. Now define $\nu_{sing}$ to be $\nu\mid_{A}$, where 

\[
A=\{\omega\in\Omega\mid k(\omega)=1\}.
\]

By definition, $\nu_{sing}(\Omega\setminus A)=0$ and by (8) applied with $g=\one_{A}$ we also have $\mu(A)=0$. Therefore 

\[
\nu_{sing}\perp\mu.
\]

We also define 

\[
\nu_{abs}=\nu\mid_{\Omega\setminus A}=\nu_{\{\omega\in\Omega\mid k(\omega)<1\}}
\]

so that $\nu=\nu_{sing}+\nu_{abs}$. Define the function $f=\frac{k}{1-k}\geq 0$ on $\Omega\setminus A$ and let $g\geq 0$ be measurable. Then by (8) we have 

\[
\int_{\Omega\setminus A}gf d\mu=\int_{\Omega\setminus A}\frac{g}{1-k}kd\mu=\int_{\Omega\setminus A}\frac{g}{1-k}(1-k)d\nu=\int_{\Omega\setminus A}gd\nu_{abs},
\]

which shows that $d\nu_{abs}=fd\mu$ and so $\nu_{abs}\ll \mu$.

\end{proof}

\begin{thm}

Let $(\Omega,\F,\p)$ be a probability space. Let $\mathcal{G}\subset \F$ be a sub $\sigma$-Algebra of $\F$ and let $X\in L^1(\Omega,\F,\p)$ be a r.v. Then there exists a unique r.v. in $L^1(\Omega,\mathcal{G},\p)$, denoted by $\E[X\mid\mathcal{G}]$, such that for all $B\in\mathcal{G}$

\[
\E[X\one_B]=\E[\E[X\mid\mathcal{G}]\one_B].
\]

More generally, for every bounded and $\mathcal{G}$-measurable r.v. $Z$ we get

\[
\E[XZ]=\E[\E[X\mid\mathcal{G}]Z]
\]

and if $X\geq 0$, then $\E[X\mid\mathcal{G}]\geq 0$.

\end{thm}

\begin{proof}

The uniqueness part was already done. To show existence, assume first that $X$ is positive. Define a new measure $\Q$ on $(\Omega,\mathcal{G})$ by 

\[
\Q[A]=\E[X\one_A]=\int_A X(\omega)d\p(\omega)
\]

for all $A\in\mathcal{G}$. Now consider the measure $\p$ restricted to $\mathcal{G}$. Then we get that 

\[
\Q\ll\p
\]

on $\mathcal{G}$. The Radon-Nikodym theorem implies that there exists a positive and $\mathcal{G}$-measurable r.v. $\tilde X$ such that 

\[
\Q[A]=\E[\tilde X\one_A]
\]
for all $A\in\mathcal{G}$. For $A\in\mathcal{G}$ we get that 

\[
\E[X\one_A]=\E[\tilde X\one_A].
\]

Now taking $A=\Omega$, we get that $\E[X]=\E[\tilde X]$. Therefore we have that $\tilde X\in L^1(\Omega,\mathcal{G},\p)$ and hence we see that $\tilde X=\E[X\mid\mathcal{G}]$ For the general case, we can just write $X=X^++X^-$ and do the same as before.

\end{proof}

\section{More properties of the conditional expectation}

\begin{thm}

Let $(\Omega,\F,\p)$ be a probability space. Let $\mathcal{G}_1\subset\F$ and $\mathcal{G}_2\subset\F$ be two sub $\sigma$-Algebras of $\F$. Then $\mathcal{G}_1$ and $\mathcal{G}_2$ are independent if and only if for every positive and $\mathcal{G}_2$-measurable r.v. $X$ (or for $X\in L^1(\Omega,\mathcal{G}_2,\p)$ or $X=\one_A$ for $A\in \mathcal{G}_2$) we have 

\[
\E[X\mid\mathcal{G}_2]=\E[X].
\]

\end{thm}

\begin{proof}

We only need to prove that the statement in the bracket implies that $\mathcal{G}_1$ and $\mathcal{G}_2$ are independent. Assume that for all $A\in \mathcal{G}_2$ we have that

\[
\E[\one_A\mid\mathcal{G}_1]=\p[A]
\]

and moreover for all $B\in \mathcal{G}_1$ we have that 

\[
\E[\one_B\one_A]=\E[\E[\one_A\mid\mathcal{G}_1]\one_B].
\]

Note that $\E[\one_A\mid\mathcal{G}_1]=\p[A]$ and therefore $\E[\one_B\one_A]=\p[A\cap B]=\p[A]\E[\one_B]=\p[A]\p[B]$ and hence the claim follows.

\end{proof}

\begin{rem}

Let $Z$ and $Y$ be two real valued r.v.'s. Then $Z$ and $Y$ are independent if and only if for all $h$ Borel measurable, such that $\E[\vert h(Z)\vert]<\infty$, we get $\E[h(Z)\mid Y]=\E[h(Z)]$. To see this we can apply the theorem with $\mathcal{G}_2=\sigma(Z)$ and note that all r.v.'s in $L^1(\Omega,\mathcal{G}_2,\p)$ are of the form $h(Z)$ with $\E[\vert h(Z)\vert]<\infty$. In particular, if $Z\in L^1(\Omega,\F,\p)$, we get $\E[Z\mid Y]=\E[Z]$. Be aware that the latter equation does not imply that $Y$ and $Z$ are independent. For example take $Z\sim\mathcal{N}(0,1)$ and $Y=\vert Z\vert$. Now for all $h$ with $\E[\vert h(\vert Z\vert)\vert]<\infty$ we get $\E[Zh(\vert Z\vert )]=0$. Thus $\E[Z\mid \vert Z\vert]=0$, but $Z$ and $\vert Z\vert$ are not independent.

\end{rem}

\begin{thm}

Let $(\Omega,\F,\p)$ be a probability space. Let $X$ and $Y$ be two r.v.'s on that space with values in the same measure space $E$ and $F$. Assume that $X$ is independent of the sub $\sigma$-Algebra $\mathcal{G}\subset \F$ and that $Y$ is $\mathcal{G}$-measurable. Then for every measurable map $g:E\times F\to \R_+$ we have 

\[
\E[g(X,Y)\mid\mathcal{G}]=\int_E g(x,y)d\p_X(x),
\]

where $\p_X$ is the law of $X$ and the right hand side has to be understood as a function $\phi(Y)$ with 

\[
\phi:y\mapsto \int_E g(x,y)d\p_X(x).
\]

\end{thm}

\begin{proof}

We need to show that for all $\mathcal{G}$-measurable r.v. $Z$ we get that 

\[
\E[g(X,Y)Z]=\E[\phi(Y)Z].
\]

Let us denote by $\p_{(X,Y,Z)}$ the distribution of $(X,Y,Z)$ on $E\times F\times \R_+$. Since $X$ is independent of $\mathcal{G}$, we have $\p_{(X,Y,Z)}=\p_X\otimes \p_{(Y,Z)}$.Thus 

\begin{align*}
\E[g(X,Y)Z]&=\int_{E\times F\times \R_+}g(x,y)zd\p_{(X,Y,Z)}(x,y,z)=\int_{F\times \R_+}z\left(\int_E g(x,y)d\p_X(x)\right) d\p_{(Y,Z)}(y,z)\\
&=\int_{F\times\R_1}z\phi(y)d\p_{(Y,Z)}(y,z)=\E[Z \phi(Y)].
\end{align*}

\end{proof}

\subsection{Important examples}We need to take a look at two important examples.

\subsubsection{Variables with densities}

Let $(X,Y)\in \R^m\times \R^n$. Assume that $(X,Y)$ has density $P(x,y)$, i.e. for all Borel measurable maps $h:\R^m\times\R^n\to \R_+$ we have 

\[
\E[h(X,Y)]=\int_{\R^m\times\R^n}h(x,y)P(x,y)dxdy.
\]

The density of $Y$ is given by 

\[
Q(y)=\int_{\R^m}P(x,y)dx.
\]

We want to compute $\E[h(X)\mid Y]$ for some measurable map $h:\R^m\to\R_+$. Therefore we have 

\begin{align*}
\E[h(X)g(Y)&=\int_{\R^m\times\R^n}h(x)g(y)P(x,y)dxdy=\int_{\R^n}\left(\int_{\R^m} h(x)P(x,y)dx\right) g(y)dy\\
&=\int_{\R^n}\frac{1}{Q(y)}\left(\int_{\R^m}h(x)P(x,y)dx\right)g(y)Q(y)\one_{\{Q(y)>0\}}dy\\
&=\int_{\R^n}\varphi(y)g(y)Q(y)\one_{\{Q(y)>0\}}dy=\E[\varphi(Y)g(y)],
\end{align*}

where 

\[
\varphi(y)=\begin{cases}\frac{1}{Q(y)}\int_{\R^n}h(x)Q(x,y)dx,& Q(y)>0\\ h(0),& Q(y)=0\end{cases}
\]

\begin{prop}

For $Y\in\R^n$, let $\nu(y,dx)$ be the probability measure on $\R^n$ defined by 

\[
\nu(x,dy)=\begin{cases}\frac{1}{Q(y)}P(x,y)& Q(y)>0\\ \delta_0(dx)& Q(y)=0\end{cases}
\]

Then for all measurable maps $h:\R^m\to\R_+$ we get 

\[
\E[h(X)\mid Y]=\int_{\R^m} h(x)\nu(Y, dx),
\]

where the right hand side has to be understood as $\phi(Y)$, where 

\[
\phi(Y)=\int_{\R^m} h(x)\nu(Y,dx).
\]

\end{prop}

\begin{rem}

In the literature, one abusively note 

\[
\E[h(X)\mid Y=y]=\int_{\R^m}h(x)\nu(y,dx),
\]

and $\nu(y,dx)$ is called the $conditional$ $distribution$ of $X$ given $Y=y$ (even though in general we have $\p[Y=y]=0$).

\end{rem}

\subsubsection{The Gaussian case}

Let $(\Omega,\F,\p)$ be a probability space. Let $X,Y_1,...,Y_p\in L^2(\Omega,\F,\p)$. We saw that $\E[X\mid Y_1,...,Y_p]$ is the orthogonal projection of $X$ on $L^2(\Omega,\sigma(Y_1,...,Y_p),\p)$. Since this conditional expectation is $\sigma(Y_1,...,Y_p)$-measurable, it is of the form $\varphi(Y_1,...,Y_p)$. In general, $L^2(\Omega,\sigma(Y_1,...,Y_p),\p)$ is of infinite dimension, so it is bad to obtain $\varphi$ explicitly. We also saw that $\varphi(Y_1,...,Y_p)$ is the best approximation of $X$ in the $L^2(\Omega,\sigma(Y_1,...,Y_p),\p)$ sense by an element of $L^2(\Omega,\sigma(Y_1,...,Y_p),\p)$. Moreover, it is well known that the best $L^2$-approximation of $X$ by an affine function of $\one,Y_1,...,Y_p$ is the best orthogonal projection of $X$ on the vector space $\{\one,Y_1,...,Y_p\}$, i.e.

\[
\E[(X-(\alpha_0+\alpha_1 Y_1+\dotsm +\alpha_p Y_p)^2)].
\]

In general, this is different from the orthogonal projection on $L^2(\Omega,\sigma(Y_1,...,Y_p),\p)$, except in the Gaussian case.

\section{Basic facts on Gaussian vectors}

A random vector $Z=(Z_1,...,Z_n)$ is said to be Gaussian, if for all $\lambda_1,...,\lambda_n\in\R$

\[
\lambda_1 Z_1+\dotsm +\lambda_n Z_n
\]
is Gaussian.Moreover, $Z$ is called centered, if $\E[Z_j]=0$ for all $1\leq  j\leq  n$. Let $Z$ be a Gaussian vector. Then for all $\xi\in\R^n$ we get 
\[
\E\left[e^{i\langle\xi, Z\rangle}\right]=\exp\left(-\frac{1}{2}\xi^t C_Z\xi\right),
\]
where $C_Z:=(C_{ij})$ and $C_{ij}=\E[Z_iZ_j]$. If $Cov(Z_i,Z_j)=0$, then $Z_i$ and $Z_j$ are independent. More generally, we have the Gaussian vectors
\[
(\underbrace{X_1,...,X_{i_1}}_{Y_1},\underbrace{X_{i_1+1},...,X_{i_2}}_{Y_2},...,\underbrace{X_{i_{n-1}+1},...,X_{i_n}}_{Y_n}).
\]
$Y_1$ and $Y_2$ are independent if and only if $Cov(X_j,X_n)=0$, where $1\leq  j\leq  i_1$ and $i_1+1\leq  k\leq  i_2$. If $Z_1,...,Z_n$ are independent Gaussian r.v.'s, we have that 
\[
Z=(Z_1,...,Z_n)
\]
is a Gaussian vector. If $Z$ is a Gaussian vector and $A\in \mathcal{M}(m\times n,\R)$ , we get that $AZ$ is again a Gaussian vector.

\begin{thm}
\label{thm1}
Let $(\Omega,\F,\p)$ be a probability space. Let $X\in L^1(\Omega,\F,\p)$ and $Y_1,...,Y_p\in L^1(\Omega,\F,\p)$ and let $(X,Y_1,...,Y_p)$ be a centered Gaussian vector. Then 

\[
\E[X\mid Y_1,...,Y_p]
\]

is the orthogonal projection of $X$ on the vector space 

\[
span\{Y_1,...,Y_p\}.
\]

Consequently, there exists real numbers $\lambda_1,...,\lambda_p$ such that 

\[
\E[X\mid Y_1,...,Y_p]=\lambda_1 Y_1+\dotsm +\lambda_p Y_p.
\]

\end{thm}

\begin{rem}
\label{rem4}
Moreover, for a measurable map $h:\R\to\R_+$ we get 

\[
\E[h(X)\mid Y_1,...,Y_p]=\int_\R h(x)Q_{\sum_{j=1}^n\lambda_j Y_j,\sigma^2}(x)dx,
\]

where $\sigma^2=\E\left[\left( X-\sum_{j=1}^n\lambda_j Y_j\right)\right]$ and 

\[
Q_{n,\sigma^2}(x)=\frac{1}{\sigma\sqrt{2\pi}}\exp\left(-\frac{(x-m)^2}{2\sigma^2}\right).
\]

\end{rem}

\begin{proof}[Proof of Remark \ref{rem4}]
Exercise.\footnote{This is done similarly to the proof of theorem \ref{thm1}}

\end{proof}

\begin{proof}[Proof of Theorem \ref{thm1}]

Let $\tilde X=\lambda_1 Y_1+\dotsm +\lambda_pY_p$ be the orthogonal projection of $X$ onto $span\{Y_1,...,Y_p\}$, meaning that for all $1\leq  j\leq  p$

\[
\E[(X-\tilde X)Y_j]=0.
\]

Note that this condition gives us explicitly the $\lambda_j's$. We obtain therefore that $(Y_1,...,Y_p,(X-\tilde X))$
is a Gaussian vector. Moreover, we get $\E[(X-\tilde X)Y_j]=Cov(X-\tilde X,Y_j)=0$ and thus $X-\tilde X$ is independent of $Y_1,...,Y_p$. Hence 

\[
\E[X\mid Y_1,...,Y_p]=\E[X-\tilde X+\tilde X\mid Y_1,...,Y_p]=\E[X-\tilde X\mid Y_1,...,Y_p]+\E[\tilde X\mid Y_1,...,Y_p]=\E[X-\tilde X]+\tilde X=\tilde X.
\]

\end{proof}

\section{Transition Kernel and Conditional distribution}

\begin{defn}[Transition Kernel]

Let $(E,\mathcal{E})$ and $(F,\F)$ be two measurable spaces. A transition kernel from $E$ to $F$ is a map

\[
\nu:E\times \F\to [0,1],
\]

such that

\begin{enumerate}[$(i)$]
\item{$\nu(x,\cdot)$ is a probability measure on $\F$ for all $x\in E$.}
\item{$x\mapsto \nu(x,A)$ is $\mathcal{E}$-measurable for all $A\in\F$.}
\end{enumerate}
\end{defn}

\begin{ex}

Let $\rho$ be a $\sigma$-finite measure on $\F$ and let $f:E\times F\to\R_+$ be a map such that 

\[
\int_F f(x,y)d\rho(y)=1.
\]

Then 

\[
\nu(x,A)=\int_A f(x,y)d\rho(y)
\]

is a transition kernel. An example for $f$ would be 

\[
f(x,y)=\frac{1}{\sigma\sqrt{2\pi}}\exp\left(-\frac{(x-y)^2}{2\sigma^2}\right).
\]

\end{ex}

\begin{prop}

The following two hold.

\begin{enumerate}[$(i)$]
\item{Let $h$ be a nonnegative (or bounded) Borel function on a measurable space $(F,\F)$. Then 

\[
\varphi(x)=\int_F h(y)\nu(x,dy)
\]

is a nonnegative (or bounded) measurable function on a measurable space $(E,\mathcal{E})$.

}
\item{If $\rho$ is a probability measure on a measurable space $(E,\mathcal{E})$, then 

\[
\mu(A)=\int_E\nu(x,A)d\rho(x),
\]

is a probability measure on a measurable space $(F,\F)$ for all $A\in \F$.

}

\end{enumerate}

\end{prop}

\begin{defn}[Conditional Distribution]

Let $X$ and $Y$ be two r.v.'s with values in a measurable space $(E,\mathcal{E})$. The conditional distribution of $Y$ given $X$ is any transition kernel $\nu$ from $E$ to $F$ such that for all nonnegative (or bounded), measurable maps $h$ on a measurable space $(F,\F)$ one has 

\[
\E[h(Y)\mid X]=\int_F h(y)\nu(X,dy)\hspace{0.5cm}a.s.,
\]

where the last equality should be understood as a map $\phi(X)$ given by 

\[
\phi:x\mapsto \int_Fh(y)\nu(x,dy).
\]

\end{defn}

\begin{rem}

If $\nu$ is the conditional distribution of $Y$ given $X$, we get for all $A\in F$

\[
\p[Y\in A\mid X]=\nu(X,A)\hspace{0.5cm}a.s.,
\]

where we have set $h=\one_A$ in the definition. If $\nu'$ is another such conditional distribution, we get 

\[
\nu(X,A)=\nu'(X,A)\hspace{0.5cm}a.s.
\]

This implies that 

\[
\nu(x,A)=\nu'(x,A)d\p_X(x)\hspace{0.5cm}a.s.
\]

\end{rem}

\begin{thm}

Assume that $(E,\mathcal{E})$ and $(F,\F)$ are two complete, separable, metric, measurable spaces endowed with their Borel $\sigma$-Algebras. Then the conditional distribution of $Y$ given $X$, exists and is a.s. unique.

\end{thm}
\begin{proof}
No proof here.
\end{proof}

\chapter{Martingales}

\section{Discrete time Martingales}

Recall that the strong law of large numbers tells us, if $(X_n)_{n\geq 1}$ is iid, $\E[\vert X_i\vert]<\infty$ and $\E[X_i]=\mu$, then 

\[
\frac{1}{n}S_n\xrightarrow{n\to\infty\atop a.s.}\mu,
\]

with $S_n=\sum_{i=1}^nX_i$. We saw that the 0-1 law of Kolmogorov implied that in this case the limit, if it exists, is constant. It is of course of interest to have a framework in which the sequence of r.v.'s converges a.s. to another r.v. This can be achieved in the framework of martingales. In this chapter, we shall consider a probability space $(\Omega,\F,\p)$ as well as an increasing family $(\F_n)_{n\geq 0}$ of sub $\sigma$-Algebras of $\F$, i.e. $\F_n\subset\F_{n+1}\subset \F$. Such a sequence is called a \emph{filtration}. The space $(\Omega,\F,(\F_n)_{n\geq 0},\p)$ is called a \emph{filtered probability} space. We shall also consider a sequence $(X_n)_{n\geq 0}$ of r.v.'s. Such a sequence is generally called a \emph{stochastic process} ($n$ is thought of as time). If for every $n\geq0$, $X_n$ is $\F_n$-measurable, we say that $(X_n)_{n\geq 0}$ is \emph{adapted} (to the filtration $(\F_n)_{n\geq0})$. One can think of $\F_n$ as the information at time $n$ and the filtration $(\F_n)_{n\geq0}$ as the flow of information in time.

\begin{rem}

Let us start with a stochastic process $(X_n)_{n\geq 0}$. We define 

\[
\F_n=\sigma(X_0,...,X_n)=\sigma(X_k\mid0\leq  k\leq  n).
\]

By construction, $\mathcal{F}_n\subset\F_{n+1}$ and $(X_n)_{n\geq0}$ is $(\F_n)_{n\geq0}$-adapted. In this case $(\F_n)_{n\geq 0}$ is called the \emph{natural filtration} of $(X_n)_{n\geq 0}$.

\end{rem}

\begin{rem}

In general, if $(\F_n)_{n\geq0}$ is a filtration, one denotes by 

\[
\F_\infty=\bigvee_{n\geq0}\F_n=\sigma\left(\bigcup_{n\geq 0}\F_n\right).
\]

the tail $\sigma$-Algebra.

\end{rem}

\begin{defn}[Martingale]

Let $(\Omega,\F,(\F_n)_{n\geq0},\p)$ be a filtered probability space. A stochastic process $(X_n)_{n\geq 0}$ is called a martingale, if 

\begin{enumerate}[$(i)$]
\item{$\E[\vert X_n\vert]<\infty$ for all $n\geq 0$.}
\item{$X_n$ is $\F_n$-measurable (adapted).}
\item{$\E[X_n\mid \F_m]=X_m$ a.s. for all $m\leq  n$.}
\end{enumerate}

The last point is equivalent to say 

\[
\E[X_{n+1}\mid\F_n]=X_n\hspace{0.5cm}a.s.,
\]

which can be obtained by using the tower property and induction.

\end{defn}

\begin{ex}
\label{ex3}
Let $(X_n)_{n\geq0}$ be a sequence independent r.v.'s such that $\E[X_n]=0$ for all $n\geq 0$ (i.e. $X_n\in L^1(\Omega,\F,(\F_n)_{n\geq 0},\p)$). Moreover, let $\F_n=\sigma(X_1,...,X_n)$ and $S_n=\sum_{i=1}^nX_i$ with $\F_0=\{\varnothing,\Omega\}$ and $S_0=0$. Then $(S_n)_{n\geq 0}$ is an $\F_n$-martingale.

\end{ex}
\begin{proof}[Proof of Example \ref{ex3}]

We need to check the assumptions for a martingale. 

\begin{enumerate}[$(i)$]
\item{The first point is clear by assumption on $X_1,...,X_n$ and linearity of the expectation. 

\[
\E[\vert S_n\vert]\leq  \sum_{i=1}\E[\vert X_i\vert]<\infty.
\]
}

\item{It is clear that $S_n$ is $\F_n$-measurable, since it is a function $X_1,...,X_n$, which are $\F_n$-measurable.

}

\item{Observe that 

\begin{multline*}
\E[S_{n+1}\mid \F_n]=\E[\underbrace{X_1+...+X_n}_{S_n}+X_{n+1}\mid\F_n]\\=\underbrace{\E[S_n\mid \F_n]}_{S_n}+\E[X_{n+1}\mid\F_n]=S_n+\E[X_{n+1}\mid\F_n]=S_n.
\end{multline*}

}

\end{enumerate}

Therefore, $(S_n)_{n\geq0}$ is a martingale.

\end{proof}

\begin{ex}
\label{ex1}
Let $(\Omega,\F,(\F_n)_{n\geq 0},\p)$ be a filtered probability space and let $Y\in L^1(\Omega,\F,(\F_n)_{n\geq0},\p)$. Define a sequence $(X_n)_{n\geq0}$ by

\[
X_n:=\E[Y\mid\F_n].
\]

Then $(X_n)_{n\geq 0}$ is an $\F_n$-martingale.

\end{ex}

\begin{proof}[Proof of Example \ref{ex1}]

Again, we show the assumptions for a martingale.

\begin{enumerate}[$(i)$]
\item{Since $\vert X_n\vert\leq \E[\vert Y\vert\mid \F_n]$, we get 

\[
\E[\vert X_n\vert]\leq  \E[\vert Y\vert]<\infty.
\]
}

\item{$\E[Y\mid \F_n]$ is $\F_n$-measurable by definition.

}
\item{With the tower property, we get

\[
\E[X_{n+1}\mid \F_n]=\E[\underbrace{\E[Y\mid\F_{n+1}]}_{X_{n+1}}\mid\F_n]=\E[Y\mid\F_n]=X_n.
\]

}

\end{enumerate}

Therefore, $(X_n)_{n\geq0}$ is a martingale.

\end{proof}

\begin{defn}[Regularity of Martingales]

Let $(\Omega,\F,(\F_n)_{n\geq0},\p)$ be a filtered probability space. A martingale $(X_n)_{n\geq 0}$ is said to be regular, if there exists a r.v. $Y\in L^1(\Omega,\F,(\F_n)_{n\geq0},\p)$ such that 

\[
X_n=\E[Y\mid\F_n]
\]

for all $n\geq 0$.

\end{defn}

\begin{prop}

Let $(\Omega,\F,(\F_n)_{n\geq0},\p)$ be a filtered probability space. Let $(X_n)_{n\geq 0}$ be a martingale. Then the map

\[
n\mapsto \E[X_n]
\]

is constant, i.e. for all $n\geq 0$ $$\E[X_n]=\E[X_0].$$

\end{prop}

\begin{proof}

By the definition of a martingale, we get 
\[
\E[X_n]=\E[\E[X_n\mid \F_0]]=X_0.
\]

\end{proof}

\begin{defn}[Discrete Stopping time]

Let $(\Omega,\F,(\F_n)_{n\geq0},\p)$ be a filtered probability space. A r.v. $T:\Omega\to\bar \N=\N\cup\{\infty\}$ is called a stopping time if for every $n\geq 0$

\[
\{T\leq  n\}\in\F_n.
\]

\end{defn}

\begin{rem}

Another, more general definition is used for continuous stochastic processes and may be given in terms of a filtration. Let $(I,\leq )$ be an ordered index set (often $I=[0,\infty)$) or a compact subset thereof, thought of as the set of possible $times$), and let $(\Omega,\F,(\F_n)_{n\geq 0},\p)$ be a filtered probability space. Then a r.v. $T:\Omega\to I$ is called a stopping time if $\{T\leq  t\}\in\F_t$ for all $t\in I$. Often, to avoid confusion, we call it a $\F_t$-stopping time and explicitly specify the filtration. Speaking concretely, for $T$ to be a stopping time, it should be possible to decide whether or not $\{T\leq  t\}$ has occurred on the basis of the knowledge of $\F_t$, i.e. $\{T\leq  t\}$ is $\F_t$-measurable.

\end{rem}

%
%
%
%
%
%

\begin{prop}

Let $(\Omega,\F,(\F_n)_{n\geq 0},\p)$ be a filtered probability space. Then 

\begin{enumerate}[$(i)$]
\item{Constant times are stopping times.}
\item{The map 

\[
T:(\Omega,\F)\to(\bar\N,\mathcal{P}(\bar\N))
\]

is a stopping time if and only if $\{T\leq  n\}\in\F_n$ for all $n\geq 0$.

}

\item{If $S$ and $T$ are stopping times, then $S\land T$, $S\lor T$ and $S+T$ are also stopping times.
}

\item{Let $(T_n)_{n\geq 0}$ be a sequence of stopping times. Then $\sup_n T$, $\inf_n T_n$, $\liminf_n T_n$ and $\limsup_n T_n$ are also stopping times.

}
\item{Let $(X_n)_{n\geq 0}$ be a sequence of adapted r.v.'s with values in some measure space $(E,\mathcal{E})$ and let $H\in\mathcal{E}$. Then, with the convention that $\inf\varnothing=\infty$,

\[
D_H=\inf\{ n\in\N\mid X_n\in H\}
\]

is a stopping time.

}

\end{enumerate}

\end{prop}

\begin{proof} We need to show all points.

\begin{enumerate}[$(i)$]

\item{This is clear.}
\item{Note that 

\[
\{T=n\}=\{T\leq  n\}\setminus \{T\leq  n-1\}\in\F_n
\]

and conversely, 

\[
\{T\leq  n\}=\bigcup_{k=0}^n\{ T=k\}\in\F_n. 
\]

}

\item{ Just observe that 

\[
\{S\lor T\leq  n\}=\{S\leq  n\}\cap\{T\leq  n\}\in\F_n
\]
\[
\{S\land T\leq  n\}=\{S\leq  n\}\cup\{T\leq  n\}\in\F_n
\]
\[
\{S+T=n\}=\bigcup_{k=0}^n\underbrace{\{S=k\}}_{\in\F_k\subset\F_n}\cap\underbrace{\{T=n-k\}}_{\in\F_{n-k}\subset\F_n}\in\F_n
\]

}
\item{First observe

\[
\{\sup_{k}T_k\leq  n\}=\bigcap_k\{T_k\leq  n\}\in\F_n
\hspace{0.5cm}\text{and}\hspace{0.5cm}
\{\inf_k T_k\leq  n\}=\bigcup_k\{T_k\leq  n\}\in\F_n.
\]

Now we can rewrite 

\[
\limsup_k T_k=\inf_k\sup_{m\geq k}T_m
\hspace{0.5cm}\text{and}\hspace{0.5cm}
\liminf_k T_k=\sup_k\inf_{m\leq  k}T_m
\]

and use the relation above.

}
\item{For all $n\in\N$, we get 

\[
\{D_H\leq  n\}=\bigcup_{k=0}^n\underbrace{\{X_k\in H\}}_{\in\F_k\subset\F_n}\in\F_n
\]

}

\end{enumerate}
\end{proof}

\begin{rem}

We say that a stopping time $T$ is bounded if there exists $C>0$ such that for all $\omega\in\Omega$

\[
T(\omega)\leq  C
\]

Without loss of generality, we can always assume that $C\in\N$. In this case we shall denote by $X_T$, the r.v. given by 

\[
X_T(\omega)=X_{T(\omega)}(\omega)=\sum_{n=0}^\infty X_n(\omega)\one_{\{T(\omega)=n\}}.
\]

Note that the sum on the right hand side perfectly defined since $T(\omega)$ is bounded.

\end{rem}

\begin{thm}

Let $(\Omega,\F,(\F_n)_{n\geq 0},\p)$ be a filtered probability space. Let $T$ be a bounded stopping time and let $(X_n)_{n\geq 0}$ be a martingale. Then we have 

\[
\E[X_T]=\E[X_0].
\]

\end{thm}

\begin{proof}

Assume that $T\leq  N\in\N$. Then 

\begin{align*}
\E[X_T]&=\E\left[\sum_{n=0}^\infty X_n\one_{\{T=n\}}\right]=\E\left[\sum_{n=0}^N X_n\one_{\{T=n\}}\right]=\sum_{n=0}^N\E[X_n\one_{\{T=n\}}]=\sum_{n=0}^N\E[\E[X_N\mid\F_n]\one_{\{T=n\}}]\\
&=\sum_{n=0}^N\E[\E[X_n\one_{\{T=n\}}\mid\F_n]]=\sum_{n=0}^N\E[X_n\one_{\{T=n\}}]=\E\left[X_n\sum_{n=0}^N\one_{\{T=n\}}\right]=\E[X_n]=\E[X_0]
\end{align*}

\end{proof}

\begin{defn}[Stopping time $\sigma$-Algebra]

Let $(\Omega,\F,(\F_n)_{n\geq0},\p)$ be a filtered probability space. Let $T$ be a stopping time for $(\F_n)_{n\geq0}$. We call the $\sigma$-Algebra of events prior of $T$ and write $\F_T$ for the $\sigma$-Algebra

\[
\F_T=\{A\in\F\mid A\cap\{T\leq  n\}\in\F_n,\forall n\geq0\}.
\]

\end{defn}

\begin{rem}
We need to show that $\F_T$ is indeed a $\sigma$-Algebra. 

\end{rem}
\begin{prop}

If $T$ is a stopping time, $\F_T$ is a $\sigma$-Algebra.

\end{prop}

\begin{proof}

It's clear that for a filtered probability space $(\Omega,\F,(\F_n)_{n\geq0},\p)$, $\Omega\in\F_T$. If $A\in\F_T$, then 

\[
A^C\cap \{T\leq  n\}=\underbrace{\{T\leq  n\}}_{\in\F_n}\setminus (\underbrace{A\cap \{T\leq  n\}}_{\in\F_n})\in\F_n
\]

and hence $A^C\in\F_n$. If $(A_i)_{i\geq0}\in\F_T$, then 

\[
\bigcup_{i\geq 0}A_i\cap \{T\leq  n\}=\bigcup_{i=1}^\infty A_i\cap \underbrace{\{T\leq  n\}}_{\in\F_n}.
\]

Hence $\bigcup_{i\geq 0}A_i\in\F_T$. Therefore $\F_T$ is a $\sigma$-Algebra.

\end{proof}

\begin{rem}

If $T=n_0$ is constant, then $\F_T=\F_{n_0}$.

\end{rem}

\begin{exer} The following exercises are important.

\begin{enumerate}[$(i)$]
\item{Show that 

\[
\F_T=\left\{\bigcup_{n\in\bar\N}A_n\cap \{T=n\}\big| A_\infty\in\F,A_n\in\F_n\right\}.
\]

}

\item{Show that a r.v. $L$ with values in $\bar \N$ is a stopping time if and only if $\left(\one_{\{L\leq  n\}}\right)_{n\geq 0}$ is $(\F_n)$-adapted and for the case it's a stopping time, we get 

\[
L=\inf\{n\geq 0\mid \one_{\{ L\leq  n\}}=1\}.
\]

}

\end{enumerate}
\end{exer}

\begin{prop}

Let $S$ and $T$ be two stopping times. 

\begin{enumerate}[$(i)$]
\item{If $S\leq  T$, then $\F_S\subset \F_T$.}
\item{$\F_{S\land T}=\F_S\cap\F_T$.}
\item{$\{S\leq  T\}$, $\{S=T\}$ and $\{S<T\}$ are $\F_S\cap\F_T$-measurable.}
\end{enumerate}
\end{prop}

\begin{proof} We need to show all three points.

\begin{enumerate}[$(i)$]

\item{For $n\in\N$ and $A\in\F_S$ we get

\[
A\cap\{T\leq  n\}= A\cap\underbrace{\{S\leq  n\}\cap\{T\leq  n\}}_{\{T\leq  n\}}=(A\cap\{S\leq  n\})\cap\underbrace{\{T\leq  n\}}_{\in\F_n}\in\F_n.
\]

Therefore $A\in\F_T$.

}

\item{Since $S\land T\leq  S$, we get by $(i)$ that $\F_{S\land T}\subset\F_S$ and similarly that $\F_{S\land T}\subset\F_T$. Let now $A\in \F_S\cap\F_T$. Then 

\[
A\cap\{S\land T\leq  n\}=\left(\underbrace{A\cap \{ S\leq  n\}}_{\in\F_n,(\text{since $A\in\F_S$})}\right)\cup \left( \underbrace{A\cap\{ T\leq  n\}}_{\in\F_n,(\text{since $A\in\F_T$})}\right)\in\F_n.
\]

Therefore $A\in\F_{S\land T}$.

}

\item{Note that 

\[
\{S\leq  T\}\cap\{T=n\}=\{S\leq  n\}\cap \{ T=n\}\in\F_n.
\]

Therefore $\{S\leq  T\}\in\F_T$. Note also that 

\[
\{S<T\}\cap\{T=n\}=\{S<n\}\cap \{T=n\}\in\F_n.
\]

Therefore $\{S<T\}\in\F_T$. Finally, note that 

\[
\{S=T\}\cap\{T=n\}=\{S=n\}\cap\{T=n\}\in\F_n.
\]

Thus $\{S=T\}\in\F_T$. Similarly one can show that these events are also $\F_S$-measurable.

}

\end{enumerate}
\end{proof}

\begin{prop}

Let $(\Omega,\F,(\F_n)_{n\geq0},\p)$ be a filtered probability space. Let $(X_n)_{n\geq 0}$ be a stochastic process, which is adapted, i.e. $X_n$ is $\F_n$-measurable for all $n\geq 0$. Let $T$ be a finite stopping time, i.e. $T<\infty$ a.s., such that $X_T$ is well defined. Then $X_T$ is $\F_T$-measurable.
\end{prop}

\begin{proof}

Let $\Lambda\in\B(\R)$ be a Borel measurable set. We want to show that 

\[
\{X_T\in\Lambda\}\in\F_T,
\]

that is, for all $n\geq 0$

\[
\{X_T\in\Lambda\}\cap\{T\leq  n\}\in\F_n.
\]

Observe that 

\[
\{X_T\in\Lambda\}\cap\{T\leq  n\}=\bigcup_{k=1}^n\{X_T\in \Lambda\}\cap\{T\leq  k\}=\bigcup_{k=1}^n\underbrace{\{X_k\in\Lambda\}}_{\in\F_k\subset\F_n}\cap\underbrace{\{T=k\}}_{\in\F_k\subset\F_n},
\]

which implies that $\{X_T\in\Lambda\}\cap\{T\leq  n\}\in\F_n$ and the claim follows.

\end{proof}

\begin{thm}

Let $(\Omega,\F,(\F_n)_{n\geq0},\p)$ be a filtered probability space. Let $(X_n)_{n\geq 0}$ be a martingale and let $S$ and $T$ be two bounded stopping times such that $S\leq  T$ a.s. Then we have 

\[
\E[X_T\mid \F_S]=X_S\hspace{0.5cm}a.s.
\]

\end{thm}

\begin{proof}

Since we assume that $T\leq  C\in\N$, we note that 

\[
\vert X_T\vert\leq  \sum_{i=0}^C\vert X_i\vert\in L^1(\Omega,\F,(\F_n)_{n\geq 0},\p).
\]

Let now $A\in\F_S$. We need to show that 

\[
\E[X_T\one_A]=\E[X_S\one_A].
\]

Let us define the random time

\[
R(\omega)=S(\omega)\one_A(\omega)+T(\omega)\one_{A^C}(\omega).
\]

We thus note that $R$ is a stopping time. Indeed, we have 

\[
\{R\leq  n\}=(\underbrace{A\cap \{S\leq  n\}}_{\in\F_n})\cup (\underbrace{A^C\cap \{T\leq  n\}}_{\in\F_n}).
\]

Consequently, since $S,T$ and $R$ are bounded, we have 

\[
\E[X_S]=\E[X_T]=\E[X_R]=\E[X_0]. 
\]

Therefore we get 

\[
\E[X_R]=\E[X_S\one_A+X_T\one_{A^C}]\hspace{0.5cm}\text{and}\hspace{0.5cm}\E[X_T]=\E[X_T\one_A+X_T\one_{A^C}]
\]

and thus 

\[
\E[X_S\one_A]=\E[X_T\one_A].
\]

Moreover, since $X_S$ is $\F_S$-measurable, we conclude that 

\[
\E[X_T\mid \F_S]=X_S\hspace{0.5cm}a.s.
\]

\end{proof}

\begin{exer}
Let $T$ be a stopping time and $\Lambda\in\F_T$. Define 

\[
T_\Lambda(\omega)=\begin{cases}T(\omega)&\text{if $\omega\in\Lambda$}\\ \infty&\text{if $\omega\not\in\Lambda$}\end{cases}
\]

Prove that $T_\Lambda$ is a stopping time.

\end{exer}

\begin{prop}

Let $(\Omega,\F,(\F_n)_{n\geq0},\p)$ be a filtered probability space. Let $(X_n)_{n\geq 0}$ be a stochastic process such that for all $n\geq 0$

\[
\E[\vert X_n\vert]<\infty
\]

and with $X_n$ being $\F_n$-measurable. If for all bounded stopping times $T$, we have 

\[
\E[X_T]=\E[X_0],
\]

then $(X_n)_{n\geq 0}$ is a martingale.

\end{prop}

\begin{proof}

Let $0\leq  m<n<\infty$ and $\Lambda\in\F_m$. Define for all $\omega\in\Omega$

\[
T(\omega)=m\one_{\Lambda^C}(\omega)+n\one_\Lambda(\omega).
\]

Then $T$ is a stopping time. Therefore

\[
\E[X_0]=\E[X_T]=\E[X_m\one_{\Lambda^C}+X_n\one_\Lambda]=\E[X_m].
\]

Hence we get 

\[
\E[X_m\one_\Lambda]=\E[X_n\one_\Lambda]
\]

and thus 

\[
\E[X_n\mid \F_m]=X_m\hspace{0.5cm}a.s.
\]

\end{proof}

\section{Submartingales and Supermartingales}

\begin{defn}[Submartingale and Supermartingale]

Let $(\Omega,\F,(\F_n)_{n\geq0},\p)$ be a filtered probability space. A stochastic process $(X_n)_{n\geq0}$ is called a submartingale (resp. supermartingale) if 

\begin{enumerate}[$(i)$]
\item{$\E[\vert X_n\vert]<\infty$ for all $n\geq 0$}
\item{$(X_n)_{n\geq 0}$ is $\F_n$-adapted.}
\item{$\E[X_n\mid \F_m]\geq X_m$ a.s. for all $m\leq  n$ (resp. $\E[X_n\mid \F_m]\leq  X_m$ a.s. for all $m\leq  n$)}
\end{enumerate}
\end{defn}

\begin{rem}

A stochastic process $(X_n)_{n\geq 0}$ is a martingale if and only if it is a submartingale and a supermartingale. A martingale is in particular a submartingale and a supermartingale. If $(X_n)_{n\geq0}$ is a submartingale, then the map $n\mapsto \E[X_n]$ is increasing. If $(X_n)_{n\geq 0}$ is a supermartingale, then the map $n\mapsto \E[X_n]$ is decreasing.

\end{rem}

\begin{ex}

Let $(\Omega,\F,(\F_n)_{n\geq0},\p)$ be a filtered probability space. Let $S_n=\sum_{j=1}^{n}Y_j$, where $(Y_n)_{n\geq1}$ is a sequence of iid r.v.'s. Moreover, let $S_0=0$, $\F_0=\{\varnothing,\Omega\}$ and $\F_n=\sigma(Y_1,...,Y_n)$. Then we get 

\[
\E[S_{n+1}\mid\F_n]=S_n+\E[Y_{n+1}].
\]

If $\E[Y_{n+1}]>0$, then $\E[S_{n+1}\mid\F_n]\geq S_n$ and thus $(S_n)_{n\geq 0}$ is a submartingale. On the other hand, if $\E[Y_{n+1}]<0$, then $\E[S_{n+1}\mid\F_n]\leq  S_n$ and thus $(S_n)_{n\geq 0}$ is a supermartingale.

\end{ex}

\begin{prop}

Let $(\Omega,\F,(\F_n)_{n\geq0},\p)$ be a filtered probability space. If $(M_n)_{n\geq 0}$ is a martingale and $\varphi$ is a convex function such that $\varphi(M_n)\in L^1(\Omega,\F,(\F_n)_{n\geq 0},\p)$ for all $n\geq0$, then 

\[
(\varphi(M_n))_{n\geq 0}
\]

is a submartingale.

\end{prop}

\begin{proof}

The first two conditions for a martingale are clearly satisfied. Now for $m\leq  n$, we get 

\[
\E[M_n\mid \F_m]=M_m\hspace{0.5cm}a.s.,
\]

since $(M_n)_{n\geq 0}$ is assumed to be a martingale. Hence, with Jensen's inequality, we get

\[
\varphi(\E[M_n\mid\F_m])=\varphi(M_m)\leq  \E[\varphi(M_n)\mid\F_m]\hspace{0.5cm}a.s.
\]

\end{proof}

\begin{cor}

Let $(\Omega,\F,(\F_n)_{n\geq0},\p)$ be a filtered probability space. If $(M_n)_{n\geq 0}$ is a martingale, then 

\begin{enumerate}[$(i)$]
\item{$(\vert M_n\vert)_{n\geq0}$ and $(M^+_n)_{n\geq 0}$ are submartingales.}
\item{if for all $n\geq 0$, $\E[M_n^2]<\infty$, then $(M_n^2)_{n\geq 0}$ is a submartingale.}
\end{enumerate}
\end{cor}

\begin{thm}

Let $(\Omega,\F,(\F_n)_{n\geq0},\p)$ be a filtered probability space. Let $(X_n)_{n\geq 0}$ be a submartingale and let $T$ be a stopping time bounded by $C\in\N$. Then 

\[
\E[X_T]\leq  \E[X_C].
\]

\end{thm}

\begin{proof}
Exercise\footnote{The proof is the same as in Theorem 7.7.}

\end{proof}

\begin{thm}[Doob's decomposition]

Let $(\Omega,\F,(\F_n)_{n\geq0},\p)$ be a filtered probability space. Let $(X_n)_{n\geq 0}$ be a submartingale. Then there exists a martingale $M=(M_n)_{n\geq 0}$ with $M_0=0$ and a sequence $A=(A_n)_{n\geq 0}$, such that $A_{n+1}\geq A_n$ a.s. with $A_0=0$ a.s., which is called an increasing process, and with $A_{n+1}$ being $\F_n$-measurable, which we will call predictable, such that 

\[
X_n=X_0+M_0+A_n.
\]

Moreover, this decomposition is a.s. unique.

\end{thm}
\begin{proof}

Let us define $A_0=0$ and for $n\geq 1$

\[
A_n=\sum_{k=1}^n\E[X_k-X_{k-1}\mid\F_{k-1}].
\]

Since $(X_n)_{n\geq 0}$ is a submartingale, we get 

\[
\E[X_k-X_{k-1}\mid\F_{k-1}]\geq 0
\]

and hence $A_{n+1}-A_n\geq 0$. Therefore $(A_n)_{n\geq 0}$ is an increasing process. Moreover, from the definition of the conditional expectation, $A_n$ is $\F_{n-1}$-measurable for $n\geq 1$. Thus $A_n$ is predictable as well. We also note that 

\[
\E[X_n\mid \F_{n-1}]-X_{n-1}=\E[X_n-X_{n-1}\mid \F_{n-1}]=A_n-A_{n-1}.
\]

Hence we get 

\[
\underbrace{\E[X_n\mid\F_{n-1}]}_{\E[X_n-A_n\mid \F_{n-1}]}-A_n=X_{n-1}-A_{n-1}.
\]

If we set $M_n=X_n-A_n-X_0$, it follows that $M=(M_n)_{n\geq 0}$ is a martingale with $M_0=0$. This proves the existence part. For uniqueness, we note that if we have two such decompositions 

\[
X_n=X_0+M_n+A_n=X_0+L_n+C_n,
\]

where $L_n$ denotes the martingale part and $C_n$ the increasing process part, it follows that 

\[
L_n-M_n=A_n-C_n.
\]

Now since $A_n-C_n$ is $\F_{n-1}$-measurable, we get that $L_n-M_n$ is also $\F_{n-1}$-measurable. Thus 

\[
L_n-M_n=\E[L_n-M_n\mid \F_{n-1}]=L_{n-1}-M_{n-1},
\]

because of the martingale property. By induction, we have a chain of equalities 

\[
L_n-M_n=L_{n-1}-M_{n-1}=\dotsm =L_0-M_0=0.
\]

Therefore $L_n=M_n$ and also $A_n=C_n$.

\end{proof}

\begin{cor}

Let $(\Omega,\F,(\F_n)_{n\geq0},\p)$ be a filtered probability space. Let $X=(X_n)_{n\geq 0}$ be a supermartingale. Then there exists a.s. a unique decomposition 

\[
X_n=X_0+M_n-A_n,
\]

where $M=(M_n)_{n\geq 0}$ is a martingale with $M_0=0$ and $A=(A_n)_{n\geq 0}$ is a increasing process with $A_0=0$. 

\end{cor}

\begin{proof}

Let $Y_n=-X_n$ for all $n\geq 0$. Then the stochastic process obtained by $(Y_n)_{n\geq0}$ is a submartingale. Theorem 8.4. tells us that there exists a unique decomposition

\[
Y_n=Y_0+L_n+C_n,
\]

where $L_n$ denotes the martingale part and $C_n$ the increasing process part. Hence we get 

\[
X_n=X_0-L_n-C_n
\]

and if we take $M_n=-L_n$ and $A_n=C_n$, the claim follows.

\end{proof}

Now consider a stopped process. Let $(\Omega,\F,(\F_n)_{n\geq0},\p)$ be a filtered probability space. Let $T$ be a stopping time and let $(X_n)_{n\geq 0}$ be a stochastic process. We denote by $X^T=(X^T_n)_{n\geq 0}$ the process $(X_{n\land T})_{n\geq 0}$.

\begin{prop}

Let $(\Omega,\F,(\F_n)_{n\geq0},\p)$ be a filtered probability space. Let $(X_n)_{n\geq 0}$ be a martingale (resp. sub- or supermartingale) and let $T$ be a stopping time. Then $(X_{n\land T})_{n\geq 0}$ is also a martingale (resp. sub- or supermartingale).

\end{prop}

\begin{proof}

Note that 

\[
\{T\geq n+1\}=\{T\leq  n\}^C\in\F_n.
\]
Hence we have 

\[
\E[X_{n+1\land T}-X_{n\land T}\mid \F_n]=\E[(X_{n+1\land T}-X_{n\land T})\one_{\{T\geq n+1\}}\mid\F_n]=\one_{\{ T\geq n+1\}}\E[X_{n+1}-X_n\mid \F_n].
\]

If $(X_n)_{n\geq 0}$ is a martingale, we deduce that 

\[
\E[X_{n+1\land T}-X_{n\land T}\mid\F_n]=0.
\]

Moreover, $X_{n\land T}$ is $\F_n$-measurable. Therefore

\[
\E[X_{n+1\land T}\mid \F_n]=X_{n\land T}.
\]

The same holds for sub-and super martingales.

\end{proof}

\begin{thm}

Let $(\Omega,\F,(\F_n)_{n\geq0},\p)$ be a filtered probability space. Let $(X_n)_{n\geq 0}$ be a submartingale (resp. supermartingale) and let $S$ and $T$ be two bounded stopping times, such that $S\leq  T$ a.s. Then 

\[
\E[X_T\mid\F_S]\geq X_S\hspace{0.5cm}a.s.\hspace{0.3cm}(resp.\hspace{0.5cm} \E[X_T\mid \F_S]\leq  X_S\hspace{0.5cm}a.s.)
\]

\end{thm}

\begin{proof}

Let us assume that $(X_n)_{n\geq 0}$ is a supermartingale. Let $A\in\F_S$ such that $S\leq  T\leq  C\in\N$. We already know that $(X_{n\land T})_{n\geq 0}$ is a supermartingale. Therefore we get 

\begin{align*}
\E[X_T\one_A]&=\sum_{j=0}^C\E[X_{\underbrace{C\land T}_{T}}\one_A\one_{\{S=j\}}]=\sum_{j=0}^C\E[X_{C\land T}\one_{\underbrace{A\cap\{S=j\}}_{\in\F_j}}]\\
&\leq  \sum_{j=0}^C\E[X_{j\land T}\one_{A\cap \{S=j\}}]=\sum_{j=0}^C\E[X_j\one_A\one_{\{S=j\}}]\\
&=\E\left[\sum_{j=0}^CX_j\one_{\{S=j\}}\one_A\right]=\E[X_S\one_A]=\E[X_T\mid\F_S]\leq  X_S.
\end{align*}
\end{proof}

\begin{cor}

Let $(\Omega,\F,(\F_n)_{n\geq0},\p)$ be a filtered probability space. Let $(X_n)_{n\geq 0}$ be a submartingale (resp. supermartingale) and let $T$ be a bounded stopping time. Then 

\[
\E[X_T]\geq \E[X_0]\hspace{0.5cm}(resp.\hspace{0.5cm}\E[X_T]\leq  \E[X_0]).
\]

Moreover, if $S\leq  T$, for $S$ and $T$ two bounded stopping times, we have

\[
\E[X_T]\geq \E[X_S]\hspace{0.5cm}(resp.\hspace{0.5cm}\E[X_T]\leq  \E[X_S]).
\]

\end{cor}

\begin{exer}

Let $(\Omega,\F,(\F_n)_{n\geq0},\p)$ be a filtered probability space. Let $X=(X_n)_{n\geq 0}$ be a supermartingale and let $T$ be a stopping time. Then $$X_T\in L^1(\Omega,\F,(\F_n)_{n\geq 0},\p)$$ and $$\E[X_T]\leq  \E[X_0]$$ in each case of the following situations.

\begin{enumerate}[$(i)$]
\item{$T$ is bounded.}
\item{$X$ is bounded and $T$ is finite.}
\item{$\E[T]<\infty$ and for some $k\geq 0$, we have 

\[
\vert X_n(\omega)-X_{n-1}(\omega)\vert\leq  k,
\]

for all $\omega\in\Omega$.

}
\end{enumerate}
\end{exer}

\section{Martingale inequalities}

Let $(\Omega,\F,(\F_n)_{n\geq0},\p)$ be a filtered probability space. Let $X=(X_n)_{n\geq 0}$ be a stochastic process, such that $X_n$ is $\F_n$-measurable for all $n\geq 0$. We denote

\[
X_n^*:=\sup_{j\leq  n}\vert X_j\vert.
\]

Note that $(X_n^*)_{n\geq 0}$ is increasing and $\F_n$-adapted. Therefore if $X_n\in L^1(\Omega,\F,(\F_n)_{n\geq 0},\p)$ for all $n\geq 0$, then $(X_n^*)_{n\geq 0}$ is a submartingale. 

\subsection{Maximal inequality and Doob's inequality}

Recall Markov's inequality in terms of $(X_n^*)_{n\geq0}$, which is given by 

\[
\p[X_n^*\geq \alpha]\leq  \frac{\E[X_n^*]}{\alpha},
\]

with the obvious bound 

\[
\E[X_n^*]\leq  \sum_{j=1}^n\E[\vert X_j\vert].
\]

We shall see for instance that when $(X_n)_{n\geq 0}$ is a martingale, one can replace $\E[X_n^*]$ by $\E[\vert X_n\vert]$. 

\begin{prop}
\label{prop1}
Let $(\Omega,\F,(\F_n)_{n\geq0},\p)$ be a filtered probability space. Let $(X_n)_{n\geq 0}$ be a submartingale and let $\lambda>0,k\in\N$. Define 

\begin{align*}
A&:=\left\{\max_{0\leq  n\leq  k}X_n\geq \lambda\right\}\\
B&:=\left\{\min_{0\leq  n\leq  k}X_n\leq  -\lambda\right\}.
\end{align*}

Then the following hold.

\begin{enumerate}[$(i)$]
\item{$$\lambda\p[A]\leq  \E[X_k\one_A],$$}
\item{$$\lambda\p[B]\leq  \E[X_k\one_{B^C}]-\E[X_0].$$}
\end{enumerate}
\end{prop}

\begin{rem}

If $(X_n)_{n\geq 0}$ is a martingale, then $(\vert X_n\vert)_{n\geq 0}$ is a submartingale. Moreover, from $(i)$ we get 

\[
\lambda\p[X^*_k\geq \alpha]\leq  \E[\vert X_k\vert\one_A]\leq  \E[\vert X_k\vert] 
\]

and hence

\[
\p[X_k^*\geq \alpha]\leq  \frac{\E[\vert X_k\vert]}{\alpha}.
\]

\end{rem}

\begin{proof}[Proof of Proposition \ref{prop1}]

We need to show both points.

\begin{enumerate}[$(i)$]
\item{Let us introduce 

\[
T=\inf\{n\in\N\mid X_n\leq  \lambda\}\land k.
\]

Then $T$ is a stopping time, which is bounded by $k$. We thus have 

\[
\E[X_T]\leq  \E[X_k].
\]

We note that $X_T=X_k$ if $T=k$, which happens for $\omega\in A^C$. Hence we get

\[
\E[X_T]=\E[X_T\one_A+X_T\one_{A^C}]=\E[X_T\one_A]+\E[X_k\one_{A^C}]\leq  \underbrace{\E[X_k]}_{\E[X_k(\one_A+\one_{A^C})]}.
\]

Now we note that 

\[
\E[X_T\one_A]\geq \lambda \E[\one_A]=\lambda \p[A].
\]

Therefore we get 

\[
\lambda\p[A]\leq  \E[X_k\one_A].
\]

}

\item{Let us define 

\[
S=\inf\{n\in\N\mid X_n\leq  -\lambda\}\land k.
\]

Again $S$ is a stopping time, which is bounded by $k$. We hence have

\[
\E[X_S]\geq \E[X_0].
\]

Thus 

\[
\E[X_0]\leq  \E[X_S\one_B]+\E[X_S\one_{B^C}]\leq  -\lambda \p[B]+\E[X_k\one_{B^C}].
\]

Therefore we get 

\[
\lambda\p[B]\leq  \E[X_k\one_{B^C}]-\E[X_0].
\]

}

\end{enumerate}

\end{proof}

\begin{prop}[Kolmogorov's inequality]

Let $(\Omega,\F,(\F_n)_{n\geq0},\p)$ be a filtered probability space. Let $(X_n)_{n\geq 0}$ be a martingale, such that for all $n\geq 0$ we have $\E[X_n^2]<\infty$. Then 

\[
\p\left[\max_{0\leq  k\leq  n}\vert X_k\vert\geq \lambda\right]\leq  \frac{\E[X_k^2]}{\lambda^2}.
\]

\end{prop}

\begin{proof}

We use the fact that $(X_n^2)_{n\geq 0}$ is a positive submartingale. Therefore we get 

\[
\lambda^2\underbrace{\p\left[\max_{0\leq  k\leq  n}\vert X_k\vert^2\geq \lambda^2\right]}_{\p\left[\max_{0\leq  k\leq  k}\vert X_k\vert\geq \lambda\right]}\leq  \E\left[X_k^2\one_{\left\{\max_{0\leq  k\leq  n}\vert X_k\vert^2\geq \lambda^2\right\}}\right]\leq  \E[X_k^2]
\]

\end{proof}

\begin{thm}[Maximal inequality]

Let $(\Omega,\F,(\F_n)_{n\geq0},\p)$ be a filtered probability space. Let $(X_n)_{n\geq 0}$ be a submartingale. Then for all $\lambda\geq 0$ and $n\in\N$, we get

\[
\lambda\p\left[\max_{0\leq  k\leq  n}\vert X_k\vert\geq \lambda\right]\leq  \E[X_0]+2\E[\vert X_n\vert].
\]

\end{thm}

\begin{proof}

Let $A$ and $B$ be defined as in Proposition 9.1. Then 

\begin{align*}
\lambda \p\left[\max_{0\leq  k\leq  n}\vert X_k\vert \geq \lambda\right]&=\lambda\p[A\cup B]\leq \E[X_k\one_A]-\E[X_0]+\E[X_k\one_{B^C}]\\
&\leq \E[\vert X_0\vert]+\E[\vert X_n\vert]+\E[\vert X_n\vert]=\E[\vert X_0\vert]+2\E[\vert X_n\vert].
\end{align*}

\end{proof}

\begin{thm}[Doob's inequality]
\label{thm2}
Let $(\Omega,\F,(\F_n)_{n\geq0},\p)$ be a filtered probability space. Let $p>1$ and $q>1$, such that $\frac{1}{p}+\frac{1}{q}=1$. 

\begin{enumerate}[$(i)$]
\item{If $(X_n)_{n\geq 0}$ is a submartingale, then for all $n\geq 0$ we have 

\[
\left\|\max_{0\leq  k\leq  n}X_k^+\right\|_p\leq  q\left\|X_n^+\right\|_p.
\]
}
\item{If $(X_n)_{n\geq 0}$ is a martingale, then for all $n\geq 0$ we have 

\[
\left\|\max_{0\leq  k\leq  n}\vert X_k\vert\right\|_p\leq  q\| X_n\|_p.
\]

}
\end{enumerate}

\end{thm}

\begin{rem}

Recall that if $X\in L^1(\Omega,\F,(\F_n)_{n\geq 0},\p)$, then 

\[
\|X\|_p=\E[\vert X\vert^p]^{1/p}.
\]

Moreover, if $p=q=2$ and $(X_n)_{n\geq 0}$ is a martingale, then for all $n\geq 0$ we have 

\[
\E\left[\max_{0\leq  k\leq  n} X_k^2\right]\leq  4\E[X_n^2].
\]

In general, we have 

\[
\vert X_n\vert^2\leq  \max_{0\leq  k\leq  n}\vert X_k\vert^p.
\]

Therefore we get 

\[
\E[\vert X_n\vert^p]\leq  \E\left[\max_{0\leq  k\leq  n}\vert X_k\vert^p\right]\leq ^{Doob}q^p\E[\vert X_n\vert^p].
\]

We shall also recall that for $X\in L^p(\Omega,\F,(\F_n)_{n\geq 0},\p)$, we can write 

\[
\E[\vert X\vert^p]=\E\left[\int_0^{\vert X\vert}p\lambda^{p-1}d\lambda\right]=\E\left[\int_0^\infty \one_{\{\vert X\vert\geq \lambda\}}p\lambda^{p-1}d\lambda\right]=\int_0^\infty p\lambda^{p-1}\p[\vert X\vert \geq \lambda]d\lambda
\]

by using Fubini's theorem.

\end{rem}

\begin{proof}[Proof of Theorem \ref{thm2}]

It is enough to prove $(ii)$. Since $(X_n)_{n\geq 0}$ is a submartingale, we know that $(X_n^+)_{n\geq 0}$ is a submartingale. Hence 

\[
\lambda\p\left[\max_{0\leq  k\leq  n}X_k^+\geq \lambda\right]\geq\E\left[X_n^+\one_{\{\max_{0\leq  k\leq  n}X_k^+\geq \lambda\}}\right].
\]

Now ler $Y_n:=\max_{0\leq  k\leq  n}X_k^+$. Then for any $k>0$, we have 

\begin{align*}
\E[(Y_n\land k)^p]&=\int_0^\infty p\lambda^{p-1}\p[Y_n\land k\geq \lambda]d\lambda=\int_0^np\lambda^{p-1}\p[Y_n\geq\lambda]d\lambda\\
&\leq  \int_0^n p\lambda^{p-1}\left(\frac{1}{\lambda}\E[X_n^+\one_{\{ Y_n\geq \lambda\}}]\right)d\lambda\\
&=\E\left[\int_0^n p\lambda^{p-1}X_n^+\one_{\{Y_n\geq\lambda\}}d\lambda\right]=\E\left[\int_0^{Y_n\land k}p\lambda^{p-2}X_n^+d\lambda\right]\\
&=\E\left[\frac{p}{p-1}(Y_n\land k)^{p-1}X_n^+\right]\\
&\leq  q\E[(X_n^+)^p]^{1/p}\E[(Y_n\land k)^p]^{1/q},
\end{align*}

where we have used that $q=\frac{p}{p-1}$ and Markov's inequality. Therefore we obtain

\[
\E[(Y_n\land k)^p]\leq  q\E[(X_n^+)^p]^{1/p}\E[(Y_n\land k)^p]^{1/q}.
\]

Since $\E[(Y_n\land k)^p]\not=0$, we can divide by it to get

\[
\E[(Y_n\land k)^p]^{1-1/q=1/p}\leq  q\E[(X_n^+)^p]^{1/p}
\]

and thus 

\[
\| (Y_k\land n)\|_p\leq  q\| X_n^+\|_p.
\]

Now for $k\to\infty$, monotone convergence implies that 

\[
\|Y_n\|_p\leq  q\|X_n^+\|_p.
\]

\end{proof}

\begin{cor}

Let $(\Omega,\F,(\F_n)_{n\geq0},\p)$ be a filtered probability space. Let $(X_n)_{n\geq 0}$ be a martingale and $p>1$, $q>1$ such that $\frac{1}{p}+\frac{1}{q}=1$. Then 

\[
\left\| \sup_{n\geq 0}\vert X_n\vert\right\|_p\leq  q\sup_{n\geq 0}\| X_n\|_p.
\]

\end{cor}

\begin{proof}
Exercise\footnote{Use Doob's inequality.}

\end{proof}



\section{Almost sure convergence for Martingales}

Let $(\Omega,\F,(\F_n)_{n\geq0},\p)$ be a filtered probability space. We start with a useful remark. If $(X_n)_{n\geq 0}$ is a submartingale, we get in particular that $X_n\in L^1(\Omega,\F,(\F_n)_{n\geq0},\p)$ for all $n\geq 0$. Moreover, we know that we can write 

\[
\E[X_n]=\E[X_n^+]-\E[X_n^-]
\]

and hence 

\[
\E[X_n^-]=\E[X_n^+]-\E[X_n].
\]

The submartingale property implies that $\E[X_0]\leq  \E[X_n]$ and thus 

\[
\E[X_n^-]\leq  \E[X_n^+]-\E[X_0].
\]

Therefore, if $\sup_{n\geq 0}\E[X_n^+]<\infty$, then $\E[X_n^-]\leq  \sup_{n\geq 0}\E[X_n^+]-\E[\vert X_0\vert]<\infty$. Since $\vert X_n\vert =X_n^++X_n^-$, we have that $\sup_{n\geq 0}\E[X_n^+]<\infty$ if and only if $\sup_{n\geq 0}\E[\vert X_n\vert]<\infty$.

\begin{lem}[Doob's upcrossing inequality]
\label{lem1}
Let $(\Omega,\F,(\F_n)_{n\geq0},\p)$ be a filtered probability space. Let $X=(X_n)_{n\geq 0}$ be a supermartingale and $a<b$ two real numbers. Then for all $n\geq 0$ we get 

\[
(b-a)\E[N_n([a,b],X)]\leq  \E[(X_n-a)^-],
\]

where $N_n([a,b],x)=\sup\{ k\geq 0\mid T_k(x)\leq  n\}$, i.e. the number of uncrossings of the interval $[a,b]$ by the sequence $x=(x_n)_n$ by time $n$, and $(T_k)_{k\geq 0}$ is a sequence of stoppping times. Moreover, as $n\to\infty$ we have 

\[
N_n([a,b],x)\uparrow N([a,b],x)=\sup\{ k\geq 0\mid T_k(x)<\infty\},
\]

i.e., the total number of up crossings of the interval $[a,b]$.

\end{lem}

\begin{lem}
\label{lem4}
A sequence of real numbers $x=(x_n)_n$ converges in $\bar \R=\R\cup\{\pm\infty\}$ if and only if $N([a,b],x)<\infty$ for all rationals $a<b$. 
\end{lem}

\begin{proof}

Suppose that $x$ converges. Then if for some $a<b$ we had that $N([a,b],x)=\infty$, that would imply that $\liminf_n x_n\leq  a<b\leq  \limsup_nx_n$, which is a contradiction. Next suppose that $x$ does not converge. Then $\liminf_nx_n<\limsup_nx_n$ and so taking $a<b$ rationals between these two numbers gives that $N([a,b],x)=\infty$.

\end{proof}

\begin{proof}[Proof of Lemma \ref{lem1}]

We will omit the dependence on $X$ from $T_k$ and $S_k$ and we will write $N=N_n([a,b],X)$ to simplify notation. By the definition of the times $(T_k)_{k\geq 0}$ and $(S_k)_{k\geq 0}$, it is clear that for all $k$

\begin{equation}
X_{T_k}-X_{S_k}\geq b-a.
\end{equation}

We have

\begin{align*}
\sum_{k=1}^n(X_{T_k\land n}-X_{S_k\land n})&=\sum_{k=1}^N(X_{T_k}-X_{S_k})+\sum_{k=N+1}^n(X_n-X_{S_k\land n})\one_{\{N<n\}}\\
&=\sum_{k=1}^N(X_{T_k}-X_{S_k})+(X_n-X_{S_{N+1}})\one_{\{S_{N+1}\leq  n\}},
\end{align*}

since the only term contributing in the second sum appearing in the middle of the last equation chain is $N+1$, by the definition of $N$. Indeed, if $S_{N+2}\leq  n$, then that would imply that $T_{N+1}\geq n$, which would contradict the definition of $N$. using induction on $k\geq 0$, it is easy to see that $(T_k)_{k\geq 0}$ and $(S_k)_{k\geq 0}$ are stopping times. Hence for all $n\geq 0$, we have that $S_k\land n\leq  T_k\land n$ are bounded stopping times and thus we get that $\E[X_{S_k\land n}]\geq \E[X_{T_k\land n}]$, for all $k\geq 0$. Therefore, taking expectations in the equations above and using the inequality (9) we get 

\[
0\geq \E\left[\sum_{k=1}^n(X_{T_k\land n}-X_{S_k\land n})\right]\geq (b-a)\E[N]-\E[(X_n-a)^-],
\]

since $(X_n-X_{S_{N+1}})\one_{\{S_{N+1}\leq  n\}}\geq -(X_n-a)^-$. Rearranging gives the desired inequality.

\end{proof}

\begin{thm}[Almost sure martingale convergence theorem]

Let $(\Omega,\F,(\F_n)_{n\geq0},\p)$ be a filtered probability space. Let $X=(X_n)_{n\geq 0}$ be a submartingale such that $\sup_{n\geq 0}\E[\vert X_n\vert]<\infty$. Then the sequence $(X_n)_{n\geq 0}$ converges a.s. to a r.v. $X_\infty\in L^1(\Omega,\F_\infty,(\F_n)_{n\geq 0},\p)$ as $n\to\infty$, where $\F_\infty=\sigma\left(\bigcup_{n\geq 0}\F_n\right)$.

\end{thm}

\begin{proof}

Let $a<b\in \Q$. By Doob's upcrossing inequality, we get that 

\[
\E[N_n([a,b],X)]\leq  (b-a)^{-1}\E[(X_n-a)^-]\leq  (b-a)^{-1}\E[\vert X_n\vert +a].
\]
By monotone convergence, since $N_n([a,b],X)\uparrow N([a,b],X)$ as $n\to\infty$, we get that 

\[
\E[N([a,b],X)]\leq  (b-a)^{-1}\left(\sup_n\E[\vert X_n\vert]+a\right)<\infty,
\]

by the assumption on $X$ being bounded in $L^1(\Omega,\F,(\F_n)_{n\geq 0},\p)$. Therefore, we get that $N([a,b],X)<\infty$ a.s. for every $a<b\in \Q$. Hence, 

\[
\p\left[\bigcap_{a<b\in \Q}\{ N([a,b),X)<\infty\}\right]=1.
\]

Writing $\Omega_0=\bigcap_{a<b\in\Q}\{ N([a,b),X)<\infty\}$, we have that $\p[\Omega_0]=1$ and by lemma \ref{lem4} on $\Omega_0$ we have that $X$ converges to a possible infinite limit $X_\infty$. So we can define 

\[
X_\infty=\begin{cases}\lim_{n\to\infty}X_n,& \text{ on $\Omega_0$,}\\ 0,&\text{ on $\Omega\setminus\Omega_0$,}\end{cases}
\]

Then $X_\infty$ is $\F_\infty$-measurable and by Fatou and the assumption on $X$ being in $L^1(\Omega,\F,(\F_n)_{n\geq 0},\p)$ we get 

\[
\E[\vert X_\infty\vert]=\E\left[\liminf_{n\to\infty} \vert X_n\vert\right]\leq  \liminf_{n\to\infty}\E[\vert X_n\vert]<\infty.
\]

Hence $X_\infty\in L^1(\Omega,\F_\infty,(\F_n)_{n\geq 0},\p)$.

\end{proof}

\begin{cor}

Let $(\Omega,\F,(\F_n)_{n\geq0},\p)$ be a filtered probability space. Let $(X_n)_{n\geq 0}$ be a nonnegative  supermartingale. Then $(X_n)_{n\geq 0}$ converges a.s. to a limit $X_\infty\in L^1(\Omega,\F_\infty,(\F_n)_{\infty},\p)$ and which satisfies 

\[
X_n\geq \E[X_\infty\mid \F_n]\hspace{0.5cm}a.s.
\]

\end{cor}

\begin{proof}

Note that $(-X_n)_{n\geq 0}$ is a submartingale, thus $(-X_n)^+=0$ for all $n\geq 0$, which implies that $\sup_{n\geq 0}\E[-X_n^+]=0<\infty$. Hence 

\[
X_n\xrightarrow{n\to\infty\atop \text{a.s. and $L^1$}}X_\infty\in L^1(\Omega,\F_\infty,(\F_n)_{n\geq 0},\p).
\]

Moreover, for all $m\geq n$ we have 

\[
X_n\geq \E[X_m\mid\F_n].
\]

By Fatou we get

\[
X_n\geq \liminf_{m\to\infty}\E[X_m\mid \F_n]\geq \E\left[\liminf_{m\to\infty}X_m\mid \F_n\right]=\E[X_\infty\mid \F_n]. 
\]

\end{proof}

\begin{ex}[Simple random walk on $\mathbb{Z}$]

Let $Y_n=1+Z_1+\dotsm +Z_n$, for $Z_j$ iid with $\p\left[Z_j=\pm1\right]=\frac{1}{2}$, $Y_0=1$, $\F_0=\{\varnothing,\Omega\}$ and $\F_n=\sigma(Z_1,...,Z_n)$. Then we have already seen that that $(Y_n)_{n\geq 0}$ is a martingale. Let $T=\inf\{n\geq 0\mid Y_n=0\}$. We need to show that $T<\infty$ a.s. Let $X_n=Y_{n\land T}$. Then $(X_n)_{n\geq 0}$ is also a martingale. Moreover, $X_n\geq 0$ for all $n\geq 0$ and $(X_n)_{n\geq 0}$ converges a.s. to a r.v. $X_\infty\in L^1(\Omega,\F_\infty,(\F_n)_{n\geq 0},\p)$. Since $X_n=Y_{n\land T}$, we get $X_\infty=Y_T$. The convergence of $(X_n)_{n\geq 0}$ implies that $T<\infty$ a.s., indeed on the set $\{T=\infty\}$ we get $\vert X_{n+1}-X_n\vert=1$. Consequently, on $\{T=\infty\}$, we get that $(X_n)_{n\geq 0}$ is not a Cauchy sequence and therefore cannot converge. This implies that $\p[T=\infty]=0$ and thus $T<\infty$. Hence 

\[
\lim_{n\to\infty}X_n=Y_T=0.
\]

We also note that $\E[X_n]=1>\E[X_\infty]=0$ for all $n\geq 0$ and so $X_n$ does not converge to $X_\infty$ in $L^1$.

\end{ex}

\section{$L^p$-convergence for Martingales}

\begin{thm}
\label{thm6}
Let $(\Omega,\F,(\F_n)_{n\geq0},\p)$ be a filtered probability space. Let $(X_n)_{n\geq 0}$ be a martingale. Then $(X_n)_{n\geq0}$ converges a.s. to a r.v. $X_\infty\in L^1(\Omega,\F_\infty,(\F_n)_{n\geq0},\p)$ if and only if there exists a r.v. $Z\in L^1(\Omega,\F,(\F_n)_{n\geq 0},\p)$ such that $X_n=\E[Z\mid \F_n]$ for all $n\geq0$, where $\F_\infty=\sigma\left(\bigcup_{n\geq 0}\F_n\right)$.

\end{thm}

\begin{rem}

We shall see that one can always represent $X_n$ as 

\[
X_n=\E[X_\infty\mid \F_n].
\]
\end{rem}

\begin{proof}[Proof of Theorem \ref{thm6}]

For the left right implication, we note that for all $m\geq n$ and for all $A\in\F_n$ we have

\[
\E[X_n\one_A]=\E[X_m\one_A].
\]

Therefore we se that $X_m\xrightarrow{m\to\infty\atop L^1}X_\infty$, implies that 

\[
\lim_{m\to\infty}\E[X_m\one_A]=\E[X_\infty\one_A]
\]

and therefore 

\[
X_n=\E[X_\infty\mid\F_n].
\]

For the left implication, we see that if $X_n=\E[Z\mid \F_n]$, then 

\[
\vert X_n\vert\leq \E[\vert Z\vert\mid \F_n]
\]

and thus

\[
\E[\vert X_n\vert]\leq  \E[\vert Z\vert].
\]

This implies that 

\[
\sup_{n\geq 0}\E[\vert X_n\vert]\leq  \E[\vert Z\vert]<\infty.
\]

Hence we know now that $X_n\xrightarrow{n\to\infty_ a.s.}X_\infty$. It remains to show that 

\[
X_n\xrightarrow{n\to\infty\atop L^1}X_\infty.
\]

First, we assume that $Z$ is bounded, i.e. for all $\omega\in\Omega$, 

\[
\vert Z(\omega)\vert\leq  M\in\R_+.
\]

Hence $\vert X_n(\omega)\vert\leq  M$ and thus $L^1$-convergence follows from dominated convergence. For the general case, let $\varepsilon>0$ and $M\in\R_+$, such that 

\[
\E[\vert Z-Z\one_{\{ \vert Z\vert \leq  M\}}\vert]<\varepsilon.
\]

Thus, for all $n\geq 0$

\[
\E[\vert X_n-\E[Z\one_{\{ \vert Z\vert \leq  M\}}\mid \F_n]\vert]\leq  \E[\E[\vert Z\vert \one_{\{\vert Z\vert >M\}}\mid \F_n]]=\E[\vert Z\vert \one_{\{ \vert Z\vert >M\}}]<\varepsilon.
\]

Moreover, from the bounded case, it follows that 

\[
\left(\E[Z\one_{\{\vert Z\vert \leq  M\}}\mid \F_n]\right)_{n\geq 0}
\]

converges in $L^1$. Hence, there exists $n_0\in\N$ such that for all $m,n\geq n_0$, we have

\[
\E[\vert \E[Z\one_{\|\vert Z\vert \leq  M\}}\mid\F_m]-\E[Z\one_{\{\vert Z\vert \leq  M\}}\mid \F_n]\vert]<\varepsilon.
\]

Now a simple application of the triangular inequality and the above estimates gives, for all $m,n\geq n_0$

\[
\E[\vert X_m-X_n\vert]\leq  3\varepsilon. 
\]

Therefore $(X_n)_{n\geq 0}$ is a Cauchy sequence in $L^1$ and hence it converges in $L^1$.
\end{proof}

\begin{cor}
\label{cor1}
Let $(\Omega,\F,(\F_n)_{n\geq0},\p)$ be a filtered probability space. Let $Z\in L^1(\Omega,\F,(\F_n)_{n\geq 0},\p)$. The unique martingale $X_n=\E[Z\mid \F_n]$ converges a.s. and in $L^1$ to $X_\infty=\E[Z\mid \F_\infty]$, where $\F_\infty=\sigma\left(\bigcup_{n\geq 0}\F_n\right)$.

\end{cor}

\begin{proof}

First, we note that $X_\infty$ is $\F_\infty$-measurable. Now choose $A\in\F_n$. Then 

\[
\lim_{n\to\infty}\E[Z\one_A]=\lim_{n\to\infty}\E[X_n\one_A]=\E[X_\infty\one_A],
\]

and hence for all $A\in \bigcup_{n\geq 0}\F_n$, 

\[
\E[Z\one_A]=\E[X_\infty\one_A].
\]

The monotone class theorem implies that for every $A\in\F_\infty$, 

\[
\E[Z\one_A]=\E[X_\infty\one_A], 
\]

which implies that 

\[
X_\infty=\E[Z\mid \F_\infty].
\]

\end{proof}

\begin{exer}

Prove Kolmogorov's 0-1 law with corollary \ref{cor1}.

\end{exer}

\begin{thm}[$L^p$ martingale convergence theorem]
\label{thm3}
Let $(\Omega,\F,(\F_n)_{n\geq0},\p)$ be a filtered probability space. Let $(X_n)_{n\geq 0}$ be a martingale. Assume that there exists $p\geq 1$ such that 

\[
\sup_{n\geq 0}\E[\vert X_n\vert^p]<\infty.
\]

Then 

\[
X_n\xrightarrow{n\to\infty\atop \text{$a.s.$ and $L^p$}}X_\infty.
\]

Moreover, we have 

\[
\E[\vert X_\infty\vert^p]=\sup_{n\geq 0}\E[\vert X_n\vert^p]
\]

and 

\[
\E[(X_\infty^*)^p]\leq \left( \frac{p}{p-1}\right)^p\E[\vert X_\infty\vert^p],
\]

where 

\[
X_\infty^*=\sup_{n\geq 0}\vert X_n\vert.
\]

\end{thm}

\begin{rem}

Let us summarize what we have seen so far.

\begin{itemize}
\item{If $(X_n)_{n\geq 0}$ is bounded in $L^1$, we get $X_n\xrightarrow{n\to\infty\atop a.s.}X_\infty\in L^1$.}
\item{$X_n\xrightarrow{n\to\infty\atop \text{$a.s.$ and $L^1$}}X_\infty$ if and only if $X_n=\E[X_\infty\mid \F_n]$.}
\item{If $(X_n)_{n\geq 0}$ is bounded in $L^p$ with $p\geq 1$, then $X_n\xrightarrow{n\to\infty\atop\text{$a.s.$ and $L^p$}}X_\infty$.}
\end{itemize}
\end{rem}

\begin{proof}[Proof of Theorem \ref{thm3}]

We first note that since $\sup_{n\geq 0}\E[\vert X_n\vert^p]<\infty$, we also have that 

\[
\sup_{n\geq 0}\E[\vert X_n\vert]<\infty.
\]

Thus $X_n\xrightarrow{n\to\infty\atop a.s.}X_\infty$. From Doob's inequality, we get 

\[
\E[(X_\infty^*)^p]\leq  \left(\frac{p}{p-1}\right)^p\sup_{n\geq 0}\E[\vert X_n\vert^p]<\infty
\]

and therefore $X_\infty^*\in L^p(\Omega,\F,(\F_n)_{n\geq 0},\p)$. Moreover, for all $n\geq 0$ we get $\vert X_n\vert\leq  X_\infty^*$ and 

\[
\vert X_n-X_\infty\vert^p\leq  2^p(X_\infty^*)p.
\]

Using dominated convergence, we get

\[
\E[\vert X_n-X_\infty\vert^p]\xrightarrow{n\to\infty}0
\]

and thus

\[
X_n\xrightarrow{n\to\infty\atop L^p}X_\infty.
\]

Finally, we note that $(\vert X_n\vert^p)_{n\geq 0}$ is a positive submartingale. Hence we know that 

\[
(\E[\vert X_n\vert^p])_{n\geq 0}
\]

is increasing, which implies that 

\[
\E[\vert X_\infty\vert^p]=\lim_{n\to\infty}\E[\vert X_n\vert^p]=\sup_{n\geq 0}\E[\vert X_n\vert^p].
\]

\end{proof}

\section{Uniform integrability}

\begin{defn}[Uniformly integrable]

Let $(\Omega,\F,(\F_n)_{n\geq0},\p)$ be a filtered probability space. A family $(X_i)_{i\in I}$ of r.v.'s in $L^1(\Omega,\F,(\F_n)_{n\geq 0},\p)$, indexed by an arbitrary index set $I$, is called uniformly integrable (denoted by u.i.) if 

\[
\lim_{a\to\infty}\sup_{i\in I}\E[\vert X_i\vert \one_{\{\vert X_i\vert >a\}}]=0.
\]

\end{defn}

\begin{rem}

A singel r.v. in $L^1(\Omega,\F,(\F_n)_{n\geq0},\p)$ is always u.i. (this follows from dominated convergence). If $\vert I\vert<\infty$, then using 

\[
\vert X_i\vert \leq  \sum_{j\in I}\vert X_j\vert,
\]

we get 

\[
\E[\vert X_i\vert \one_{\{\vert X_i\vert >a\}}]\leq  \sum_{j\in I}\E[\vert X_j\vert\one_{\{ X_j\vert>a\}}]\xrightarrow{a\to\infty}0.
\]

Let $(X_i)_{i\in I}$ be u.i. For $a$ large enough, we then have 

\[
\sup_{i\in I}\E[\vert X_i\vert \one_{\{\vert X_i\vert >a\}}]\leq  1.
\]

Hence 

\[
\sup_{i\in I}\E[\vert X_i\vert]=\sup_{i\in I}\E[\vert X_i\vert (\one_{\{\vert X_i\vert\leq  a\}}+\one_{\{ \vert X_i\vert >a\}})]\leq  1+a<\infty,
\]

which implies that $(X_i)_{i\in I}$ is bounded in $L^1(\Omega,\F,(\F_n)_{n\geq 0},\p)$.
\end{rem}

\begin{ex}

Let $Z\in L^1(\Omega,\F,(\F_n)_{n\geq0},\p)$. Then the family 

\[
\Theta=\{X\in L^1(\Omega,\F,(\F_n)_{n\geq 0},\p)\mid \vert X\vert\leq  Z\}
\]

is u.i. Indeed, we have 

\[
\sup_{X\in\Theta}\E[\vert X\vert \one_{\{\vert X\vert>a\}}]\leq  \E[Z\one_{\{Z>a\}}]\xrightarrow{a\to\infty}0.
\]

\end{ex}
\begin{ex}

Let $\phi:\R_+\to\R_+$ be a measurable map, such that 

\[
\frac{\phi(x)}{x}\xrightarrow{x\to\infty}\infty.
\]

Then for all $C>0$, the family 

\[
\Theta_C=\{X\in L^1(\Omega,\F,(\F_n)_{n\geq0},\p)\mid \E[\phi(\vert X\vert)]\leq  C\}
\]

is u.i. Indeed, for $a$ large enough, we have

\[
\E[\vert X\vert \one_{\{\vert X\vert >a\}}]=\E\left[\frac{\vert X\vert}{\phi(\vert X\vert)}\phi(\vert X\vert)\one_{\{\vert X\vert>a\}}\right]\leq \sup_{x>a}\left(\frac{x}{\phi(x)}\right)\E[\phi(\vert X\vert)]\leq  C\sup_{x>a}\left(\frac{x}{\phi(x)}\right)\xrightarrow{a\to\infty}0.
\]

Thus 

\[
\sup_{X\in\Theta}\E[\vert X\vert\one_{\{\vert X\vert>a\}}]\leq  C\sup_{x>a}\left(\frac{x}{\phi(x)}\right)
\]

\end{ex}

\begin{prop}

Let $(\Omega,\F,(\F_n)_{n\geq0},\p)$ be a filtered probability space. Let $(X_i)_{i\in I}$ be a family of r.v.'s bounded in $L^1(\Omega,\F,(\F_n)_{n\geq 0},\p)$, i.e. $\sup_{i\in I}\E[\vert X_i\vert]<\infty$. Then $(X_i)_{i\in I}$ is u.i. if and only if for all $\varepsilon>0$ there is a $\delta>0$ such that for all $\A\in \F$, if $\p[A]<\delta$ then $\sup_{i\in I}\E[\vert X_i\vert \one_A]<\varepsilon$.

\end{prop}
\begin{proof}

For the right implication, let $\varepsilon>0$. Then there exists $a>0$ such that 

\[
\sup_{i\in I}\E[\vert X_i\vert \one_{\{\vert X_i\vert >a\}}]<\frac{\varepsilon}{2}.
\]

Now let $\delta=\frac{\varepsilon}{2a}$ and $A\in\F$ such that $\p[A]<\delta$. Then 

\[
\E[\vert X_i\vert \one_A]=\E[\vert X_i\vert \one_A\one_{\{\vert X_i\vert>a\}}]+\E[\vert X_i\vert \one_A\one_{\vert X_i\vert \leq  a\}}]\leq  \E[\vert X_i\vert \one_{\{\vert X_i\vert>a\}}]+a\underbrace{\E[\one_A]}_{\p[A]}<\frac{\varepsilon}{2}+a\delta<\frac{\varepsilon}{2}+a\frac{\varepsilon}{2a}=\varepsilon.
\]

For the left implication, let $C=\sup_{i\in I}\E[\vert X_i\vert]<\infty$. From Markov's inequality,we get 

\[
\p[\vert X_i\vert>a]\leq  \frac{\E[\vert X_i\vert]}{a}\leq  \frac{C}{a}.
\]

Now let $\delta >0$ such that $(ii)$ holds. If $\frac{C}{a}<\delta$, then for all $i\in I$

\[
\E[\vert X_i\vert \one_{\{\vert X_i\vert >a\}}]<\varepsilon.
\]

\end{proof}

\begin{cor}
Let $(\Omega,\F,(\F_n)_{n\geq0},\p)$ be a filtered probability space. Let $X$ be a bounded r.v., i.e.  $X\in L^1(\Omega,\F,(\F_n)_{n\geq 0},\p)$. Then the family 

\[
\{\E[X\mid\mathcal{G}]\mid \mathcal{G}\subset \F,\text{ $\mathcal{G}$ is a $\sigma$-Algebra}\}
\]

is u.i.

\end{cor}

\begin{proof}

Let $\varepsilon>0$. Then there exists a $\delta >0$ such that for all $A\in \F$, if $\p[A]<\delta$ then $\E[\vert X\vert \one_A]<\varepsilon$. Then for all $a>0$, we have 

\[
\p[\E[X\mid \mathcal{G}]>a]\leq  \frac{\E[\vert \E[X\mid \mathcal{G}]\vert]}{a}\leq  \frac{\E[\vert X\vert]}{a}.
\]

For $a$ large enough, i.e. $\frac{\E[\vert X\vert]}{a}<\delta$, we have 

\[
\E[\vert \E[X\mid \mathcal{G}]\vert \one_{\{ \vert \E[X\mid \mathcal{G}]\vert >a\}}]\leq  \E[\E[\vert X\vert \one_{\vert \E[X\mid \mathcal{G}]\vert >a\}}\mid \mathcal{G}]]\leq  \E[\vert X\vert \one_{\{\E[X\mid \mathcal{G}]\vert >a\}}]<\varepsilon.
\]

\end{proof}

\begin{thm}

Let $(\Omega,\F,(\F_n)_{n\geq0},\p)$ be a filtered probability space. Let $(X_n)_{n\geq 0}$ be a sequence of r.v.'s in $L^1(\Omega,\F,(\F_n)_{n\geq 0},\p)$, which converges in probability to $X_\infty$. Then $X_n\xrightarrow{n\to\infty\atop L^1}X_\infty$ if and only if $(X_n)_{n\geq 0}$ is u.i.

\end{thm}

\begin{proof}

For the right implication, we first note that $(X_n)_{n\geq 0}$ is bounded in $L^1(\Omega,\F,(\F_n)_{n\geq 0},\p)$ since it converges in $L^1$. For $\varepsilon>0$, there exists $N\in\N$ such that for $n\geq N$ we get 

\[
\E[\vert X_N-X_n\vert]<\frac{\varepsilon}{2}.
\]

Next we note that $\{X_0,X_1,...,X_N\}$ is u.i. since it is a finite family of bounded r.v.'s. Therefore, there exists a $\delta>0$, such that for all $A\in\F$, if $\p[A]<\delta$ then $\E[\vert X_n\vert\one_A]<\frac{\varepsilon}{2}$ for all $n\in\{0,1,...,N\}$. Finally for $n\geq N$, we get 

\[
\E[\vert X_n\vert \one_A]\leq  \E[\vert X_N-X_n\vert \one_A]+\E[\vert X_N\vert \one_A]<\frac{\varepsilon}{2}+\frac{\varepsilon}{2}=\varepsilon.
\]

Thus $(X_n)_{n\geq 0}$ is u.i.

\vspace{0.5cm}

For the left implication, we note that if $(X_n)_{n\geq 0}$ is u.i., then the family $(X_n-X_m)_{(n,m)\in\N^2}$ is also u.i. since 

\[
\E[\vert X_n-X_m\vert \one_A]\leq  \E[\vert X_n\vert \one_A]+\E[X_m\vert \one_A]
\]

for all $A\in\F$. Now for $\varepsilon>0$, there exists $a$ sufficiently large, such that 

\[
\E[\vert X_n-X_m\vert \one_{\{\vert X_n-X_m\vert >a\}}]<\varepsilon.
\]

Moreover, we note that 

\[
\E[\vert X_m-X_m\vert]\leq  \E[\vert X_n-X_m\vert \one_{\{ \vert X_n-X_m\vert<\varepsilon\}}]+\E[\vert X_n-X_m\vert \one_{\{\vert X_n-X_m\vert\geq \varepsilon\}}]
\]
\[
\leq  \varepsilon +\E[\vert X_n-X_m\vert\one_{\{\vert X_n-X_m\vert \geq\varepsilon\}}\one_{\{\vert X_n-X_m\vert \leq  a\}}]+\E[\vert X_n-X_m\vert\one_{\{\vert X_n-X_m\vert\geq \varepsilon\}}\one_{\{\vert X_n-X_m\vert >a\}}]
\]
\[
\leq  \varepsilon+a\p[\vert X_n-X_m\vert\geq\varepsilon]+\E[\vert X_n-X_m\vert \one_{\{\vert X_n-X_m\vert>a\}}].
\]

Then, using that $\lim_{n\to\infty}\p[\vert X_n-X_m\vert\geq\varepsilon]=0$, we can show that the right hand side converges to zero for $a$ large enough, which implies that $(X_n)_{n\geq 0}$ is a Cauchy sequence $L^1(\Omega,\F,(\F_n)_{n\geq 0},\p)$ and hence converges in $L^1$.

\end{proof}

Combining all our previous results, we have; if $(X_n)_{n\geq 0}$ is a martingale, then the following are equivalent

\begin{enumerate}[$(i)$]
\item{$(X_n)_{n\geq 0}$ converges a.s. and in $L^1$.}
\item{$(X_n)_{n\geq 0}$ is u.i.}
\item{$(X_n)_{n\geq 0}$ is regular and $X_n=\E[X_\infty\mid \F_n]$ a.s.}
\end{enumerate}

\section{Stopping theorems}

If $S\leq  T$ are two bounded stopping times and $(X_n)_{n\geq 0}$ a martingale, then 

\[
\E[X_T\mid \F_S]=X_S\hspace{0.5cm}a.s.
\]

If $(X_n)_{n\geq 0}$ is an adapted process, which converges a.s. to $X_\infty$, we can define $X_T$ for all stopping times (finite or not) by 

\[
X_T=\sum_{n=0}^\infty X_n\one_{\{T=n\}}+X_\infty\one_{\{ T=\infty\}}.
\]

\begin{thm}

Let $(\Omega,\F,(\F_n)_{n\geq0},\p)$ be a filtered probability space. Let $(X_n)_{n\geq 0}$ be u.i. martingale. Then for any stopping time $T$, we have that 

\[
\E[X_\infty\mid \F_T]=X_T\hspace{0.5cm}a.s.
\]

In particular 

\[
\E[X_T]=\E[X_\infty]=\E[X_n]
\]

for all $n\geq 0$. If $S$ and $T$ are two stopping times, such that $S\leq  T$, then 

\[
\E[X_T\mid \F_S]=X_S\hspace{0.5cm}a.s.
\]
\end{thm}

\begin{proof}

We first want to check that $X_T$ is in $L^1(\Omega,\F,(\F_n)_{n\geq 0},\p)$. Therefore we have

\begin{align*}
\E[\vert X_T\vert]&=\sum_{n=0}^\infty \E[\vert X_n\vert\one_{\{ T=n\}}]+\E[\vert X_\infty\vert\one_{\{ T=\infty\}}]\leq  \sum_{n=0}^\infty \E[\E[\vert X_\infty\vert\mid \F_n]\one_{\{T=n\}}]+\E[\vert X_\infty\vert \one_{\{ T=\infty\}}]\\
&=\sum_{n=0}^\infty\E[\E[\vert X_\infty\vert \one_{\{ T=n\}}\mid \F_n]]+\E[\vert X_\infty\vert\one_{\{ T=\infty\}}]=\sum_{n=0}^\infty \E[\vert X_\infty\vert \one_{\{ T=n\}}]+\E[\vert X_\infty\vert \one_{\{T=\infty\}}]\\
&=\E[\vert X_\infty\vert]<\infty
\end{align*}

Now let $A\in\F_T$. Then

\begin{align*}
\E[X_T\one_A]&=\sum_{n\in\N\cup\{\infty\}}\E[X_T\one_{A\cap\{T=n\}}]=\sum_{n\in\N\cup\{\infty\}}\E[X_n\one_{A\cap\{T=n\}}]=\sum_{n\in\N\cup\{\infty\}}\E[\E[X_\infty\mid \F_n]\one_{A\cap\{ T=n\}}]\\
&=\sum_{n\in\N\cup\{\infty\}}\E[X_\infty\one_{A\cap\{ T=n\}}]=\E[X_\infty\one_A]
\end{align*}

where we have used that $X_\infty\in L^1(\Omega,\F_\infty,(\F_n)_{n\geq 0},\p)$ and Fubini for the first and last equation. Now since $X_T$ is $\F_T$-measurable, we get that $\E[X_\infty\mid \F_T]=X_T$ a.s. Finally for $S\leq  T$, we have $\F_S\subset\F_T$ and thus 

\[
X_S=\E[X_\infty\mid \F_S]=\E[\E[X_\infty\mid \F_T]\mid \F_S]=\E[X_T\mid \F_S].
\]

\end{proof}

\begin{rem}

If $(X_n)_{n\geq 0}$ is a u.i. martingale, then the family 

\[
\{X_T\mid \text{$T$ a stopping time}\}
\]

is u.i. Indeed, we note that 

\[
\{X_T\mid \text{$T$ a stopping time}\}=\{\E[X_\infty\mid \F_T]\mid \text{$T$ a stopping time}\}\subset \E[X_\infty\mid \mathcal{G}]\mid \text{$\mathcal{G}$ a $\sigma$-Algebra, $\mathcal{G}\subset\F$}\},
\]

where the superset is u.i. If $N\in \N$, then $(X_{n\land N})_{n\geq 0}$ is u.i. Indeed, if $Y_n=X_{n\land N}$, then $(Y_n)_{n\geq 0}$ is a martingale and 

\[
Y_n=\E[X_N\mid \F_n]=\E[Y_\infty\mid \F_n].
\]

\end{rem}

\begin{ex}[Another random walk]

Consider a simple random walk with $X_0=k\geq 0$. Let $m\geq 0$, $0\leq  k\leq  m$ and 

\[
T=\inf\{n\geq 1\mid \text{$X_n=0$ or $X_n=m$}\}
\]

with $X_n=k+Y_1+\dotsm +Y_n$, where $(Y_n)_{n\geq 1}$ are iid and $\p[Y_n=\pm 1]=\frac{1}{2}$. We have already seen that $T<\infty$ a.s. Now let $Z_n=X_{n\land T}$. Then $(Z_n)_{n\geq 0}$ is a martingale and $0\leq  Z_n\leq  m$. Therefore, $(Z_n)_{n\geq 0}$ is u.i. and hence $Z_n$ converges a.s. and in $L^1$ to $Z_\infty$ with 

\[
\E[Z_\infty]=\E[Z_0]=k.
\]

Moreover, 

\[
\E[Z_\infty]=\E[X_T]=\E[m\one_{\{ X_T=m\}}]+\E[0\one_{\{X_T=0\}}]=m\p[X_T=m],
\]

which implies that $\p[X_T=m]=\frac{k}{m}$ and thus $\p[X_T=0]=1-\p[X_T=m]=\frac{m-k}{m}$. 

\vspace{0.5cm}

Now let us assume that $\p[Y_n=1]=p$ and $\p[Y_n=-1]=1-p=q$ for $p\in (0,1)$ and $p\not=\frac{1}{2}$. Let us consider 

\[
Z_n=\left(\frac{q}{p}\right)^{X_n}.
\]

Then $(Z_n)_{n\geq 0}$ is a martingale. Indeed, by definition $Z_n$ is adapted and $Z_n\in L^1(\Omega,\F,(\F_n)_{n\geq 0},\p)$ because 

\[
k-n\leq  X_n\leq  k+n
\]

and thus $Z_n$ is bounded. Moreover, 

\[
\E[Z_{n+1}\mid \F_n]=\E\left[\left(\frac{q}{p}\right)^{X_n}\left(\frac{q}{p}\right)^{Y_n}\mid \F_n\right]=Z_n\E\left[\left(\frac{q}{p}\right)^{Y_n}\mid \F_n\right]=Z_n\E\left[\left(\frac{q}{p}\right)^{Y_n}\right]=Z_n\{q+p\}=Z_n,
\]
and therefore $(Z_n)_{n\geq 0}$ is a martingale. Now $(Z_{n\land T})_{n\geq 0}$ is bounded and hence u.i., which implies that $(Z_{n\land T})_{n\geq 0}$ converges a.s. and in $L^1$. We also have that 

\[
\E[Z_T]=\E[Z_0]=\left(\frac{q}{p}\right)^k.
\]

On the other hand we have 

\[
\E[Z_T]=\left(\frac{q}{p}\right)^m\p[X_T=m]+(1-\p[X_T=m]).
\]

Hence we get 

\[
\p[X_T=m]=\frac{\left(\frac{q}{p}\right)^k-1}{\left(\frac{q}{p}\right)^m-1}.
\]
\end{ex}

\begin{thm}

Let $(\Omega,\F,(\F_n)_{n\geq0},\p)$ be a filtered probability space. Let $(X_n)_{n\geq 0}$ be a supermartingale. Assume that one of the following two conditions is satisfied

\begin{enumerate}[$(i)$]
\item{$X_n\geq 0$ for all $n\geq 0$.}
\item{$(X_n)_{n\geq 0}$ is u.i.}
\end{enumerate}

Then for every stopping time $T$ (finite or not) we get that $X_T\in L^1(\Omega,\F,(\F_n)_{n\geq 0},\p)$. Moreover, if $S$ and $T$ are two stopping time, such that $S\leq  T$, then in case

\begin{enumerate}[$(i)$]
\item{$\one_{\{S<\infty}X_S\geq \E[X_T\one_{\{ T<\infty\}}\mid \F_S]$ a.s.}
\item{$X_S\geq \E[X_T\mid \F_S]$ a.s.}
\end{enumerate}

\end{thm}

\begin{proof}

We first deal with the case $(i)$. We have already seen that if $T$ is a bounded stopping time, we have 

\[
\E[X_T]\leq  \E[X_0].
\]

With Fatou we get 

\[
\E\left[\liminf_{n\to\infty}X_{n\land T}\right]\leq \liminf_{n\to\infty}\E[X_{n\land T}]\leq  \E[X_0],
\]

which implies that $X_T\in L^1(\Omega,\F,(\F_n)_{n\geq 0},\p)$. Now let $S\leq  T$ be two stopping times. First assume that $S\leq  T\leq  N\in\N$. Then we know that 

\[
\E[X_S]\geq \E[X_T].
\]

Now let $A\in\F_S$ and consider the stopping times

\begin{align*}
S^A(\omega)&=\begin{cases}S(\omega)&\omega\in A\\ 0&\omega\not\in A\end{cases}\\
T^A(\omega)&=\begin{cases}T(\omega)&\omega\in A\\ 0&\omega\not\in A\end{cases}
\end{align*}

Then we get that 

\[
S^A\leq  T^A\leq  N,
\]

thus $\E[X_{S^A}]\geq \E[X_{T^A}]$ and therefore $\E[X_S\one_A]\geq \E[X_T\one_A]$ for all $A\in \F_S$.

Let us now go back to the general case $S\leq  T$ and let $B\in \F_S$. Let us now apply the above to $S\land k$, $T\land k$ and $A=B\cap\{S\leq  k\}\in\F_S$. Then we get 

\[
\E[X_{S\land k}\one_{B\cap \{ S\leq  k\}}]\geq \E[X_{T\land k}\one_{B\cap\{S\leq  k\}}]\geq \E[X_{T\land k}\one_{B\cap\{T\leq  k\}}].
\]

Hence we get 

\[
\E[X_{S\land k}\one_{B\cap\{S\leq  k\}}]\geq \E[X_{T\land k}\one_{B\cap \{ T\leq  k\}}]
\]

and thus

\[
\E[X_S\one_B\one_{\{S\leq  k\}}]\geq \E[X_T\one_B\one_{\{ T\leq  k\}}].
\]

By dominated convergence, we obtain 

\[
\E[X_S\one_B\one_{\{ S<\infty\}}]\geq \E[X_T\one_B\one_{\{T<\infty\}}].
\]

Now let $\tilde X_S=\one_{\{S<\infty\}}X_S$ and $\tilde X_T=\one_{\{ T<\infty\}}X_T$. Then for any $B\in \F_S$, we get 

\[
\E[\tilde X_S\one_B]\geq \E[\tilde X_T\one_B]=\E[\one_B\E[\tilde X_T\mid \F_S]].
\]

Since the last equality is true for all $B\in \F_S$, we can conclude that 

\[
\tilde X_S\geq \E[\tilde X_T\mid \F_S].
\]

Now let us prove $(ii)$. We know from previous results that in this case $X_n\xrightarrow{n\to\infty\atop \text{a.s. and $L^1$}}X_\infty$. We have $X_n\geq\E[X_m\mid \F_n]$ for all $m\geq n$. The $L^1$-convergence as $n\to\infty$ gives

\[
X_n\geq \E[X_\infty\mid \F_n].
\]

Moreover, the martingale $Z_n:=\E[X_\infty\mid \F_n]$ converges a.s. to $X_\infty$. Set $Y_n:=X_n-Z_n$. Then $(Y_n)_{n\geq 0}$ is a positive supermartingale and hence it converges a.s. to $X_\infty-Z_\infty=0$. We now apply case $(i)$ to deduce that $X_T=Y_T+Z_T\in L^1(\Omega,\F,(\F_n)_{n\geq 0},\p)$ and $Y_S\geq \E[Y_T\mid \F_S]$\footnote{We use the fact that $Y_S\one_{\{S=\infty\}}=0$ and $Y_T\one_{\{T=\infty\}}=0$, since $Y_\infty=0$.}. The stopping theorem applied to $(Z_n)_{n\geq 0}$ gives

\[
Z_S=\E[Z_T\mid \F_S]. 
\]

This implies that 

\[
Y_S+Z_S\geq \E[Z_T+Y_T\mid \F_S]
\]

and thus 

\[
X_S\geq \E[X_T\mid \F_S].
\]

\end{proof}

\section{Applications of Martingale limit theorems}

\subsection{Backward Martingales and the law of large numbers}

\begin{defn}[Backward Filtration]

Let $(\Omega,\F,\p)$ be a probability space. A backward filtration is a family $(\F_n)_{n\in-\N}$ of $\sigma$-Algebras indexed by the negative integers, which we will denote by $(\F_n)_{n\leq  0}$, such that for all $n\leq  m\leq  0$ we have 

\[
\F_n\subset\F_m.
\]

\end{defn}

\begin{rem}

We will write 

\[\F_{-\infty}:=\bigcap_{n\leq  0}\F_n.
\]

It is clear that $\F_{-\infty}$ is also a $\sigma$-Algebra included in $\F$. A stochastic process $(X_n)_{n\leq  0}$, indexed by the negative integers, is called a backwards martingale (resp. backwards sub- or supermartingale) if for all $n\leq 0$, $X_n$ is $\F_n$-measurable, $\E[\vert X_n\vert]<\infty$ and for all $n\leq  m$ we have 

\[
\E[X_m\mid \F_n]=X_n\hspace{0.5cm}\text{(resp. $\E[X_m\mid \F_n]< X_n$ or $\E[X_m\mid \F_n]\geq X_n$).}
\]

\end{rem}

\begin{thm}[Backward convergence theorem]

Let $(\Omega,\F,(\F_n)_{n\leq 0},\p)$ be a backward filtered probability space. Let $(X_n)_{n\leq  0}$ be a backward supermartingale. Assume that 

\begin{equation}
\sup_{n\leq  0}\E[\vert X_n\vert]<\infty.
\end{equation}

Then $(X_n)_{n\leq  0}$ is u.i. and converges a.s. and in $L^1$ to $X_{-\infty}$ as $n\to-\infty$. Moreover, for all $n\leq  0$, we have 

\[
\E[X_n\mid \F_{-\infty}]\leq  X_{-\infty}\hspace{0.3cm}a.s.
\]

\end{thm}

\begin{proof}

First we show a.s. convergence. Let therefore $k\geq 1$ be a fixed integer. For $n\in\{1,...,k\}$, let $Y_n^k=X_{n-k}$ and $\mathcal{G}_n^k=\F_{n-k}$. For $n>k$, we take $Y_n^k=X_0$ and $\mathcal{G}_n^k=\F_0$. Then $(Y_n^k)_{n\geq 0}$ is a supermartingale with respect to $(\mathcal{G}_n^k)_{n\geq 0}$. We now apply Doob's upcrossing inequality to the submartingale $(-Y_n^k)_{n\geq 0}$ to obtain that for $a<b$

\[
(b-a)\E[N_k([a,b],-Y^k)]\leq  \E[(-Y_n^k-a)^+]=\E[(-X_0-a)^+]\leq  \vert a\vert +\E[\vert X_0\vert].
\]

We note that when $k\uparrow \infty$, $N_k([a,b],-Y^k)$ increases and 

\begin{multline*}
N([a,b],-X):=\sup\{k\in\N\mid \exists m_1<n_1<\dotsm <m_k<n_k\leq  0;\\
-X_{m_1}\leq  a,-X_{n_1}\geq b,...,-X_{m_k}\leq  a,-X_{n_k}\geq b\}.
\end{multline*}

With monotone convergence we get 

\[
(b-a)\E[N([a,b],-X)]\leq  \vert a\vert +\E[\vert X_0\vert]<\infty.
\]

One can easily show that $(X_n)_{n\leq  0}$ converges a.s. to $X_\infty$ as $n\to-\infty$ and Fatou implies then that $\E[\vert X_{-\infty}\vert]<\infty$. We want to show that $(X_n)_{n\leq  0}$ is u.i. Thus, let $\varepsilon>0$. The sequence $(\E[X_n])_{n\geq 0}$ is increasing and bounded; we can take $k\leq  0$ small enough to get for $n\leq  k$, 

\[
\E[X_n]\leq  \E[X_k]+\frac{\varepsilon}{2}.
\]

Moreover, the finite family $\{X_k,X_{k+1},...,X_{-1},X_0\}$ is u.i. and one can then choose $\alpha>0$ large enough such that for all $k\leq  n\leq  0$

\[
\E[\vert X_n\vert\one_{\{\vert X_n\vert>\alpha}]<\varepsilon.
\]

We can also choose $\delta>0$ sufficiently small such that for all $A\in\F$, $\p[A]<\delta$ implies that $\E[\vert X_n\mid \one_A]<\frac{\varepsilon}{2}$. Now if $n<k$, we get 

\begin{align*}
\E[\vert X_n\vert \one_{\{ \vert X_n\vert >\alpha}]&=\E[-X_n\one_{\{ X_n<-\alpha\}}]+\E[X_n\one_{\{ X_n>\alpha\}}]=-\E[X_n\one_{\{ X_n<-\alpha\}}]+\E[X_n]-\E[X_n\one_{\{X_n\leq  \alpha\}}]\\
&\leq  -\E[\E[X_k\mid \F_n]\one_{\{ X_n<-\alpha\}}]+\E[X_k]+\frac{\varepsilon}{2}-\E[\E[X_k\mid \F_n]\one_{X_\leq  \alpha}]\\
&=-\E[X_k\one{\{ X_n<-\alpha\}}]+\E[X_k]+\frac{\varepsilon}{2}-\E[X_k\one_{\{ X_n\leq  \alpha\}}]\\
&=-\E[X_n\one_{\{ X_n<-\alpha\}}]+\E[X_n\one_{X_n>\alpha\}}]+\frac{\varepsilon}{2}\leq  \E[\vert X_k\vert \one_{\{ \vert X_n\vert>\alpha\}}]+\frac{\varepsilon}{2}.
\end{align*}

Next, we observe that 

\[
\p[\vert X_n\vert>\alpha]\leq \frac{1}{\alpha}\E[X_n]\leq  \frac{C}{\alpha},
\]

where $C=\sup_{n\leq  0}\E[\vert X_n\vert]<\infty$. Choose $\alpha$ such that $\frac{C}{\alpha}<\delta$. Consequently, we get 

\[
\E[\vert X_k\vert\one_{\{\vert X_n\vert>\alpha\}}]<\frac{\varepsilon}{2}.
\]

Hence, for all $n< k$, $\E[\vert X_n\vert \one_{\{ \vert X_n\vert >\alpha\}}]<\varepsilon$. This inequality is also true for $k\leq  n\leq  0$ and thus we have that $(X_n)_{n\leq  0}$ is u.i. To conclude, we note that u.i. and a.s. convergence implies $L^1$ convergence. Then, for $m\leq  n$ and $A\in\F_{-\infty}\subset \F_m$, we have 

\[
\E[X_n\one_A]\leq  \E[X_m\one_A]\xrightarrow{m\to-\infty}\E[X_{-\infty}\one_A].
\]

Therefore, $\E[\E[X_n\mid \F_{-\infty}]\one_A]\leq  \E[X_{-\infty}\one_A]$ and hence 

\[
\E[X_n\mid \F_{-\infty}]\leq  X_{-\infty}.
\]

\end{proof}

\begin{rem}

Equation $(1)$ is always satisfied for backward martingales. Indeed, for all $n\leq  0$ we get 

\[
\E[X_0\mid \F_n]=X_n,
\]

which implies that $\E[\vert X_n\vert]\leq  \E[\vert X_0\vert]$ and thus 

\[
\sup_{n\leq  0}\E[\vert X_n\vert]<\infty.
\]

Backward martingales are therefore always u.i.

\end{rem}

\begin{cor}

Let $(\Omega,\F,\p)$ be a probability space. Let $Z$ be a r.v. in $L^1(\Omega,\F,\p)$ and let $(\mathcal{G}_n)_{n\geq 0}$ be a decreasing family of $\sigma$-Algebras. Then 

\[
\E[Z\mid \mathcal{G}_n]\xrightarrow{n\to\infty\atop\text{a.s. and $L^1$}}\E[Z\mid \mathcal{G}_\infty],
\]

where 

\[
\mathcal{G}_\infty:=\bigcap_{n\geq 0}\mathcal{G}_n.
\]

\end{cor}

\begin{proof}

For $n\geq 0$ define $X_{-n}:=\E[Z\mid \F_n]$, where $\F_{-n}=\mathcal{G}_n$. Then $(X_n)_{n\leq  0}$ is a backward martingale with respect to $(\F_n)_{n\leq  0}$. Hence theorem 14.1. implies that $(X_n)_{n\leq  0}$ converges a.s. and in $L^1$ for $n\to\infty$. Moreover\footnote{This follows from the last part of theorem 14.1.}, 

\[
X_\infty=\E[X_0\mid \F_{-\infty}]=\E[\E[Z\mid \F_0]\mid \F_{-\infty}]=\E[Z\mid \F_{-\infty}]=\E[Z\mid \mathcal{G}_\infty].
\]

\end{proof}

\begin{thm}[Kolmogorov's 0-1 law]

Let $(\Omega,\F,\p)$ be a probability space. Let $(X_n)_{n\geq 1}$ be a sequence of independent r.v.'s with values in arbitrary measure spaces. For $n\geq 1$, define the $\sigma$-Algebra

$$\B_n:=\sigma(\{X_k\mid k\geq n\}).$$

The tail $\sigma$-Algebra $\B_\infty$ is defined as 

$$\B_\infty:=\bigcap_{n=1}^\infty\B_n.$$

Then $\B_\infty$ is trivial in the sense that for all $B\in \B_\infty$ we have $\p[B]\in\{0,1\}$.

\end{thm}

\begin{proof}
This proof can be found in the stochastics I notes.
\end{proof}

\begin{lem}
\label{lem5}
Let $(\Omega,\F,\p)$ be a probability space. Let $Z\in L^1(\Omega,\F,\p)$ and $\mathcal{H}_1$ and $\mathcal{H}_2$ two $\sigma$-Algebras included in $\F$. Assume that $\mathcal{H}_2$ is independent of $\sigma(Z)\lor \mathcal{H}_1$. Then 

\[
\E[Z\mid \mathcal{H}_1\lor \mathcal{H}_2]=\E[Z\mid \mathcal{H}_1].
\]

\end{lem}

\begin{proof}

Let $A\in \mathcal{H}_1\lor \mathcal{H}_2$ such that $A=B\cap C$, where $B\in \mathcal{H}_1$ and $C\in\mathcal{H}_2$. Then 

\begin{multline*}
\E[Z\one_A]=\E[Z\one_B\one_C]=\E[Z\one_B]\E[\one_C]=\E[\one_A]\E[\E[Z\mid \mathcal{H}_1]\one_B]\\=\E[\E[Z\mid \mathcal{H}_1]\one_B\one_C]=\E[\E[Z\mid \mathcal{H}_1]\mid \one_A].
\end{multline*}

Now we note that 

\[
\sigma(W=\{ B\cap C\mid B\in\mathcal{H}_1,C\in \mathcal{H}_2\})=\mathcal{H}_1\lor \mathcal{H}_2
\]

and $W$ is stable under finite intersections. Thus the monotone class theorem implies that for all $A\in \mathcal{H}_1\lor \mathcal{H}_2$ we have

\[
\E[Z\one_A]=\E[\E[Z\mid\mathcal{H}_1]\one_A]
\]

\end{proof}

\begin{thm}[Strong law of large numbers]

Let $(\Omega,\F,\p)$ be a probability space. Let $(\xi_n)_{n\geq 1}$ be a sequence of iid r.v.'s such that for all $n\geq 1$ we have $\E[\vert \xi_n\vert]<\infty$. Moreover, let $S_0=0$ and $S_n=\sum_{j=1}^n\xi_j$. Then 

\[
\frac{S_n}{n}\xrightarrow{n\to\infty\atop\text{a.s. and $L^1$}}\E[\xi_1].
\]

\end{thm}

\begin{proof}

At first, we want to show that $\E[\xi_1\mid S_n]=\frac{S_n}{n}$. Indeed, we know that there is a measurable map $g$ such that 

\[
\E[\xi_1\mid S_n]=g(S_n).
\]

Moreover, we know that for $k\in\{1,...,n\}$ we have $(\xi_1,S_n)$ and $(\xi_k,S_n)$ have the same law. Now for all bounded and Borel measurable maps $h$ we have 

\[
\E[\xi_k h(S_n)]=\E[\xi_1 h(S_n)]=\E[g(S_n)h(S_n)].
\]

Thus $\E[\xi_k\mid S_n]=g(S_n)$. We have 
\[
\sum_{j=1}^n\E[\xi_j\mid S_n]=\E\left[\sum_{j=1}^n \xi_j\mid S_n\right]=S_n, 
\]

but on the other hand we have that 

\[
\sum_{j=1}^n\E[\xi_j\mid S_n]=ng(s_n).
\]

Hence $g(S_n)=\frac{S_n}{n}$. Now take $\mathcal{H}_1=\sigma(S_n)$ and $\mathcal{H}_2=\sigma(S_n,\xi_{n+1},\xi_{n+2},...)$. Thus, by lemma \ref{lem5}, we get 

\[
\E[\xi_1\mid S_n,\xi_{n+1},\xi_{n+2},...]=\E[\xi_1\mid S_n].
\]

Now define $\mathcal{G}:=\sigma(S_n,\xi_{n+1},\xi_{n+2},...)$. Then we have $\mathcal{G}_{n+1}\subset\mathcal{G}_n$ because $S_{n+1}=S_n+\xi_{n+1}$. Hence, it follows that $\E[\xi_1\mid \mathcal{G}_n]=\E[\xi_1\mid S_n]=\frac{S_n}{n}$ converges a.s. and in $L^1$ to some r.v., but Kolmogorov's 0-1 law implies that this limit is a.s. constant. In particular, $\E\left[\frac{S_n}{n}\right]=\E[\xi_1]$ converges in $L^1$ to this limit, which is thus $\E[\xi_1]$.

\end{proof}

\begin{exer}[Hewitt-Savage 0-1 law]
Let $(\xi_n)_{n\geq 1}$ be iid r.v.'s with values in some measurable space $(E,\mathcal{E})$. The map $\omega\mapsto (\xi_1(\omega),\xi_2(\omega),...)$ defines a r.v. without values in $E^{\N^\times}$.  A measurable map $F$ defined on $E^{\N^\times}$ is said to be symmetric if 

\[
F(x_1,x_2,...)=F(x_{\pi(1)},x_{\pi(2)},...)
\]

for all permutations $\pi$ of $\N^\times$ with finite support. 

\vspace{0.5cm}

Prove that if $F$ is a symmetric function on $E^{\N^\times}$, then $F(\xi_1,\xi_2,...)$ is a.s. constant.

$Hint:$ Consider $\F_n=\sigma(\xi_1,x_2,...,\xi_n)$, $\mathcal{G}_n=\sigma(\xi_{n+1},\xi_{n+2},...)$, $Y=F(\xi_1,\xi_2,...)$, $X=\E[Y\mid \F_n]$ and $Z_n=\E[Y\mid \mathcal{G}_n]$.

\end{exer}

\subsection{Martingales bounded in $L^2$ and random series}

Let $(\Omega,\F,(\F_n)_{n\geq 0},\p)$ be a filtered probability space. Let $(M_n)_{n\geq 0}$ be a martingale in $L^2(\Omega,\F,(\F_n)_{n\geq 0},\p)$, i.e. $\E[M_n^2]<\infty$ for all $n\geq 0$. We say that $(M_n)_{n\geq 0}$ is bounded in $L^2(\Omega,\F,(\F_n)_{n\geq 0},\p)$ if $\sup_{n\geq 0}\E[M_n^2]<\infty$. For $n\leq  \nu$, we have that $\E[M_\nu\mid \F_n]=M_n$ implies that $(M_\nu-M_n)$ is orthogonal to $L^2(\Omega,\F,(\F_n)_{n\geq 0},\p)$. Hence, for all $s\leq  t\leq  n\leq  \nu$, $(M_\nu-M_n)$ is orthogonal to $(M_t-M_s)$. 

\[
\left\langle M_\nu-M_n,M_t-M_s\right\rangle=0\Longleftrightarrow \E[(M_\nu-M_n)(M_t-M_s)]=0.
\]

Now write $M_n=M_0+\sum_{k=1}^n(M_k-M_{k-1})$. $M_n$ is then a sum of orthogonal terms and therefore

\[
\E[M_n^2]=\E[M_0^2]+\sum_{k=1}^n\E[(M_k-M_{k-1})^2].
\]

\begin{thm}
Let $(M_n)_{n\geq 0}$ be a martingale in $L^2(\Omega,\F,(\F_n)_{n\geq 0},\p)$. Then $(M_n)_{n\geq 0}$ is bounded in $L^2(\Omega,\F,(\F_n)_{n\geq 0},\p)$ if and only if $\sum_{k\geq 1}\E[(M_k-M_{k-1})^2]<\infty$ and in this case $$M_n\xrightarrow{n\to\infty\atop \text{a.s. and $L^1$}}M_\infty.$$ 

\end{thm}

\begin{thm}

Suppose that $(X_n)_{n\geq 1}$ is a sequence of independent r.v.'s such that for all $k\geq 1$, $\E[X_k]=0$ and $\sigma_k^2=Var(X_k)<\infty$. Then 

\begin{enumerate}[$(i)$]
\item{$\sum_{k\geq1}\sigma^2_k<\infty$ implies that $\sum_{k\geq 1}X_k$ converges a.s.}
\item{If there is a $C>0$ such that for all $\omega\in\Omega$ and $k\geq1$, $\vert X_k(\omega)\vert\leq  C$, then $\sum_{k\geq 1}X_k$ converges a.s. implies that $\sum_{k\geq 1}\sigma_k^2<\infty$.}
\end{enumerate}
\end{thm}

\begin{proof}

Consider $\F_n=\sigma(X_1,...,x_n)$ with $F_0=\{\varnothing,\Omega\}$, $S_n=\sum_{j=1}^nX_j$ with $S_0=0$ and $A_n=\sum_{k=1}^n\sigma_k^2$ with $A_0=0$. Moreover, set $M_n=S_n^2-A_n$. Then $(S_n)_{n\geq 0}$ is a martingale and 

\[
\E[(S_n-S_{n-1})^2]=\E[X_n^2]=\sigma_n^2.
\]

Thus $\sum_{n\geq 1}\sigma_n^2<\infty$ inplies $\sum_{n\geq 1}\E[(S_n-S_{n-1})^2]<\infty$ and hence $(S_n)_{n\geq 0}$ 
is bounded in 

$L^2(\Omega,\F,(\F_n)_{n\geq 0},\p)$, which means that $S_n$ converges a.s. Next we show that $(M_n)_{n\geq 0}$ is a martingale. We have 

\[
\E[(S_n-S_{n-1})^2\mid \F_{n-1}]=\E[X_n^2\mid \F_{n-1}]=\E[X_n^2]=\sigma_n^2.
\]

Hence we get 

\begin{multline*}
\sigma_n^2=\E[(S_n-S_{n-1})^2\mid \F_{n-1}]=\E[S_n^2-2S_{n-1}S_n+S_{n-1}^2\mid F_{n-1}]=\\
=\E[S_n^2\mid\F_{n-1}]-2S_{n-1}^2+S_{n-1}^2=\E[S_n^2\mid \F_{n-1}]-S_{n-1}^2,
\end{multline*}

which implies that $(M_n)_{n\geq 0}$ is a martingale. Let $T:=\inf\{n\in\N\mid \vert S_n\vert>\alpha\}$ for some constant $\alpha$. Then $T$ is a stopping time. $(M_{n\land T})_{n\geq 1}$ is a martingale and hence 

\[
\E[M_{n\land T}]=\E[S_{n\land T}^2]-\E[A_{n\land T}]=0.
\]

Therefore $\E[S_{n\land T}^2]=\E[A_{n\land T}]$ and if $T$ is finite, $\vert S_T-S_{T-1}\vert<\vert X_T\vert\leq  C$ for some constant $C$, thus $\vert S_{n\land T}\vert \leq  C+\alpha$ and hence $\E[A_{n\land T}]\leq  (C+\alpha)^2$ for all $n\geq 0$. Now since $A_n$ is increasing we get that $\E[A_{n\land T}]\leq  (C+\alpha)^2<\infty.$ Since $\sum_{n\geq 0}X_n$ converges a.s., $\sum_{k=1}^n X_k$ is bounded and there exists $\alpha>0$ such that $\p[T=\alpha]>0$. Choosing\footnote{Note that $\E\left[\sum_{k\geq 1}\sigma_k^2\one_{\{ T=\infty\}}\right]=\sum_{k\geq 1}\sigma_k^2\p[T=\infty]$} $\alpha$ right yields $\sum_{k\geq 1}\sigma_k^2<\infty$.

\end{proof}

\begin{ex}

Let $(a_n)_{n\geq 1}$ be a sequence of real numbers and let $(\xi_n)_{n\geq 1}$ be iid r.v.'s with $\p[\xi=\pm 1]=\frac{1}{2}$. Then $\sum_{n\geq 1}a_n\xi_n$ converges a.s. if and only if $\sum_{n\geq 1}a_n^2<\infty$. Indeed, we get $\vert a_n\xi_n\vert=\vert a_n\vert\xrightarrow{n\to\infty}0$ and therefore there exists a $C>0$ such that for all $n\geq 1$, $\vert a_n\vert \leq  C$. Now for a r.v. $X$ recall that we write $\Phi_X(t)=\E\left[e^{itX}\right]$. We also know

\[
e^{itx}=\sum_{n\geq 0}\frac{i^nt^nx^n}{n!},\hspace{0.5cm}e^{itX}=\sum_{n\geq 0}\frac{i^nt^nX^n}{n!}.
\]

Moreover, define $R_n(x)=e^{ix}-\sum_{k=0}^n\frac{i^kx^k}{k!}$. Therefore we get 

\[
\vert R_n(x)\vert\leq  \min \left(2\frac{\vert x\vert^n}{n!},\frac{\vert x\vert^{n+1}}{(n+1)!}\right).
\]

Indeed, $\vert R_0(x)\vert=\vert e^{ix}-1\vert=\left\vert\int_0^x ie^{iy}dy\right\vert\leq  \min(2,\vert x\vert).$ Moreover, we have $\vert R_n(x)\vert=\left\vert \int_0^x iR_{n-1}(y)dy\right\vert$. Hence the claim follows by a simple induction on $n$. If $X$ is such that $\E[X]=0$, $\E[X^2]=\sigma^2<\infty$ and $e^{itX}-\left( 1+itX-\frac{t^2X^2}{2}\right)=R_2(tX)$ we get 

\[
\E\left[e^{itX}\right]=1-\frac{\sigma^2t^2}{2}+\E[R_2(tX)]
\]

and $\E[R_2(tX)]\leq  t^2\E[\vert X\vert^2\land tX^3]$. With dominated convergence it follows that $\Phi(t)=1-\frac{t^2\sigma^2}{2}+o(t^2)$ as $t\to 0$.

\end{ex}

\begin{lem}

Let $(\Omega,\F,(\F_n)_{n\geq 0},\p)$ be a filtered probability space. Let $(X_n)_{n\geq 0}$ be a sequence of independent r.v.'s bounded by $k>0$. Then if $\sum_{n\geq 0}X_n$ converges a.s., $\sum_{n\geq 0}\E[X_n]$ and $\sum_{n\geq 0}Var(X_n)$ both converge.

\end{lem}

\begin{proof}

If $Z$ is a r.v. such that $\vert Z\vert<k$, $\E[Z]=0$ and $\sigma^2=Var(Z)<\infty$, then for $\vert t\vert\leq  \frac{1}{k}$ we get 

\[
\vert \Phi_Z(t)\vert\leq  1-\frac{t^2\sigma^2}{2}+\frac{t^3k\E[Z^2]}{6}\leq  1-\frac{t^2\sigma^2}{2}+\frac{t^2\sigma^2}{6}=1-\frac{t^2\sigma^2}{3}\leq  \exp\left(-\frac{t^2\sigma^2}{3}\right).
\]

Let $Z_n=X_n-\E[X_n]$. Then $\vert \Phi_{Z_n}(t)\vert=\vert \Phi_{X_n}(t)\vert$ and $\vert Z_n\vert\leq  2k$. If $\sum_{n\geq 0}X_n=\infty$, we get

\[
\prod_{n\geq0}\vert\Phi_{X_n}(t)\vert=\prod_{n\geq 0}\vert\Phi_{Z_n}(t)\vert\leq  \exp\left(-\frac{1}{3}t^2\sum_{n\geq 0}Var(X_n)\right)=0.
\]

This is a contradiction, since $\vert \Phi_{\sum_{n\geq 0}X_n}(t)\vert\xrightarrow{n\to\infty}\vert\Phi(t)\vert$ with $\Phi$ continuous and $\Phi(0)=1$. Hence $\sum_{n\geq 0}Var(X_n)=\sum_{n\geq 0}Var(Z_n)<\infty$. Since $\E[Z_n]=0$ and $\sum_{n\geq 0}Var(Z_n)<\infty$, we have $\sum_{n\geq 0}Z_n$ converges a.s., but $\sum_{n\geq 0}-Z_n=\sum_{n\geq 0}X_n-\sum_{n\geq 0}\E[X_n]$ and thus since $\sum_{n\geq 0}X_n$ converges a.s. it follows that $\sum_{n\geq 0}\E[X_n]$ converges.

\end{proof}

\begin{thm}[Kolmogorov's three series theorem]

Let $(\Omega,\F,\p)$ be a probability space. Let $(X_n)_{n\geq 0}$ be a sequence of independent r.v.'s. Then $\sum_{n\geq 0}X_n$ converges a.s. if and only if for some $k>0$ (then for every $k>0$) the following properties hold.
\begin{enumerate}[$(i)$]
\item{$\sum_{n\geq 0}\p[\vert X_n\vert>k]<\infty$}
\item{$\sum_{n\geq 0}\E\left[X_n^{(k)}\right]$ converges, where $X_n^{(k)}=X_n\one_{\{\vert X_n\vert\leq  k\}}$}
\item{$\sum_{n\geq 0}Var\left(X_n^{(k)}\right)<\infty$}
\end{enumerate}

\end{thm}

\begin{proof}

Suppose that for some $k>0$, $(i)$, $(ii)$ and $(iii)$ hold. Then 

\[
\sum_{n\geq 0}\p\left[X_n\not=X_n^{(k)}\right]=\sum_{n\geq 0}\p[\vert X_n\vert >k]<\infty.
\]

It follows from the Borel-Cantelli lemma that $\p\left[X_n=X_n^{(k)}\text{for all but finitely many $n$}\right]=1$. Hence we only need to show that $\sum_{n\geq 0}X_n^{k}$ converges a.s. Because of $(ii)$ it is enough to show that $\sum_{n\geq 0}Y_n^{(k)}$ converges, where $Y_n^{(k)}=X_n^{(k)}-\E\left[X_n^{(k)}\right]$. The convergence of $\sum_{n\geq 0}Y_n^{(k)}$ follows then from $(iii)$. Conversely assume that $\sum_{n\geq 0}X_n$ converges a.s. and that $k\in(0,\infty)$. Since $X\xrightarrow{n\to\infty}0$ a.s., we have that $\vert X_n\vert>k$ for only finitely many $n$. Therefore, the Borel-Cantelli lemma implies $(i)$. Since $X_n=X_n^{(k)}$ for all but finitely many $n$, $\sum_{n\geq 0}X_n^{(k)}$ converges and it follows from lemma 14.8. that $(ii)$ and $(iii)$ have to hold.

\end{proof}

\begin{lem}[Ces\`aro]

Suppose $(b_n)_{n\geq 1}$ is a sequence of strictly positive real numbers with $b_n\uparrow\infty$ as $n\to\infty$. Let $(v_n)_{n\geq 1}$ be a sequence of real numbers such that $v_n\xrightarrow{n\to\infty}v$. Then 

\[
\frac{1}{b_n}\sum_{k=1}^n(b_k-b_{k-1})v_k\xrightarrow{n\to\infty}v\hspace{0.5cm}(b_0=0)
\]

\end{lem}

\begin{proof}

Note that 

\begin{multline*}
\left\vert \frac{1}{b_n}\sum_{k=1}^n(b_k-b_{k-1})v_k-v\right\vert=\left\vert \frac{1}{b_n}\sum_{k=1}^n(b_k-b_{k-1})(v_k-v)\right\vert\leq  \frac{1}{b_n}\sum_{k=1}^N(b_k-b_{k-1})\vert v_k-v\vert\\+\frac{1}{b_n}\sum_{k=N+1}^n(b_k-b_{k-1})\vert v_k-v\vert.
\end{multline*}

Now we only have to choose $N$ such that $n\geq N$ and $\vert v_k-v\vert<\varepsilon$ for any $\varepsilon>0$.
\end{proof}

\begin{lem}[Kronecker]

Let $(b_n)_{n\geq 1}$ be a sequence of real numbers, strictly positive with $b_n\uparrow \infty$ as $n\to\infty$. Let $(x_n)_{n\geq 1}$ be a sequence of real numbers. Then if $\sum_{n\geq 1}\frac{x_n}{b_n}$ converges, we get that 

\[
\frac{x_1+\dotsm +x_n}{b_n}\xrightarrow{n\to\infty}0.
\]

\end{lem}

\begin{proof}

Let $v_n=\sum_{k=1}^n\frac{x_n}{b_n}$ and $v=\lim_{n\to\infty}v_n$. Then $v_n-v_{n-1}=\frac{x_n}{b_n}$. Moreover, we note that

\[
\sum_{k=1}^nx_k=\sum_{k=1}^nb_k(v_k-v_{k-1})=b_nv_n-\sum_{k=1}^n(b_k-b_{k-1})v_k,
\]

which implies that 

\[
\frac{x_1+\dotsm+x_n}{b_n}=v_n-\frac{1}{b_n}\sum_{k=1}^n(b_k-b_{k-1})v_k\xrightarrow{n\to\infty}v-v=0.
\]

\end{proof}

\begin{prop}

Let $(\Omega,\F,\p)$ be a probability space. Let $(w_n)_{n\geq 1}$ be a sequence of r.v.'s such that $\E[w_n]=0$ for all $n\geq 1$ and $\sum_{n\geq 1}\frac{Var(w_n)}{n^2}<\infty$. Then 

\[
\frac{1}{n}\sum_{n\geq 1}w_n\xrightarrow{n\to\infty \atop a.s.}0.
\]

\end{prop}

\begin{proof}

Exercise.\footnote{From Kronecker's lemma it is enough to prove that $\sum_{k\geq 1}\frac{w_k}{k}$ converges a.s.} 

\end{proof}

\begin{thm}

Let $(\Omega,\F,\p)$ be a probability space. Let $(X_n)_{n\geq 1}$ be independent and non-negative r.v.'s such that $\E[X_n]=1$ for all $n\geq 1$. Define $M_0=1$ and for $n\in\N$, let 

\[
M_n=\prod_{j=1}^nX_j.
\]

Then $(M_n)_{n\geq 1}$ is a non-negative martingale, so that $M_\infty:=\lim_{n\to\infty}M_n$ exists a.s. Then the following are equivalent.

\begin{enumerate}[$(i)$]
\item{$\E[M_\infty]=1$.}
\item{$M_n\xrightarrow{n\to\infty\atop L^1}M_\infty$.}
\item{$(M_n)_{n\geq 1}$ is u.i.}
\item{$\prod_{n}a_n>0$, where $0<a_n=\E[X_n^{1/2}]\leq  1$.}
\item{$\sum_{n}(1-a_n)<\infty$.}
\end{enumerate}

Moreover, if one of the following (then every one) statements hold, then 

\[
\p[M_\infty=0]=1.
\]

\end{thm}
\begin{proof}
Exercise.

\end{proof}

\subsection{A martingale central limit theorem} 

\begin{thm}
Let $(\Omega,\F,(\F_n)_{n\geq 0},\p)$ be a filtered probability space. Let $(X_n)_{n\geq 0}$ be a sequence of real valued r.v.'s such that for all $n\geq 1$

\begin{enumerate}[$(i)$]
\item{$\E[X_n\mid \F_{n-1}]=0$.}
\item{$\E[X_n^2\mid \F_{n-1}]=1$.}
\item{$\E[\vert X_n\vert^3\mid \F_{n-1}]\leq  k<\infty$.}
\end{enumerate}

Let $S_n=\sum_{j=1}^nX_j$. Then 

\[
\frac{S_n}{\sqrt{n}}\xrightarrow{n\to\infty\atop law}\mathcal{N}(0,1).
\]

\end{thm}

\begin{proof}

Define $\Phi_{n,j}(u)=\E\left[e^{iu\frac{X_j}{\sqrt{n}}}\mid \F_{j-1}\right]$. A Taylor expansion yields

\[
\exp\left(iu\frac{X_j}{\sqrt{n}}\right)=1+iu\frac{X_j}{\sqrt{n}}-\frac{u^2}{2n}X_j^2-\frac{iu^3}{6n^{3/2}}\bar X_j^3,
\]

where $\bar X_j$ is a random number between 0 and $X_j$. Therefore we get 

\[
\Phi_{n,j}(u)=1+iu\frac{1}{\sqrt{n}}\E[X_j\mid \F_{j-1}]-\frac{u^2}{2n}\E[X_j^2\mid \F_{j-1}]-\frac{iu^3}{6n^{3/2}}\E[\bar X_j^3\mid \F_{j-1}]
\]

and thus 

\[
\Phi_{n,j}(u)-1+\frac{u^2}{2n}=-\frac{iu^3}{6n^{3/2}}\E[\bar X_j^2\mid \F_{j-1}].
\]

Hence we get 

\[
\E\left[e^{iu\frac{S_p}{\sqrt{n}}}\right]=\E\left[e^{iu\frac{S_{p-1}}{\sqrt{n}}}e^{iu\frac{X_p}{\sqrt{n}}}\right]=\E\left[e^{iu\frac{S_{p-1}}{\sqrt{n}}}\E\left[e^{iu\frac{X_p}{\sqrt{n}}}\mid \F_{p-1}\right]\right]=\E\left[e^{iu\frac{S_{p-1}}{\sqrt{n}}}\Phi_{n,p}(u)\right].
\]

Consequently, we get 

\[
\E\left[e^{iu\frac{S_p}{\sqrt{n}}}\right]=\E\left[e^{iu\frac{S_{p-1}}{\sqrt{n}}}\left(1+\frac{u^2}{2n}-\frac{iu^3}{6n^{3/2}}\bar X_p^3\right)\right].
\]

Thus we get that

\[
\E\left[e^{iu\frac{S_p}{\sqrt{n}}}-\left(1-\frac{u^2}{2n}\right)e^{iu\frac{S_{p-1}}{\sqrt{n}}}\right]=\E\left[e^{iu\frac{S_{p-1}}{\sqrt{n}}}\frac{iu^3}{6n^{3/2}}\bar X_p^3\right],
\]

which implies that 

\[
\left\vert \E\left[e^{iu \frac{S_p}{\sqrt{n}}}-\left(1-\frac{u^2}{2n}\right)e^{iu\frac{S_{p-1}}{\sqrt{n}}}\right]\right\vert\leq  \frac{K\vert u\vert ^3}{6n^{3/2}}.\hspace{1cm}(\star)
\]

Let us fix $n\in\N$. For $n$ large enough, we have $0\leq  1-\frac{u^2}{2n}\leq  1$. Multiplying both sides of $(\star)$ by $\left(1-\frac{u^2}{2n}\right)^{n-p}$, we get 

\[
\left\vert \left(1-\frac{u^2}{2n}\right)^{n-p}\E\left[e^{iu\frac{S_p}{\sqrt{n}}}\right]-\left(1-\frac{u^2}{2n}\right)^{n-p+1}\E\left[e^{iu\frac{S_{p-1}}{\sqrt{n}}}\right]\right\vert\leq  \frac{K\vert u\vert^3}{6n^{3/2}}.\hspace{1cm}(\star\star)
\]

By taking $K$ sufficiently large, we can assume that $(\star\star)$ holds for all $n$. Now we note that 

\[
\E\left[e^{iu\frac{S_n}{\sqrt{n}}}\right]-\left(1-\frac{u^2}{2n}\right)^n=\sum_{p=1}^n\left\{\left(1-\frac{u^2}{2n}\right)^{n-p}\E\left[e^{iu\frac{S_p}{\sqrt{n}}}\right]-\left(1-\frac{u^2}{2n}\right)^{n-p+1}\E\left[e^{iu\frac{S_{p-1}}{\sqrt{n}}}\right]\right\}.
\]

Therefore we get 

\[
\left\vert\E\left[e^{iu\frac{S_n}{\sqrt{n}}}\right]-\left(1-\frac{u^2}{2n}\right)^n\right\vert\leq  n\frac{K\vert u\vert^3}{n^{3/2}}=\frac{K\vert u\vert^3}{6\sqrt{n}},
\]

which implies that 

\[
\lim_{n\to\infty}\E\left[e^{iu\frac{S_n}{\sqrt{n}}}\right]=e^{-\frac{u^2}{2}}.
\]

\end{proof}

\chapter{Markov Chains}

\section{Definition and first properties}

In this chapter $E$ will be a finite or a countable set, endowed with the $\sigma$-Algebra $\mathcal{P}(E)$. A stochastic matrix on $E$ is a family $(Q(x,y))_{x,y\in E}$ of real numbers satisfying

\begin{enumerate}[$(i)$]
\item{$0\leq  Q(x,y)\leq  1$ for all $x,y\in E$.}
\item{$\sum_{y\in E}Q(x,y)=1$ for all $x\in E$.}
\end{enumerate}

This is a transition probability from $E$ to $E$ in the following sense: if for $x\in E$ and $A\subset E$ we write

\[
\nu(x,A)=\sum_{y\in A}\nu(x,y),
\]

we see that $\nu$ is a transition kernel probability from $E$ to $E$. Conversely, if we start with such a transition kernel, the formula 

\[
Q(x,y)=\nu(x,\{y\})
\]

defines a stochastic matrix on $E$. For $n\geq 1$, we can define $Q_n=Q^n$. Indeed, $Q_1=Q$ and by induction

\[
Q_{n+1}(x,y)=\sum_{z\in E}Q_n(x,z)Q(z,y).
\]

One can check that $Q_n$ is also a stochastic matrix on $E$. For $n=0$ we take $Q_0(x,y)=\one_{\{x=y\}}$. For a measurable map $f:E\to\R_+$ we write $Qf$ as the function defined by 

\[
Qf(x)=\sum_{y\in E}Q(x,y)f(y).
\]

\begin{defn}[Markov chain]

Let $(\Omega,\F,\p)$ be a probability space. Let $Q$ be a stochastic matrix on $E$ and let $(X_n)_{n\geq 0}$ be a stochastic process with values in $E$. We say that $(X_n)_{n\geq 0}$ is a Markov chain with transition matrix $Q$ if for all $n\geq 0$, the conditional distribution of $X_{n+1}$ given $(X_0,X_1,...,X_n)$ is $Q(X_n,y)$, or equivalently if for all $x_0,x_1,...,x_n,y\in E$

\[
\p[X_{n+1}=y\mid X_0=x_0,X_1=x_1,...,X_{n-1}=x_{n-1},X_n=x_n]=Q(x_n,y)
\]

such that $\p[X_0=x_0,....,X_n=x_n]>0$.

\end{defn}

\begin{rem}

In general, the conditional distribution of $X_{n+1}$ given $X_0,X_1,...,X_n$ depends on all the variables $X_0,X_1,...,X_n$. The fact that this conditional distribution only depend on $X_n$ is called the Markov property.

\end{rem}

\begin{rem}
$Q(x,\cdot)$, which is the distribution of $X_{n+1}$ given $X_n=x_n$ does not depend on $n$: this is the homogeneity of the Markov chain.

\end{rem}

\begin{prop}

Let $(\Omega,\F,\p)$ be a probability space. A stochastic process $(X_n)_{n\geq 0}$ with values in $E$ is a Markov chain with transition kernel $Q$ if and only if for all $n\geq 0$ and for all $x_0,x_1,...,x_n\in E$

\begin{equation}
\label{markov}
\p[X_0=x_0,X_1=x_1,...,X_n=x_n]=\p[X_0=x_0]\prod_{j=1}^nQ(x_{j-1},x_j).
\end{equation}

In particular, if $\p[X_0=x_0]>0$, then 

\[
\p[X_n=x_n\mid X_0=x_0]=Q_n(x_0,x_n).
\]

\end{prop}

\begin{proof}

If $(X_n)_{n\geq 0}$ is a Markov chain with transition matrix $Q$, we have 

\begin{align*}
\p[X_0=x_0,...,X_n=x_n,X_{n+1}=x_{n+1}]&=\p[X_0=x_0,...,X_n=x_n]\p[X_{n+1}=x_{n+1}\mid X_0=x_0,...,X_n=x_n]\\
&=\p[X_1=x_1,...,X_n=x_n]\\
&=Q(x_{n+1},x_n).
\end{align*}

Thus we can conclude by induction. Conversely, if (\ref{markov}) is satisfied, then 

\[
\p[X_{n+1}=y\mid X_0=x_0,...,X_n=x_n]=\frac{\p[X_0=x_0]\prod_{j=0}^{n-1}Q(x_j,x_{j+1})Q(x_n,y)}{\p[X_0=x_0]\prod_{j=0}^{n-1}Q(x_j,x_{j+1})}=Q(x_n,y).
\]

To conclude, we note that 

\[
Q_n(x_0,x)=\sum_{x_1,...x_{n+1}\in E}\prod_{j=1}^nQ(x_{j-1}x_j).
\]

\end{proof}

\begin{rem}

Proposition 15.1. shows that for a Markov chain, $(X_0,X_1,...,X_n)$ is completely determined by the initial distribution (that of $X_0$) and the transition matrix $Q$. For now we want to note $\p[A\mid Z]$ for $\E[\one_A\mid Z]$.

\end{rem}

\begin{prop}

Let $(\Omega,\F,\p)$ be a probability space. Let $(X_n)_{n\geq 0}$ be a Markov chain with transition matrix $Q$. 

\begin{enumerate}[$(i)$]
\item{For all $n\geq 0$ and for all measurable maps $f:E\to \R_+$ we have 

\[
\E[f(X_{n+1})\mid X_0,X_1,...,X_n]=\E[f(X_{n+1})\mid X_n]=Qf(X_n).
\]

More generally for all $\{ i_1,...,i_n\}\subset \{0,1,...,n-1\}$, we have

\[
\E[f(X_{n+1})\mid X_{i_1},X_{i_2},...,X_{i_n},X_n]=Qf(X_n).
\]

}
\item{For all $n\geq 0, p\geq 1$ and for all $y_1,...,y_p\in E$ we have 

\[
\p[X_{n+1}=y_1,...,X_{n+p}=y_p\mid X_0,...,X_n]=Q(X_n,y_1)\prod_{j=1}^{p-1}Q(y_j,y_{j+1})
\]

and hence 

\[
\p[X_{n+p}=y_p\mid X_n]=Q_p(X_n,y_p).
\]

If we take $Y_p=X_{n+p}$, then $(Y_p)_{p\geq 0}$ is also a Markov chain with transition matrix $Q$.

}

\end{enumerate}

\end{prop}

\begin{proof}

For $(i)$, we have 

\[
\E[f(X_{n++})\mid X_0,X_1,...,X_n]=\sum_{y\in E}Q(X_n,y)f(y)=Qf(X_n).
\]

Now if $\{i_1,...,i_n\}\subset \{ 0,1,...,n-1\}$, then\footnote{Recall that if $X\in(E,\mathcal{E})$, $Y\in(F,\F)$, we call the conditional distribution of $Y$ given $X$ any transition kernel $\nu:E\to F$ such that for all measurable maps $h:(F;\F)\to \R_+$, $\E[h(y)\mid X]=\int_\R h(y)\nu(X,dy)$.} 

\begin{multline*}
\E[f(X_{n+1})\mid X_{i_1},...,X_{i_n},X_n]=\E[\E[f(X_{n+1})\mid X_0,X_1,...,X_n]\mid X_{i_1,...,X_{i_n}},X_n]\\=\E[Qf(X_n)\mid X_{i_1},...,X_{i_n},X_n]=Qf(X_n).
\end{multline*}

For $(ii)$, note that it follows immediately from $(\ref{markov})$ that 

\[
\p[X_{n+1}=y_1,...,X_{n-p}=y_p\mid X_0=x_0,...,X_n=x_n]=Q(x_n,y_1)\prod_{j=1}^{p-1}Q(y_j,y_{j+1}). 
\]
The formula for $\p[X_{n-p}=y_p\mid X_n]$ follows for $y_1,...,y_{p-1}\in E$. Finally we note that 

\[
\p[Y_0=,Y_1=y_1,...,Y_p=y_p]=\p[X_n=y_n]\prod_{j=0}^pQ(y_{j-1},y_j)
\]

and we can apply proposition 15.1.

\end{proof}

\begin{ex}

Let $(\Omega,\F,\p)$ be a probability space. Let $(X_n)_{n\geq 0}$ be a sequence of independent r.v.'s with values in $E$, with the same distribution $\mu$. Then $(X_n)_{n\geq 0}$ is a Markov chain with transition matrix

\[
Q(x,y)=\mu(y)
\]

for all $x,y\in E$.
 
\end{ex}

\begin{ex}

Let $(\Omega,\F,\p)$ be a probability space. Let $\eta,\xi_1,...,\xi_n,...$ be independent r.v.'s in $\mathbb{Z}^d$. We assume that $\xi_1,...\xi_n,...$ have the same distribution $\mu$. For $n\geq 0$ we write 

\[
X_n=\eta+\xi_1+\dotsm+\xi_n.
\]

Then $(X_n)_{n\geq 0}$ is a Markov chain with transition matrix 

\[
Q(x,y)=\mu(y-x)
\]

for all $x,y\in E$. Indeed, 

\begin{multline*}
\p[X_{n+1}=y\mid X_0=x_0,...,X_n=x_n]=\p[\xi_{n+1}=y-x_n\mid X_0=x_0,...,X_n=x_n]\\=\p[\xi_{n+1}=y-x_n]=\mu(y-x_n).
\end{multline*}

If $(e_1,...,e_d)$ is the canonical basis of $\R^d$, and if $\mu(e_i)=\mu(-e_i)=\frac{1}{2d}$ for all $i\in\{1,...,d\}$, then the Markov chain is called a simple random walk on $\mathbb{Z}^d$.

\end{ex}

\begin{ex}

Let $\mathcal{P}_2(E)$ be the subsets of $E$ with two elements and let $A\subset\mathcal{P}_2(E)$. For $x\in E$, we note 

\[
A_x=\{y\in E\mid \{x,y\}\in A\}.
\]
We assume that $A_x\not=\varnothing$ for all $x\in E$. We then define a transition matrix on $E$ by setting for $x,y\in E$

\[
Q(x,y)=\begin{cases}\frac{1}{\vert A_x\vert},&\text{if $\{x,y\}\in A$}\\ 0,&\text{otherwise}\end{cases}
\]

A Markov chain with transition matrix $Q$ is called a simple random walk on the graph $(E,A)$.

\end{ex}

\section{The Canonical Markov chain}

We start with the existence of a Markov chain associated with a given transition matrix.

\begin{prop}
\label{prop2}
Let $(\Omega,\F,\p)$ be a probability space. Let $Q$ be a stochastic matrix on $E$. There exists a probability space $(\Omega',\F',\p')$ on which there exists, for all $x\in E$, a stochastic process $(X_n^x)_{n\geq 0}$ which is a Markov chain with transition matrix $Q$, starting from $X_0^x=x$.

\end{prop}

\begin{proof}

We can take $\Omega'=[0,1)$, with the Borel $\sigma$-Algebra and the Lebesgue measure. For $\omega\in[0,1)$, we set 

\[
\omega=\sum_{n=0}^\infty \mathcal{E}_n(\omega)2^{-n-1},
\]

with $\mathcal{E}_n(\omega)\in\{0,1\}$. We can further take the sequence $(\mathcal{E}_n)_{n\geq 0}$ of iid r.v.'s with $\p[\mathcal{E}_n=1]=\p[\mathcal{E}_n=0]=\frac{1}{2}$. If $\varphi:\N\times\N\to \N$ is a bijection, then the r.v.'s $\eta_{i,j}=\mathcal{E}_{\varphi(i,j)}$ for $i,j\in\N$ are also iid. Now define 

\[
U_i=\sum_{j=0}^\infty\eta_{i,j}2^{-j-1}.
\]

Therefore $U_0,U_1,U_2,...$ are iid. Denote $E=\{y_1,y_2,...,y_k,...\}$ and fix $x\in E$. We set $X_0^x=x$ and $X_1^x=y_k$ if 

\[
\sum_{i\leq  j<k}Q(x,y_j)<U_1\leq  \sum_{1\leq  j\leq  k}Q(x,y_j),
\]

such that $\p[X_1^x=y]=Q(x,y)$ for all $y\in E$. We then proceed by induction $X_{n+1}^x=y_k$ if 

\[
\sum_{1\leq  j<k} Q(X_n^x,y_j)<U_{n+1}\leq  \sum_{1\leq  j\leq  k}Q(X_n^x,y_j).
\]

Using the independence of the $U_i$'s, we check that for all $k\geq 1$, we have 

\begin{multline*}
\p[X_{n+1}^x=y_k\mid X_0^x=x,X_1^x=x_1,...,X_n^x=x_n]\\=\p\left[\sum_{1\leq  j<k}Q(x_n,y_j)<U_{n+1}\leq  \sum_{1\leq  j\leq  k}Q(x_n,y_j)\mid X_0^x=x,...,X_n^x=x_n\right]=Q(x_n,y_k).
\end{multline*}

\end{proof}

In the sequel, we shall take $\Omega=E^\N$. An element $\omega\in \Omega$ is a sequence $\omega=(\omega_0,\omega_1,...)$ of elements in $E$. Define the coordinate map 

\[
X_n(\omega)=\omega_n.
\]

We take $\F$ to be the smallest $\sigma$-Algebra on $\Omega$ which make all $X_n$'s measurable. This $\sigma$-Algebra is generated by 

\[
C=\{\omega\in\Omega\mid \omega_0=x_0,\omega_1=x_1,...,\omega_n=x_n\}
\]

for $n\in\N$ and $x_0,...,x_n\in E$.

\begin{lem}
\label{lem2}
Let $(G,\mathcal{G})$ be a measurable space and let $\psi:G\to\Omega$ be a map. Then $\psi$ is measurable if and only if for all $n\geq 0$, $X_n\circ\psi$ is measurable. 

\end{lem}

\begin{proof}

We only need to show that if $X_n\circ \psi$ is measurable for all $n\geq 0$, then $\psi$ is measurable as well. Consider the $\sigma$-Algebra $\A$ on $\Omega$ given by 

\[
\A:=\{A\in\F\mid \psi^{-1}(A)\in\mathcal{G}\},
\]

which contains all the sets of the form $X_n^{-1}(y)$ for $y\in E$. Hence all the $X_n$'s are measurable with respect to $\A$ and the claim follows.

\end{proof}

\begin{thm}

Let $(\Omega,\F,\p)$ be a probability space. Let $Q$ be a stochastic matrix on $E$. For all $x\in E$, there exists a unique probability measure $\p_x$, on $\Omega=E^\N$, such that the sequence of coordinate functions $(X_n)_{n\geq 0}$ is a Markov chain with transition matrix $Q$ and $\p_x[X_0=x]=1$.

\end{thm}

\begin{proof}

From proposition \ref{prop2}, there exists $(\Omega',\F',\p')$ and $(X_n^x)_{n\geq 0}$ which is a Markov chain with transition matrix $Q$ and such that $X_0^x=x$. We then define $\p_x$ as the image of $\p'$ by the map 

\begin{align*}
\Omega'&\longrightarrow\Omega\\
\omega'&\longmapsto (X_n^x(\omega'))_{n\geq 0}
\end{align*}

This map is measurable because of lemma \ref{lem2}. We have $\p_x[X_0=x]=\p'[X_0^x=x]=1$ and for $x_0,x_1,...,x_n\in E$ we have 

\begin{multline*}
\p_x[X_0=x_0,X_1=x_1,...,X_n=x_n]=\p'[X_0^x=x_0,X_1^x=x_1,...,X_n^x=x_n]\\=\p'[X_0^x=x_0]\prod_{j=1}^nQ(x_{j-1},x_j)=\p_x[X_0=x_0]\prod_{j=1}Q(x_{j-1},x_j)
\end{multline*}

Therefore it follows that the sequence of the coordinate functions is a Markov chain. For uniqueness, we note that if $\p'_x$ is another such probability measure, $\p_x$ and $\p_x'$ coincide on cylinders. It then follows from an application of the monotone class theorem that $\p_x=\p_x'$.

\end{proof}

\begin{rem}

It follows for all $n\geq 0$ and for all $x,y\in E$, that we have 

\[
\p_x[X_n=y]=Q_n(x,y).
\]

\end{rem}

\begin{rem}
If $\mu$ is a probability measure on $E$, we write  

\[
\p_\mu:=\sum_{x\in E}\mu(x)\p_x,
\]

which defines a new probability measure on $\Omega$. By writing an explicit formula for $\p_\mu[X_0=x_0,...,X_n=x_n]$, one immediately checks that under $\p_\mu$ we get that $(X_n)_{n\geq 0}$ is a Markov chain with transition matrix $Q$, where $X_0$ has the distribution $\mu$.

\end{rem}

\begin{rem}

Let $(X'_n)_{n\geq 0}$ be a Markov chain with transition matrix $Q$ and initial distribution $\mu$. Then for all measurable subsets $B\subset\Omega=E^\N$, we have 

\[
\p[(X'_n)_{n\geq 0}\in B]=\p_\mu[B].
\]

This is only true when $B$ is a cylinder as in the proof above. This shows that all results proven for the canonical Markov chains are true for a general Markov chain with the same transition matrix.  One of the advantages of using a canonical Markov chain is that one can use translation operators. For all $k\in\N$, we define 

\[
\Theta_k:\Omega\to\Omega,\hspace{0.5cm}(\omega_n)_{n\geq 0}\mapsto \Theta_k((\omega_n)_{n\geq 0})=(\omega_{k+n})_{n\geq 0}.
\]

Lemma 16.2. shows that these applications are measurable. We write $\F_n=\sigma(X_0,...,X_n)$ and we use the notation $\E_x$ when we integrate with respect to $\p_x$. 
\end{rem}

\begin{thm}[Simple Markov property]
\label{thm5}
Let $(\Omega,\F,\p)$ be a probability space. Let $F$ and $G$ be two measurable, positive maps on $\Omega$ and let $n\geq 0$. Assume that $F$ is $\F_n$-measurable. Then for all $x\in E$, 

\[
\E_x[FG\circ \Theta_n]=\E_x[F\E_{X_n}[G]].
\]

Equivalently, we have 

\[
\E_x[G\circ \Theta_n\mid \F_n]=\E_{X_n}[G].
\]

We say that the conditional distribution of $\Theta_n(\omega)$ given $(X_0,X_1,...,X_n)$ is $\p_{X_n}$.

\end{thm}

\begin{rem}

Theorem \ref{thm5} is also true if one replace $\E_x$ with $\E_\mu$.
\end{rem}

\begin{proof}[Proof of Theorem \ref{thm5}]

It is enough to prove the first statement. For this we can restrict ourselves to the case where $F=\one_{\{ X_0=x_0,X_1=x_1,...,X_n=x_n\}}$ for $x_0,x_1,...,x_n\in E$. Let us first consider the case where $G=\one_{\{X_0=y_0,...,X_p=y_p\}}$. In this case, if $y\in E$, then 

\[
\E_y[G]=\one_{\{ Y_0=y\}}\prod_{j=1}^pQ(y_{j-1},y_j)
\]

and 

\begin{multline*}
\E_x[FG\circ \Theta_n]=\p_x[X_0=x_0,...,X_{n}=x_{n},X_n=y_0,X_{n+1}=y_{1},...,X_{n+p}=y_p]\\
=\one_{\{X_0=x\}}\prod_{j=1}^nQ(x_{j-1},x_j)\one_{\{y_0=y_n\}}\prod_{j=1}^pQ(y_{j-1},y_j)
\end{multline*}

Now it follows from the monotone class theorem that the above holds for $G=\one_A$, for $A\in\F_n$ and thus we can conclude.
\end{proof}

\begin{rem}

We would like to make sense of theorem 16.4. for $n$ replaced by a stopping time $T$. Indeed let us assume that we want to  the problem of knowing whether a Markov chain starting from $x$ visits back $x$. In other words, let us write 

\[
N_x=\sum_{n=0}^\infty \one_{\{X_n=x\}}
\]

and ask the question, whether we have that $\p_x[N_x=\infty]=1$. It is in fact enough to check that it comes back at least once at $x$. If $H_x=\inf\{n\geq 1\mid X_n=x\}$, with the convention $\inf\varnothing=\infty$, we have $\p_x[N_x=\infty]=1$ if and only if $\p_x[H_x<\infty]=1$, which is trivial. If $\p_x[N_x<\infty]=1$, and if theorem 16.4. is true for stopping times, then $\theta_{H_x}(\omega)=\left(\omega_{H_x(\omega)=n}\right)_{n\geq 0}$ has the law $\p_x$. But then, since $N_x(\omega)=1+N_x(\Theta_{H_x}(\omega))$, we see that $N_x$ has the same distribution as $1+N_x$ under $\p_x$. This is only possible if $N_x=\infty$ a.s.

\end{rem}

\begin{thm}[Strong Markov property]

Let $(\Omega,\F,(\F_n)_{n\geq 0},\p)$ be a filtered probability space. Let $T$ be a stopping time for the filtration $(\F_n)_{n\geq 0}$. Let $F$ and $G$ be two measurable and positive functions on $\Omega$. Assume further that $F$ is $\F_T$-measurable. Then for all $x\in E$, we get 

\[
\E_x[\one_{\{ T<\infty\}}FG\circ \Theta_T]=\E_x[\one_{\{ T<\infty\}}F\E_x[G]]
\]

or equivalently 

\[
\E_x[\one_{\{T<\infty\}}G\circ \Theta_T\mid \F_T]=\one_{\{ T<\infty\}}\E_{X_T}[G].
\]

\end{thm}

\begin{proof}

For $n\geq 0$, we have 

\[
\E_x[\one_{\{ T=n\}}FG\circ\Theta_T]=\E_x[\one_{\{ T=n\}}FG\circ \Theta_n]=\E_n[\one_{\{ T=n\}}\E_{X_n}[G]].
\]

The last equality follows from a previous theorem after obtaining that $\one_{\{T=n\}}F$ is $\F_n$-measurable. We then sum over $n$.

\end{proof}

\begin{cor}

Let $(\Omega,\F,(\F_n)_{n\geq 0},\p)$ be a filtered probability space. Let $T$ be a stopping time such that $\p_x[T<\infty]=1$. Let us assume that there exists $y\in E$ such that $\p_x[X_T=y]=1$. Then under $\p_x$, $\Theta_T(\omega)$ is independent of $\F_T$ and has the law $\p_y$.

\end{cor}

\begin{proof}

We only have to note that 

\[
\E_x[FG\circ \Theta_T(\omega)]=\E_x[F\E_{X_T}[G]]=\E_x[F\E_y[G]]=\E_x[F]\E_y[G].
\]

\end{proof}

\section{Classification of states}

From now on we will only use the canonical Markov chain. Recall that for $x\in E$ we have 

\begin{align*}
H_x&=\inf\{n\geq 1\mid X_n=x\},\\
N_x&=\sum_{n\geq 0}\one_{\{ X_n=x\}}.
\end{align*}

\begin{prop}[Recurrence and Transience]

Let $x\in E$. Then we have two situations, which can occur.

\begin{enumerate}[$(i)$]
\item{$\p_x[H_x<\infty]=1$ and $N_x=\infty$ $\p_x$-a.s. In this case $x$ is said to be reccurrent.}
\item{$\p_x[H_x<\infty]<1$ and $N_x<\infty$ $\p_x$-a.s. More precisely 

\[
\E_x[N_x]=\frac{1}{\p_x[H_x<\infty]}<\infty.
\]
In this case $x$ is said to be transient.

}

\end{enumerate}

\end{prop}

\begin{proof}

For $k\geq 1$, it follows from the strong Markov property that 

\[
\p_x[N_x\geq k+1]=\E_x[\one_{H_x<\infty}\one_{N_x\geq k}\circ \Theta_{H_x}]=\E_x[\one_{H_x<\infty}\E_x[\one_{N_x\geq k}]]=\p_x[H_x<\infty]\p_x[N_x\geq k].
\]

Now since $\p_x[N_x\geq 1]=1$, it follows by induction that 

\[
\p_x[N_x\geq k]=\p_x[H_x<\infty]^{k-1}.
\]

If $\p_x[H_x<\infty]=1$, then $\p_x[N_x=\infty]=1$. If $\p_x[H_x<\infty]<1$, then 

\[
\E_x[N_x]=\sum_{k\geq 1}\p_x[N_x\geq k]=\sum_{k\geq 1}\p_x[H_x<\infty]^{k-1}=\frac{1}{1-\p_x[H_x<\infty]}=\frac{1}{\p_x[H_x<\infty]}<\infty
\]

\end{proof}

\begin{defn}[Potential kernel]

The potential kernel of the chain is given by the function 

\[
U:E\times E\to [0,\infty],\hspace{1cm}(x,y)\mapsto \E_x[N_y].
\]

\end{defn}

\begin{prop}

The following hold.

\begin{enumerate}[$(i)$]
\item{For all $x,y\in E$ we have

\[
U(x,y)=\sum_{n\geq 0}Q_n(x,y).
\]

}
\item{$U(x,x)=\infty$ if and only if $x$ is recurrent.

}
\item{For all $x,y\in E$ it follows that if $x\not=y$ then 

\[
U(x,y)=\p_x[H_y<\infty]U(y,y).
\]
}

\end{enumerate}

\end{prop}

\begin{proof}

We need to show all three points. 

\begin{enumerate}[$(i)$]
\item{Note that 

\[
U(x,y)=\E_x\left[\sum_{n\geq 0}\one_{\{N_x=y\}}\right]=\sum_{n\geq 0}\p_x[X_n=y]=\sum_{n\geq 0}Q_n(x,y).
\]

}

\item{Exercise.}

\item{This follows from the strong Markov property. We get

\[
\E_x[N_y]=\E_x[\one_{\{H_y<\infty\}}N_y\circ\Theta_{H_y}]=\E_x[\one_{\{H_y<\infty\}}\E_y[N_y]]=\p_x[H_y<\infty]U(y,y).
\]

}

\end{enumerate}

\end{proof}

\begin{ex}

Let us consider the Markov chain, starting at 0, on $\mathbb{Z}^d$ with transition matrix given by 

\[
Q((x_1,...,x_d),(y_1,...,y_d))=\frac{1}{2^d}\prod_{i=1}^d\one_{\{\vert y_i-x_i\vert=1\}}.
\]

This Markov chain has the same distribution as $(Y_n^1,...,Y_n^d)_{n\geq 0}$, where $Y^1,...,Y^d$ are iid r.v.'s of a simple random walk on $\mathbb{Z}$, starting at 0. Hence we get 

\[
Q_n(0,0)=\p[Y_n^1=0,...,Y_n^d=0]=\p[Y_n^1=0]^d.
\]

Moreover, $\p[Y_n^1=0]=0$ if $n$ is odd and for $n=2k$ we get 

\[
\p[Y_{2k}^1=0]=2^{-2k}\binom{2k}{k}.
\]

Therefore we get that 

\[
U(0,0)=\sum_{k\geq 0}Q_{2k}(0,0)=\sum_{k\geq 0}\left(2^{-2k}\binom{2k}{k}\right)^d.
\]

Now with Stirling's formula\footnote{Recall that $n!\approx \sqrt{2\pi n}\left(\frac{n}{e}\right)^n$} we get

\[
2^{-2k}\binom{2k}{k}\approx_{k\to\infty}\frac{\left(\frac{2k}{e}\right)^{2k}\sqrt{4\pi k}}{2^{2k}\left(\left(\frac{k}{e}\right)^k\sqrt{2\pi k}\right)^2}\approx_{k\to\infty}\sqrt{\frac{1}{\pi k}}.
\]

Consequently, 0 is recurrent if $d=1$ and transient if $d\geq 2$.

\end{ex}

Now let us denote the set of all recurrent points by $R$. Then we can obtain the following lemma.

\begin{lem}

Let $x\in R$ and let $y\in E$ such that $U(x,y)>0$. Then $y\in R$ and 

\[
\p_y[H_x<\infty]=1.
\]

In particular we get that $U(y,x)>0$.

\end{lem}

\begin{proof}

Let us first show that $\p_y[H_x<\infty]=1$. We have that 

\begin{multline*}
0=\p_x[N_x<\infty]\geq \p_x[H_y<\infty, H_x\circ\Theta_{H_y}=\infty]=\E_x[\one_{\{H_y<\infty\}}\one_{\{H_x=\infty\}}\circ\Theta_{H_y}]\\=\E_x[\one_{\{H_y<\infty\}}\p_y[H_x=\infty]]=\p_x[H_y<\infty]\p_y[H_x=\infty].
\end{multline*}

Then $U(x,y)>0$ implies that $\p_x[H_y<\infty]>0$ and thus $\p_y[H_x=\infty]=0$. Hence we can find $n_1,n_2\geq 1$ such that 

\[
Q_{n_1}(x,y)\geq 0, \hspace{1cm} Q_{n_2}(y,x)>0.
\]
 
Then for all $p\geq 0$ we get that 

\[
Q_{n_1+p+n_2}(y,y)\geq Q_{n_2}(y,x)Q_p(x,x)Q_{n_1}(x,y),
\]
 
and thus we have 
 
\[
U(y,y)\geq\sum_{p=0}^\infty Q_{n_1+p+n_2}(y,y)\geq Q_{n_2}(y,x)\left(\sum_{p=0}^\infty Q_p(x,x)\right)Q_{n_1}(x,y)=\infty,
\]
 
since $x\in R$.
\end{proof}

\begin{rem}
The above lemma has the following important consequence. If $x\in R$ and $y\in E\setminus R$, then $U(x,y)=0$, which means that one cannot go from a recurrent point to a transient point.
\end{rem}

\begin{thm}[Classification of states]

Let $R$ be the set of all recurrent points and let $N_x$ be defined as above. Then there exists a partition of $R$, denoted $(R_i)_{i\in I}$ for some index set $I$, i.e. 

\[
R=\bigcup_{i\in I}R_i,
\]

such that the following properties hold.

\begin{enumerate}[$(i)$]
\item{If $x\in R$ and $i\in I$ such that $x\in R_i$, we get that 

\begin{itemize}
\item{$N_y=\infty$, $\p_x$-a.s. for all $y\in R_i$,}
\item{$N_y=0$, $\p_x$-a.s. for all $y\in E\setminus R_i$.}
\end{itemize}

}
\item{If $x\in E\setminus R$ and $T=\inf\{n\geq 0\mid X_n\in R\}$, we get that 

\begin{itemize}
\item{$T=\infty$ and $N_y<\infty$, $\p_x$-a.s. for all $y\in E$,}
\item{$T<\infty$ and then there exists $j\in I$ such that for all $n\geq T$ we get $X_n\in R_j$, $\p_x$-a.s.}
\end{itemize}

}
\end{enumerate}

\end{thm}

\begin{proof}

For $x,y\in R$, let us note $x\sim Y$ if $U(x,y)>0$. It follows from lemma 3.3. that this is an equivalence relation\footnote{for transitivity observe that if $Q_n(x,y)>0$ and $Q_m(y,z)>0$, then $Q_{n+m}(x,z)$.} on $R$. Let $i\in I$ and $x\in R_i$. We have that $U(x,y)=0$ for all $y\in E\setminus R_i$ and hence $N_y=0$, $\p_x$-a.s. for all $y\in E\setminus R_i$. If $y\in R_i$, we have from lemma 3.3. that $\p_x[H_y<\infty]=1$ and it follows from the strong Markov property that 

\[
\p_x[N_y=\infty]=\E_x[\one_{\{H_y<\infty\}}\one_{\{N_y=\infty\}}\circ\Theta_{H_y}]=\p_x[H_y<\infty]\p_y[N_y=\infty]=1.
\]

If $x\in E\setminus R$ and $T=\infty$, it easily follows from the strong Markov property that $N_y<\infty$ for all $y\in E\setminus R$. If $T<\infty$, let $j\in I$ such that $X_T\in R_j$. Then apply the strong Markov property with $T$ and the first part of the theorem, which then gives us that $X_n\in R_j$ for all $n\geq T$.

\end{proof}

\begin{defn}[Irreducibility]

A chain is called irreducible if $U(x,y)>0$ for all $x,y\in E$.

\end{defn}

\begin{cor}

If the chain is irreducible, then the following hold.

\begin{enumerate}[$(i)$]
\item{Either all states are recurrent and there exists a recurrent chain and for all $x\in E$ we get 

\[
\p_x[N_y=\infty, \forall y\in E]=1,
\]

}

\item{or all states are transient and then for all $x\in E$ we get 

\[
\p_x[N_y<\infty,\forall y\in E]=1.
\]

}

Moreover, if $\vert E\vert<\infty$, then only the first case can occur.

\end{enumerate}

\end{cor}

\begin{proof}

If there exists a recurrent state, the lemma 3.3. shows that all states are recurrent with $U(x,y)>0$ for all $x,y\in E$ and there is only one recurrent chain. If $\vert E\vert<\infty$ and if all states are transient, then we get $\p_x$-a.s. that 

\[
\sum_{y\in E}N_y<\infty,
\]

but we know 
\[
\sum_{y\in E}N_y=\sum_{y\in E}\sum_{n\geq 0}\one_{\{X_n=y\}}=\sum_{n\geq 0}\sum_{y\in E}\one_{\{X_n=y\}}=\infty.
\]

\end{proof}

\begin{ex}

Let us first recall that the results obtained for the canonical Markov chain hold for any arbitrary Markov chain. Now let $(Y_n)_{n\geq 0}$ be a Markov chain with transition matrix $Q$. For instance, if $Y_0=y$ and $N_x^Y=\sum_{n\geq0} \one_{\{ Y-n=x\}}$, we have for that $k\in\bar\N$  

\[
\p[N_x^Y=k]=\p_y[N_x=k],
\]

since the left hand side is $\p[(Y_n)_{n\geq 0}\in B]$ with $B:=\{\omega\in E^\N\mid N_x(\omega)=k\}$. 

\end{ex}

\begin{thm}

Let $(\xi_j)_{j\geq 1}$ be iid r.v.'s having the law $\mu$ with $\xi_j\in\mathbb{Z}$ for all $j\geq 1$ and assume that $\E[\vert\xi_1\vert]<\infty$ and $m=\E[\xi_1]$. Moreover, set $Y_n=Y_0+\sum_{j=1}^n\xi_j$. Then 

\begin{enumerate}[$(i)$]
\item{if $m\not=0$, then all states are transient.}
\item{if $m=0$, then all states are recurrent. Moreover, the chain is irreducible if and only if the subgroup given by $\{y\in\mathbb{Z}\mid \mu(y)>0\}$ is in $\mathbb{Z}$.}
\end{enumerate}

\end{thm}

\begin{proof}

We need to show both points.

\begin{enumerate}[$(i)$]
\item{If $m\not=0$, we know from the strong law of large numbers that $(Y_n)\xrightarrow{n\to\infty\atop a.s.}\infty$ and hence all states are transient.}

\item{Assume $m=0$ and $0$ is transient. Thus $U(0,0)<\infty$. We can assume that $Y_0=0$. For $x\in\mathbb{Z}$ we get 

\[
U(0,x)\leq  U(x,x)=U(0,0).
\]

Therefore we get that for all $n\geq 1$

\[
\sum_{\vert x\vert\leq  n}U(0,x)\leq  (2n+1)U(0,0)\leq  Cn,
\]

where $C=3 U(0,0)$. Moreover, by the law of large numbers, there exists some $N$ sufficiently large such that for all $n\geq N$

\[
\p[\vert Y_n\vert\leq  \varepsilon n]>Y_2
\]

with $\varepsilon=\frac{1}{6C}$, or equivalently 

\[
\sum_{\vert x\vert\leq  \varepsilon n}Q_n(0,x)\geq \frac{1}{2}.
\]

If $p\geq n\geq N$, then 

\[
\sum_{\vert x\vert\leq  \varepsilon n}Q_p(0,x)\geq \sum_{\vert x\vert\leq  \varepsilon p}Q_p(0,x)>\frac{1}{2},
\]

and by summing over $p$, we get 

\[
\sum_{\vert x\vert\leq  \varepsilon n}U(0,n)\geq \sum_{p=N}^n\sum_{\vert x\vert\leq  \varepsilon p}Q_p(0,x)>\frac{n-N}{\varepsilon}.
\]

But if $\varepsilon n\geq 1$, then 

\[
\sum_{\vert x\vert\leq  \varepsilon n}U(0,x)\leq  C\varepsilon n=\frac{n}{6},
\]

which leads to a contradiction for $n$ sufficiently large. For the last statement, write $G=\langle x\in\mathbb{Z}\mid \mu(x)>0\rangle$. Then, since $Y_0=0$, we get 

\[
\p[Y_n\in G\mid \forall n\in \N]=1.
\]

If $G\not=\mathbb{Z}$, then obviously $(Y_n)_{n\geq 0}$ is irreducible. If $G=\mathbb{Z}$, then write 

\[
H:=\{x\in\mathbb{Z}\mid U(0,x)>0\}.
\]

It follows that $H$ is a subgroup of $\mathbb{Z}$, indeed, for $x,y\in H$, we get that

\[
Q_{n+p}(0,x+y)\geq Q_n(0,x)Q_p(x,x+y)=Q_n(0,x)Q_p(0,y),
\]

and thus we get that $x+y\in H$. For $x\in H$ and with $0$ being recurrent, we get that $U(0,x)>0$ implies that $U(x,x)>0$ and therefore $U(x,0)=U(0,-x)$, which implies that $-x\in H$. Now since $H\supset\{ x\in\mathbb{Z}\mid \mu(x)>0\}$, we get that $H=\mathbb{Z}$ and the claim follows.

}

\end{enumerate}

\end{proof}

\begin{ex}

Let $\mu=\frac{1}{2}\delta_{-2}+\frac{1}{2}\delta_{2}$. Then all states are recurrent but there are two recurrence classes.

\end{ex}

\subsection{Random variable on a graph}

Assume that $\vert E\vert<\infty$. Moreover, let $A\subset\mathcal{P}(E)$ and for $x\in E$ let 

\[
A_x=\{y\in E\mid \{x,y\}\in A\}\not=\varnothing.
\]

A graph $\mathcal{G}$ is said to be connected if every pair of vertices in the graph is connected.

\begin{prop}

The simple random walk on a finite graph is recurrent and irreducible.
\end{prop}

\begin{proof}

Exercise\footnote{Connected implies irreducible. Then apply corollary 3.5.}.

\end{proof}

\section{Invariant measures}

\begin{defn}[Invariant measure]

Let $\mu$ be a positive measure on $E$ such that $\mu(x)<\infty$ for all $x\in E$ and $\mu\not\equiv 0$. We say that $\mu$ is invariant for the transition matrix $Q$ if for all $y\in E$ we get 

\[
\mu(y)=\sum_{x\in E}\mu(x)Q(x,y).
\]

With matrix notation, this means

\[
\mu Q=\mu.
\]

Since for all $n\in\N$, $Q_n=Q^n$, we have $\mu Q_n\mu=\mu$.

\end{defn}

\subsection{Interpretation}

Assume that $\mu(E)<\infty$, which is always the case when $E$ is finite. We can assume that $\mu(E)=1$. Then for all measurable maps $f:E\to\R_+$, we get that 

\[
\E_\mu[f(X_1)]=\sum_{x\in E}\mu(x)\sum_{y\in E}Q(x,y)f(y)=\sum_{y\in E}f(y)\sum_{x\in E}\mu(x)Q(x,y)=\sum_{y\in E}\mu(y)f(y).
\]

Hence under $\p_\mu$, we get that $X_1$ ${law}\atop{=}$ $X_0\sim \mu$. Using the fact that $\mu Q_n=\mu$, we show that under $\p_\mu$, we get that $X_n\sim \mu$. For all measurable maps $F:\Omega\to\R_+$ we have 

\[ 
\E_\mu[F\circ \Theta_\one]=\E_\mu[\E_{X_1}[F]]=\sum_{x\in E}\mu(x)\E_x[F]=\E_\mu[F],
\]

which implies that under $\p_\mu$, we get that $(X_{1+n})_{n\geq 0}$ has the same law\footnote{Same holds for $(X_{k+n})_{n\geq 0}$ for all $k\geq 0$} as $(X_n)_{n\geq 0}$.

\begin{ex}

For the r.v. on $\mathbb{Z}^d$, we get $Q(x,y)=(y,x)$. One then immediately checks that the counting measure on $\mathbb{Z}^d$ is invariant.

\end{ex}

\begin{defn}[Reversible measure]

Let $\mu$ be  positive measure on $E$ such that $\mu(x)<\infty$ for all $x\in E$. $\mu$ is said to be reversible if for all $x,y\in E$ we get that 

\[
\mu(x)Q(x,y)=\mu(y)Q(y,x).
\]

\end{defn}

\begin{prop}

A reversible measure is invariant.

\end{prop}

\begin{proof}

If $\mu$ is reversible, we get that 

\[
\sum_{x\in E}\mu(x)Q(x,y)=\sum_{x\in E}\mu(y)Q(y,x)=\mu(y).
\]

\end{proof}

\begin{rem}

There exists invariant measures which are not reversible, for example the counting measure is not reversible if the $\mathcal{G}$ is not symmetric, i.e. $\mathcal{G}(x)=\mathcal{G}(-x)$.  

\end{rem}

\begin{ex}[Random walk on a graph]

The measure $\mu(x)=\vert A_x\vert$ is reversible. Indeed, if $\{x,y\}\in A$, we get 

\[
\mu(x)Q(x,y)=\vert A_x\vert\frac{1}{\vert A_x\vert}=1=\mu(y)Q(y,x).
\]

\end{ex}

\begin{ex}[Ehrenfest's model]

This is the Markov chain on $\{1,...,k\}$ with transition matrix

\[
\begin{cases}Q(j,j+1)=\frac{k-j}{k},& 0\leq  j\leq  k-1\\ Q(j,j-1)=\frac{j}{k},&1\leq  j\leq  k\end{cases}
\]

A measure $\mu$ is reversible if and only if for $0\leq  j\leq  k-1$ we have

\[
\mu(j)\frac{k-j}{k}=\mu(j+1)\frac{j+1}{k}.
\]

One can check that $\mu(j)=\binom{k}{j}$ is a solution. 

\end{ex}

\begin{thm}

Let $x\in E$ be recurrent. The formula 

\[
\mu(y)=\E_x\left[\sum_{k=0}^{H_x-1}\one_{\{ X_k=y\}}\right]
\]

defines an invariant measure. Moreover, $\mu(y)>0$ if and only if $y$ is in the same recurrence class as $x$.

\end{thm}

\begin{proof}

Let us first note that if $y$ is not in the same recurrence class as $x$, then 

\[
\E_x[N_y]=U(x,y)=0,
\]

which implies that $\mu(y)=0$. For $y\in E$, we get that

\begin{multline*}
\mu(y)=\E_x\left[\sum_{k=1}^{H_x}\one_{\{ X_k=y\}}\right]=\sum_{z\in E}\E_x\left[\sum_{k=1}^{H_x}\one_{\{ X_{k-1}=z,X_k=y\}}\right]=\sum_{z\in E}\sum_{k\geq 1}\E_x\left[\one_{\{k\leq  H_x,X_{k-1}=z\}}\one_{\{ X_k=y\}}\right]\\=\sum_{z\in E}\sum_{k\geq 1}\E_x\left[\one_{\{k\leq  H_x,X_{k-1}=z\}}\right]=\sum_{z\in E}\E_x\left[\sum_{k=1}^{H_x}\one_{X_{k-1}=z}\right]Q(z,y)=\sum_{z\in E}\mu(z)Q(z,y).
\end{multline*}

We showed that $\mu Q=\mu$. Thus, it follows that $\mu Q_n=\mu$ for all $n\geq 0$. In particular, we get 

\[
\mu(x)=1=\sum_{z\in E}\mu(z)Q_n(z,x).
\]

Let $y$ be in the same recurrence class as $x$. Then there exists some $n\geq 0$ such that $Q_n(y,x)>0$, which implies that 

\[
\mu(y)<\infty.
\]

We can also find $m\geq 0$ such that $Q_m(x,y)>0$ and 

\[
\mu(y)=\sum_{z\in E}\mu(z)Q_m(z,y)\geq Q_m(x,x)>0.
\]
\end{proof}

\begin{rem}

If there exists several recurrence classes $R_i$ with $i\in I$ for some index set $I$ and if we set

\[
\mu_i(y)=\E_{x_i}\left[\sum_{k=0}^{H_{x_i}-1}\one_{\{ X_k=y\}}\right],
\]

we obtain a invariant measure with disjoint supports.

\end{rem}

\begin{thm}

Let us assume that the Markov chain is irreducible and recurrent. Then the invariant measure is unique up to a multiplicative constant.
\end{thm}

\begin{proof}

Let $\mu$ be an invariant measure. We can show by induction that for $p\geq 0$ and for all $x,y\in E$, we get

\[
\mu(y)\geq \mu(x)\E_x\left[\sum_{k=0}^{p\land (H_x-1)}\one_{\{ X_k=y\}}\right].\hspace{1.5cm}(\star) 
\]

First, if $x=y$, this is obvious. Let us thus suppose that $x\not=y$. If $p=0$, the inequality is immediate. Let us assume $(\star)$ holds for $p$. Then

\begin{align*}
\mu(y)&=\sum_{z\in E}\mu(z)Q(z,y)\geq \mu(x)\sum_{z\in E}\E_x\left[\sum_{k=0}^{p\land (H_x-1)}\one_{\{ X_k=z\}}\right]Q(z,y)\\
&=\mu(x)\sum_{z\in E}\sum_{k=0}^p\E_x[\one_{\{ X_k=z,k\leq  H_x-1\}}]Q(z,y)\\
&=\mu(x)\sum_{z\in E}\sum_{k=0}^p
\E_x[\one_{\{ X_k=z,k\leq  H_x-1\}}]\one_{\{ X_{k-1}=y\}}\\
&=\mu(x)\E_x\left[\sum_{k=0}^{p\land (H_x-1)}\one_{\{ X_{k-1}=y\}}\right]=\mu(x)\E_x\left[\sum_{k=1}^{(p+1)\land (H_x)}\one_{\{ X_k=y\}}\right].
\end{align*}

This establishes the result for $p+1$. Now, if we let $p\to\infty$ in $(\star)$, we get 

\[
\mu(y)\geq \mu(x)\E_x\left[\sum_{k=0}^{H_x-1}\one_{\{X_k=y\}}\right].
\]

Let us fix $x\in E$. The measure $\nu_x(y)=\E_x\left[\sum_{k=0}^{H_x-1}\one_{\{ X_k=y\}}\right]$ is invariant and we have $\mu(y)\geq \mu(x)\nu_x(y)$. Hence for all $n\geq 1$ we get 

\[
\mu(x)=\sum_{z\in E}\mu(z)Q_n(z,x)\geq \sum_{x\in E}\mu(x)\nu_x(z)Q_n(z,x)=\mu(x)\nu_x(x)=\mu(x).
\]

Therefore we get that $\mu(z)=\mu(x)=\nu_x(z)$ for all $z$ such that $Q_n(z,x)>0$. Since the chain is irreducible, there exists some $n\geq 0$ such that $Q_n(z,x)>0$. This implies finally that 

\[
\mu(z)=\mu(x)\nu_x(z).
\]

\end{proof}

\begin{cor}
\label{cor2}
Let us assume the chain is irreducible and recurrent. Then one of the following hold.

\begin{enumerate}[$(i)$]
\item{There exists an invariant probability measure $\mu$ and for $x\in E$ we have 

\[
\E_x[H_x]=\frac{1}{\mu(x)}.
\]

}

\item{
All invariant measures have infinite total mass and for $x\in E$ we get 

\[
\E_x[H_x]=\infty.
\]

}

\end{enumerate}

In the first case, the basis is said to be positive recurrent and in the second case it is said to be negative recurrent.

\end{cor}

\begin{rem}
If $E$ is finite, then only the first case can occur.

\end{rem}

\begin{proof}[Proof of Corollary \ref{cor2}]

We know that in this situation all invariant measures are proportional. Hence they all have finite mass or infinite mass. For case $(i)$, let $\mu$ be the invariant probability measure and let $x\in E$. Moreover, let 

\[
\nu_x(y)=\E_x\left[\sum_{k=0}^{H_x-1}\one_{\{ X_k=y\}}\right].
\]

Then for some $C>0$ we get $\mu=C\nu$. We can determine $C$ by 

\[
1=\mu(E)=C\nu_x(E),
\]

which implies that $C=\frac{1}{\nu_x(E)}$ and thus 

\[
\mu(x)=\frac{\nu_x(x)}{\nu_x(E)}=\frac{1}{\nu_x(E)}.
\]

But on the other hand we have 

\[
\nu_x(E)=\sum_{y\in E}\E_x\left[\sum_{k=0}^{H_x-1}\one_{\{ X_k=y\}}\right]=\E_x\left[\sum_{k=0}^{H_x-1}\left(\sum_{y\in E}\one_{\{ X_k=y\}}\right)\right]=\E_x[H_x].
\]

In case $(ii)$, $\nu_x$ is infinite and thus 

\[
\E_x[H_x]=\nu_x(E)=\infty.
\]
\end{proof}

\begin{appendix}
\appendixpage
\noappendicestocpagenum
\addappheadtotoc

\chapter{Measure theory}

\section{Measurable Spaces}

To start with measure theory, we want to handle the abstract setting of a measure space at first. This definitions should lead to a formal understanding of abstract measure theoretical background. The most important notion is that of a $\sigma$-Algebra.

\begin{defn}[$\sigma$-Algebra and measurable sets]

Let $E$ be a Set. A $\sigma$-Algebra $\mathcal{A}$ on $E$ is a collection of subsets of $E$, which satisfies the following conditions.

\begin{enumerate}[$(i)$]
\item{The ground space has to be in $\A$, i.e. $E\in\mathcal{A}$,}
\item{If $A\in\mathcal{A}$ then $A^C\in\mathcal{A}$, where $A^C$ denotes the complement of $A$,}
\item{If $(A_n)_{n\in\mathbb{N}}\subset \mathcal{A}$ is a collection of elements in $\A$ then $\bigcup_{n\in\mathbb{N}}A_n\in\mathcal{A}$.}
\end{enumerate}

Moreover, the elements of $\mathcal{A}$ are called measurable sets. The tupel $(E,\A)$, that is the set $E$ endowed with the $\sigma$-Algebra $\A$, is called a measurable space. 
\end{defn}

\begin{rem} This definition implies the following.

\begin{enumerate}[$(i)$]
\item{Every $\sigma$-Algebra $\A$ is a subset of $\mathcal{P}(E)$, i.e. $\mathcal{A}\subseteq\mathcal{P}(E)$, where $\mathcal{P}(E)$ denotes the power set of $E$, that is the set of all subsets of $E$.}
\item{The empty set has to be in $\A$, i.e. $\varnothing\in\mathcal{A}$,} 
\item{If $(A_n)_{n\in\N}\subset\mathcal{A}$ is a collection of elements of $\A$ then $\bigcap_{n\in\mathbb{N}}A_n\in\mathcal{A}$, i.e. 
$$ \bigcap_{n\in\mathbb{N}}A_n=\left(\bigcup_{n\in\mathbb{N}}A_n^C\right)^C.$$ 
}
\end{enumerate}

\end{rem}

\begin{ex}[Examples of $\sigma$-Algebras]

We give the following simple examples for $\sigma$-Algebras on a set $E$.

\begin{enumerate}[$(i)$]
\item{$\mathcal{A}=\{\emptyset,E\}$ is called the \emph{trivial} or the \emph{smallest} $\sigma$-Algebra on $E$.}
\item{$\mathcal{A}=\mathcal{P}(E)$ is the \emph{largest} $\sigma$-Algebra\footnote{This is convenient for finite and countable measureable spaces} on $E$.}
\item{$\mathcal{A}=\{A\subset E\mid A$ is countable or $A^C$ is countable$\}$}.
\end{enumerate}

\end{ex}

\begin{exer}
Show that the examples above are indeed $\sigma$-Algebras.
\end{exer}

Let us consider a set $A_n\in \A$ for $n\in\mathbb{N}$. The following observation are useful

\begin{enumerate}[$(i)$]
\item{If $A_n$ is a countable set for all $n\in\N$, then $\bigcup_{n\in\mathbb{N}}A_n$ is also a countable set and we know that $$\bigcup_{n\in\mathbb{N}}A_n\in\mathcal{A}.$$}
\item{If there is a $n_0\in\N$ such that $A_{n_0}$ is an uncountable set, it follows that $A_{n_0}^C$ is a countable set, i.e.
$$\left(\bigcup_{n\in\mathbb{N}}A_n\right)^C=\bigcap_{n\in\mathbb{N}}A_n^C\subset A_{n_0}^C,$$ which implies that $\left(\bigcup_{n\in\N}A_n\right)^C$ is countable.}
\end{enumerate}

We can construct many more interesting $\sigma$-Algebras by noting that any arbitrary intersection of $\sigma$-Algebras is again a $\sigma$-Algebra. Let therefore $(\mathcal{A}_i)_{i\in I}$ be a family of $\sigma$-Algebras and $I$ an arbitrary Indexset, then the set

$$\mathcal{A}:=\bigcap_{i\in I}\mathcal{A}_i$$

is also a $\sigma$-Algebra.

\begin{defn}[Generated $\sigma$-Alegbra]

Let $E$ be a set and let $\mathcal{C}$ be a subset of $\mathcal{P}(E)$. Then there exists a smallest $\sigma$-Algebra, denoted by $\sigma(\mathcal{C})$, which contains $\mathcal{C}$. This $\sigma$-Algebra may be defined as

$$\sigma(\mathcal{C})=\bigcap_{\mathcal{C}\subset\mathcal{A}\atop\mathcal{A}\hspace{0.1cm} \text{a $\sigma$-Algebra} }\mathcal{A}.$$

\end{defn}

\begin{rem} We can observe that if $\mathcal{C}$ is a $\sigma$-Algebra itself, then clearly $\sigma(\mathcal{C})=\mathcal{C}$. Moreover, for two subsets $\mathcal{C}\subset\mathcal{P}(E)$ and $\mathcal{C}'\subset\mathcal{P}(E)$ with $\mathcal{C}\subset\mathcal{C}'$ we get that $\sigma(\mathcal{C})\subset\sigma(\mathcal{C}')$.


\end{rem}

\begin{ex}

Let $E$ be a set and let $A\subset E$ be a subset. Moreover, let $\mathcal{C}=A$. Then we would get 

\[
\sigma(\mathcal{C})=\{\varnothing,A,A^C,E\}.
\]

More generally, let $E=\bigcup_{i\in I}E_i$,  where $I$ is a finite or countable index set and $E_i\cap E_j=\varnothing$ for $i\not=j$. Then we call $(E_i)_{i\in I}$ a partition of $E$ and the set 

$$\mathcal{A}=\left\{\bigcup_{j\in J}E_j\mid \hspace{0.2cm}J\subset I\right\}$$

has the structure of a $\sigma$-Algebra. Now let $\mathcal{C}=\left\{\{x\}\mid x\in E\right\}$. Then we would get that 

$$\sigma(\mathcal{C})=\left\{A\subset E\mid \text{$A$ is countable or $A^C$ is countable}\right\}.$$

\end{ex}

\section{Topological Spaces}

The notion of a $\sigma$-Algebra is related to one of the most general constructions, that of a \emph{topology}. To deal with Euclidean spaces and for the description of a natural notion of $\sigma$-Algebra, we need to take a closer look at the \emph{topological} point of view. Therefore we want to describe some point set aspects of topology, basically also introduced in analysis.

\begin{defn}[Topological space]

Let $X$ be a set. A topology on X is a family $\mathcal{O}$ of subsets of $X$ satisfying:

\begin{enumerate}[$(i)$]
\item{The ground set is in $\mathcal{O}$, i.e. $X\in\mathcal{O}$,}
\item{The empty set is in $\mathcal{O}$, i.e. $\varnothing\in\mathcal{O}$,}
\item{(finite intersection) For $O_1,...,O_n\in\mathcal{O}$ with $n\geq1$ we get that $\bigcap_{i=1}^{n}O_i\in\mathcal{O}$,}
\item{(arbitrary union) For $O_i\in\mathcal{O}$ with $i\in I$, where $I$ is any index set $I$, we get that $\bigcup_{i\in I}O_i\in\mathcal{O}$.}
\end{enumerate}

The elements of $\mathcal{O}$ are called open sets and the complements are called closed sets. Moreover, we call the tupel $(X,\mathcal{O})$ a topological space.

\end{defn}

\begin{rem}[Hausdorff]

Let $(X,\mathcal{O})$ be a topological space. The topology $\mathcal{O}$ on $X$ is said to be Hausdorff if and only if for all $x,y\in X$ with $x\not=y$ there is a $O_x\in\mathcal{O}$ with $x\in O_x$ and there exists a $O_y\in\mathcal{O}$ with $y\in O_y$, i.e. $O_x\cap O_y=\varnothing$.

\end{rem}

\begin{defn}[Metric space]

Let $X$ be a set. A metric is a map $d:X\times X\longrightarrow \mathbb{R}_{+}$, which, for all $x,y,z\in X$, satisfies the following.

\begin{enumerate}[$(i)$]
\item{(zero distance) $d(x,y)=0\Longleftrightarrow x=y,$}
\item{(symmetry) $d(x,y)=d(y,x),$}
\item{(triangle inequality) $d(x,y)\leq d(x,z)+d(z, y),$}
\end{enumerate}

We call a set $X$ endowed with a metric, written $(X,d)$, a metric space.
If we have a norm $\|\cdot\|$ on $X$, and we then consider a normed space $(X,\|\cdot\|)$, we get the relation $d(x,y):=\|x-y\|$, which defines a distance.
If $(X,d)$ is a metric space, the topology on $X$ is associated 
with $d$ and is, for some arbitrary index set $I$, given by

$$\mathcal{O}_X^d:=\left\{\bigcup_{i\in I}B_{r_i}(x_i)\mid x_i\in X,\hspace{0.1cm} r_i\in\mathbb{R}_{+}\right\},\hspace{0.5cm}B_{r_i}(x_i):=\left\{y\in X\mid d(x_i,y)<r_i\right\}$$

\end{defn}

\begin{defn}[Basis and separability]

A topological space $(X,\mathcal{O})$ is said to have a countable basis of open sets $\{w_n\}_{n\in\N}$ if for every open set $O\in\mathcal{O}$, there exists a countable index set $I\subset\mathbb{N}$, such that

$$O=\bigcup_{n\in I}w_n.$$

Moreover, a metric space $(X,d)$ is said to be separable, if it contains a sequence $(x_n)_{n\in\mathbb{N}}$ which is dense in $X$, that is, for all $x\in X$ there exists a subsequence $(x_{n_k})_{k\in\N}$ of $(x_n)_{n\in N}$, such that $d(x,x_{n_k})\xrightarrow{k\to\infty}0$.

\end{defn}

\begin{prop}
A metric space is separable if and only if it has a countable basis of open sets.
\end{prop}

\begin{proof}
We first prove the direction $\Longrightarrow$. Therefore we can observe that $\mathscr{B}:=\{B_r(x_n)\}_{r_\in\Q_+,n\in\N}$ is a basis of open sets, and thus we can write every open set as $O=\bigcup_{b\in\mathscr{B}\atop b\subset O}b$. Now let $x\in O$. Then there exists an $\varepsilon>0$, such that $B_\varepsilon(x)\subset O$ and there is a $n_0$ with $d(x_{n_0},x)<\frac{\varepsilon}{4}$ and thus $x\in B_{\frac{\varepsilon}{2}}(x_{n_0})\subset O$. Now we prove the direction $\Longleftarrow$. Let therefore $\{w_n\}_{n\in\N}$ of open sets. Then we can choose a $x_n\in w_n$ and check that the sequence $(x_n)_{n\in \N}$ is a dense subset, which gives the claim. 
\end{proof}

\begin{defn}[Product topology]

Let $(X,\mathcal{O}_X)$ and $(Y,\mathcal{O}_Y)$ be two topological spaces. The product topology for the product space $X\times Y$ is defined, with an arbitrary index set $I$, by the family of open sets

$$\mathcal{O}_{X\times Y}:=\left\{\bigcup_{i\in I}O^X_i\times \mathcal{O}^Y_i\mid O^X\in \mathcal{O}_X,O^Y\in\mathcal{O}_Y\right\}.$$

\end{defn}

\begin{defn}[Continuity]

Let $(X,\mathcal{O}_X)$ and $(Y,\mathcal{O}_Y)$ be two topological spaces. A map $f:(X,\mathcal{O}_X)\longrightarrow(Y,\mathcal{O}_Y)$ is continuous if and only if for all $O^Y \in\mathcal{O}_Y$, the image of $O^Y$ under $f^{-1}$ is open, that is

$$f^{-1}(O^Y)=\{x\in X\mid f(x)\in O^Y\}\in\mathcal{O}_X$$

\end{defn}

\begin{defn}[Canonical projection]

Let $X$ and $Y$ be two sets. Then we can define the canonical projections to be the surjective maps

\begin{align*}
\pi_X: X\times Y &\longrightarrow X\\
(x,y)&\longmapsto x
\end{align*}
\begin{align*}
\pi_Y:X\times Y&\longrightarrow Y\\
(x,y)&\longmapsto y
\end{align*}

\end{defn}

\begin{rem}

A useful observation is that the product topology is defined in such a way that the canonical projections are continuous, that is 

$$\pi_X^{-1}(O^X)=O^X\times Y\in \mathcal{O}_{X\times Y},\hspace{1cm}\pi_Y^{-1}(O^Y)=X\times O^Y\in\mathcal{O}_{X\times Y}$$

\end{rem}

\begin{rem}

We can also define a metric on two metric spaces $(X,d)$ and $(Y,\delta)$ given by

$$D_p((x,y),(x',y'))=(d^p(x,x')+\delta^p(y,y'))^{\frac{1}{p}},$$

for $p\geq 1$.

\end{rem}

\begin{prop} 

Let $(X,\mathcal{O}_X)$ and $(\mathcal{O}_Y)$ be two topological spaces. 
If $(X,\mathcal{O}_X)$ and $(Y,\mathcal{O}_Y)$ have a countable basis of open sets, then $(X\times Y,\mathcal{O}_{X\times Y})$ also has a countable basis of open sets. Moreover, Let $(X,d)$ and $(Y,\delta)$ be two metric spaces. If $(X,d)$ and $(Y,\delta)$ are separable, then $(X\times Y, D_p)$ is also separable.

\end{prop}

\begin{proof}

First, let $\mathcal{U}_X=\{U_n\}_{n\in\N}$ be a basis of open sets on $X$ and $\mathcal{V}_Y=\{V_n\}_{n\in\N}$ a basis of open sets on $Y$. Then $\{U_n\times V_m\}_{(n,m)\in\N^2}$ is a basis of open sets for $X\times Y$, which proves the first claim. We leave the second claim as an exercise for the reader.
\end{proof}

\section{Borel sets}

Topologically, the Borel sets in a topological space are the $\sigma$-Algebra generated by the open sets. One can build up the Borel sets from the open sets by iterating the operations of complementation and taking countable unions.

\begin{defn}[Borel $\sigma$-Algebra]

Let $(E,\mathcal{O})$ be a topological space. Then $\sigma(\mathcal{O})$ is called the Borel $\sigma$-Algebra of $E$ and is denoted by $\mathcal{B}(E)$. Moreover, the elements of $\mathcal{B}(E)$ are called Borel sets.

\end{defn}

\begin{rem}

Observe that if $E=\mathbb{R}$, then $\mathcal{B}(\mathbb{R})\not=\mathcal{P}(\mathbb{R})$. That means that there exist subsets which are not Borel measurable.

\end{rem}

\begin{prop}

Let $(E,\mathcal{O})$ be a topological space with a countable basis of open sets $\{w_n\}_{n\in\N}$. Then 

$$\mathcal{B}(E)=\sigma(\{w_n\}_{n\in\N})$$

\end{prop}

\begin{proof}

Since $\{w_n\}_{n\in\N}\subset\mathcal{O}$, we get that $\sigma(\{w_n\}_{n\in\N})\subset\sigma(\mathcal{O})=\mathcal{B}(E)$. Moreover, since every open set $O$ can be written as $O=\bigcup_{j\in J\subset \N}w_j$, we deduce that for all $O\in\mathcal{O}$ we get $O\in \sigma(\{ w_n\}_{n\in\N})$ and thus $\sigma(\mathcal{O})\subset \sigma(\{w_n\}_{n\in\N})$.

\end{proof}

\begin{rem}
\label{rem1}
An important observation is also that the $\sigma$-Algebra generated by open sets equals the $\sigma$-Algebra generated by closed sets of the form, that is, if we denote by $\mathcal{O}^C:=\{F\subset E\mid \text{$F$ is closed with respect to the topology $\mathcal{O}$}\}$, 

$$\mathcal{B}(E)=\sigma(\{F\}_{F\in \mathcal{O}^C}).$$

\end{rem}

\begin{proof}[Proof of Remark \ref{rem1}]

We show the direction $\Longrightarrow$. For $O\in\mathcal{O}$ set $F:=O^C$, which is closed, that is $F\in \mathcal{O}^C$. The fact that $F$ is closed implies that $F\in \sigma(\{F\}_{F\in\mathcal{O}^C})$ and thus $F^C=O\in\sigma(\{ F\}_{F\in\mathcal{O}^C})$, because of the properties of a $\sigma$-Algebra. Hence $\mathcal{O}\subset\sigma(\{ F\}_{F\in\mathcal{O}^C})$ and therefore $\sigma(\mathcal{O})\subset\sigma(\{ F\}_{F\in\mathcal{O}^C})$. The other direction is similar, hence we leave it as an exercise.
\end{proof}

\begin{rem}
\label{rem2}
Consider the case $E=\mathbb{R}$. Then we would get 

$$\mathcal{B}(\mathbb{R})=\sigma(\{[a,\infty)\}_{a\in \Q})=\sigma(\{(a,\infty)_{a\in \Q})\})=\sigma(\{(-\infty,a]\}_{a\in \Q}))=\sigma(\{(-\infty,a)_{a\in \Q})\}).$$

\end{rem}

\begin{proof}[Proof of Remark \ref{rem2}]

Recall that $\mathbb{Q}$ is a dense subset of $\mathbb{R}$. Therefore it follows that

$$\{(\alpha,\beta)\mid \alpha,\beta \in \mathbb{Q},\alpha<\beta\}=\{(\rho-r,\rho+r)\mid \rho\in\mathbb{Q},r\in\mathbb{Q}_+\}=\{B_r(\rho)\mid \rho\in\mathbb{Q},r\in\mathbb{Q}_+\}$$

is a countable basis of open sets in $\mathbb{R}$ and thus

$$\mathcal{B}(\mathbb{R})=\sigma\left(\{(\alpha,\beta)\}_{\alpha,\beta\in\Q\atop \alpha>\beta}\right).$$

Moreover, it is important to observe that $(\alpha,\beta)=(\alpha,\infty)\cap[\beta,\infty)^C$ with 

\[
(\alpha,\infty)=\bigcup_{n\in\N}\left[\frac{\alpha}{n},\infty\right).
\]

Therefore we get that $(\alpha,\infty)\in\sigma(\{(\alpha,\infty)\}_{\alpha\in\Q})$ and thus $(\alpha,\infty)\cap[\beta,\infty)^C\in\sigma(\{[\alpha,\infty)\}_{\alpha\in\Q})$. It follows from the definition of the Borel $\sigma$-Algebra that

\[
\mathcal{B}(\R)=\sigma(\{(\alpha,\beta)\}_{\alpha,\beta\in\R})\subset\sigma(\{[\alpha,\infty\}_{\alpha\in\R})\subset\sigma(\{ F\}_{F\in\mathcal{O}_\R^C})=\B(\R),
\]

which finally implies that $\sigma(\{[\alpha,\infty)\}_{\alpha\in\R})=\B(\R)$.
\end{proof}

\section{Positive Measures}

\begin{defn}[Positive measure]
Let $(E,\A)$ be a measurable space. A positive measure $\mu$ on $(E,\mathcal{A})$ is an application $\mu:\mathcal{A}\longrightarrow[0,\infty]\subset\overline{\mathbb{R}}$, which satisfies the following.

\begin{enumerate}[$(i)$]
\item{(measure of the empty set is zero) $\mu(\varnothing)=0$,}
\item{($\sigma$-additivity) For all sequences $(A_n)_{n\in\mathbb{N}}\in\mathcal{A}$ of disjoint measurable sets, that is $A_i\cap A_j=\varnothing$ for $i\not=j$, we have

$$\mu\left(\bigcup_{n\in\N}A_n \right)=\sum_{n\in\N}\mu(A_n).$$}

Moreover, we call a triple $(E,\A,\mu)$, that is a measurable space endowed with a specific measure, a measure space.

\end{enumerate}

\end{defn}


\begin{rem}

A nice observation of $\bar\R_+$ is that all sums are convergent, that is, for any sequence $(x_n)_{n\in\N}\subset[0,\infty]$, we get $\sum_{n\in\N}x_n\in[0,\infty]$. We can formulate an equivalent definition of this sum as $$\sum_{n\in\N}x_n=\sup_{I\subset\N\atop \text{$I$ finite}}\sum_{n\in I}x_n.$$
More general, for any sequence $(x_\alpha)_{\alpha\in A}$ of real numbers $x_\alpha\in[0,\infty]$ with an arbitrary index set $A$, either countable or uncountable, we can define $$\sum_{\alpha\in A}x_\alpha:=\sup_{I\subset A\atop\text{$A$ finite}}\sum_{\alpha\in I}x_\alpha.$$ Moreover, for any two index sets $A$ and $B$ and for any bijection $\phi:B\xrightarrow{\sim}A$, we get $$\sum_{\alpha\in A}x_\alpha=\sum_{\beta\in B}x_{\phi(\beta)}.$$
\end{rem}

\begin{prop}

Let $(E,\mathcal{A},\mu)$ be a measure space. Then the following hold.

\begin{enumerate}[$(i)$]

\item{Let $A,B\in\A$ be two measurable sets such that $A\subset B$. Then $\mu(A)\leq \mu(B)$. Moreover, if $\mu(A)<\infty$, then $\mu(B\setminus A)=\mu(B)-\mu(A)$.
}

\item{
(Inclusion-exclusion) Let $A,B\in \A$ be two measurable sets. Then $\mu(A)+\mu(B)=\mu(A\cup B)+\mu(A\cap B)$.
}
\item{
Let $(A_n)_{n\in\N}\subset\A$ be an increasing sequence of measurable sets. Then 

$$\mu\left(\bigcup_{n\in\mathbb{N}}A_n\right)=\lim_{n\to\infty}\uparrow\mu(A_n).$$
}

\item{
Let $(B_n)_{n\in\N}\subset\A$ be a decreasing sequence of measurable sets. Moreover, let $\mu(B_0)<\infty$. Then

$$\mu\left(\bigcap_{n\in\N}B_n\right)=\lim_{n\to\infty}\downarrow\mu(B_n).$$
}

\item{
($\sigma$-subadditivity) Let $(A_n)_{n\in\N}\subset\A$ be a sequence of measurable sets. Then

$$\mu\left(\bigcup_{n\in\mathbb{N}}A_n\right)\leq \sum_{n\in\mathbb{N}}\mu(A_n).$$
}
\end{enumerate}
\end{prop}

\begin{proof}

For $(i)$, observe that $B=(B\setminus A)\sqcup A$ and thus $\mu(B)=\mu(B\setminus A)+\mu(A)$. Moreover, if $\mu(A)<\infty$, then $\mu(B\setminus A)=\mu(B)-\mu(A)$. For $(ii)$, observe that $A\cup B= (A\setminus B)\sqcup (B\setminus A)\sqcup(A\cap B)$ and thus $\mu(A\cup B)=\mu(A\setminus B)+\mu(B\setminus A)+\mu(A\cap B)$. Assume $\mu(A\cap B)=\infty$, then we get $(ii)$, since $\mu(A\cup B),\mu(A),\mu(B)\geq \mu(A\cap B)=\infty$. On the other hand, assume $\mu(A\cap B)<\infty$, then $\mu(A\setminus B)=\mu(A\cap B^C)$ and $\mu(B\setminus A)=\mu(B\cap A^C)$, which implies that $\mu(A\cup B)=\mu(A)+\mu(B)-2\mu(A\cap B)+\mu(A\cap B)$. Rearranging things, we get the claim. For $(iii)$, let $C_0=A_0$ and $C_n=A_n\setminus A_{n-1}$ for all $n\geq 1$. Then we get 

\[
\bigcup_{n\in\N}C_n=\bigcup_{n\in\N}A_n,
\]

where the $C_n$'s are disjoint and moreover, $\mu(C_n)=\mu(A_n)-\mu(A_{n-1})$. Therefore we get 

\[
\mu\left(\bigcup_{n\in\N}A_n\right)=\mu\left(\bigcup_{n\in\N}C_n\right)=\sum_{n\in\N}\mu(C_n)=\lim_{N\to\infty}\sum_{n=0}^N\mu(C_n)=\lim_{N\to\infty}\uparrow\mu(A_N).
\]

For $(iv)$, let $A_n=B_0\setminus B_1$, which implies that $A_n\subset A_{n+1}$ for all $n\in \N$. Thus we get 

\begin{multline*}
\mu(A_n)=\mu(B_0)-\mu\left(\bigcup_{n\in\N}B_n\right)=\mu\left(b_0\setminus \bigcup_{n\in\N}B_n\right)\\=\mu\left(\bigcup_{n\in\N}A_n\right)=\lim_{n\to\infty}\uparrow \mu(A_n)=\lim_{n\to\infty}(\mu(B_0)-\mu(B_n)).
\end{multline*}

Now, since $\mu(B_0)<\infty$, we get that

\[
\mu\left(\bigcap_{n\in\N}B_n\right)=\lim_{n\to\infty}\downarrow\mu(B_n),
\]

where we have used the fact that 

\[
\bigcup_{n\in\N}A_n=\bigcup_{n\in\N}(B_0\setminus B_n)=\bigcup_{n\in\N}(B_0\cap B_n^C)=B_0\cap\left(\bigcup_{n\in\N}B_n^C\right)=B_0\cap\left(\bigcap_{n\in\N}B_n\right)^C=B_0\setminus\bigcap_{n\in\N}B_n,
\]

and that $\bigcap_{n\in\N}\subseteq B_0$. For $(v)$, let $C_0=A_0$ and for $n\geq 1$ let 

\[
C_n=A_n\setminus\bigcup_{k=0}^{n-1}A_k,
\]

where the $C_n$'s are disjoint and moreover, $\bigcup_{n\in\N}A_n=\bigcup_{n\in\N}C_n\subset A_n$. Therefore we get

\[
\mu\left(\bigcup_{n\in\N}A_n\right)=\mu\left(\bigcup_{n\in\N}C_n\right)=\sum_{n\in\N}\mu(C_n)\leq  \sum_{n\in\N}\mu(A_n).
\]

This can also been proved by induction, which we leave as an exercise for the reader.
\end{proof}

\begin{ex}[Dirac measure]

Let $(E,\A)$ be a measurable space such that for any $x\in E$, we get that $\{ x\}\in\A$. we can define a measure $\delta:E\times \A\longrightarrow \{1,0\}$ by 

\[
\delta_x(A)=\one_A(x)=\begin{cases}1,&x\in A\\ 0,& x\not\in A\end{cases} 
\]

This measure is called the \emph{Dirac measure} or the \emph{Dirac mass at $x$}. More generally, if we consider sequences $(x_n)_{n\in\N}\subset E$ and $(\alpha_n)_{n\in\N}\subset [0,\infty]$, we can define a measure $\mathscr{D}_{(x_n)}^{(\alpha_n)}:\A\longrightarrow \bar\R_+$, which is defined by 

\[
\mathscr{D}_{(x_n)}^{(\alpha_n)}(A)=\left(\sum_{n\in\N}\alpha_n\delta_{x_n}\right)=\sum_{n\in\N}\alpha_n\delta_{x_n}(A).
\]
\end{ex}

\begin{ex}[Lebesgue measure]

There exists a unique measure on the measurable space $(\mathbb{R},\mathcal{B}(\mathbb{R}))$, which is denoted by $\lambda$, such that for all open intervals $(a,b)\in \B(\R)$ it is given by 

$$\lambda((a,b))=b-a.$$

\end{ex}

\begin{defn}[Finite-, $\sigma$-finite- and probability measures]

Let $(E,\A)$ be a measurable space. We say that a measure $\mu$ is 

\begin{enumerate}[$(i)$]
\item{finite if $\mu(E)<\infty$.}
\item{a probability measure if $\mu(E)=1$.}
\item{$\sigma$-finite if there exists an increasing sequence (partition of the total space) $(E_n)_{n\in\N}\subset\mathcal{A}$, such that $E=\bigcup_{n\in\N}E_n$ and 
with $\mu(E_n)<\infty$ for all $n\in\N$.}
\end{enumerate}
\end{defn}

\begin{defn}[Atom]

Let $(E,\A,\mu)$ be a measure space. An element $x\in E$ is called an atom for $\mu$ if the set $\{x\}\in\A$ and $$\mu(\{x\})>0.$$
\end{defn}

%
%

\begin{defn}[Product $\sigma$-Algebra]

Let $(E_1,\mathcal{A}_1)$ and $(E_2,\mathcal{A}_2)$ be two measurable spaces. Then we can define the product $\sigma$-Algebra $\A_1\otimes\A_2$ on the product space $E_1\times E_2$ by

\[
\A_1\otimes\A_2=\sigma(A_1\times A_2),
\]

where $A_1\in\A_1$ and $A_2\in\A_2$. This is actually the $\sigma$-Algebra which contains all sets of the form $(A_1\times A_2)$.

\end{defn}

Let $(E,\A)$ and $(F,\B)$ be two measure spaces. Consider a map $f:E\longrightarrow F$. Moreover, let $I$ be an arbitrary index set and for $i\in I$, let $A_i\subset E$ and $B_i\subset F$. We can write, for $A\subset E$

\[
f(A):=\{ f(x)\mid x\in A\},
\]

and similarly, for $B\subset F$, we can write 

\[
f^{-1}(B)=\{ x\in E\mid f(x)\in B\}.
\]

Moreover, it is easy to observe the following relations.

\begin{enumerate}[$(i)$]
\item{$f\left(\bigcup_{i\in I}A_i\right)=\bigcup_{i\in I}f(A_i).$}
\item{$f\left(\bigcap_{i\in I}A_i\right)\subseteq \bigcap_{i\in I}f(A_i). \hspace{0.3cm}\text{(where equality holds if $f$ is injective)}$}
\item{$f^{-1}\left(\bigcup_{i\in I}B_i\right)=\bigcup_{i\in I}f^{-1}(B_i).$
}
\item{$f^{-1}\left(\bigcap_{i\in I}B_i\right)=\bigcap_{i\in I}f^{-1}(B_i).$}
\item{$f^{-1}(B)^C=f^{-1}(B)^C.$}
\item{

If $\mathcal{C}\subset \mathcal{P}(F)$, then $f^{-1}(\mathcal{C}):=f^{-1}(C)\mid C\in\mathcal{C}\}$.

}
\end{enumerate}

\begin{prop}

Let $E$ and $F$ be two measurable spaces, where $\mathcal{B}$ is a $\sigma$-Algebra on $F$. Then $$\A:=f^{-1}(\B)=\{f^{-1}(B)\mid B\in \B\}$$ is a $\sigma$-Algebra on $E$.

\end{prop}

\begin{proof}

First, it is obvious that $f^{-1}(F)=E$, which implies that if $F\in\B$ then $E\in \A$. Moreover, it holds that $$f^{-1}(B)^C=f^{-1}(B^C)\in\A$$ for all $B\in\B$ since arbitrary unions of elements in $\B$ are again in $\B$.

\end{proof}

\begin{rem}
It is sometimes usual to write $\sigma(f)$ instead of $f^{-1}(B)$.

\end{rem}

\begin{ex}

Let $(E,\A)$ be a measurable space and let $F\subset E$ be a subset of $E$. Moreover, let $\iota:F\hookrightarrow (E,\A)$ be the canonical injection. Then we get 

\[
\iota^{-1}(\A)=\{\iota^{-1}(A)\mid A\in\A\}=\{F\cap A\mid A\in\A\}.
\]
\end{ex}

\begin{ex}

Let $(E,\A)$ be a measurable space and let $F\subset E$ be a subset of $E$. Moreover, let $\pi_E:E\times F\longrightarrow (E,\A)$ be the canonical projection. Then we get 

\[
\pi_E^{-1}(\A)=\{\pi_E^{-1}(A)\mid A\in\A\}=\{A\times F\mid A\in\A\}.
\]

\end{ex}

\begin{defn}[Image $\sigma$-Algebra]

Let $(E.\A)$ and $(F,\B)$ be measurable spaces and let $f:E\longrightarrow F$ be a map. The image $\sigma$-Algebra of $\A$ by $f$ is defined by 

\[
\mathcal{I}=\{I\in\mathcal{P}(F)\mid f^{-1}(I)\in\A\}.
\]

\end{defn}

\begin{prop} 

Let $(X,d)$ be a metric space and let $Y\subset X$. Then the Borel $\sigma$-Algebra of $Y$ is given by 

\[
\B(Y)=\{A\cap Y\mid A\in\B(X)\}.
\]

Moreover, if $Y\in\B(X)$ then $\B(Y)\subset \B(X)$ and $\B(Y)=\{A\in\B(X)\mid A\subset Y\}$.

\end{prop}

\begin{proof}
Let $\iota:Y\hookrightarrow X$ be the canonical injection of $Y$ into $X$. Then 

\[
\mathcal{O}_Y:=\{O\cap Y\mid O\in\mathcal{O}_X\}=\iota^{-1}(\mathcal{O}_X).
\]

Moreover, we get 

\[
\B(Y)=\sigma(\mathcal{O}_Y)=\sigma(\iota^{-1}(\mathcal{O}_X))=\iota^{-1}(\sigma(\mathcal{O}_X))=\iota^{-1}(\B(X))=\{A\cap Y\mid A\subset \B(X)\},
\]

which proves the first part of the proposition. The second part is easily obtained from the fact that $\sigma$-Algebras are stable under finite intersections.

\end{proof}

We have the following examples of Borel $\sigma$-Algebras.

\begin{ex}
Let $X=\R_+$. Then 
 
\[
\B(\R_+)=\{A\subseteq \R_+\mid A\in\B(\R)\}.
\]

\end{ex}

\begin{ex}

Let $X=\R^\times=\R\setminus\{0\}$. Then 

\[
\B(\R^\times)=\{A\in\B(\R)\mid 0\not\in A\}.
\]

\end{ex}

\begin{ex}[Borel sets on $\bar \R$]

Let us define $\bar \R=\R\cup\{-\infty,\infty\}$ and let us consider the map 

\begin{align*}
f:\R&\longrightarrow (-1,1)\\
x&\longmapsto \frac{x}{\sqrt{x^2+1}}
\end{align*}

We can now consider an extension $\tilde f$ of $f$, which is defined on $\bar \R$ such that $\tilde f\mid_\R=f$ with $\tilde f(-\infty)=-1$ and $\tilde f(\infty)=1$. Moreover, we can consider $\bar \R$ as a metric space by the considering the distance given for all $x,y\in\bar \R$ as $\|\tilde f(x)-\tilde f(y)\|$. We write therefore $(\bar\R,\delta)$ as a metric space with the metric $\delta$. Thus we can define the Borel $\sigma$-Algebra of $\bar \R$ by the Borel sets, which are described by the metric topology of $\bar \R$. This concept is important as we will work many times with the space $\bar \R$.

\end{ex}

\begin{rem}

It is useful to note that $\bar \R$ describes a totally ordered set, since $\leq $ arises with the usual naturalness as in $\R$. Moreover, the identity map 

\begin{align*}
id:(\R,\delta\mid_\R)&\longrightarrow (\R,\|\cdot\|)\\
x&\longmapsto x
\end{align*}

is a homeomorphism. Another useful observation is that $(\bar \R,\delta)$ is a compact space and homeomorphic to the interval $[-1,1]$ and eventually $\R$ is an open subset of $\bar\R$.

\end{rem}

\begin{exer}

Show that 

\[
\B(\bar \R)=\sigma(\{[a,\infty]\}_{a\in\Q})=\sigma(\{(a,\infty]\}_{a\in\Q}).
\]
\end{exer}

\section{Measurable Maps}

Measure theory and the notion of integration require special structures on functions, which need to satisfy different properties, such as being measurable or bounded. The notion of a measurable map is important for the study of the integration with respect to a certain measure and the fact that we can only consider integration with respect to a measure if the integrating function satisfies measurability. It is now important to use the $\sigma$-Algebras of the underlying spaces similar to the topological notion of continuity where the topologies of the underlying spaces are used. Let us therefore define a measurable map. 

\begin{defn}[Measurable Map]

Let $(E,\A)$ and $(F,\B)$ be two measurable spaces and let $f:E\longrightarrow F$ be a map. We say that $f$ is measurable, if for all $B\in\B$ we get $f^{-1}(B)\in\A$.

\end{defn}

\begin{prop}
Let $(X,\A),(Y,\B)$ and $(Z,\mathcal{G})$ be measurable space and consider the composition 

\[
(X,\A)\xrightarrow{f}(X,\B)\xrightarrow{g}(Z,\mathcal{G}).
\]

If $f$ and $g$ are both measurable, then $g\circ f$ is also measurable.

\end{prop}

\begin{proof}

Exercise.\footnote{Use the definition of a measurable map.}

\end{proof}

\begin{prop}
Let $(E,\A)$ and $(F,\B)$ be two measurable spaces and let $f:E\longrightarrow F$ be a map. Moreover, assume that there exists $\mathcal{C}\subset\mathcal{P}(F)$ such that $\sigma(\mathcal{C})=\B$. Then $f$ is measurable if and only if for all $C\in\mathcal{C}$ we have $f^{-1}(C)\in\A$.

\end{prop}

\begin{proof}

Let us first define the $\sigma$-Algebra $\mathcal{G}$ by 

\[
\mathcal{G}:=\{B\in\B\mid f^{-1}(B)\in\A\}\supset \mathcal{C}.
\]

Now, since $\mathcal{G}$ is a $\sigma$-Algebra, we get that $\sigma(\mathcal{C})\subset\mathcal{G}$ and thus $\mathcal{G}=\B$, which proves the claim.

\end{proof}

\begin{ex}

Let\footnote{We can also take $(\bar \R,\B(\bar \R))$.} $(F,\B)=(\R,\B(\R))$. To show that $f$ is measurable, it is enough to show either that $f^{-1}((a,b))\in\A$ or $f^{-1}((-\infty,a))\in\A$, for $a,b\in\R$ with $a<b$.

\end{ex}

\begin{ex}[continuous maps are measurable]
Assume that $E$ and $F$ are two metric spaces (or topological spaces), endowed with their Borel $\sigma$-Algebra respectively. Then $f$ is measurable if for every open set $O$ of $F$ we have $f^{-1}(O)\in\B(E)$. In particular we can say that \emph{continuous maps are measurable maps}.
\end{ex}

\begin{ex}

Let $A\subset E$ be a subset of $E$. Then the map $\one_A$ is measurable if and only if $A\in\A$.

\end{ex}

\begin{rem}

The notion of measurability of a map $f:E\longrightarrow F$, between two measurable spaces $(E,\A)$ and $(F,\B)$, means that $f^{-1}(\B)\in \A$. The smallest $\sigma$-Algebra on $E$ which makes $f$ measurable is given by $f^{-1}(\B)$ and we denote it by $\sigma(f)$. Moreover, we want to emphasize that we can write $\{f\in B\}$ for $f^{-1}(B)=\{x\in E\mid f(x)\in B\}$. Hence we can write $\{f\geq b\}$ instead of $f^{-1}([b,\infty))$ or $\{ f=b\}$ instead of $f^{-1}(\{b\})$. If $f$ is constant and if for some $C\in F$, we have $f(x)=C$ for all $x\in E$, then $f$ is always measurable, since $f^{-1}(\B)=\{\varnothing, E\}$.

\end{rem}

\begin{lem}

Let $(E,\A),(F_1,\B_1)$ and $(F_2,\B_2)$ be measurable spaces and let $f_1:E\longrightarrow F_1$ and $f_2:E\longrightarrow F_2$ be two measurable maps. Then  the map 

\begin{align*}
f:(E,\A)&\longrightarrow (F_1\times F_2,\B_1\otimes \B_2)\\
x&\longmapsto (f_1(x),f_2(x))
\end{align*}

is measurable.

\end{lem}

\begin{proof}

Let us define $\mathcal{C}:=\{B_1\times B_2\mid\hspace{0.1cm}B_1\in\mathcal{B}_1,\hspace{0.1cm}B_2\in\mathcal{B}_2\}$. Then we get, by definition of the product $\sigma$-Algebra, that $\sigma(\mathcal{C})=\mathcal{B}_1\otimes\mathcal{B}_2$. Now, for $B_1\times B_2\in \mathcal{C}$ we get that $f^{-1}(B_1\times B_2)=\underbrace{\underbrace{f^{-1}(B_1)}_{\in\mathcal{A}}\cap \underbrace{f^{-1}(B_2)}_{\in\mathcal{A}}}_{\in\mathcal{A}}\in\A$. Therefore, it follows that $f$ is measurable.
\end{proof}

\begin{rem}

Consider $f_1=\pi_1\circ f$ and $f_2=\pi_2\circ f$ with 

\begin{align*}
\pi_i:F_1\times F_2&\longrightarrow F_i\\
(y_1,y_2)&\longmapsto y_i
\end{align*}

for $i\in\{1,2\}$ with $\pi_1$ and $\pi_2$ are measurable. Then $f_1$ and $f_2$ are measurable.

\end{rem}

\begin{cor}
Let $(E,\A)$ be a measurable space and let $f,g:E\longrightarrow \R$ be two measurable maps, where $\R$ is endowed with its Borel $\sigma$-Algebra $\B(\R)$. Then 

\begin{enumerate}[$(i)$]
\item{$f+g$}
\item{$f\cdot g$}
\item{$f^+=\max\{f,0\}$}
\item{$f^-=\max\{-f,0\}$}
\item{$\vert f\vert$}
\end{enumerate}

are measurable, where $f = f^+-f^-$ and $\vert f\vert =f^++f^-$.

\end{cor}

\begin{proof}
We will only show $(i)$ and leave the other points as an exercise for the reader. The map $f+g$ is a composition of the map $h:x\longmapsto (f(x),g(x))$ and $r:(a,b)\longmapsto a+b$. The map $h$ is clearly measurable, since $f$ and $g$ are measurable and the map $r$ is clearly continuous and thus measurable. As we have seen, the composition of two measurable maps is again measurable, which shows that $f+g$ is measurable. The proof of the other points is similar. 

\end{proof}

\begin{rem}

Let us consider the field of complex numbers $\C$ and make the identification $\C\simeq \R^2$. Then we can naturally make sense of the measurability of the map $f:E\longrightarrow \C$, where $\C$ is endowed with its Borel $\sigma$-Algebra $\B(\C)$, by saying that $f$ is measurable if and only if $Re(f)$ and $Im(f)$ are measurable.
\end{rem}

\section{The Theorems of Lusin and Egorov}

There are two important theorems which make statements about convergence types of measurable maps. They are important to understand the behavior of sequence of measurable maps and to understand the importance of uniform convergence.

\begin{thm}[Egorov]

Let $(E,\A,\mu)$ be a measure space. Let $f_k:E\longrightarrow \bar \R$ be measurable for all $k\in\mathbb{N}$ and $f:E\longrightarrow\bar \R$ be measurable and $\mu$-a.e. finite. Moreover $f_k(x)\xrightarrow{k\to\infty} f(x)$ $\mu$-a.e. for $x\in E$. Then for all $\delta>0$ there exists $F\subset E$,  with $F$ compact and $\mu(E\setminus F)<\delta$ and 

$$\sup_{x\in F}\vert f_k(x)-f(x)\vert\xrightarrow{k\to\infty} 0,$$

i.e. $(f_k)_{k\in\mathbb{N}}$ converges uniformly to $f$ in $F$.

\end{thm}

\begin{proof}

Let $\delta >0$. For $i,j\in\mathbb{N}$ set 

$$C_{i,j}:=\bigcup_{k=j}^{\infty}\{x\in E\mid \vert f_k(x)-f(x)\vert > 2^{-1}\}.$$

$C_{i,j}$ is $\mu$-measurable, because $f$ and $f_k$ are $\mu$-measurable and $C_{i,(j+1)}\subset C_{i,j}$, $\forall i,j$.
We also know that $f_k(x)\xrightarrow{k\to\infty}f(x)$ for $\mu$-a.e. $x\in E$ and since $\mu(E)<\infty$ it follows that for all $i\in\N$

$$\lim_{j\to\infty}\mu(C_{i,j})=\mu\left(\bigcup_{j=1}^\infty C_{i,j}\right)=0.$$

So for every $i$ there exists a $N(i)\in\N$ with 

$$\mu(C_{i,N(i)})<\delta\cdot 2^{-i-1}.$$

Now set $A=E\setminus \bigcup_{i=1}^\infty C_{i,N(i)}$. Then 

$$\mu(E\setminus A)\leq \sum_{i=1}^{\infty}\mu(C_{i,N(i)})<\delta/2,$$

and for all $i\in\N$ and $k\geq N(i)$

$$\sup_{x\in A}\vert f_k(x)-f(x)\vert \leq 2^{-i}.$$

Choose a $F\subset A$, where $F$ is compact with $\mu(A\setminus F)<\delta/2$. Hence we have

$$\mu(E\setminus F)\leq \mu(E\setminus A)+\mu(A\setminus F)<\delta.$$

\end{proof}

\begin{thm}[Lusin]

Let $(E,\A,\mu)$ be a measure space. Let $f:E\longrightarrow\bar \R$ be measurable and $\mu$-a.e. finite. Then for all $\delta>0$ there exists $F\subset E$, $F$ compact with $\mu(E\setminus F)<\delta$ and $f\mid_F:F\longrightarrow \R$ is continuous.

\end{thm}

\begin{proof}

We split the proof onto two parts.

\begin{enumerate}[$(i)$]

\item{We are going to show this theorem for step functions of the form

$$g=\sum_{i=1}^Ib_i\one_{B_i},$$

where we set $E=\bigsqcup_{i=1}^{I}B_i$ with $B_i\cap B_j=\varnothing$ for $i\not=j$. For $\delta>0$ choose $F_i\subset B_i$ compact with 

$$\mu(B_i\setminus F_i)<\delta\cdot 2^{-i},\hspace{0.2cm}1\leq i\leq I$$

Since the sets $B_i$ are disjoint, it follows that the sets $F_i$ are also disjoint, because of the fact that they are also compact it follows that $d(F_i,F_j)>0$ for $i\not=j$. Therefore we notice that $g$ is locally constant, i.e. continuous on $F:=\bigcup_{i=1}^IF_i\subset E.$ Moreover $F\subset E$ and 

$$\mu(E\setminus F)=\mu\left(\bigcup_{i=1}^I(B_i\setminus F_i)\right)\leq \sum_{i=1}^I\mu(B_i\setminus F_i)<\delta.$$

}
\item{Let $f_k:E\longrightarrow \R$ be a step function with

$$f(x)=\lim_{k\to\infty}f_k(x),\hspace{0.2cm}x\in E,$$

where 

$$f_k=\sum_{j=1}^k\frac{1}{j}\one_{A_j}=\sum_{i=1}^{I_k}b_{ik}\one_{B_{ik}},\hspace{0.2cm}k\in\N,$$

with $B_{ik}\cap B_{jk}=\varnothing$ for $i\not=j$ and $\bigsqcup_{i=1}^{I_k}B_{ik}=E$ and with 

$$b_{ik}=\sum_{B_{ik}\subset A_j}\frac{1}{j},\hspace{0.2cm}1\leq i\leq I_k,\hspace{0.2cm}k\in\N.$$

For $\delta>0$, $g=f_k$ choose compact sets $F_k\in E$ as in part $(i)$ with 

$$\mu(E\setminus F_k)<\delta\cdot 2^{-k-1},\hspace{0.2cm}f_k|_{F_k}:F_k\longrightarrow \R\hspace{0.1cm}\text{continuous,}\hspace{0.1cm}k\in\N.$$

Choose also $F_0\subset E$ compact with 

$$\mu(E\setminus F_0)<\delta/2,\hspace{0.2cm}\sup_{x\in F_0}\vert f_k(x)-f(x)\vert\xrightarrow{k\to\infty} 0.$$

Finally let $F=\bigcap_{k=0}^\infty F_k\subset E$ Note that $F$ is compact with 

$$\mu(E\setminus F)\leq \mu\left(\bigcup_{k=0}^\infty(E\setminus F_k)\right)\leq \sum_{k=0}^\infty\mu(E\setminus F_k)<\delta.$$

and because of the fact that $F\subset F_0$ it follows that 

$$\sup_{x\in F}\vert f_k(x)-f(x)\vert\xrightarrow{k\to\infty} 0.$$

The continuity of $f_k\mid_{F}$, $k\in\N$, gives us now the continuity of $f\mid_F:F\longrightarrow \R.$

}

\end{enumerate}
\end{proof}

\section{The limit superior and limit inferior}

The notion of a limit plays a very big role in measure and integration theory. The way how limits interact with integrals and how they behave under certain situations (for example changing the order of taking limits and integrating) lead to the famous limit theorems of Lebesgue integration. We need to recall the notion of the limsup and the one for the liminf in order to get a better intuition of how sequences of measurable function behave. Let therefore $(a_n)_{n\in\mathbb{N}}$ be a sequence in $\bar \R$ and define the limit superior and the limit inferior of $(a_n)$ as 

$$\limsup_{n\to\infty} a_n=\lim_{n\rightarrow\infty}\swarrow\left(\sup_{k\geq  n} a_k\right)=\inf_n \sup_{k\geq n}a_k$$
$$\liminf_{n\to\infty} a_n=\lim_{n\rightarrow\infty}\nearrow\left(\inf_{k\geq n} a_k\right)=\sup_n\inf_{k\geq n}a_k$$

\begin{rem}
Note the crucial thing that the above limits always exists in $\bar \R$.
\end{rem}

\begin{prop}

Let $(E,\A)$ be a measurable space and let $(f_n)_{n\in\N}$ be a sequence of measurable maps such that $f_n:E\longrightarrow \bar \R$. Then 

\begin{enumerate}[$(i)$]
\item{$\sup_n f_n$}
\item{$\inf_n f_n$}
\item{$\limsup_n f_n$}
\item{$\liminf_n f_n$}
\end{enumerate}

are measurable. Im particular, if $f_n\xrightarrow{n\to\infty} f$ then $f$ is also measurable. In general we can say that $\{x\in E\mid \lim_{n\to\infty}f_n(x)\text{ exists}\}$ is measurable.

\end{prop}

\begin{proof}

Let us first define $f(x):=\inf_n f_n(x)$. Now we see that it is enough to show that for all $a\in\R$ we have $f^{-1}([-\infty,a))\in\A$. Indeed, we can observe that 

\[
f^{-1}([-\infty,a))=\left\{ x\in E\mid \inf_n f_n(x)<a\right\}=\bigcup_n\left\{ x\in E\mid f_n(x)<a\right\}\in\A,
\]

and therefore we can say that $f^{-1}([-\infty,a))=\bigcup_n\left\{ f_n<a\right\}$. Moreover, we have that 

\[
\left\{ x\in E\mid \lim_{n\to\infty} f_n(x) \text{ exists}\right\}=\left\{ x\in E\mid \liminf_n f_n(x)=\limsup_n f_n(x)\right\} =G^{-1}(\Delta)\in \A,
\]

where $G$ is the map given by $G:x\longmapsto (\liminf_n f_n(x),\limsup_n f_n(x))$ and $\Delta$ is the diagonal of $\bar \R^2$, which is closed and hence measurable.

\end{proof}

\begin{ex}
If a map $f:\R\longrightarrow \R$ is differentiable, its derivative $f'$ will be measurable and we can hence write it as a limit of measurable functions as 

\[
f'(x)=\lim_{n\to\infty}n\left(f\left(x+\frac{1}{n}\right)-f(x)\right).
\]

\end{ex}

\begin{rem}

Let $(E,\A)$ be a measurable space and let $(f_n)_{n\in\N}$ be a sequence of measurable functions $f_n:E\longrightarrow X$ to some space $X$. If $\bar \R$ is described as the metric space $(X,d)$, we get that $f_n\xrightarrow{n\to\infty}f$ implies that $f$ is measurable. Moreover, $f_n\xrightarrow{n\to\infty}f$ if and only if for all $x\in E$ we get $\lim_{n\to\infty}f_n(x)=f(x)$ if and only if for all $x\in E$ we get $\lim_{n\to\infty}d(f_n(x),f(x))=0$. If we consider a closed set $F\subset X$, we get 

\begin{align*}
f^{-1}(F)&=\left\{x\in E\mid d(f(x),F)=0\right\}=\left\{ x\in E\mid \lim_{n\to\infty}d(f_n(x),F)=0\right\}\\
&=\left\{x\in E\mid \forall p\geq 1,\exists N\in\N \text{ such that $\forall n\geq N$ } d(f_n(x),F)\leq  \frac{1}{p}\right\}\\
&=\bigcap_{p\geq 1}\bigcup_{N\in\N}\bigcap_{n\geq N}\left\{ x\in E\mid d(f_n(x),F)\geq \frac{1}{p}\right\}\in\A.
\end{align*}

If we consider a complete metric space $(E,d)$, then one can show that $\{x\in E\mid f_n(x) \text{ converges}\}\in \A$. We leave this as an exercise for the reader.
\end{rem}

\begin{defn}[Push-forward measure]

Let $(E,\A)$ and $(F,\B)$ be two measurable spaces and let $\mu$ be a positive measure on $(E,\A)$. Moreover, let $f:E\longrightarrow F$ be a measurable map. Then the push-forward of the measure $\mu$ by $f$, denoted by $f_*\mu$ is defined for all $B\in\B$ as 

\[
f_*\mu(B):=\mu(f^{-1}(B)).
\]

\end{defn}

\section{Simple Functions}

After we have developed the notion of a measurable map, we need to discuss a class of very powerful and, as the name points out, \emph{simple functions}. The advantage of these type of functions are exactly the fact that they are simple to handle and moreover one can basically proof many things for measurable functions by proving it for simple functions and deduce the general case out of that. We will later see the advantage of them being \emph{dense} in different spaces. Let us start with the definition of a simple function.

\begin{defn}[Simple function]

Let $(E,\A)$ be a measurable space. A map $f:E\longrightarrow \R$ is called simple, if it is measurable and if it takes a finite number of values. Recall again that $\R$ is considered as a measurable space endowed with its Borel $\sigma$-Algebra $\B(\R)$.
\end{defn}

\begin{rem}
\label{rem3}
By definition, one can therefore write any simple function $f$ as 

\[
f=\sum_{i\in I}\alpha_i\one_{A_i},
\]

where $I$ is a finite index set, $\alpha_i$ are real numbers and the sets $(A_i)_{i\in I}$ form a $\A$-measurable partition of $E$, i.e. $E=\bigcup_{i\in I}A_i$, $A_i\cap A_j=\varnothing$ if $i\not=j$ and $A_i\in \A$ for all $i\in I$.
\end{rem}

\begin{proof}[Proof of Remark \ref{rem3}]

Note first that for all $x\in E$, we get that $f(x)\in\{\alpha_1,...,\alpha_k\}$, where the $\alpha_i$'s are distinct if and only if $f^{-1}(\{\alpha_i\})=\{x\in E\mid f(x)=\alpha_i\}=A_i\in\A$, since $f$ is measurable. Hence we get that $A_i\cap A_j=\varnothing$ for $i\not=j$ and $\bigcup_{i=1}^kA_i=E$. This representations is unique if the $\alpha_i's$ are distinct. We can write the canonical form therefore as 

\[
f=\sum_{\alpha\in f(E)}\alpha\one_{f=\alpha}.
\]

\end{proof}

\begin{rem}

We can notice the fact that the simple functions form a \emph{commutative algebra}. indeed, let $f=\sum_{i\in I}\alpha_i\one_{A_i}$ and $g=\sum_{j\in J}\beta_j\one_{B_j}$ be two simple functions (in canonical form) and let $\lambda\in\R$. Then we can easily obtain that 

\[
\lambda f+g=\sum_{i\in I\atop j\in J}(\lambda\alpha_i+\beta_j)\one_{A_i\cap B_j} 
\]

and since $(A_i\cap B_j)_{(i,j)\in I\times J}$ forms an $\A$-partition of $E$, we get that 

\[
f\cdot g=\sum_{i\in I\atop j\in J}\alpha_i\beta_j\one_{A_i\cap B_j}
\]

is a simple function and moreover, 

\[
\max(f,g)=\sum_{i\in I\atop j\in J}\max(\alpha_i,\beta_j)\one_{A_i\cap B_j}.
\]

\end{rem}

\begin{thm}

Let $(E,\A)$ be a measurable space and let $f:E\longrightarrow \R$ be a measurable map. Then there exists a sequence of simple functions $(f_n)_{n\in\N}$ such that for all $x\in E$ we get 

\[
f_n(x)\xrightarrow{n\to\infty}f(x).
\]

Moreover, if 

\begin{enumerate}[$(i)$]
\item{$f\geq 0$, we can choose an increasing $f_n\geq 1$ ($0\leq  f_n\leq  f_{n+1}$).
}
\item{$f$ is bounded, $f_n$ can be chosen such that the convergence is uniformly, i.e. 

\[
\sup_{x\in E}\vert f_n(x)-f(x)\vert\xrightarrow{n\to\infty}0.
\]
}
\end{enumerate}

\end{thm}

\begin{proof}

Let us first assume that $f\geq 0$. For $n\in\mathbb{N}$, we set

$$E_{n,\infty}:=\{f\geq n\},\hspace{0.3cm}E_{n,k}:=\left\{\frac{k}{2^n}\leq f<\frac{k+1}{2^n}\right\},\hspace{0.4cm}k\in\{0,1,...,n2^n-1\}.$$

Thus we have $E_{n,k},E_{n,\infty}\in\mathcal{A}$. Now define

$$f_n:=\sum_{k=0}^{n2^n-1}\frac{k}{2^n}\cdot\one_{E_{n,k}}+n\one_{E_{n,\infty}}$$

and obtain that $f_n$ is simple by construction. For $x\in E_{n,k}$ we get

$$f_{n+1}(x)=\begin{cases}f_n(x), &\frac{2k}{2^{n+1}}\leq f(x)<\frac{2k+1}{2^{n+1}}\\ f_n(x)+\frac{1}{2^{n+1}}, &\frac{2k+1}{2^{n+1}}\leq f(x)<\frac{2(k+1)}{2^{n+1}}\end{cases}$$

and if $x\in E_{n,\infty}$ we get

$$f_{n+1}(x)=\begin{cases} n+1, & f(x)\geq n+1\\ \frac{n2^{n+1}+l}{2^{n+1}}, &\frac{n2^{n+1}+l}{2^{n+1}}\leq f(x)<\frac{n2^{n+1}+l+1}{2^{n+1}}\end{cases}$$

It follows that $0\leq f_n(x)\leq f_{n+1}(x)$ for all $ x\in E$.
If furthermore $x\in \{f<n\}$,  then $$0\leq f(x)-f_n(x)\leq 2^{-n}=\frac{1}{2^n}=\frac{k+1}{2^n}-\frac{k}{2^n}\xrightarrow{n\to\infty} 0$$ or equivalently

$$f_n(x)\xrightarrow{n\to\infty} f(x)$$

on $\left\{x\in E\mid f(x)<\infty\right\}=\bigcup_{k\geq 1}\{f<k\}$, and if $x\in \{f=\infty\}=\bigcap_{n\in\mathbb{N}}\{f\geq n\}$, then

$$f_n(x)=n\xrightarrow{n\to\infty}\infty.$$

If we have a function $f:E\longrightarrow \mathbb{R}^+$ and we assume that there exists aome $M\in(0,\infty)$ such that $0\leq f_n(x)\leq M$ for all $x\in E$, then for $n>M$ with $\{f\geq n\}=\varnothing$ it follows that for all $x\in E$ we get

$$0\leq f(x)-f_n(x)\leq 2^{-n},$$

which implies that

$$\sup_{x\in E}\vert f(x)-f_n(x)\vert \leq 2^{-n}\xrightarrow{n\to\infty} 0.$$

Let us emphasize the real case. If we have a function $f:E\longrightarrow \overline{\mathbb{R}}$ we have the decompositions as

$$f=f^+-f^-,\hspace{0.4cm}\vert f\vert =f^1+f^-,\hspace{0.4cm}\text{with}\hspace{0.2cm}f^+,f^-\geq 0.$$

If we now take $f_n^+$ and $f_n^-$ as constructed above we can obtain $f_n^+ \uparrow f^+$ and $f_n^- \uparrow f^-$ for $n\to\infty$. Moreover, we notice that for $x\in E$ the sequences $(f_n^+(x))_{n\in\N}$ and $(f_n^-(x))_{n\in\N}$ cannot be simultaneously nonzero and therefore

$$f_n=f_n^+-f_n^-\xrightarrow{n\to\infty} f=f^+-f^-.$$
\end{proof}

\section{Monotone classes}

A very important notion is that of a monotone class. We will see that there are many things which can be deduced by using the \emph{monotone class lemma}. 

\begin{defn}[Monotone Class]

Let $E$ be some topological space and let $M\subset\mathcal{P}(E)$. $M$ is called a monotone class if the following holds.

\begin{enumerate}[$(i)$]
\item{$E\in M$.}
\item{Let $A\in M$ and $B\in M$. If $A\subset B$ $\Longrightarrow B\setminus A\in M$.}
\item{Let $(A_n)_{n\in\N}\in M$. If $A_n\subset A_{n+1}$ $\Longrightarrow \bigcup_{n\in\mathbb{N}}A_n \in M$.}
\end{enumerate}
\end{defn}

\begin{rem}
A $\sigma$-Algebra is a monotone class\footnote{Prove that as an exercise.}
\end{rem}

\begin{rem}
As for $\sigma$-Algebras, we notice that an arbitrary intersection of monotone classes is again a monotone class. Thus, if $\mathcal{C}\in \mathcal{P}(E)$, we can define the monotone class generated by $\mathcal{C}$ as

$$M(\mathcal{C})=\bigcap_{\mathcal{C}\subset M\atop M\hspace{0.1cm}\text{mon.cl.}}M.$$

This is also by construction the smallest monotone class containing $\mathcal{C}$.
\end{rem}

\begin{thm}[Monotone Classes lemma]

Let $E$ be a topological space. If $\mathcal{C}\subset\mathcal{P}(E)$ is stable under finite intersection, i.e. for $A\in \mathcal{C}$ and $B\in\mathcal{C}\Longrightarrow A\cap B\in \mathcal{C}$, then 

$$\sigma(\mathcal{C})=M(\mathcal{C}).$$
\end{thm}

\begin{proof}

It is obvious that, $M(\mathcal{C})\subset\sigma(\mathcal{C})$ since a $\sigma$-Algebra is also a monotone class. Next we want to show that $M(\mathcal{C})$ is a $\sigma$-Algebra to conclude that $\sigma(\mathcal{C})\subset M(\mathcal{C})$ and hence then $M(\mathcal{C})$ contains $\mathcal{C}$, i.e. $\sigma(\mathcal{C})$. It is  not difficult to see that a monotone class, which is stable under finite intersection, is a $\sigma$-Algebra. Let us therefore show that $M(\mathcal{C})$ is stable under finite intersections. First, we fix $A\in\mathcal{C}$ and define

$$M_A:=\{B\in\mathcal{C}\mid A\cap B\in M(\mathcal{C})\}.$$

Then we get that $\mathcal{C}\in M_A$ since $\mathcal{C}$ is stable under finite intersections and obviously $E\in M_A$. We can also note that If $B,B'\in M_A$ and $B\subset B'$, with 

$$A\cap (B'\setminus B)=\underbrace{\underbrace{(A\cap B')}_{\in M(\mathcal{C})}\setminus \underbrace{(A\cap B)}_{\in M(\mathcal{C})}}_{\in M(\mathcal{C})}$$

then $B'\setminus B\in M_A$. Moreover, if $(B_n)_{n\in\N}\in M_A$ and $A\cap\bigcup_n\underbrace{(A\cap B_n)}_{\in M(\mathcal{C})}$ we get the implication

$$\bigcup_{n\geq 1}(A\cap B_n)\in M(\mathcal{C})\Longrightarrow \bigcup_{n\geq 1}B_n\in M_A.$$

since $(A\cap B_n)$ is increasing. Finally we can conclude the above facts. That means if $M_A$ is a monotone class containing $\mathcal{C}$, then $M_A=M(\mathcal{C})$, which shows that for all $A\in\mathcal{C}$ and $B\in M(\mathcal{C})$ we get $A\cap B\in M(\mathcal{C})$. We can now apply the same idea another time. Fix $B\in M(\mathcal{C})$ and define 

\[
M:=\{A\in M(\mathcal{C})\mid A\cap B\in M(\mathcal{C})\}.
\]

Now, from above, we get that $M$ is a monotone class, i.e. $M\subset M(\mathcal{C})$ and thus for all $A,B\in M(\mathcal{C})$ we get $A\cap B\in M(\mathcal{C})$. Hence it follows that $M(\mathcal{C})$ is stable under finite intersections and is therefore a $\sigma$-Algebra.
\end{proof}

\begin{cor}

Let $(E,\A)$ be a measurable space and let $\mu,\nu$ be two measures on $(E,\mathcal{A})$. Moreover, assume that there exists a family of subsets $\mathcal{C}$, which is stable under finite intersections, such that $\sigma(\mathcal{C})=\mathcal{A}$ and $\mu(A)=\nu(A)$ for all $A\in\mathcal{C}$. Then the following hold.

\begin{enumerate}[$(i)$]

\item{If $\mu(E)=\nu(E)<\infty$, then we get $\mu=\nu$.
}

\item{If there exists an increasing family $(E_n)_{n\in\N}$ with $E_n\in\mathcal{C}$ such that 

$$E=\bigcup_{n\in\N}E_n$$

and $\mu(E_n)=\nu(E_n)<\infty$, then it follows that $\mu=\nu$.
}
\end{enumerate}
\end{cor}

\begin{proof}

Let us first define the set $G:=\{A\in\A\mid \mu(A)=\nu(A)\}$. By assumption we get that $\mathcal{C}\subset G$. Moreover, we note that $G$ is a monotone class. Note at first that $E\in G$ by assumption since $\mu(E)=\nu(E)$. Now let $A,B\in G$ such that $A\subset B$ and since 

\[
\mu(B\setminus A)=\mu(B)-\mu(A)=\nu(B)-\nu(A)=\nu(B\setminus A),
\]

we get that $(B\setminus A)\in G$. Now let $(A_n)_{n\in\N}$ be an increasing sequence in $G$. Then the fact 

\[
\mu\left(\bigcup_{n\in\N}A_n\right)=\lim_{n\to\infty}\mu(A_n)=\lim_{n\to\infty}\nu(A_n)=\nu\left(\bigcup_{n\in\N}A_n\right)
\]

implies that $\bigcup_{n\in\N}A_n\in G$. Moreover, since $G$ is a monotone class containing $\mathcal{C}$, we get that $G$ contains $M(\mathcal{C})$. On the other hand we know that $\mathcal{C}$ is stable under finite intersections and therefore it follows that $M(\mathcal{C})=\sigma(\mathcal{C})$ and that $G=\sigma(\mathcal{C})=\A$. Now define for all $n\in\N$ and $A\in \A$ the two sequences 

\begin{align*}
\mu_n(A)&=\mu(A\cap E_n)\\
\nu_n(A)&=\nu(A\cap E_n)
\end{align*}

Now since $\mu(E)=\nu(E)$, we get the same for the sequence elements and obtain therefore that $\mu_n=\nu_n$. Moreover, for $A\in \A$, we have 

\[
\mu(A)=\lim_{n\to\infty}\mu(A\cap E_n)=\lim_{n\to\infty}\uparrow\nu(A\cap E_n)=\nu(A).
\]

Hence we get that $\mu=\nu$.
\end{proof}

\begin{rem}

There are several applications of this corollary. Let us emphasize a first one, by giving already a small introduction to the Lebesgue measure. Assume that $\lambda$ is a measure on the measurable space $(\R,\B(\R))$ such that $\lambda((a,b))=b-a$ for $a<b$ and let $\mathcal{C}$ be the class of intervals $E_n=(-n,n)$ for $n\geq 1$. With the corollary above, it follows that $\lambda$ is unique. We will call $\lambda$ the Lebesgue measure. A second application is that a finite measure on $(\R,\B(\R))$ is uniquely characterized, for $a\in\R$, by the values 

\[
\mu((-\infty,a])=\mu((-\infty,a)).
\]

\end{rem}

\chapter{Integration with respect to a positive measure}

\section{Integration for positive (nonnegative) functions}

In this chapter we will introduce the integral as a new concept in terms of a measure, which is a more general point of view instead of the Riemann integral, which has several drawbacks. Consider for example the function 

\[
\one_\mathbb{Q}(x):=\begin{cases}1,&x\in\mathbb{Q}\\0,&x\not\in\mathbb{Q}\end{cases}
\]

There is no way where we can say that this function would be Riemann integrable and thus we cant make sense of the integral 

\[
\int_0^1\one_\mathbb{Q}(x)dx
\]

in terms of the theory of Riemann integration. However, with the notion of an integral with respect to a positive measure we can also deal with such integrals as we will see. Let, in this chapter, $(E,\mathcal{A})$ be a measurable space and let $f(x)=\sum_{i=1}^n\alpha_i\one_{A_i}(x)$ be a simple function, where $A_i\in \A$ and $\alpha_i\in\R$ for all $1\leq  i\leq  n$ and $n\in\N$. Moreover, let us assume without loss of generality that 

$$\alpha_1<\alpha_2<...<\alpha_n.$$

Then $A_i=f^{-1}(\{\alpha_i\})\in\mathcal{A}$. Let us also denote by $\mu$ a positive measure on $(E,\A)$. Then we can define the integral for simple functions as follows.

\begin{defn}[Integral with respect to a measure]

Assume that $f$ takes values in $\mathbb{R}_+$ with  $0\leq \alpha_1<...<\alpha_n$. Then the Integral of $f$ with respect to $\mu$ is defined by

$$\int f d\mu=\sum_{i=1}^{n}\alpha_i\mu(A_i),$$

where we use the convention $0\cdot \infty=0$, in case that $\alpha_i=0$ and $\mu(A_i)=\infty$. Moreover, if $f(x)=\one_{A}(x)+0\cdot \one_{A^C}(x)$, for $A\in\A$, then

$$\int f d\mu=\mu(A)(+0\cdot \mu(A^C)).$$

\end{defn}

\begin{rem}
It is easy to obtain that by definition, if $\int fd\mu\in[0,\infty]$, then if $f=0$ we get that $\int fd\mu=0$. 
\end{rem}

\begin{rem}
This integral is well defined. Indeed, let $$f(x)=\sum_{i=1}^n\alpha_i\one_{A_i}(x)$$ be the canonical form of the simple function of $f$, i.e. $\alpha_i$ distinct, $A_i=f^{-1}(\{\alpha_i\})$. Then, by definition, it follows that 

$$\int fd\mu =\sum_{i=1}^n\alpha_i\mu(A_i).$$

On the other hand, we can also write $f$ as  

$$f(x)=\sum_{j=1}^m\beta_j\one_{B_j}(x),$$

where $(\beta_j)_{1\leq  j\leq  m}$ forms an $\mathcal{A}$-partition, with $\beta_j\geq 0$ and where the $\beta_j$'s are not necessarily distinct. We want to show that the integral is still the same. Note that for each $i\in\{1,...,n\}$, we get that $A_i$ is the disjoint union of the sets $B_j$ for which $\alpha_i=\beta_j$. Then the additivity property of the measure shows that

$$\mu(A_i)=\sum_{\{j\mid \beta_j=\alpha_j\}}\mu(B_j),$$

and therefore the integral doesn't depend on the representation of $f$.
\end{rem}

\begin{prop}
Let $f$ and $g$ be two simple, positive and measurable functions on $E$. Then

\begin{enumerate}[$(i)$]
\item{For $a,b\geq 0$ we get

\[
\int (af+bg)d\mu=a\int fd\mu+b\int gd\mu.
\]

}
\item{If $f\leq  g$, then 

\[
\int fd\mu\leq \int gd\mu.
\]
}
\end{enumerate}
\end{prop}

\begin{proof}

For $(i)$, let us first set

$$f=\sum_{i=1}^n\alpha_i\one_{A_i}\hspace{01cm}g=\sum_{k=1}^m\alpha_k'\one_{A_k'}.$$

Moreover, we can note that

$$A_i=\bigcup_{k=1}^m(A_i\cap A_k'),\hspace{1cm}A_k'=\bigcup_{i=1}^n(A_k'\cap A_i)$$

and therefore we can write

$$f=\sum_{j=1}^p\beta_j\one_{B_j}\hspace{1cm}g=\sum_{j=1}^p\gamma_j\one_{B_j},$$

where $(B_j)_{1\leq  j\leq  pj}$ is an $\mathcal{A}$-partition of $E$ obtained from a reordering of $(A_i\cap A_k')$. Thus we get 

$$\int fd\mu=\sum_{j=1}^p\beta_j\mu(B_j),\hspace{1cm}\int gd\mu=\sum_{j=1}^p\gamma_j\mu(B_j).$$

It follows that 

$$\int (af+bg)d\mu=\sum_{j=1}^p(a\beta_j+b\gamma_j)\mu(B_j)=a\sum_{j=1}^p\beta_j\mu(B_j)+b\sum_{j=1}^p\gamma_j\mu(B_j)=a \int f d\mu+b\int gd\mu.$$

For $(ii)$, we can write $g$ as $g=f+g-f$, where $f\geq 0$ and $g-f\geq 0.$ Then

$$\int g d\mu=\underbrace{\int(g-f)d\mu}_{\geq 0}+\int fd\mu\geq \int fd\mu.$$

\end{proof}

\begin{defn}[Integral for measurable maps]

Let $(E,\mathcal{A},\mu)$ be a measure space. Moreover, let us denote by $\mathcal{E}^+$ the set of all nonnegative simple functions and let $f:E\longrightarrow [0,\infty)$ be a measurable map. Then we can define the integral for $f$ to be given as 

$$\int fd\mu:=\sup_{h\leq f\atop h\in\mathcal{E}^+}\int hd\mu.$$

\end{defn}

\begin{rem}

Sometimes we have different notation for the same.

$$\int fd\mu=\int f(x)d\mu(x)=\int f(x)\mu(dx).$$

\end{rem}

\begin{prop}

Let $(E,\A,\mu)$ be a measure space and let $f,g:E\longrightarrow [0,\infty)$ be measurable maps. Then

\begin{enumerate}[$(i)$]
\item{If $f\leq g$, then 

$$\int fd\mu\leq \int gd\mu.$$
}

\item{If $\mu(\{x\in E\mid f(x)>0\})=0,$ then 

$$\int fd\mu=0.$$
}
\end{enumerate}
\end{prop}

\begin{proof}
Exercise.\footnote{It is enough to show this for simple functions.}
\end{proof}

\begin{thm}[Monotone convergence theorem]

Let $(f_n)_{n\in\N}$ be an increasing sequence of positive and measurable functions (with values in $[0,\infty)$), and let $f=\lim_{n\to \infty}\uparrow f_n$. Then 

$$\int fd\mu=\lim_{n\to\infty}\int f_nd\mu.$$
\end{thm}

\begin{proof}

We know that since $f_n\leq f$, we have $\int f_nd\mu\leq \int f d\mu$ and hence

$$\lim_{n\to\infty}\int f_nd\mu\leq \int fd\mu.$$

We need to show that 

$$\int fd\mu\leq \lim_{n\to\infty}\int f_nd\mu.$$

Let therefore $h=\sum_{i=1}^{m}\alpha_i\one_{A_i}$,
such that $h\leq f$, and let $a\in[0,1)$. Moreover let then  

$$E_n:=\{x\in E\mid ah(x)\leq f_n(x)\}.$$

We see immediately that $E_n$ is measurable\footnote{Prove this as an exercise.} and since $f_n\uparrow f$ as $n\to\infty$ and $a<1$, we see that 

$$E=\bigcup_{n\geq 1}E_n,\hspace{0.5cm}E_n\subset{E_{n+1}}.$$

We also note that $f_n\geq a\one_{E_n}h$. To see this, wee need to emphasize two cases.

\begin{enumerate}[$(i)$]

\item{If $x\in E_n$, then $a\one_{E_n}(x)h(x)=ah(x)$ and by definition of $E_n$ we get that$f_n(x)\geq ah(x).$
}
\item{If $x\not\in E_n$, then $a\one_{E_n}(x)h(x)=0$ and thus $f_n(x)\geq 0$ holds since $f_n$ is positive.
}
\end{enumerate}

Hence it follows that

$$\int f_nd\mu\geq \int a\one_{E_n}hd\mu=a\int\one_{E_n}hd\mu=a\sum_{i=1}^m\alpha_i\mu(A_i\cap E_n).$$

Since $E_n\uparrow E$ and $A_i\cap E_n\uparrow A_i$ as $n\to\infty$, we have $\mu(A_i\cap E_n)\uparrow \mu(A_i)$ as $n\to\infty$. Thus  we get that

$$\lim_{n\to\infty}\int f_nd\mu\geq a\sum_{i=1}^m\alpha_i\mu(A_i)=a\int hd\mu.$$

Note that the left side doesn't depend on $a$. Now we let $a\to 1$ and obtain then

$$\lim_{n\to\infty}\int f_nd\mu\geq \int hd\mu.$$

This is now true for every $h\in\mathcal{E}^+$ with $h\leq f$ and the left hand side doesn't depend on $h$. Therefore we have

$$\lim_{n\to\infty}\uparrow \int f_nd\mu\geq \sup_{h\in\mathcal{E}^+\atop h\leq f}\int hd\mu=\int fd\mu.$$

Let us recall that for any positive measurable map $f$ (values in $[0,\infty)$) there exists an increasing sequence $(f_n)_{n\in\N}$ of simple positive functions such that $f=\lim_{n\to\infty}f_n$.
\end{proof}

\begin{prop}

Let $(E,\A,\mu)$ be a measure space. Then

\begin{enumerate}[$(i)$]

\item{If $f$ and $g$ are two positive and measurable on $E$ and $a,b\in\mathbb{R}_+$, then 

$$\int (af+bg)d\mu=a\int fd\mu + b\int gd\mu.$$
}

\item{If $(f_n)_{n\in\N}$ is a sequence of measurable and positive functions on $E$, then 

$$\int\sum_n f_nd\mu=\sum_n\int f_nd\mu.$$
}
\end{enumerate}
\end{prop}

\begin{proof}

For $(i)$, take two positive sequences $(f_n)_{n\in\N}$ and $(g_n)_{n\in\N}\geq 0$ of simple functions such that $f_n\uparrow f$ and $g_n\uparrow g$ as $n\to\infty$. Then we get

$$\int (af+bg)d\mu\stackrel{MCV}{=}\lim_{n\to\infty}\int (af_n+bg_n)d\mu=\lim_{n\to\infty}\left(a\int f_n d\mu+b\int g_nd\mu\right)\stackrel{MCV}{=}a\int fd\mu +b\int gd\mu,$$

where $MCV$ stands for monotone convergence. Now $(ii)$ is an Immediate consequence of the monotone convergence theorem (MCV). Indeed, set

$$g_N=\sum_{n=1}^Nf_n,$$

with $g_N\geq 0$ and $\lim_{N\to\infty}\uparrow g_N=\sum_{n=1}^\infty f_n$. If we now apply the monotone convergence theorem to $(g_N)_{N\in\N}$, we get
\begin{multline*}
\int\sum_{n=1}^{\infty}f_n d\mu=\int \lim_{N\to\infty}g_Nd\mu \stackrel{MCV}{=}\lim_{N\to\infty}\int g_Nd\mu\\=\lim_{N\to\infty}\int\sum_{n=1}^Nf_nd\mu\stackrel{(i)}{=}\lim_{N\to\infty}\sum_{n=1}^N\int f_nd\mu=\sum_{n=1}^{\infty}\int f_n d\mu,
\end{multline*}

which proves the proposition.
\end{proof}

\begin{ex}
Let us consider the Dirac measure $\delta_x$ for $x\in E$ and let $f:E\longrightarrow \mathbb{R}_+$ be a measurable map. Then

$$\int fd\delta_x=f(x).$$
\end{ex}

\begin{ex}
Let us consider the space $\mathbb{N}$, the $\sigma$-Algebra $\mathcal{P}(\mathbb{N})$ and the counting measure on $\N$, which is the measure $\mu$ satisfying that for all $A\subset \N$ we get $\mu(A)=\vert A\vert$. Then for a measurable map $f:\mathbb{N}\longrightarrow \mathbb{R}_+$ we get

$$\int fd\mu=\sum_{n\in\mathbb{N}}f(n).$$ 

Moreover, we also get for every positive sequence $(a_{n,k})_{n\in\mathbb{N},k\in\mathbb{N}}\geq 0$ that 

$$\sum_{k\in\mathbb{N}}\sum_{n\in\mathbb{N}}a_{n,k}=\sum_{n\in\mathbb{N}}\sum_{k\in\mathbb{N}}a_{n,k}.$$
\end{ex}

\begin{cor}
\label{cor1}
Let $(E,\A)$ be a measurable space and let $f$ be a positive measurable map on $E$. Moreover, let us define for $A\in\mathcal{A}$ a map $\nu$ on $E$ by

$$\nu(A):=\int_{E}\one_{A}(x)f(x)d\mu=\int_Af(x)d\mu.$$

Then $\nu$ is a measure on $(E,\mathcal{A})$, which is called the measure with density $f$ with respect to $\mu$ and we write

$$\nu=f\circ \mu.$$
\end{cor}

\begin{rem}

It is clear that if $\mu(A)=0$ then $\nu(A)=0$ for all $A\in\A$.
\end{rem}

\begin{proof}[Proof of corollary \ref{cor1}]

First of all $\nu(\varnothing)=0$ follows from the given proposition 1.4. Now let $(A_n)_{n\in\mathbb{N}}\in\mathcal{A}$ be a sequence of measurable sets with $A_n\cap A_m=\varnothing$ for all $n\not=m$. Then 

\[
\nu\left(\bigcup_{n\geq 1}A_n\right)=\int\one_{\bigcup_{n\geq 1}A_n}(x)f(x)d\mu=\int\sum_{n\geq 1}\one_{A_n}(x)f(x)d\mu=\sum_{n\geq 1}\int \one_{A_n}(x)f(x)d\mu=\sum_{n\geq 1}\nu(A_n),
\]

which actually shows that $\nu$ is a measure on $E$.
\end{proof}

\begin{rem} 
We say that a property is true $\mu$-almost everywhere and we write $\mu$-a.e. (or a.e. if it is clear for which measure), if this property holds on a set $A\in\mathcal{A}$ with $\mu(A^C)=0$. For example, if $f$ and $g$ are both measurable maps on $E$, then $f=g$ a.e. means that

$$\mu(\{x\in E\mid f(x)\not=g(x)\})=0.$$
\end{rem}

\begin{prop}
Let $(E,\A,\mu)$ be a measure space and Let $f,g:E\longrightarrow \R$ be measurable and positive functions. Then the following hold.
\begin{enumerate}[$(i)$]
\item{$\mu(\{x\in E\mid f(x)>a\})\leq  \frac{1}{a}\int fd\mu$ for all $a>0$.
}
\item{If $\int fd\mu<\infty$, then $f<\infty$ a.e.
}
\item{$\int fd\mu=0$ if and only if $f=0$ a.e.
}
\item{If $f=g$ a.e., then $\int fd\mu=\int gd\mu$.
}
\end{enumerate}
\end{prop}

\begin{proof} We show each points seperately.
\begin{enumerate}[$(i)$]
\item{Consider the set $A_a=\{x\in E\mid f(x)>a\}$. Then we can observe that $f(x)\geq a\one_{A_a}(x)$ for all $x\in A_a$. Then it follows that 
\[
\int f(x)d\mu\geq \int a\one_{A_a}(x)d\mu
\]
and 
\[
\mu(A_a)\leq  \frac{1}{a}\int fd\mu.
\]
}
\item{For $n\geq 1$ consider the sets $A_n=\{x\in E\mid f(x)\geq n\}$ and $A_\infty=\{x\in E\mid f(x)=\infty\}=\bigcap_{n\in\N}A_n$. Then we get that
\[
\mu(A_\infty)=\lim_{n\to\infty}\mu(A_n),
\]
which implies that 
\[
\mu(A_n)\leq  \underbrace{\frac{1}{n}\int f(x)d\mu}_{\xrightarrow{n\to\infty} 0}<\infty.
\]
}
\item{We have already seen that if $f=0$ a.e., then $\int fd\mu=0$. Therefore, it is enough to show the other direction. Let $n>1$. Then 
\[
\mu(B_n)\leq  n\underbrace{\int fd\mu}_{0}=0.
\]
Thus we get that
\[
\mu(\{x\in E\mid f(x)>0\})=\mu\left(\bigcup_{n\geq 1}\right)\leq \sum_{n\geq 1}\mu(B_n)=0.
\]
}
\item{Let us introduce a special notation at this point. We write $f\lor g$ for $\sup(f,g)$ and $f\land g$ for $\inf(f,g)$. Now assume that $f=g$ a.e., which implies that $f\lor g=f\land g$ and hence we get
\[
\int (f\lor g)d\mu=\int (f\land g)d\mu+\int (f\lor g-f\land g)d\mu=\int (f\land g)d\mu.
\]
Because of the fact that
\begin{align*}
f\land g\leq  f\leq  f\lor g&\Longrightarrow \int fd\mu=\int(f\lor g)d\mu=\int (f\land g)d\mu\\
f\land g\leq  g\leq  f\lor g&\Longrightarrow \int gd\mu=\int (f\lor g)d\mu=\int (f\land g)d\mu,
\end{align*}
we finally get 
\[
\int fd\mu=\int gd\mu.
\]
}
\end{enumerate}
\end{proof}

\begin{thm}[Fatou's lemma]

Let $(f_n)_{n\in\N}$ be a sequence of real valued, measurable and positive functions on a measure space $(E,\A,\mu)$. Then

$$\int \liminf_{n\to\infty}f_nd\mu\leq \liminf_{n\to\infty}\int f_nd\mu.$$

\end{thm}

\begin{proof}

Recall that we actually have

$$\liminf_{n\to\infty} f_n=\lim_{n\to\infty}\nearrow \left(\inf_{k\geq n}f_n\right),\hspace{0.2cm}\left(=\sup_n\inf_{k\geq n}f_k\right).$$

From the monotone convergence theorem (MCV) we get

$$\int \liminf_{n\to\infty} f_nd\mu\stackrel{MCV}{=}\lim_{n\to\infty}\int \inf_{k\geq n}f_kd\mu.$$

Now for all $p\geq k$ we have

\begin{align*}
\inf_{n\geq k}f_n &\leq f_p\\
\Longrightarrow\int \inf_{n\geq k}f_nd\mu &\leq \int f_pd\mu\\
\Longrightarrow\int \inf_{n\geq k}f_nd\mu &\leq \inf_{p\geq k}\int f_pd\mu\\
\Longrightarrow\lim_{k\to\infty}\int\inf_{n\geq k}f_nd\mu &\leq \lim_{k\to\infty}\uparrow\inf_{p\geq k}\int f_pd\mu=\liminf_{n\to\infty}\int f_nd\mu,
\end{align*}
which proves the claim.
\end{proof}

\section{Integrable functions}
\begin{defn}[Integrable]
Let $(E,\A,\mu)$ be a measure space and let $f:E\longrightarrow \R$ be a measurable map. We say that $f$ is integrable with respect to $\mu$ if 

$$\int \vert f\vert d\mu<\infty.$$

Moreover, if $f$ is integrable, we define its integral to be given by
$$\int fd\mu =\int f^+d\mu -\int f^-d\mu.$$
\end{defn}

\begin{rem} It always holds that

$$\int f^{\pm}d\mu\leq \int \vert f\vert d\mu.$$

For instance 

$$\int_{\mathbb{R}^+}\frac{\sin(x)}{x}dx=\frac{\pi}{2},\hspace{0.3cm}\text{but}\hspace{0.3cm}\int_{\mathbb{R}^+}\left\vert\frac{\sin(x)}{x}\right\vert dx=\infty.$$

Moreover, we will denote by $\mathcal{L}^1(E,\mathcal{A},\mu)$ the space of integrable (and measurable) functions. Furthermore, we denote by $\mathcal{L}^1_+(E,\mathcal{A},\mu)$ the same space, but containing only positive functions.
\end{rem}

\begin{prop}
Let $(E,\A,\mu)$ be a measure space. Then the following hold.
\begin{enumerate}[$(i)$]
\item{If $f\in\mathcal{L}^1(E,\mathcal{A},\mu)$, then  $\left\vert\int fd\mu\right\vert \leq \int \vert f\vert d\mu.$
}
\item{$\mathcal{L}^1(E,\mathcal{A},\mu)$ is a vector space and the map $f\mapsto \int\vert f\vert d\mu$ is a linear form.
}
\item{If $f,g\in\mathcal{L}^1(E,\mathcal{A},\mu)$ and $f\leq g$, then $\int fd\mu\leq \int gd\mu.$
}
\item{If $f,g\in\mathcal{L}^1(E,\mathcal{A},\mu)$ and $f=g$ a.e., then  $\int fd\mu=\int gd\mu.$}
\end{enumerate}
\end{prop}

\begin{proof}We show each point seperately.
\begin{enumerate}[$(i)$]
\item{Let $f\in\mathcal{L}^1(E,\mathcal{A},\mu)$. Then we can obtain

\begin{align*}
\left\vert \int fd\mu\right\vert &=\left\vert \int f^+d\mu-\int f^-d\mu\right\vert\\
&\leq \left\vert \int f^+d\mu\right\vert +\left\vert\int f^-d\mu\right\vert\\
&=\int f^+d\mu +\int f^- d\mu\\
&=\int \left(f^++f^-\right)d\mu\\
&=\int \vert f\vert d\mu.
\end{align*}
}
\item{Indeed, $\mathcal{L}^1(E,\mathcal{A},\mu)$ is a linear space and for $f\in \mathcal{L}^1(E,\A,\mu)$ we get that the map
\begin{equation}
\label{map}
f\longmapsto\int fd\mu
\end{equation}
is a linear form. First we want to show that $\int afd\mu=a\int fd\mu$ for some $a\in\R$. Let therefore $a\in \R$ and consider the case where $a\geq 0$ and the one where $a<0$ as follows.
\begin{align*}
\underline{\underline{a\geq 0}}:\hspace{0.5cm}\int(af)d\mu&=\int(af)^+d\mu-\int (af)^-d\mu\\
&= a\int f^+d\mu -a\int f^-d\mu\\
&= a\int fd\mu.
\end{align*}

\begin{align*}
\underline{\underline{a<0}}:\hspace{0.5cm}\int (af)d\mu&=\int (af)^+d\mu-\int (af)^-d\mu\\
&=(-a)\int f^-d\mu -\int f^+d\mu\\
&=a\int fd\mu.
\end{align*}

If $f$ and $g$ are in $\mathcal{L}^1(E,\mathcal{A},\mu)$, the inequality $\vert f+g\vert\leq \vert f\vert +\vert g\vert$ implies that $(f+g)\in\mathcal{L}^1(E,\mathcal{A},\mu)$. One has to check the linearity of the map (\ref{map}). It is easy to obtain the following implications.
\begin{align*}
(f+g)^+-(f+g)^-&=(f+g)=f^+-f^-+g^+-g^-\\
\Longrightarrow(f+g)^++f^-+g^-&=(f+g)^-+f^++g^+\\
\Longrightarrow\int (f+g)^+d\mu+\int f^-d\mu+\int g^-d\mu&=\int (f+g)^-d\mu+\int f^+d\mu+\int g^+d\mu\\
\Longrightarrow\int (f+g)^+d\mu-\int(f+g)^-d\mu&=\int f^+d\mu-\int f^-d\mu+\int g^+d\mu-\int g^-d\mu\\
\Longrightarrow\int (f+g)d\mu&=\int fd\mu+\int gd\mu.
\end{align*}
}
\item{Let $f,g\in\mathcal{L}^1(E,\mathcal{A},\mu)$ with $f\leq g$. We can write $g=f+(g-f)$ and by assumption $g-f\geq 0$, which also implies that $\int (g-f)d\mu\geq 0$. Therefore we have
\[
\int gd\mu=\int fd\mu+\int(g-f)d\mu\geq \int f d\mu, 
\]
which proves the claim.
}

\item{Let $f,g\in\mathcal{L}^1(E,\A,\mu)$ with $f=g$ a.e., which also implies that $f^+=g^+$ a.e. and $f^-=g^-$ a.e. Thus it follows that 
\[
\int f^+d\mu=\int g^+d\mu\hspace{0.3cm}\text{and}\hspace{0.3cm}\int f^-d\mu=\int g^-d\mu.
\]
Therefore $\int fd\mu=\int gd\mu$.
}
\end{enumerate}
\end{proof}

\begin{exer}
Show that if $f,g\in\mathcal{L}^1(E,\mathcal{A},\mu)$ and $f\leq g$ a.e., then 
$$\int fd\mu\leq \int gd\mu.$$
\end{exer}

\subsection{Extension to the complex case}
Let $(E,\A,\mu)$ be a measure space and let $f:E\longrightarrow \C(\cong \mathbb{R}^2)$ be a measurable map, which basically means that $Re(f)$ and $Im(f)$ are both measurable maps. We say that $f$ is integrable if $Re(f)$ and $Im(f)$ are both integrable, or equivalently
\[
\int \vert f\vert d\mu<\infty,
\]
and we write $f\in\mathcal{L}^1_\C(E,\A,\mu)$. This is simply because of the fact that
\[
\int fd\mu=\int (Re(f)+Im(f))d\mu=\int Re(f)d\mu+\int Im(f)d\mu.
\]
Moreover, the properties $(i),(ii)$ and $(iv)$ of proposition also hold for the complex case.

\section{Lebesgue's dominated convergence theorem}

An important question of integration theory is whether the interchanging of limits and integrals is actually possible and under which condition on the sequence of maps on some measure space. We have already seen the monotone convergence theorem and Fatou's lemma, giving some simple conditions for such an interchange. Another way of achieving the same with different conditions is due to Lebesgue, who gave a more general condition, which is going to be discussed now. 

\begin{thm}[Lebesgue's dominated convergence theorem]
Let $(E,\A,\mu)$ be a measure space and let $(f_n)_{n\in\N}$ be a sequence of functions in $\mathcal{L}^1(E,\mathcal{A},\mu)$ (resp. $\mathcal{L}^1_{\mathbb{C}}(E,\mathcal{A},\mu)$). Moreover, assume that the following hold.
\begin{enumerate}[$(i)$]
\item{There exists a measurable map $f$ on $E$ with values in $\mathbb{R}$ (resp. $\mathbb{C}$) such that for all $x\in E$
$$\lim_{n\to\infty}f_n(x)=f(x).$$
}
\item{There exists a psotive and measurable map $g:E\longrightarrow \mathbb{R}$ such that 
$$\int gd\mu<\infty$$
and such that $\vert f_n\vert \leq g$ a.e. for all $n\in\N$.
}
\end{enumerate}
Then $f\in \mathcal{L}^1(E,\mathcal{A},\mu)$ (resp. $\mathcal{L}^1_{\mathbb{C}}(E,\mathcal{A},\mu))$ and we have
\[
\lim_{n\to\infty}\int\vert f_n-f\vert d\mu=0
\]
and hence 
\[
\lim_{n\to\infty}\int f_nd\mu=\int fd\mu.
\]
\end{thm}

\begin{proof}
Let us first assume some stronger assumptions.
\begin{enumerate}[$(i)$]
\item{$\lim_{n\to\infty}f_n(x)=f(x)$ for all $x\in E$.
}
\item{There exists a positive and measurable map $g:E\longrightarrow \mathbb{R}$ such that 

$$\int gd\mu<\infty$$

and $\vert f_n(x)\vert \geq g(x)$ for all $x\in E$. 
}
\end{enumerate}
Consider the general case where $(i)$ and $(ii)$ hold. Now define the set

$$A=\left\{x\in E\mid f_n(x)\xrightarrow{n\to\infty}f(x)\hspace{0.1cm}\text{and}\hspace{0.1cm}\forall n, \vert f_n(x)\vert\leq g(x)\right\}.$$

Then $\mu(A^C)=0$. Let us now apply the first part of the proof to 
\begin{align*}
\tilde{f}_n(x)&=\one_{A}(x)f_n(x)\\
\tilde{f}(x)&=\one_{A}(x)f(x).
\end{align*}
Thus we have that $f=\tilde{f}$ a.e. and $f_n=\tilde{f}_n$ a.e. Therefore, we get the following equations.

\begin{align*}
\int f_nd\mu&=\int \tilde f_nd\mu\\
\int fd\mu&=\int \tilde fd\mu\\
\lim_{n\to\infty}\int f_nd\mu&=\int fd\mu\\
\lim_{n\to\infty}\int\vert f-f_n\vert d\mu&=0
\end{align*}

Since now $\vert f\vert \leq g$ and $\int gd\mu<\infty$, we get that
$\int \vert f\vert d\mu<\infty$. Hence we have
$$\vert f-f_n\vert \leq 2g\hspace{0.3cm}\text{and}\hspace{0.3cm}\vert f-f_n\vert\xrightarrow{n\to\infty}0.$$
Now, applying Fatou's lemma, we can observe that
\[
\liminf_{n\to\infty}\int 2g-\vert f-f_n\vert d\mu\geq \int \liminf_{n\to\infty} 2g-\vert f-f_n\vert d\mu,
\]
and therefore 
$$\liminf_{n\to\infty} \int 2g-\vert f-f_n\vert d\mu\geq \int 2g d\mu,$$
which implies that
$$\int 2gd\mu-\limsup_{n\to\infty}\int \vert f-f_n\vert d\mu\geq \int 2g d\mu.$$
it is easy now to observe that $\limsup_{n\to\infty}\int \vert f-f_n\vert d\mu\leq  0,$ which basically implies that $\lim_{n\to\infty}\int \vert f-f_n\vert d\mu =0$ and thus we can finally deduce
$$\left\vert \int fd\mu-\int f_nd\mu\right\vert=\left\vert\int(f-f_n)d\mu\right\vert\leq  \int \vert f-f_n\vert d\mu\xrightarrow{n\to\infty}0$$
which simply means that
$$\lim_{n\to\infty}\int f_n d\mu=\int fd\mu.$$
\end{proof}

\section{Parameter Integrals}
In this section we want to consider the special case of an integral. Basically, we want to look at integrals of the form  
$$\int f(u,x)d\mu,$$
for some $u\in E$, which actually gives rise to a map 
\begin{align*}
F:E&\longrightarrow \R\\
u&\longmapsto F(u)=\int f(u,x)d\mu
\end{align*}
\begin{ex}[Gamma function]
Let us start with a special example of such a map. The Gamma function is defined as the map
\begin{align*}
\Gamma:\C&\longrightarrow \R\\
s&\longmapsto \Gamma(s)=\int_0^\infty t^{s-1}e^{-t}dt,
\end{align*}
The question is, whether this integral converges for all $s\in\C$. This is certainly not the case and only possible if $Re(s)>0$. An important functional equation is given by 
\[
\Gamma(n+1)=n!,
\]
where $n\in\N$. We will need this function later to describe the volume of the unit ball in $\R^d$.
\end{ex}
\begin{thm}
Let $(E,\A,\mu)$ be a measure space. Let $(U,d)$ be a metric space and let $f:U\times E\longrightarrow\Big|_\C^\R$ with $u_0\in U$. Moreover, assume that the following hold.
\begin{enumerate}[$(i)$]
\item{The map $x\mapsto f(u,x)$ is measurable for all $x\in U$.
}
\item{The map $u\mapsto f(u,x)$ is continuous at $u_0$ a.e.
}
\item{There exists a measurable function $g\in\mathcal{L}^1(E,\mathcal{A},\mu)$ such that for all $u\in U$
$$\vert f(u,x)\vert \leq  g(x)\hspace{0.3cm}\text{-a.e.}$$ 
}
\end{enumerate}
Then the map $F(u)=\int_Ef(u,x)d\mu$ is well defined and continuous at $u_0$.
\end{thm}

\begin{proof}
From $(iii)$ follows that the map $x\mapsto f(u,x)$ is integrable for every $u\in U$, and so $F(u)$ is well defined. Take a sequence $(u_n)_{n\in\N}\in U$ such that $u_n\xrightarrow{n\to \infty} u_0$, which basically means that $d(u_n,u_0)\xrightarrow{n\to \infty} 0$. Then by continuity from $(ii)$, we get that $f(u_n,x)\xrightarrow{n\to \infty} f(u_0,x)$ a.e. and from $(iii)$ we can apply Lebesgue's dominated convergence theorem to obtain that
$$\lim_{n\to\infty}F(u_n)=\lim_{n\to\infty}\int f(u_n,x)d\mu=\int f(u_0,x)d\mu=F(u_0).$$
\end{proof}
\begin{rem}
Observe that $F(u)$ is continuous if $F$ is continuous at every point $u\in U$.
\end{rem}

\begin{ex}[Fourier analysis] we want to give several examples of Fourier analysis at this point.
\begin{enumerate}
\item{Let $\mu$ be a measure on $(\mathbb{R},\mathcal{B}(\mathbb{R}))$ such that $\mu(\{x\})=0$ for all $x\in\R$. Moreover, let $\phi\in\mathcal{L}^1(\R,\B(\R),\mu)$. Then the map
\[
F(u)=\int_\R\one_{(-\infty,u]}(x)\phi(x)d\mu=\int_{(-\infty,u]}\phi(x)d\mu
\]
is continuous. Here we have $f(u,x)=\one_{(-\infty,u]}(x)\phi(x)$. The map $u\mapsto f(u,x)$ is continuous at $u_0$ for all $x\in\R\setminus\{u_0\}$ for some $u_0\in\R$. But $\mu(\{u_0\})=0$, which implies that the map $u\mapsto f(u,x)$ is a.e. continuous at $u_0$ with $\vert f(u,x)\vert\leq  \phi(x)\vert$ and $\vert\phi\vert$ is integrable by assumption.
}
\item{Consider now the Lebesgue measure $\lambda$ and $\phi\in\mathcal{L}^1(\mathbb{R},\mathcal{B}(\mathbb{R}),\lambda)$. Define moreover

$$\hat{\phi}(u)=\int_{\mathbb{R}}e^{iux}\phi(x)d\lambda,$$

which is called the \emph{Fourier-transform} of $\phi$. The map $u\mapsto e^{iux}\phi(x)$ is actually continuous for all $x\in \R$, which implies that $\hat{\phi}(u)$ is continuous at any $u\in\mathbb{R}$.
}
\item{Let again $\phi\in\mathcal{L}^1(\mathbb{R},\mathcal{B}(\mathbb{R}),\lambda)$ and $h:\mathbb{R}\longrightarrow \mathbb{R}$ be a continuous and bounded map. The \emph{convolution} of $h$ and $\phi$ is given by

$$(h*\phi)(u)=\int_{\mathbb{R}}h(u-x)\phi(x)d\lambda(x).$$

The map $(h*\phi)$ is continuous at any $u\in\mathbb{R}$. Moreover, the map $u\mapsto h(u-x)\phi(x)$ is continuous for all $x\in\mathbb{R}$. Since $h$ is bounded, there exists a constant $k>0$, such that $\vert h(x)\vert\leq  k$ for all $x\in\R$.
}
\end{enumerate}
\end{ex}

\section{Differentiation of Parameter Integrals}

\begin{thm}[Differentiation of Integrals]
\label{diff}
Let $I\in \R$ be a real interval and $(E,\A,\mu)$ a measure space. Let $f:I\times E\longrightarrow \Big|^\R_\C$ and let $u_0\in I$. Moreover, assume that the following hold.
\begin{enumerate}[$(i)$]
\item{The map $x\mapsto f(u,x)$ is in $\mathcal{L}^1(E,\A,\mu)$ for all $u\in I$.
}
\item{The map $u\mapsto f(u,x)$ is a.e. differentiable at $u_0$  with derivation denoted by $\partial_uf(u_0,x)$.
}
\item{There exists a map $g\in\mathcal{L}^1(E,\A,\mu)$ such that for all $u\in I$ 
\[
\vert f(u,x)-f(u_0,x)\vert\leq  g(x)\vert u-u_0\vert. 
\]
}
\end{enumerate}
Then the map $F(u)=\int_Ef(u,x)d\mu$ is differentiable at $u_0$ and its derivative is given by 
\[
F'(u)=\int_E\partial_uf(u_0,x)d\mu.
\]
\end{thm}

\begin{rem}
It is often useful to replace $(ii)$ and $(iii)$ with the following points respectively.
\begin{enumerate}[$(i)$]

\item{The map $u\mapsto f(u,x)$ is a.e. differentiable at any point in $I$.
}
\item{There exists a map $g\in\mathcal{L}^1(E,\mathcal{A},\mu)$ such that $\vert\partial_uf(u,x)\vert\leq  g(x)$ a.e. for all $u\in I$.
}
\end{enumerate}
Moreover, if $f$ is differentiable on a interval $[a,b]$, there exists a $\theta_{a,b}$, by the mean value theorem, such that
\[
f'(\theta_{a,b})=\frac{f(b)-f(a)}{b-a}.
\]
If $f'$ is also bounded on $[a,b]$, i.e. there is a $M$ such that $\vert f'(x)\vert\leq  M$, then 
\[
\left\vert\frac{f(b)-f(a)}{b-a}\right\vert\leq  M.
\]
\end{rem}

\begin{proof}[Proof of Theorem \ref{diff}]
Let $(u_n)_{n\geq 1}$ be a sequence in $I$ such that $u_n\xrightarrow{n\to\infty}u_0$ and assume that $u_n\not=u_0$ for all $n\geq 1$. Now define the sequence 
\[
\phi_n(x).=\frac{f(u_n,x)-f(u_0,x)}{u_n-u_0}.
\]
Then we can obtain that 
\[
\lim_{n\to\infty}\phi_n(x)=\partial_uf(u_0,x)\hspace{0.3cm}\text{a.e.}
\]
Now $(iii)$ allows us to use Lebesgue's dominated convergence theorem. Hence we get
\[
\lim_{n\to\infty}\frac{F(u_n)-F(u_0)}{u_n-u_0}=\lim_{n\to\infty}\int_E\phi_n(x)d\mu=\int_E\partial_uf(u_0,x)d\mu.
\]
\end{proof}

\begin{ex}[Fourier analysis] Let us give the following examples of Fourier analysis.
\begin{enumerate}[$(i)$]
\item{Let $\phi\in \mathcal{L}^1(\mathbb{R},\mathcal{B}(\mathbb{R}),\lambda)$ be an integrable map such that 

$$\int_\mathbb{R}\vert x\phi(x)\vert d\lambda<\infty.$$

Then its Fourier-transform is given by

$$\hat{\phi}(u)=\int_\mathbb{R}\underbrace{e^{iux}\phi(x)}_{f(u,x)}d\lambda,$$

which is differentiable and the derivative of it is then

$$\hat{\phi}'(u)=i\int_{\mathbb{R}}e^{iux}x\phi(x)d\lambda.$$

Therefore we can write
$$\vert\partial_uf(u,x)\vert=\vert ie^{iux}x\phi(x)\vert=\vert x\phi(x)\vert.$$
}

\item{Let $\phi\in \mathcal{L}^1(\mathbb{R},\mathcal{B}(\mathbb{R}),\lambda)$ and $h:\mathbb{R}\longrightarrow \mathbb{R}$ be a bounded $C^1$-map with bounded derivative $h'$. Then the convolution $(h*\phi)$ is differentiable and its derivative is given by $(h*\phi)'=h'*\phi'$. Recall that the convolution $(h*\phi)$ is given by 

$$(h*\phi)(u)=\int_{\mathbb{R}}\underbrace{h(u-x)\phi(x)}_{f(u,x)}d\lambda.$$

Moreover, $\vert h(u-x)\phi(x)\vert\leq M\vert\phi(x)\vert$, where $M$ is such that $\vert h(x)\vert\leq M$ and $\vert h'(x)\vert\leq M$ for all $x\in\R$. Then

$$\partial_uf(u,x)=h'(u-x)\phi(x),\hspace{0.4cm}\text{and}\hspace{0.4cm}\vert\partial_uf(u,x)\vert \leq M\vert\phi(x)\vert.$$
}
\end{enumerate}
\end{ex}

\begin{exer}
Let $\mu$ be a measure on $(\mathbb{R},\mathcal{B}(\mathbb{R}))$ such that for all $x\in\mathbb{R}$, $\mu(\{x\})=0$. Moreover, let $\phi\in\mathcal{L}^1(\mathbb{R},\mathcal{B}(\mathbb{R}),\mu)$ such that

$$\int_\mathbb{R}\vert x\phi(x)\vert d\mu<\infty.$$

Furthermore, for $u\in \mathbb{R}$, define 

$$F(u)=\int_{\mathbb{R}}(u-x)^+\phi(x)d\mu.$$

Show that $F$ is differentiable and that

$$F'(u)=\int_{(-\infty,u]}\phi(x)d\mu.$$
\end{exer}

\end{appendix}

\end{document}